\documentclass[reqno]{memo-l}

\usepackage[colorlinks=true, pdfstartview=FitV, linkcolor=blue,
citecolor=blue]{hyperref}

\usepackage{amsmath,amssymb}
\usepackage{bbm}
\usepackage{xcolor}
\usepackage{tikz-cd}
\usepackage{mathrsfs}
\usepackage[refpage]{nomencl}
\makenomenclature

\newcommand{\cp}{(G,H,\cL)}

\usepackage{imakeidx}
\makeindex

\usepackage{scalerel,stackengine}

\stackMath
\newcommand\reallywidehat[1]{%
\savestack{\tmpbox}{\stretchto{%
  \scaleto{%
    \scalerel*[\widthof{\ensuremath{#1}}]{\kern-.6pt\bigwedge\kern-.6pt}%
    {\rule[-\textheight/2]{1ex}{\textheight}}
  }{\textheight}%
}{0.5ex}}%
\stackon[1pt]{#1}{\tmpbox}%
}

\newcommand\reallywidecheck[1]{%
\savestack{\tmpbox}{\stretchto{%
  \scaleto{%
    \scalerel*[\widthof{\ensuremath{#1}}]{\kern-.6pt\bigwedge\kern-.6pt}%
    {\rule[-\textheight/2]{1ex}{\textheight}}
  }{\textheight}%
}{0.5ex}}%
\stackon[1pt]{#1}{\scalebox{-1}{\tmpbox}}%
}

\DeclareFontFamily{U}{mathx}{}
\DeclareFontShape{U}{mathx}{m}{n}{<-> mathx10}{}
\DeclareSymbolFont{mathx}{U}{mathx}{m}{n}
\DeclareMathAccent{\widehat}{0}{mathx}{"70}
\DeclareMathAccent{\widecheck}{0}{mathx}{"71}

\newcommand{\ver}{{\vert\kern-0.25ex\vert\kern-0.25ex\vert }}

\allowdisplaybreaks

\newcommand{\fS}{{\mathscr S}}
\newcommand{\ZZ}{{\mathbb Z}}
\newcommand{\QQ}{{\mathbb Q}}
\newcommand{\RR}{{\mathbb R}}
\newcommand{\NN}{{\mathbb N}}
\newcommand{\CC}{{\mathbb C}}

\newcommand{\BB}{{\mathbb B}}
\newcommand{\EE}{{\mathbb E}}
\newcommand{\XX}{{\mathbb X}}

\newcommand{\TT}{{\mathbb T}}
\newcommand{\SSS}{{\mathbb S}}
\newcommand{\AAA}{{\mathbb A}}
\newcommand{\cA}{{\mathcal A}}
\newcommand{\cB}{{\mathcal B}}
\newcommand{\cN}{{\mathcal N}}
\newcommand{\cH}{{\mathcal H}}
\newcommand{\im}{{\mathrm{i}}}
\newcommand{\AP}{{\mathsf{AP}}}
\newcommand{\vL}{\varLambda}
\newcommand{\vG}{\varGamma}
\newcommand{\supp}{{\mbox{supp}}}
\newcommand{\dens}{{\operatorname{dens}}}

\newcommand{\mc}{\mathcal}

\newcommand{\card}{\mbox{\rm card}}
\newcommand{\dd}{\mbox{\rm d}}

\newcommand{\mm}{\mathfrak{m}}
\newcommand{\re}{\operatorname{Re}}
\newcommand{\imm}{\operatorname{Im}}
\newcommand{\eps}{\varepsilon}
\newcommand{\cM}{{\mathcal M}}
\newcommand{\cL}{{\mathcal L}}
\newcommand{\cF}{{\mathcal F}}

\newcommand{\be}{\mathsf{B}}
\newcommand{\we}{\mathsf{W}}

\newcommand{\BL}{B\hspace*{-1pt}L}

\newcommand{\WL}{W\hspace*{-1pt}L}

\newcommand{\sap}{Ap\text{-}B\hspace*{-1.1pt}ohr}

\newcommand{\SAP}{\mathcal{AP}\text{-}{Strong} }

\newcommand{\weak}{Ap\text{-}W\hspace*{-1.1pt}eak}

\newcommand{\WAP}{\mathcal{AP}\text{-}W\hspace*{-1pt}{eak}}

\newcommand{\Wap}{\mathcal{AP}\text{-}W\hspace*{-1pt}{eyl}}

\newcommand{\wap}{Ap\text{-}W\hspace*{-1pt}eyl}

\newcommand{\Bap}{\mathcal{AP}\text{-}{Bes}}

\newcommand{\bap}{Ap\text{-}B\hspace*{-1pt}es_{\cA}}

\newcommand{\bappe}{Ap\text{-}B\hspace*{-1pt}es^p_{\cA}(G)/\equiv}

\newcommand{\bapte}{Ap\text{-}B\hspace*{-1pt}es^2_{\cA}(G)/\equiv}

\newcommand{\Map}{\mathcal{AP}\text{-}M\mathsf{ean}}

\newcommand{\mean}{Ap\text{-}M\hspace*{-1pt}ean}

\newcommand{\Mean}{\mathsf{Ap\text{-}Mean}} 

\newcommand{\amen}{\mbox{Amen}(G)}

\newcommand{\oplam}{\mbox{\Large $\curlywedge$}}
\newcommand{\exend}{\hfill $\Diamond$}
\newcommand{\lm}{\ensuremath{\lambda\!\!\!\lambda}}

\newcommand{\smoplam}{\mbox{$\curlywedge$}}

\newcommand{\Cu}{C_{\mathsf{u}}}
\newcommand{\Cc}{C_{\mathsf{c}}}
\newcommand{\Cz}{C^{}_{0}}

\newcommand{\lb}{\text{\textlquill} }
\newcommand{\rb}{\text{\textrquill} }

\index{g-a-p|see{generalized almost periodic}}
\index{TMDS|see{dynamical~system~on~the~translation~bounded~measures}}
\index{CPS|see{cut~and~project~scheme}}
\index{CPP|see{consistent~phase~property}}

\newtheorem{theorem}{Theorem}[chapter]
\newtheorem{lemma}[theorem]{Lemma}
\newtheorem{prop}[theorem]{Proposition}
\newtheorem{coro}[theorem]{Corollary}

\theoremstyle{definition}
\newtheorem{definition}[theorem]{Definition}
\newtheorem{example}[theorem]{Example}

\theoremstyle{remark}
\newtheorem{remark}[theorem]{Remark}

\numberwithin{section}{chapter}
\numberwithin{equation}{chapter}

\begin{document}

\frontmatter

\title{Pure Point Diffraction and Mean, Besicovitch and Weyl Almost Periodicity}

\author{Daniel Lenz}
\address{Mathematisches Institut, Friedrich Schiller Universit\"at Jena, 07743 Jena, Germany}
\email{daniel.lenz@uni-jena.de}
\urladdr{http://www.analysis-lenz.uni-jena.de}

\author{Timo Spindeler}
\address{Fakult\"at f\"ur Mathematik, Universit\"at Bielefeld, \newline
\hspace*{\parindent}Postfach 100131, 33501 Bielefeld, Germany}
\email{tspindel@math.uni-bielefeld.de}
\thanks{The second author was supported in part by DFG with grant \#415818660.}

\author{Nicolae Strungaru}
\address{Department of Mathematics and Statistics, MacEwan University \\
10700 -- 104 Avenue, Edmonton, AB, T5J 4S2, Canada\\
and \\
Institute of Mathematics ``Simon Stoilow''\\
Bucharest, Romania}
\email{strungarun@macewan.ca}
\urladdr{https://sites.google.com/macewan.ca/nicolae-strungaru/home}
\thanks{The last author was supported in part by NSERC with grants \#2020-00038 and \#2024-04853.}


\dedicatory{We dedicate this work to Robert V. Moody on the occasion of his $80^{\mbox{th}}$ birthday.}

\subjclass[2020]{37A30, 43A05, 43A07, 43A25, 43A60, 22D10}
\keywords{\texttt{Mean almost periodicity, Besicovitch almost periodicity,
Weyl almost periodicity, pure point spectrum}}

\dedicatory{We dedicate this work to Robert V. Moody on the occasion of his $80^{\mbox{th}}$ birthday.}

\begin{abstract}
We show that a translation bounded measure has pure point
diffraction if and only if it is mean almost periodic. We then go on
and show that  a translation bounded measure has pure point diffraction and
satisfies the so called consistent
phase property if and only if it is Besicovitch almost periodic.
Finally, we show that a translation bounded measure has pure point diffraction and
satisfies the consistent phase property independent of the underlying van Hove sequence if and only
if it is Weyl almost periodic. These results solve fundamental
issues in the theory of pure point diffraction and answer questions
of Lagarias.
\end{abstract}

\maketitle

\tableofcontents

%
%
%


\chapter*{Introduction}
This book deals with the harmonic analysis behind pure-point
diffraction. This topic has received  substantial attention  since
the discovery of quasicrystals some forty years ago. Indeed,
the pivotal article \cite{BT} by Bombieri and Taylor written immediately
after the discovery of quasicrystals has in its title the
question `Which distributions of matter diffract?'. Of course, in
order to answer this question one needs to be more specific: How is
the distribution of matter modeled? What is meant by `diffract'?  In
this section, we will discuss this and  present our results. For
further details and precise definitions of certain concepts appearing
within our discussion, we refer the reader to Chapter~\ref{sec-key-player}.

As has become the custom in the last two decades, distribution of
matter is modeled by a measure in Euclidean space. This measure
should satisfy a uniform boundedness condition, known as translation
boundedness. This setting covers both Delone sets and bounded
densities. In fact, as far as our investigation here is concerned,
there is no reason to restrict to the Euclidean space. Instead we will
from now on consider translation bounded measures on a locally
compact, $\sigma$-compact Abelian group $G$. The dual group of $G$,
i.e. the set of all continuous group homomorphisms from $G$ into the
circle, is denoted by $\widehat{G}$.

According to the  mathematical setup developed by Hof  \cite{Hof1,Hof2}
(dealing  with the Euclidean case) and extended by Schlottmann
\cite{Martin2} (considering the group case), diffraction comes
about after one fixes a van Hove sequence $\cA = (A_n)$ of subsets
of the group $G$. Such a sequence is characterized by having the
boundary of its members become arbitrarily small compared to the
volume for large $n$.  The \textit{Eberlein (average) convolution} of the measure $\mu$
with the complex conjugate of its reflection along the sequence
$\cA$  is then known as the \textit{autocorrelation} of $\mu$ and denoted by
$\gamma_{\cA}$ (provided  it exists). This autocorrelation measure is
positive definite and, hence, possesses a Fourier transform. This
Fourier transform is known as the \textit{diffraction measure} of $\mu$ along
$\cA$ and is denoted by $\reallywidehat{\gamma_{\cA}}$.  It is this
diffraction measure that models the outcome of diffraction
experiments. Models for quasicrystals are
distinguished by the (pure) point nature of $\reallywidehat{\gamma_{\cA}}$.
Hence, in our context the natural first question is the following:

\medskip

\textbf{Question 1 (Characterization of pure-point diffraction).}
Let a van Hove sequence $\cA$ be given, and let $\mu$ be a measure
with autocorrelation $\gamma_{\cA}$. Under which conditions on $\mu$
is the diffraction measure $\reallywidehat{\gamma_{\cA}}$ a pure-point
measure?

\medskip

This is a long standing problem and clearly a most relevant question
in our context. However, it only deals with a partial aspect of the
situation as the diffraction measure only gives information on the
diffraction amplitudes. It does not contain any information on the
phases.  The real issue of diffraction concerns the phases.
Accordingly, the topic of phases is a central  focus of Lagarias'
article \cite{LAG} on the problem of diffraction.

The paper \cite{LAG} has been fairly influential. In particular,
various works in recent years have been devoted to answer questions
and treat problems from this article. This includes the work of Lev
and Olevski \cite{LevOle} on Poisson summation type formulae and
generalizations of Cordoba's theorem, the work of Favorov \cite{Fav}
on the failure of certain such generalizations in dimension bigger
than one, and the work of Kellendonk and Sadun \cite{KS} and
Kellendonk and Lenz \cite{KL} on the existence of sets with pure-point
diffraction without finite local complexity or Meyer property.

Following Lagarias\footnote{We use notation slightly different from
the notation in \cite{LAG}. Note also that the setting of \cite{LAG}
is restricted  to Delone sets in Euclidean space, and \cite{LAG}
assumes that the autocorrelation exists for any van Hove sequence,
whereas here we just assume existence along one fixed  van Hove
sequence.},  we look at the following \textbf{problem}: \label{problem-lagarias}
Consider a measure $\mu$ with pure-point diffraction supported on
the set $E\subset \widehat{G}$. How can one associate complex numbers $a_\chi\in \CC$, $\chi \in E$,
such that the Fourier transform of $\mu$ formally equals $\sum_{\chi \in E} a_\chi
\delta_\chi$ and the consistent phase property
\begin{equation}
 \widehat{\gamma} = \sum_{\chi \in E} |a_\chi|^2 \,
\delta_\chi  \tag{CPP}
\end{equation}
holds? As already discussed in \cite{LAG}, when dealing with this question, one first has to deal with what is meant by the
Fourier transform of $\mu$ being formally equal to $\sum_{\chi \in
E} a_\chi \delta_\chi$. Here, we take the point of view that this together with (CPP)
means that, for each $\chi\in \widehat{G}$, the Fourier--Bohr
coefficient
\[
a_{\chi}^{\cA}(\mu)=\lim_{n\to \infty} \frac{1}{|A_n|} \int_{A_n} \overline{\chi (t)}\, \dd\mu(t)
\]
of $\mu$ exists along the sequence $\cA$  and the consistent phase property holds with
\[
E = \{ \chi : a_{\chi}^{\cA} (\mu) \neq 0\} \mbox{ and } a_\chi = a_{\chi}^{\cA}(\mu) \mbox{  for each }  \chi \in E \,.
\]
Given this,  the problem
can be stated in a precise mathematical form as our second question:

\medskip

\textbf{Question 2.} Let a van Hove sequence $\cA$
be given. When does a measure $\mu$ satisfy the following three
properties?

\smallskip

\begin{itemize}
\item[(P1)]  The autocorrelation $\gamma_{\cA}$ of $\mu$ exists along $\cA$ and $\reallywidehat{\gamma_{\cA}}$ is a pure-point measure.

\item[(P2)] For each $\chi \in \widehat{G}$ the Fourier--Bohr coefficient $a_{\chi}^\cA (\mu)$  of $\mu$  exists along $\cA$.

\item[(P3)]  The consistent phase property
\[
\reallywidehat{\gamma_{\cA}} = \sum_{\chi \in \widehat{G}} |a_\chi^{\cA} (\mu)|^2 \,
\delta_\chi
\]
holds.
\end{itemize}

\medskip

So far, everything is developed with respect to a fixed van Hove
sequence. However, it is natural to aim for independence of the  van
Hove sequence. This leads us to the third question:

\medskip

\textbf{Question 3.} When does a measure
$\mu$ solve Question 2 for every van Hove sequence $\cA$ with
Fourier--Bohr coefficients and diffraction independent of the actual
van Hove sequence?

\medskip

In our work we provide complete answers to all three questions in
terms of  almost periodicity properties of $\mu$. Our main results
can be stated as follows:

\medskip

\textbf{Result 1 }(Theorem \ref{single element}). Let $\mu$ be
translation bounded with autocorrelation $\gamma_{\cA}$ along $\cA$.
Then, $\reallywidehat{\gamma_{\cA}}$ is pure point if and only if $\mu$ is
mean almost periodic.

\medskip

\textbf{Result 2 } (Theorem \ref{Bap and BT}).
 Let a van Hove sequence $\cA$ be given. Then,
the translation bounded measure  $\mu$ provides a solution to  Question 2
(along  $\cA$) if and only if $\mu$ is Besicovitch almost periodic.

\medskip

\textbf{ Result 3 }(Theorem \ref{theorem-uniform-phase-problem}).
The translation bounded measure $\mu$ provides a solution to  Question 3 if and only if $\mu$ is Weyl
almost periodic.

\medskip

Result 1 solves a long standing open problem with some partial
results obtained earlier. For Delone sets in $\RR^d$, a
characterization of pure-point diffraction has been given by
Gou\'{e}r\'{e} in \cite{Gouere-1}. As we will discuss below, his
characterization is just an alternative description of mean almost
periodicity for  Delone sets. Thus, his result is a special case of
our result (see Theorem \ref{gou}). For measures  supported inside
Meyer sets  a  condition for pure-point diffraction is
given by Baake and Moody in \cite[Theorem~5]{bm}. Here, again, we can show that
their condition actually is a description of mean almost
periodicity in the context of measures with Meyer set support
(see Theorem~\ref{ba}). In fact, as our proof shows, their condition implies
mean almost periodicity for all translation bounded measures. The Meyer condition
is only necessary when establishing that mean almost periodicity (and hence pure point diffraction)
implies their condition. At the same time, our result can be seen to
imply a result of Solomyak \cite{SOL}. We also note that in
retrospect one may also see how  mean almost periodicity is
implicitly present in Solomyak's dealing with pure-point
(diffraction) spectrum for primitive substitutions in \cite{SOL1}.

Result 2  and Result 3 are new. They settle
fundamental issues as witnessed by the mentioned article  of
Lagarias \cite{LAG}. Indeed, it has already been discussed how that
article focuses on the questions solved by Result 2 and Result
3. Moreover, the discussion in that article suggests to tackle the
problem via suitable notions of almost periodicity.  To be more
specific, we need some more notation. A \textit{Patterson set} in the sense
of \cite{LAG} is a Delone set in Euclidean space such that its
autocorrelation exists for any van Hove sequence, is independent of
the van Hove sequence, and has a pure-point
measure as its Fourier transform. Let now  $\mathcal{B}$ be a suitable vector space  of
almost periodic functions in Euclidean space satisfying three
natural additional assumptions via a Parseval type condition, a
Riesz--Fischer property and translation invariance. Then, Lagarias
calls a Delone set $\vL$ in Euclidean space a
$\mathcal{B}$-quasicrystal or a $\mathcal{B}$-Besicovitch almost
periodic set if
\[
\sum_{x\in \vL} \varphi (\cdot - x) \in \mathcal{B}
\]
holds  for each  infinitely many times differentiable function  $\varphi$
with compact support. In the introduction to his article, Lagarias
writes (p.~64): `\textit{...it remains to determine a good class
$\mathcal{B}$ that gives a reasonable theory.}' and further on
`\textit{It is natural to hope that a suitable class of
$\mathcal{B}$-Besicovitch almost periodic sets will all be Patterson
sets and have the consistent phase property given in (3.9), but this
is an open problem.}' Now, our results can clearly be understood to
answer these issues. In fact, our result specifically can be seen as
answers to Problems 4.6, 4.7 and 4.8 mentioned in the problem
section of \cite{LAG}. This deserves some further discussion:
Problem 4.6 asks for a class $\mathcal{B}$ of almost periodic
functions on Euclidean space such that their $\mathcal{B}$-quasicrystals satisfy:

\begin{itemize}
\item[(a)] each $\mathcal{B}$-quasicrystal satisfies (CPP).
\item[(b)] any Patterson set coming from a cut-and-project scheme (CPS) is a
$\mathcal{B}$-quasicrystal.
\item[(c)] any self-replicating Delone set which is a Patterson set is a
$\mathcal{B}$-quasicrystal.
\end{itemize}

Our results show that the choice $\mathcal{B}$ as the $2$-Besicovitch almost
periodic functions provides a solution to Problem 4.6:
This choice of $\mathcal{B}$ entails that every
$\mathcal{B}$-quasicrystal is a Besicovitch almost periodic measure
and, by Result 2, each such measure satisfies (CPP). Moreover, by
Result 2 again, each Patterson set satisfying (CPP) belongs to
$\mathcal{B}$. Hence, (b) and (c) are satisfied\footnote{The article
\cite{LAG} does not completely specify what is meant by Patterson
set coming from a cut-and-project scheme. We understand this to mean
regular model sets.}.

Problem 4.7 asks whether every $\mathcal{B}$-quasicrystal is a
Patterson set. Now, this is not the case for the  choice of
$\mathcal{B}$ as $2$-Besicovitch almost
periodic functions. The
reason is that in Result 2 we do not obtain existence of the
autocorrelation along any van Hove sequence (but just along one fixed
van Hove sequence). So, our Result 2 solves only a weakened version
of Problem 4.7. On the other hand, our Result 3 implies existence of
the autocorrelation along any van Hove sequence. So, Problem 4.7 is
solved if one takes as $\mathcal{B}$ the Weyl almost periodic
functions.

Strictly speaking, however, the Weyl almost periodic functions do
not qualify as a $\mathcal{B}$-class in the sense of Lagarias as
they do not satisfy the Riesz--Fischer property. On the other hand,
we can show in Chapter~\ref{sec:unavoidable} that any
$\mathcal{B}$-class satisfying the Riesz--Fischer property and a Parseval
type condition
must actually agree with the $2$-Besicovitch almost periodic
functions under some  mild additional assumptions. Now, with the
choice of $\mathcal{B}$ as $2$-Besicovitch almost periodic functions
one always ends up with some quasicrystals for which the
autocorrelation does not exist for all van Hove sequences. Thus, it
seems that one cannot expect a full solution to Problem 4.7 when
insisting on the Riesz--Fischer property and Parseval equality.

Problem 4.8 deals with a translation bounded  measure $\mu$ of
$\mathcal{B}$ whose Fourier transform is formally given by $\sum a_\xi
\delta_\xi$. It asks whether $a_\xi$ must be the Fourier--Bohr
coefficient (if this Fourier--Bohr coefficient exists). Now, this is
(trivially) true in our context if one chooses for $\mathcal{B}$ the
space of $2$-Besicovitch almost
periodic functions, as we have just defined the formal Fourier
expansion via the Fourier--Bohr coefficients.

Result 2 and Result 3 are not only of conceptual interest but also
of direct consequence.  Result 2 sheds a new and different light on
weak model sets of maximal density. Such model sets have received
attention in recent years  in work of Keller and Richard \cite{KR} and
Baake, Huck and Strungaru \cite{BHS}.  They have the particular
feature that -- unlike most other basic models for Aperiodic Order --
here the actual choice of the van Hove sequence matters. As we show
below, they can rather directly be seen to be Besicovitch almost
periodic. Given this, Result 2 allows one to recover most of the
fundamental results obtained for such weak model sets in the mentioned works
(Corollary  \ref{coro:weak-model-set}). Result 3 gives a new
perspective on a class of almost periodic measures recently
introduced by Meyer \cite{Mey2}. Meyer showed that this class
contains all regular model sets (in Euclidean space)  but did not
discuss  diffraction for this class. Here we show that every measure
in the class of Meyer is Weyl almost periodic (Corollary~\ref{cor12345}).  Then, Result 3 provides a rather complete description
of diffraction for this class.

A few words on our methods are in order. As discussed above, diffraction
theory starts with a translation bounded measure. The reflected Eberlein
convolution (along a given van Hove sequence) is then used to form
its autocorrelation. The Fourier transform of the autocorrelation is
the diffraction measure. A key insight in the  present work is
that this theory  can naturally be placed within the context of
group representations. Specifically, the autocorrelation gives rise
to a (pre-)Hilbert space structure on a certain space of functions,
on which the group acts continuously by isometries. The diffraction
measure  then appears as a  kind of `universal' spectral measure of
this group representation. In this way, tools from representation
theory become available in the study of diffraction. A convenient
way of formalizing this part of our approach is given by the concept
 of $\mathcal{A}$-representation introduced below (for a van Hove sequence $\cA =
(A_n)$). Such a representation is a  linear $G$-invariant map $N :
\Cc (G)\longrightarrow L^1_{loc} (G)$ with the additional property
that the means
\[
\lim_{n\to\infty} \frac{1}{|A_n|}\int_{A_n} N(\varphi) (s)\,
\overline{N(\psi) (s)} \, \dd s=:\langle N(\varphi),N(\psi)\rangle
\]
exist for all $\varphi,\psi \in\Cc(G)$. Under a mild additional
assumption any such representation comes with a measure $\sigma$ on
$\widehat{G}$ such that $t\mapsto \langle N(\varphi), T_t
N(\varphi)\rangle $ is just the Fourier transform of the finite
measure $|\widehat{\varphi}|^2 \sigma$ for any $\varphi \in\Cc (G)$.
To apply  this general approach  to diffraction of measures  we
consider, for a measure $\mu$ on $G$,  the map $N=N_\mu$ defined on
$\Cc (G)$ by
\[
N_\mu (\varphi) := \mu *\varphi \,.
\]
In this situation, the measure $\sigma$ can then be seen to be just
the `usual' diffraction measure considered in the literature.

A  second key insight of the present work is that almost
periodicity properties of the functions in the range of such an
$\mathcal{A}$-representation $N$ store all pieces of information
relevant to us to deal with pure-point diffraction and its
strengthened variants. Our main results are then obtained by
combining the framework of  $\mathcal{A}$-representations with a
thorough study of suitable  sets of almost periodic functions
(together with the  translation action on them). To provide such a
study can be seen as a core directive of this book.

We single out three types of almost periodicity. These are mean
almost periodicity, Besicovitch almost periodicity and Weyl almost
periodicity. All these concepts are natural generalizations of Bohr
almost periodicity. They arise by replacing the supremum norm by a
seminorm arising from averaging along a van Hove sequence in
respective characterizations of Bohr almost periodic functions.
 While very natural, the  concept of mean almost periodicity seems
not to have been investigated before. On the other hand, Besicovitch
and Weyl almost periodic functions have been considered in the
literature, mostly in connection with differential equations, i.e.
in the one dimensional Euclidean situation. Still,  parts of the
theory have also been considered for more general groups.  Here, we
thoroughly develop the  theory in the context of general
$\sigma$-compact, locally compact Abelian groups pointing out
earlier results along the way.  Central to our treatment are the
group action on Besicovitch spaces by translation  and the Eberlein
convolution of Besicovitch almost periodic functions. This part of
the theory is completely new. As for Weyl almost periodicity, a key
element for us is to characterize this within the Besicovitch class
by uniformity with respect to the van Hove sequences. To our
knowledge, this is a new characterization.

Two additional advantages of  our concept of $\cA$-representation
may be worth mentioning. Firstly, it makes the underlying
mathematics very transparent. In particular, it is clear that the
domain of  $N$ is rather irrelevant.  The crucial ingredient is the
range of $N$ being contained in certain classes of almost periodic
functions. In fact, the domain  $\Cc (G)$ of $N$ could be replaced
by any other subalgebra of continuous functions which is closed
under convolution and whose image under Fourier transform is dense
in a suitable $L^2$-space.  In particular, if $G$ is the Euclidean
space, we can develop a completely analogous theory based on
(tempered) distributions by considering
$\mathcal{A}$-representations $N$ mapping smooth functions with
compact support (or elements of the Schwartz space) into the set of
functions on $G$. Then, any (tempered) distribution $\varrho$ gives
a map $N = N_\varrho$ defined by $N_\varrho (\varphi) := \varrho\ast
\varphi$.

Secondly, we feel that this concept seems to be appropriate in terms
of modeling. After all, there is no intrinsic reason to prefer
measures over distributions. The only thing relevant is that --
irrespective of  how the distribution of matter in question is
modeled -- one should be able to pair it with functions. This is
exactly what is achieved by our concept of
$\mathcal{A}$-representation. In the context of dynamical systems
related  ideas were developed by Lenz and Moody in \cite{LM}.

This book is organized as follows: In Chapter~
\ref{sec-key-player}, we present the setting and discuss the
necessary concepts and tools for our considerations.  In particular,
we define the autocorrelation and  the Fourier--Bohr coefficients of
a measure. The material of this chapter is used throughout the whole book.

We then  discuss the fundamental seminorms and the
associated Besicovitch and Weyl type spaces and introduce the
framework of $\mathcal{A}$-representations, see Chapter~\ref{ch:besweyl}.

Chapter~\ref{sec-mean} is devoted to mean almost periodicity. We
introduce and study this notion for functions and measures and then
turn to our first main result, Theorem~\ref{single element}.
Finally, we discuss applications and show how our result contains
the mentioned earlier  results of \cite{bm} and \cite{Gouere-1}.

Chapter~\ref{sec-Besicovitch} deals with Besicovitch almost
periodicity. We first present a thorough study of Besicovitch almost
periodic functions. In particular, we show that the $p$-Besicovitch
almost periodic functions form a Banach space for every $p\geq 1$.
For $p=2$ this space is  a Hilbert space with a natural orthonormal
basis given by the characters of the group. Expansion with respect
to this orthonormal basis gives a  Fourier type theory and is the
basis for our solution to Problem 2. This solution is presented in
Theorem \ref{Bap and BT}. As an application, we give in Corollary
\ref{coro:weak-model-set} and its proof  a new approach to results
of \cite{BHS,KR} concerning weak model sets.

Our study of Weyl almost periodicity is given in Chapter
\ref{sec:Weyl}. We note
 that Weyl almost periodicity can
be understood as simultaneous  Besicovitch almost periodicity for
all van Hove sequences. Given this, Result 3 is a rather direct
consequence of Result 2. Details are given in Theorem
\ref{theorem-uniform-phase-problem} and its proof.  As an
application, we discuss a (slight generalization of a) concept of
almost periodicity recently introduced by Meyer in \cite{Mey2}. As
shown by Meyer, this type of almost periodicity is present in
regular model sets in Euclidean space. Indeed, finding a concept of
almost periodicity present in such models was exactly the motivation
for Meyer. Here, we show  in Corollary \ref{cor12345}
that this form of almost periodicity entails Weyl almost
periodicity. Given Result 3, this complements the results of Meyer
by providing the missing diffraction theory for this form of almost
periodicity.

Given a certain (weak) continuity assumption, Besicovitch almost periodic functions and Weyl
almost periodic functions are unavoidable when one deals with pure point
diffraction. In this sense, there is a uniqueness to our
solution of Questions 2 and 3. This is discussed in Chapter
\ref{sec:unavoidable}.

In Chapter~\ref{sec:Dynamics}, we discuss the connections to dynamical systems. The use of dynamical systems has been  fundamental to various parts of the study of Aperiodic Order, see e.g. \cite{TAO,KLS,Rad} and references therein. A general setup for diffraction in terms of dynamical systems was developed in \cite{BL} using dynamical systems of translation bounded measures (TMDS). Here, we study the link  between almost periodicity and
pure point spectrum for such TMDS.
We show that the hierarchy of mean/Besicovitch/Weyl almost periodicity can be fully understood
at the level of dynamical systems. More precisely, we show that
\begin{itemize}
  \item{} A measure $\mu$ is mean almost periodic if and only if it is generic for a $G$-invariant measure $\mm$ with pure-point spectrum (see Theorem~\ref{thm:DS-char-map}).
  \item{} A measure $\mu$ is Besicovitch almost periodic if and only if it is generic for an ergodic $G$-invariant measure $\mm$ with pure-point spectrum, and the Wiener--Wintner theorem applies to $\mu$ (i.e. $\mu$ is a Wiener--Wintner point for $\mm$). This is Theorem~\ref{bap-sg}.
  \item{} A measure $\mu$ is Weyl almost periodic if and only if its dynamical system is uniquely ergodic and has pure-point spectrum and continuous eigenfunctions (Theorem~\ref{thm:wap}).
\end{itemize}

The book is concluded by appendices, dealing with cut-and-project
schemes, semi-measures, counterexamples and uniform convergence
respectively.


\mainmatter
%
%
%

\chapter[Preliminaries]{Preliminaries or  framework,  notation and concepts  used throughout the text}\label{sec-key-player}

In this chapter, we introduce the central concepts  for our
investigations. These are the  autocorrelation and its Fourier
transform, the diffraction measure, as well as the associated
Fourier--Bohr coefficients and certain seminorms arising from
averaging. All these quantities live on  a $\sigma$-compact, locally
compact Abelian group,  and we start this section with a discussion
of basic concepts related to such groups.

\section[General setting]{The general setting: locally compact Abelian groups, functions, measures}\label{subsection-basic-setting}
In this section, we present the general framework used in our investigation. The concepts and results  are all well-known. We point  to specific (recent) references in some places for the convenience of the reader.

For the whole book, $G$ denotes a locally compact (Hausdorff),
$\sigma$-compact Abelian group. The associated Haar measure is
denoted by $\theta_G$\index{Haar~measure}. For the Haar measure of a measurable set $A\subseteq G$,
we often write $|A|$ instead of $\theta_G (A)$. Integration of an
integrable function $f: G\longrightarrow \CC$, with respect to
$\theta_G$, is often written as $\int_G f(s)\, \dd s$.
For $1\leq p <\infty$, we denote by $\mathcal{L}^p (G)$\nomenclature{$\mathcal{L}^p(G)$}{space of all measurable and $p$-integrable functions} the space of all measurable $f : G\longrightarrow \CC$ that satisfy $\int |f(s)|^p\, \dd s <\infty$. For $p =\infty$ we denote by $\mathcal{L}^\infty (G)$ the space of measurable $ f: G\longrightarrow \CC$ that are essentially bounded (i.e. for which there exists a number $C>0$ with $ |f(s)|\leq C$ for almost every $s\in G$). Similarly, for $1\leq p <\infty$ we denote by  $\mathcal{L}^p_{loc}(G)$ \nomenclature{$\mathcal{L}^p_{loc}(G)$}{space of locally $p$-integrable functions from $G$ to $\CC$} the space of all measurable $f : G\longrightarrow \CC$ that satisfy $\int_K |f(t)|^p\, \dd t <\infty$ for all
compact $K\subseteq G$. We call these functions \textit{locally integrable}.

Next, let us briefly recall some standard definitions.

\begin{definition}[Relatively dense and uniformly discrete sets]
A subset $\vL$ of $G$ is \textit{relatively dense}\index{relatively~dense} if there
exists a compact set $K \subseteq G$ such that
\[
G = \bigcup_{t\in \vL} (t+ K) \,.
\]
A subset $\vL$ of $G$  is \textit{uniformly discrete}\index{uniformly~discrete} if there
exists an open neighborhood $U$ of the identity such that
\[
(s+ U) \cap
(t + U) = \varnothing
\]
holds for all $s,t\in \vL$ with $s\neq t$. Such a subset
$\vL$ will also be called \textit{$U$-uniformly discrete}\index{uniformly~discrete!$U$~uniformly~discrete}.
A subset $\vL$ of $G$ is a \textit{Delone set}\index{Delone~set} if it is
relatively dense and uniformly discrete.       \exend
\end{definition}

\subsection{Functions and measures}

We will use the familiar symbols $C(G)$\nomenclature{$C(G)$}{set of continuous functions from $G$ to $\CC$} for the space of continuous functions from $G$ to $\CC$.
We denote by $\Cc(G)$\nomenclature{$C_{\mathsf{c}}(G)$}{set of continuous functions from $G$ to $\CC$ with compact support}, $\Cu(G)$ \nomenclature{$C_{\mathsf{u}}(G)$}{set of uniformly and continuous functions from $G$ to $\CC$} and $\Cz(G)$\nomenclature{$C^{}_{0}(G)$}{set of continuous functions from $G$ to $\CC$ vanishing at infinity} the subspaces of $C(G)$ consisting of
compactly supported, bounded uniformly continuous, and
continuous functions vanishing at infinity.
For a  bounded  function $u$  on $G$, we define the supremum
norm of $u$ by
\[
\|u\|_\infty:=\sup_{t\in G} |u(t)| \,.
\]
\nomenclature{$\| \, \|_\infty$}{supremum norm}

For any function $g$ on $G$ and any element $t\in G$, the functions $\widetilde{g}$\nomenclature{$\tilde{f}$}{complex conjugated and involuted version of a function $f$}, $\tau_t g$\nomenclature{$\tau_t f$}{shifted version of a function $f$} and $g^{\dagger}$\nomenclature{$f^{\dagger}$}{involuted version of a function $f$} on $G$ are defined
by
\[
\widetilde{g}(s) := \overline{g(-s)}\,, \qquad
(\tau_t g)(s) :=g(s-t) \qquad
\text{ and }  \qquad
g^{\dagger}(s):=g(-s) \,.
\]
The dual group\index{dual~group} $\widehat{G}$ of $G$ is the set of all  continuous group homomorphisms from $G$ to the unit circle $\{z\in \CC : |z| =1\}$. It becomes a topological space in a natural way, see e.g. \cite{BF,rud}. \nomenclature{$\widehat{G}$}{dual group of the group $G$}

\smallskip

The \textit{Fourier transform}\index{Fourier~transform!Fourier~transform~of~function} of  $g\in \mathcal{L}^1(G)$ is the function \nomenclature{$\widehat{f}$}{Fourier transform of the function $f$}
\begin{align*}
\widehat{g} &: \widehat{G}\longrightarrow \CC \\
\widehat{g}(\chi)(\chi) \, &:= \,  \int_G \overline{\chi (t)}\, g(t)\, \dd  t \,.
\end{align*}
We will also encounter the inverse Fourier transform\index{Fourier~transform!inverse~Fourier~transform~of~function} of $g$ defined by
\[
\widecheck{g} (\chi) \, :=\,  \widehat{g} (\chi^{-1}) \, = \,  \int_G \chi (t)\, g(t)\, \dd  t \,.
\] For basic properties of the Fourier transform, we refer the reader to \cite{rud}. \nomenclature{$\check{f}$}{inverse Fourier transform of the function $f$}

\medskip

Next, let us briefly review the concept of measures.
For a compact set $K\subseteq G$,
we define $C_K(G)$\nomenclature{$C_K(G)$}{set of continuous functions from $G$ to $\CC$ that are supported inside $K$} to be the subspace of
functions in $\Cc(G)$ that vanish outside of $K$. The space $\Cc
(G)$ is naturally equipped with the inductive limit topology\index{inductive~topology}, which is the largest
topology that makes the inclusion
\[
C_K (G)\hookrightarrow
\Cc (G) \,;\, f\mapsto f \,,
\]
continuous for any compact $K\subseteq G$.
A \textit{Radon measure}\index{measure!Radon~measure} $\mu$ on $G$ is a linear functional
$\mu: \Cc(G) \to \CC$ which is continuous in the inductive topology on $\Cc(G)$.
Subsequently, we will simply call $\mu$ a \textit{measure}\index{measure}.

\begin{remark} \phantom{XX}
\begin{itemize}
   \item[(a)] A linear functional $L : \Cc(G) \to \CC$ is a measure if and only if, for each compact $K \subseteq G$, there exists a constant $C_K>0$ such that
       \begin{equation}\label{eq-rm}
       \left| L(\varphi) \right| \leq C_K \| \varphi \|_\infty
       \end{equation}
holds for all $\varphi \in \Cc(G)$ with $\supp(\varphi) \subseteq K$.
   \item[(b)] Any positive linear mapping $L: \Cc(G) \to \CC$ is a measure \cite{CRS}. Moreover, a linear mapping $L : \Cc(G) \to \CC$ is a measure if and only if it is a linear combination of positive linear mappings \cite{CRS}.
   \item[(c)] By the Riesz representation theorem \cite{rud2}, the class of positive Radon measures coincides with the class of positive, regular Borel measures. In particular,
   a Radon measure is simply a linear combination of positive, regular Borel measures.  \exend
 \end{itemize}
\end{remark}

For a measure $\mu$ on $G$ and $t\in G$, we define the measures $\widetilde{\mu}$\nomenclature{$\widetilde{\mu}$}{complex conjugated and involuted version of a measure $\mu$}, $\tau_t \mu$\nomenclature{$\tau_t\mu$}{shifted version of a measure $\mu$} and $\mu^{\dagger}$
\nomenclature{$\mu^{\dagger}$}{involuted version of a measure $\mu$}
on $G$ by
\[
\widetilde{\mu}(g) := \overline{\mu(\widetilde{g})} \,, \qquad
(\tau_t\mu)(g):= \mu(\tau_{-t}g) \qquad  \text{ and }  \qquad
\mu^{\dagger}(g):=\mu(g^{\dagger}) \,.
\]

Moreover, for a measure $\mu$ and a function $\varphi \in \Cc (G)$ we define the \textit{convolution} \index{convolution! between a measure and a function}  $\mu \ast \varphi : G\longrightarrow \CC$ by
\[
(\mu \ast \varphi) (t) = \int_G \varphi (t-s)\, \dd\mu (s) \qquad \forall t \in G \,.
\]

Any measure $\mu$ gives rise to a positive measure $|\mu|$ satisfying
\[
|\mu|(\varphi) = \sup \{ \left| \mu(\psi) \right| : \psi \in \Cc(G), |\psi| \leq \varphi \}
\]
for all non-negative $\varphi \in \Cc(G)$, see \cite[Thm.~6.5.6]{Ped}.
The measure $|\mu|$ is called the \textit{total variation of $\mu$}\index{total~variation~measure}, and it (is the smallest positive measure on $G$ which) satisfies
\[
\left|\mu (\varphi)\right| \leq  |\mu|(|\varphi|) \qquad \text{ for all } \varphi \in
\Cc(G) \,.
\]
A measure $\mu$ on $G$ is called \textit{finite}\index{measure!finite~measure} if $|\mu| (G)<\infty$
holds. In particular, this means that the constant $C_K$ in \eqref{eq-rm} can be chosen to be the same for all compact sets $K \subseteq G$.

\smallskip

\begin{definition}[Translation bounded measure]
A measure $\mu$ on $G$ is called \textit{translation bounded}\index{measure!translation~bounded~measure}
if \nomenclature{$\| \, \|_{V}$}{$V$-norm}
\[
\| \mu \|_{V} := \sup_{t\in G}|\mu|(t+V) < \infty
\]
holds for one (and thus each) open and relatively compact subset
$V\subseteq G$. The space of all
translation bounded measures on $G$ is denoted by
$\mc{M}^{\infty}(G)$\nomenclature{$\mathcal{M}^{\infty}(G)$}{space of translation bounded measures on $G$}.\exend
\end{definition}

A measure $\mu$ is translation bounded if and only if  $\mu* \varphi \in \Cu(G)$ for
all $\varphi \in \Cc(G)$ \cite{ARMA1,MoSt}.

\smallskip
For $\varphi \in \Cu (G)$ and $\psi\in \Cc(G)$ or $\varphi \in \Cc(G)$ and $\psi
\in \Cu (G)$, the \textit{convolution}\index{convolution!of functions} $\varphi \ast \psi$
is defined via
\[
(\varphi \ast \psi) (t) := \int_{G} \varphi(t-s)\, \psi(s)\, \dd s
\]
for $t\in G$. It is not difficult to see that $\varphi \ast \psi\in \Cu(G)$. If both $\varphi$ and $\psi$ belong to $\Cc (G)$, so
does $\varphi \ast \psi$.  The convolution $\mu \ast
\nu$\index{convolution!of measures} of a finite measure $\mu$
and a translation-bounded measure $\nu$ is defined by
\[
(\mu \ast \nu) (\varphi) := \int_G \int_G \varphi (s +t)\, \dd \mu (s)\, \dd\nu (t) \,.
\]

\smallskip

For $f\in\Cu (G)$ and $\varepsilon>0$, there exists a neighborhood
$U$ of $0$ in $G$ such that
\[
\|f  - f*\varphi\|_\infty\leq \varepsilon
\]
for all $\varphi  \in\Cc (G)$ with support contained in $U$ and $\int_G \varphi(t)\, \dd t = 1$. In other words, this allows one to find an
\textit{approximate unit}\index{approximate~unit} (or \textit{approximate identity})\index{approximate~identity|see{approximate~unit}}, i.e. a net $(\varphi_\alpha)$  in $\Cc
(G)$ with $\varphi_\alpha *f \to f$ with respect to
$\|\cdot\|_\infty$ for all $f\in\Cu (G)$ and such that the support of every
$\varphi_\alpha$ is contained in one fixed open and relatively compact
neighborhood of $0$ in $G$.

\smallskip
Finally, recall that a function $g : G\longrightarrow \CC$ is \textit{positive definite}\index{positive~definite!positive~definite~function} if, for any finite set $F\subseteq G$ and any function $c : F\longrightarrow \CC$, we have
\[
\sum_{t,s\in F} c_t \cdot \overline{c_s} \cdot g(t-s)\geq 0 \,.
\]
A measure $\mu$  on $G$ is called \textit{positive
definite}\index{positive~definite!positive~definite~measure} if
\[
\mu(\varphi*\widetilde{\varphi})\geq 0 \qquad \text{ for all } \varphi\in \Cc(G) \,.
\]
A continuous function $g : G\longrightarrow \CC$ is positive definite if and only if the measure $g\theta_G$ is positive definite \cite[Prop.~4.1]{BF}.
Any positive definite measure $\gamma$
admits a (unique) positive measure  $\reallywidehat{\gamma}$ on
$\widehat{G}$ which satisfies
\[
\int_{\widehat{G}} |\widecheck{\varphi}(\chi)|^2\, \dd\reallywidehat{\gamma}(\chi)  = \int_{G}
(\varphi \ast\widetilde{\varphi})(x)\,  \dd\gamma(x) \,,
\]
for all $\varphi\in
\Cc(G)$ \cite{BF,MoSt}. The measure $\reallywidehat{\gamma}$ is called the
\textit{Fourier transform}\index{Fourier~transform!Fourier~transform~of~measure} of $\gamma$.  \nomenclature{$\widehat{\gamma}$}{Fourier transform of the measure $\gamma$}

\subsection{Bohr almost periodic functions, strongly almost periodic measures and the Bohr compactification}
Almost periodicity is the core concept in this book. Its strongest form is given by Bohr almost periodic functions and the corresponding measures. This is discussed in this section and builds the foundation for the more elaborate concepts of almost periodicity to be introduced later.

Bohr almost periodic functions were introduced by Bohr in \cite{Bohr1,Bohr2,Bohr3}. In these papers, Bohr showed the equivalence of the function having a relatively dense set of almost periods (property (i) in Proposition~\ref{prop:SAPchar}) and being the uniform limit of a sequence of trigonometric polynomials (property (iii) in Proposition~\ref{prop:SAPchar} below). He also studied the Fourier--Bohr coefficients as well as their square summability and relationship to the mean of $|f|^2$. Bohr's work built on earlier results in this direction by Bohl \cite{Bohl1,Bohl2} and Esclangon \cite{Esc1,Esc2}. Bochner proved \cite{Boch} that this definition is also equivalent to the compactness of the translation hull of the function (property (ii) in Proposition~\ref{prop:SAPchar}). These notions were extended from $\RR^d$ to arbitrary LCAG by von Neumann \cite{vNeu} and Bochner and von Neumann \cite{BvN}.

\smallskip

Over the years, generalizations of those notions were introduced by Stepanov \cite{Step1,Step2}, Weyl \cite{Weyl}, Besicovitch \cite{Bes26} and Eberlein \cite{Eb}, to name just a few. Almost periodicity in the sense of Weyl and Besicovitch will play a central role in this book.

We start by briefly reviewing some of the properties of Bohr almost periodic functions. For a full review, we recommend \cite{ABG,Bes,Boh,Cor,Cor2}.

\medskip


Recall that a \textit{trigonometric polynomial}\index{trigonometric~polynomial} on $G$ is a function of the form
\[
P(t):=\sum_{j=1}^n c_j \chi_j(t)
\]
for some $n \in \NN, c_1, \ldots, c_n \in \CC$ and $\chi_1, \ldots, \chi_n \in \widehat{G}$.

Recall also that to  any locally compact Abelian group $G$ there exists a unique (up to group isomorphism)  compact group $G_{\mathsf{b}}$ together with a  group homomorphism
$$i_{\mathsf{b}} : G\longrightarrow  G_{\mathsf{b}}$$
 with dense range such that any group homomorphism  $i: G\longrightarrow \mathbb{K}$ with dense range and compact $\mathbb{K}$ factors through $i_{\mathsf{b}}$ (i.e. admits a (necessarily unique) group homomorphism $\kappa : G_{\mathsf{b}} \longrightarrow \mathbb{K}$ with $\kappa \circ i_{\mathsf{b}} = i$).
The group  $G_{\mathsf{b}}$ \nomenclature{$G_{\mathsf{b}}$}{Bohr compactification of $G$} is called the Bohr compactification of $G$\index{Bohr~compactification} (see \cite{MoSt} for further details and properties of Bohr compactification).

\smallskip

Let us now review the following characterization which we can use as the definition of a Bohr almost periodic function. For a proof, see for example \cite[Thm.~4.3.5]{MoSt}.

\begin{prop}\label{prop:SAPchar} Let $f \in \Cu(G)$. Then, the following are equivalent:
\begin{itemize}
  \item[(i)] For all $\varepsilon >0$, the set
\[
\AP_{\infty}(f, \eps):= \{ t \in G\, :\, \| f-\tau_t f \|_{\infty} < \eps \}
\]
of \textit{$\eps$-Bohr almost periods of $f$}\index{almost~periods!Bohr-almost~periods} is relatively dense. \nomenclature{$\AP_{\infty}(f, \eps)$}{Bohr-almost periods}
  \item[(ii)] The closure of $\{\tau_t f\ : \ t\in G\}$ is compact in $(\Cu(G), \| \cdot \|_\infty)$.
  \item[(iii)] For each $\varepsilon >0$, there exists a trigonometric polynomial $P$ such that $\| f-P \|_\infty < \varepsilon$.
  \item[(iv)] There exists a continuous function $F : G_{\mathsf{b}} \to \CC$ such that $f= F \circ i_{\mathsf{b}}$.\qed
\end{itemize}
\end{prop}

 This leads to the following classical definition.

\begin{definition}[Bohr almost periodic functions and measures]
A function $f \in \Cu(G)$ is called \textit{almost periodic in the sense of Bohr}, or simply \textit{Bohr almost periodic}\index{almost~periodic!Bohr~almost~periodic~function}, if
it satisfies the equivalent conditions from Proposition~\ref{prop:SAPchar}.
The set of Bohr almost periodic functions is denoted by $\sap(G)$. \nomenclature{$\sap(G)$}{space of Bohr almost periodic functions on $G$}
A measure $\mu \in \cM^\infty(G)$ is called \textit{strongly almost periodic}\index{almost~periodic!strongly~almost~periodic~measure} if $\mu*\varphi \in \sap(G)$ holds for all $\varphi \in \Cc(G)$. The space of strongly almost periodic measures is denoted by $\SAP(G)$. \nomenclature{$\SAP(G)$}{space of strongly almost periodic measure}  \exend
\end{definition}

The set $\sap(G)$ is a closed subalgebra of $(\Cu(G),\|\cdot\|_\infty)$ \cite[Prop.~4.3.4]{MoSt}.
Note that the set of trigonometric polynomials is a dense subalgebra of $\sap(G)$.

\medskip

Weak almost periodicity is another classical concept of almost periodicity.  As this will at least occasionally show up in the book by means of comparison, we finish this subsection with a short discussion on this class of functions.

A function  $f\in \Cu(G)$  is called \textit{weakly almost periodic}\index{almost~periodic!weakly~almost~periodic~function}
if the closure $\{\tau_t  f : t\in G\}$ is compact in the weak
topology of the Banach space $(\Cu(G),\|\cdot\|_\infty)$. The space of weakly almost periodic functions is denoted by $\weak(G)$. \nomenclature{$\weak(G)$}{set of Weakly almost periodic functions}

Any weakly almost periodic $f$ admits a unique decomposition $f = g + h$ with
$g$ being Bohr almost periodic and $h$ being amenable with $M(|h|)
=0$ (see below for the definition of mean), see e.g. \cite{Ebe3} or \cite[Prop.~4.5.9]{MoSt}. As any Bohr almost periodic function
is amenable, we infer in particular that any weakly almost periodic function is amenable.

A measure $\mu$ is called \textit{weakly
almost periodic}\index{almost~periodic!weakly~almost~periodic~measure} if $\mu\ast \varphi$ is a weakly almost periodic
function for all $\varphi \in\Cc (G)$. The space of weakly almost periodic measures is denoted by $\WAP(G)$. \nomenclature{$\WAP(G)$}{set of Weakly almost periodic measures on $G$}

\subsection{Averaging and means}
Averaging will play an important role in our considerations. Thus, we need to discuss the appropriate basic concepts, such as F\o lner and van Hove sequences. While the notion of a F\o lner
sequence is sufficient for dealing with means of functions, the van Hove property
is essential for dealing with means of measures. Roughly speaking convergence along  F\o lner sequences ensures that means of functions and their translates agree whereas van Hove sequences are needed to ensure that the means of measures and their translates agree, see for example \cite[Lem.~4.10.6, Lem.~4.10.7]{MoSt} for more details on this.

\begin{definition} [F\o lner and van Hove sequences] \phantom{XX}
\begin{itemize}
\item[(a)] A sequence $(F_n)$ of measurable, pre-compact sets of positive measure is called a
\textit{F\o lner sequence}\index{F\o lner~sequence} if
\[
\lim_{n\to\infty} \frac{|F_n \triangle (t+F_n)|}{|F_n|} =0
\]
holds for all $t \in G$.

\item[(b)] A sequence $(A_n)$ of relatively compact open subsets of $G$ is
called a \textit{van Hove}\index{van~Hove~sequence} sequence if, for each compact set $K
\subseteq G$, one has
\[
\lim_{n\to\infty} \frac{|\partial^{K} A_{n}|}{|A_{n}|}  =  0 \,,
 \]
where the \textit{$K$-boundary $\partial^{K} A$}\index{van~Hove~boundary} of a set $A$
is defined as
\[
\partial^{K} A := \bigl( (\overline{A+K}) \cap \overline{A^c} \bigr) \cup
\bigl((\overline{A^c  - K})\cap \overline{A}\, \bigr)
\,,
\]
where we write $A^c$ for the complement $G\setminus A$ of $A$. \exend
\end{itemize}
\end{definition}
For open sets $A$ some of the sets appearing in the definition of the $K$-boundary are automatically closed, and the $K$-boundary takes the  (somewhat simpler, albeit less symmetric) form
\[
\partial^{K} A = \bigl( (\overline{A+K}) \setminus A\bigr) \cup
\bigl((\left(G \backslash A\right) - K)\cap \overline{A}\, \bigr)
\,.
\]

It is not difficult to see that every van Hove sequence is a F\o lner~sequence.
However, not every locally compact Abelian group admits a van Hove or a F\o lner~sequence. A characterization is given in  \cite{Martin2,SS}.

\begin{lemma} \cite[Prop.~B.6]{SS}
   For an LCAG $G$, the following assertions are equivalent:
  \begin{itemize}
    \item[(i)] $G$ is $\sigma$-compact.
    \item[(ii)]  $G$ admits a van Hove sequence.
    \item[(iii)] $G$ admits a F\o lner sequence.    \qed
  \end{itemize}
\end{lemma}

\begin{remark}
If the group is not $\sigma$-compact, it does not admit van Hove or F\o lner sequences. However, it will still admit van Hove and F\o lner nets. This is sufficient to carry out most of the subsequent considerations with the notable exception of the proof of completeness of the spaces $BL^p_\cA (G)$ (see below). We refrain from providing further details.  \exend
\end{remark}

Now, we can talk about the mean of a function. Whenever $f$ belongs to $\mathcal{L}^1_{loc}(G)$  and $\cA$ is a F\o lner
sequence, we define the \textit{mean of $f$  along $\cA$}\index{mean!mean~of~function} by \nomenclature{$M(f)$}{mean of the function $f$}
\[
M_{\cA} (f):= \lim_{n\to\infty} \frac{1}{|A_n|} \int_{A_n} f(t)\, \dd t
\]
if the limit exists. In this case, we also say that \textit{the
mean $M_{\cA}(f)$ exists}. For bounded measurable functions, the existence of the mean along arbitrary van Hove sequences is characterized as follows.

\begin{prop}[Amenability of functions]\label{prop: amenable}
For any bounded and measurable function $f : G\longrightarrow \CC$ and any van Hove sequence $\cA$, the following statements are equivalent:
\begin{itemize}
  \item[(i)] The limit
  \[
  \lim_{n\to\infty} \frac{1}{|A_n|} \int_{s+A_n} f(t)\, \dd t
  \]
exists uniformly in $s\in G$.
  \item[(ii)] For each van Hove sequence $\cB$, the limit
  \[
  \lim_{n\to\infty} \frac{1}{|B_n|} \int_{s+B_n} f(t)\, \dd t
  \]
exists uniformly in $s\in G$ and is independent of the van Hove sequence.
  \item[(iii)]
The limit
\[
\lim_{n\to\infty} \frac{1}{|B_n|}
\int_{B_n} f(t)\, \dd t \]
exists for every van Hove sequence $\cB$.
\end{itemize}
\end{prop}
\begin{proof}
(i)$\Longrightarrow$(ii): This follows from Proposition~\ref{prop:mother-of-uniform-van-Hove-results} (compare \cite[Thm.~3.1]{Eb} or \cite[Prop.~4.5.6]{MoSt} for $f \in \Cu(G)$).
\smallskip

\noindent (ii)$\Longrightarrow$(iii): This is obvious.

\smallskip

\noindent (iii)$\Longrightarrow$(i): Assume to the contrary that
\[
\lim_{n\to\infty} \frac{1}{|A_n|} \int_{s +A_n} f(t)\, \dd t
\]
does
not exist uniformly in $s\in G$. Note that by (iii) the limit
\[
M_{\cA}(f)=\lim_{n\to\infty} \frac{1}{|A_n|} \int_{A_n} f(t)\, \dd t
\]
exists. Then, there exists some $\eps >0$, an increasing sequence
$(n_k)$ and some $s_k \in G$ such that,
for all $k$, we have
\begin{equation}\label{eq ame}
\bigg| \frac{1}{|A_{n_k}|} \int_{ s_k +A_{n_k}} f(t)\, \dd t -
M_{\cA} (f)\bigg| >\eps \,.
\end{equation}
Next, define $B_k=s_k+A_{n_k}$. Then, $\cB=(B_k)$ is a van Hove
sequence, and hence, by (iii),  $M_{\cB}$
 exists and $|M_{\cB}(f) - M_{\cA}(f)|\geq
\varepsilon$ holds by \eqref{eq ame}. On the other hand, we can
consider the van Hove sequence $\mathcal{C} = (C_n)$ with
\[
C_n = \begin{cases}
A_{m} &\mbox{ if } n=2m \\
B_{m} &\mbox{ if } n=2m+1 \,.
\end{cases}
\]
By (iii)  the
limit of $f$ along this sequence exists giving the contradiction
$M_{\cA}(f) = M_{\cB}(f)$.
\end{proof}

A bounded and measurable function $f: G\longrightarrow\CC$ satisfying
one of the equivalent conditions of Proposition~\ref{prop: amenable} is
called \textit{amenable}\index{amenable~function}\footnote{Amenability is typically defined for $f \in \Cu(G)$ using F\o lner sequences. For such $f$, Proposition~\ref{prop: amenable} holds for F\o lner sequences. As each van Hove sequence is a F\o lner sequence, condition (ii) tells us that for uniformly continuous bounded functions, our seemingly weaker condition is actually equivalent to the standard one.}.

\section[Diffraction theory]{Diffraction theory and its key players:  Eberlein convolution,  autocorrelation and Fourier--Bohr coefficients} \label{sub:autocorrelation} In this section, we
introduce the autocorrelation and study some of its properties as
well as its Fourier transform. These are the main objects of
interest in this book. The autocorrelation arises from a measure by taking the  Eberlein convolution with its reflection. In this sense the reflected Eberlein convolution is at the heart of diffraction theory and we start with a discussion of it.

\subsection{Eberlein convolution}

Whenever $\cA$ is a F\o lner sequence on $G$, the \textit{Eberlein
convolution}\index{Eberlein~convolution!Eberlein~convolution~of~functions} $f\circledast_{\cA} g$ (along $\cA$)\nomenclature{$f\circledast_{\cA} g$}{Eberlein
convolution of functions}  of measurable
functions $f,g : G\longrightarrow \CC$ is defined as the function
\begin{align*}
f\circledast_{\cA} g &: G\longrightarrow \CC\ \\
f\circledast_{\cA} g(t) \, &:= \, M_{\cA}( f \tau_{t}g^{\dagger}) = \lim_{n\to\infty} \frac{1}{|A_n|} \int_{A_n} f(s)\, g(t - s)\, \dd s
\end{align*}
if the integral and the limit in question exist for all $t\in G$
\cite{Eb,MoSt}.

Similarly, when $\cA$ is a van Hove sequence, the \textit{Eberlein convolution} $\mu
\circledast_{\cA} \nu$ \index{Eberlein~convolution!Eberlein~convolution~of~measures} of measures $\mu$ and $\nu$\nomenclature{$\mu\circledast_{\cA} \nu$}{Eberlein
convolution of measures}  on $G$ is
defined as the vague limit
\[
\mu \circledast_{\cA} \nu = \lim_{n\to\infty} \frac{1}{|A_n|} (\mu|_{A_n}
*\nu|_{-A_n})
\]
if this limit exists. Here,  $\mu|_{A_n}$  and $\mu|_{-A_n}$ are the
finite measures arising as the restrictions of $\mu$ to the compact
sets  $A_n$ and $-A_n$ respectively.

In the spirit of \cite{LSS4}, we also introduce the following definition.

\begin{definition}[Reflected Eberlein convolution]
Let  $\cA$ be a van Hove sequence.

Let measurable
functions $f,g : G\longrightarrow \CC$ be given. If  $f \circledast_{\cA} \tilde{g}$ exists we
call it the
 \textit{reflected Eberlein
convolution}\index{Eberlein~convolution!reflected~Eberlein~convolution~of~functions} of $f$ and $g$. We then also say that the reflected Eberlein  convolution of $f$ and $g$ exists and use the notation
\[
\lb f,g \rb_{\cA}(t):= f \circledast_{\cA} \tilde{g}(t) = \lim_{n\to\infty} \frac{1}{|A_n|} \int_{A_n} f(s)\, \overline{ g(s - t)} \, \dd s \,.
\]\nomenclature{$\lb f,g \rb_{\cA}$}{reflected Eberlein
convolution of functions}
Similarly, we say that the \textit{reflected Eberlein
convolution}\index{Eberlein~convolution!reflected~Eberlein~convolution~of~measures} $\lb \mu,\nu \rb_{\cA}$ of the measures $\mu, \nu$ exists with respect to $\cA$, if $\mu \circledast_{\cA} \tilde{\nu}$ exists. In this case, we use the notation
\[
\lb \mu,\nu \rb_{\cA}:= \mu \circledast_{\cA} \tilde{\nu}  \,.  \tag*{$ \Diamond $}
\]\nomenclature{$\lb \mu,\nu \rb_{\cA}$}{reflected Eberlein
convolution of measures}
\end{definition}

The following result will be important later.

\begin{prop}[Characterization of the existence of the reflected Eberlein convolution] \label{prop-compute-autocorrelation}
Let $\mu$ and $\nu$ be translation-bounded measures on $G$. Let
$\cA$ be a van Hove sequence. Then, the following assertions are
equivalent:
\begin{itemize}
\item[(i)] The reflected Eberlein convolution $\lb \mu, \nu\rb_{\cA}$ exists.
\item[(ii)]  For all  $\varphi, \psi \in \Cc (G)$, the mean $M_{\cA}\big((\mu * \varphi) \cdot \overline{(\nu *\psi)}\big)$ exists.
\end{itemize}
If (i) and (ii) hold, then
\[
M_{\cA}\big((\mu * \varphi) \cdot \overline{(\nu *\psi)}\big)  = \big(\lb \mu, \nu \rb_{\cA}*\varphi*\widetilde{\psi}\big)(0)
\]
holds for all $\varphi, \psi \in \Cc(G)$.
\end{prop}
\begin{proof} This is essentially contained in (the proof of)
\cite[Lem.~7.1]{LS2}. Specifically, this lemma states that (i)
implies (ii) and that the last statement holds. Moreover, the proof
allows one to conclude  that the reverse implication holds as well:
For $n\in\NN$, we define the measure $m_n$ on $G$ by
\[
m_n:=\frac{1}{|A_n|} \mu|_{A_n} \ast \widetilde{\nu}|_{-A_n}
\]
and
the map $M_n : \Cc (G)\longrightarrow \CC$ by
\[
M_n ( \varphi) =
\frac{1}{|A_n}| \int_{A_n} \varphi(t)\, \dd t \,.
\]
By (i), the sequence $\big(m_n (\varphi)\big)$
converges for all $\varphi \in \Cc(G)$. Of course, this is
equivalent to the convergence of the sequence  $\big((m_n\ast \varphi)(0)\big)$ for all $\varphi \in\Cc (G)$. In fact, (i) is actually equivalent to the convergence of the sequence $\big((m_n*\varphi*\widetilde{\psi})(0)\big)$ for all $\varphi,\psi \in \Cc (G)$. To see this, consider an arbitrary $\varphi\in \Cc (G)$. Let $K\subseteq  G$ be compact with $\varphi$ vanishing outside of $K$. Choose an open, relatively compact neighborhood $U$ of $0$ in $G$. Then, for any $\varepsilon >0$, we can find a $\psi \in \Cc (G)$ supported in $U$ with $\|\varphi - \varphi \ast \widetilde{\psi}\|_\infty < \varepsilon$. This gives
\[
|m_n (\varphi) - m_n (\varphi\ast \widetilde{\psi})| \leq |m_n| (K
+ \overline{U})\, \|\varphi - \varphi \ast\widetilde{\psi}\|_\infty
\leq \varepsilon\, |m_n| (K + \overline{U}) \,.
\]
Due to the translation
boundedness of $\mu$ and $\nu$, the sequence $\left(|m_n|
(K + \overline{U})\right)$ can be
seen to be  bounded, compare
\cite[Lem.~1.1]{Martin2}, and the desired statement follows.

Now, the proof of \cite[Lem.~7.1]{LS2} contains the line
\[
\lim_{n\to\infty} \left|\big(m_n \ast \varphi \ast \widetilde{\psi}\big) (0)
- M_n \big((\mu \ast \varphi) \cdot \overline{(\nu \ast \psi)} \big)\right| = 0 \,.
\]
This shows that
(ii) is equivalent to the convergence of $\big((m_n \ast \varphi \ast
\widetilde{\psi})(0)\big)$ for all $\varphi,\psi \in\Cc (G)$, which is
equivalent to (i) by our considerations above. This finishes the
proof.
\end{proof}

As a consequence, we get the following results which we will often use.

\begin{coro}\label{coro:ref-ebe}
Let $\mu$ and $\nu$ be translation bounded measures on $G$, and let
$\cA$ be a van Hove sequence. Then, $\lb \mu, \nu\rb_{\cA}$ exists if and only if $\lb \mu*\varphi, \nu*\psi \rb_{\cA}$ exists for all $\varphi, \psi \in \Cc(G)$.

Moreover, in this case, we have
\begin{align*}
\big(\lb \mu, \nu\rb_{\cA}*\varphi*\widetilde{\psi}\big)(t)
\,&=\, \lb \mu*\varphi, \nu*\psi \rb_{\cA}(t) \\
\,&=\, \lim_{n\to\infty} \frac{1}{|A_n|} \int_{A_n} (\mu*\varphi)(s)\, \overline{(\nu*\psi)(s-t)}\, \dd s
\end{align*}
for all $\varphi, \psi \in \Cc(G)$ and all $t \in G$. \qed
\end{coro}

\bigskip

Proposition~\ref{prop-compute-autocorrelation} implies that (reflected) Eberlein convolution of
functions in $\Cu (G)$ agrees with the (reflected) Eberlein convolution of the
corresponding measures.

\begin{prop}\label{prop:Eberlein-functions}
Let $\cA$ be a van Hove sequence. For any $f,g\in \Cu (G)$, the following assertions are equivalent:
\begin{itemize}
\item[(i)] The Eberlein convolution $f\circledast_\cA g$ exists.
\item[(ii)] The Eberlein convolution $(f\theta_G)\circledast_\cA (g\theta_G)
$ exists.
\end{itemize}
If (i) and (ii) hold,  we have
\[
(f\theta_G)\circledast_\cA
(g\theta_G)=(f\circledast_\cA g)\,  \theta_G \ ,
\]
and $f\circledast_{\cA} g$ belongs to $\Cu (G)$.
\end{prop}
\begin{proof}
As $(f\theta_G)*\varphi = f*\varphi$ and $(g\theta_G)*\psi = g*\psi$ for all $\varphi,\psi \in \Cc (G)$, Proposition~\ref{prop-compute-autocorrelation} easily gives that (ii) is equivalent to the existence of the means $M_{\cA}
\big(f\ast \varphi \cdot g \ast \psi (- \cdot)\big)$ for all $\varphi,\psi\in\Cc (G)$. As $f,g$ are uniformly continuous, this can easily be seen to be equivalent to the existence of $f\circledast_{\cA} g$.
Moreover, the uniform continuity of $f,g$ gives that $t\mapsto
M_{\cA} \big(f \cdot g (t- \cdot)\big)$ is uniformly continuous (if it exists at all).
\end{proof}

Replacing $g$ by $\tilde{g}$, we immediately obtain the following consequence.

\begin{coro}
Let $\cA$ be a van Hove sequence, and let $f,g\in \Cu (G)$ be given.
Then, the following assertions are equivalent:
\begin{itemize}
\item[(i)] The reflected Eberlein convolution $\lb f,g \rb_{\cA}$ exists.
\item[(ii)] The reflected Eberlein convolution $\lb f\theta_G, g\theta_G \rb_{\cA}$ exists.
\end{itemize}
If (i) and (ii) hold,  we have
\[
\lb f\theta_G, g\theta_G \rb_{\cA}= \lb f,g \rb_{\cA} \theta_G \ ,
\]
and
$\lb f,g \rb_{\cA} \in \Cu (G)$. \qed
\end{coro}

\subsection{Autocorrelation and diffraction}

\begin{definition}[Autocorrelation]
Let $\mu$ be a measure on $G$, and let $\cA =(A_n)$ be a van Hove
sequence. If the reflected Eberlein convolution $\lb \mu ,
\mu \rb_\cA$ exists, it is called the \textit{autocorrelation of $\mu$ along
$\cA$}\index{autocorrelation!autocorrelation~measure} and denoted by $\gamma_\cA$ or just $\gamma$  \nomenclature{$\gamma$}{autocorrelation measure} if $\cA$ is clear from the context.       \exend
\end{definition}

\begin{remark}
For translation bounded measures in second-countable groups $G$, given any van Hove sequence $(A_n)$, the autocorrelation always exists along a subsequence of $(A_n)$ \cite[Theorem~4.15]{LSS3}.        \exend
\end{remark}

The following result follows immediately from Corollary~\ref{coro:ref-ebe} and the polarisation identity.

\begin{coro}[Existence of the autocorrelation]
Let $\mu$ be a translation bounded measures on $G$, and let $\cA$ be a van Hove sequence. Then, the autocorrelation $\gamma_{\mu}$ exists if and only if, for all $\varphi \in \Cc(G)$ and all $t \in G$, the following limit exist:
\[
 \lim_{n\to\infty} \frac{1}{|A_n|} \int_{A_n} (\mu*\varphi)(s) \overline{(\mu*\varphi)(s-t)}\, \dd s \,.
\]
Moreover, in this case,
\[
 \lim_{n\to\infty} \frac{1}{|A_n|} \int_{A_n} (\mu*\varphi)(s) \overline{(\mu*\varphi)(s-t)}\, \dd s = \big(\gamma_{\mu}*\varphi*\widetilde{\varphi}\big)(t) \,.       \tag*{$\qed$}
\]
\end{coro}

By the  standard argument of `mixing' van Hove sequences (compare the proof of Proposition \ref{prop: amenable}), we can see that the autocorrelation of a translation bounded measure must be independent of the van Hove sequence if it exists along each
van Hove sequence. We then say that the translation bounded measure has a  \textit{unique autocorrelation}.

It is easy to see that, for any finite measure $\nu$, the measure
$\nu*\widetilde{\nu}$ is positive definite, and that vague limits of
positive definite measures are positive definite \cite{MoSt}.
As the Fourier transform of any positive definite measure
exists (compare discussion in Section \ref{subsection-basic-setting}), this gives the following result.

\begin{coro}
Let $\mu$ be a measure, and let $\cA$ be a van Hove sequence along which the autocorrelation $\gamma$ of $\mu$ exists. Then, $\gamma$ is positive
definite. In particular, its Fourier transform  $\widehat{\gamma}$
exists and is a positive measure. \qed
\end{coro}

In the situation of the corollary, we refer to the positive measure $\widehat{\gamma}$ as
the \textit{diffraction}\index{diffraction!diffraction~measure} or \textit{diffraction measure} of $\mu$ (with respect to $\cA$).

\medskip
Let us now recall the concept of Fourier--Bohr coefficients.

\begin{definition}[Fourier--Bohr coefficient of function]
Let a function $f \in \mathcal{L}^1_{loc}(G)$, a F\o lner sequence $\cA$ and a character $\chi \in \widehat{G}$ be given. If the limit
\[
\lim_{n\to\infty} \frac{1}{|A_n|} \int_{A_n}
\overline{\chi(t)} f(t)\, \dd t \,
\]
exists we call it the
 \textit{Fourier--Bohr coefficient}\index{Fourier--Bohr~coefficient!Fourier--Bohr~coefficient~of~function}  of $f$ at $\chi$. We then also say that the Fourier--Bohr coefficient of $f$ at $\chi$ exists and denote the limit by
$a_{\chi}^\cA(f)$.  \nomenclature{$a_{\chi}^\cA(f)$}{Fourier--Bohr coefficient of $f$ at $\chi$}

If $(A_n)$ can be replaced by $A_n + s_n$ for any sequence $(s_n)$ in $G$, we say that the \textit{Fourier--Bohr coefficient of $f$ exists uniformly on $G$ (along $\cA$)}.       \exend
\end{definition}

\begin{definition}[Fourier--Bohr coefficient of measure]
Given a measure $\mu$, a van Hove sequence $\cA$ and a character $\chi \in \widehat{G}$, we say that the \textit{Fourier--Bohr
coefficient}\index{Fourier--Bohr~coefficient!Fourier--Bohr~coefficient~of~measure} $a_{\chi}^\cA(\mu)$ of the measure  $\mu$ on $G$  exists at $\chi \in \widehat{G}$
if the following limit exists: \nomenclature{$a_{\chi}^\cA(\mu)$}{Fourier--Bohr coefficient of $\mu$ at $\chi$}
\[
a_{\chi}^{\cA}(\mu)=\lim_{n\to\infty} \frac{1}{|A_n|} \int_{A_n}
\overline{\chi(t)}\, \dd \mu(t) \,.
\]
If $(A_n)$ can be replaced by $A_n + s_n$ for any  sequence $(s_n)$ in $G$, we say that the \textit{Fourier--Bohr coefficient of $\mu$ exists uniformly on $G$ (along $\cA$)}.      \exend
\end{definition}

We complete the section by discussing the connection between the
Fourier--Bohr coefficients of a measure $\mu \in \cM^\infty(G)$ and
the Fourier--Bohr coefficients of $\mu*\varphi$ for $\varphi \in
\Cc(G)$. We first need the following Lemma.

\begin{lemma}
Let $\cA$ be a van Hove sequence, $\mu \in \cM^\infty(G), \varphi \in \Cc(G)$ and $\chi \in \widehat{G}$. Set $K := \supp(\varphi)$. Then, for all $s \in G$, we have
\begin{displaymath}
\left| \left( \int_{s+A_n} (\varphi* \mu)(t)\, \overline{\chi(t)}\, \dd t
\right) - \big( \widehat{\varphi}(\{ \chi \}) \big)  \int_{s+A_n}
\overline{\chi(t)}\, \dd  \mu(t) \right| \leq \bigl\| |\varphi|* |\mu|
\bigr\|_\infty  |\partial^K(A_n)| \,.
\end{displaymath}
\end{lemma}
\begin{proof}
By a standard application of Fubini's theorem, we have
\begin{align*}
D:=&\left| \left( \int_{s +A_n} (\varphi* \mu)(z)\, \overline{\chi(z)}\, \dd z \right) - \left( \widehat{\varphi}(\{ \chi \}) \right)  \int_{s+A_n} \overline{\chi(t)}\, \dd  \mu(t) \right|   \\
&=\left| \int_{G}\int_G \big( (1_{x+A_n}(r)-1_{x+A_n}(t))\,
\varphi(r-t) \big)\overline{\chi(r)}\,  \dd r \,\dd  \mu(t) \right| \,.
\end{align*}
A simple  computation shows that $(1_{s+A_n}(z)-1_{s+A_n}(t)) \,
\varphi(r-t)=0$ unless $r \in G$ satisfies $r\in
\partial^{K}(s+A_n)=s+\partial^{K}(A_n)$. Therefore, we obtain
\[
D \leq \int_{G}\int_{s+\partial^{K}(A_n)} | \varphi(r-t)
|\,  \dd r\, \dd |\mu | (t)\leq  \bigl\| |\varphi|*
|\mu| \bigr\|_\infty |\partial^{K}(A_n)| \,.
\]
This finishes the proof.
\end{proof}

We note an immediate consequence of the previous lemma.

\begin{coro}\label{FB measure relations}
Let $\cA$ be a van Hove sequence, $\mu \in \cM^\infty(G)$ and $\chi \in
\widehat{G}$ be given.
\begin{itemize}
  \item [(a)] If $\varphi \in \Cc(G)$ is a function such that $\widehat{\varphi}(\chi) \neq 0$,
and the Fourier--Bohr coefficient $a_{\chi}^{\cA}(\mu*\varphi)$
exists (uniformly on $G$), then the Fourier--Bohr coefficient
$a_{\chi}^{\cA}(\mu)$ exists (uniformly on $G$).
  \item [(b)]
If the Fourier--Bohr coefficient $a_{\chi}^{\cA}(\mu)$ exists
(uniformly on $G$), then for all $\varphi \in \Cc(G)$, the
Fourier--Bohr coefficient $ a_{\chi}^{\cA}(\mu*\varphi)$ exists
(uniformly on $G$) and satisfies the identity
\[
a_{\chi}^{\cA}(\mu*\varphi)=\widehat{\varphi}(\chi)
a_{\chi}^{\cA}(\mu) \,.       \tag*{$\qed$}
\]
\end{itemize}
\end{coro}

By combining Corollary~\ref{FB measure relations} with
Proposition~\ref{prop: amenable}, we also obtain the next result.

\begin{coro}\label{cor unif FB}
Let $\cA$ be a van Hove sequence, $\mu \in \cM^\infty(G)$ and $\chi \in
\widehat{G}$  be given. Then, the Fourier--Bohr coefficient $a_\chi^{\cA}(\mu)$ exists uniformly in $s \in G$ if and only if the Fourier--Bohr coefficient
$a_\chi^{\cB}(\mu)$ exists with respect to any van Hove sequence $\cB$.\qed
\end{coro}

Because of this result, when the Fourier--Bohr coefficient $a_{\chi}^{\cA}(f)$ or $a_{\chi}^\cA(\mu)$ of a function or measure exists uniformly, we will simply denote it by $a_{\chi}(f)$ or $a_\chi(\mu)$.

%
%
%

\chapter{Besicovitch and Weyl seminorms and the associated spaces} \label{ch:besweyl}

We introduce the Besicovitch and Weyl seminorms $\|\cdot\|_{\be,p,\cA}$ and $\|\cdot\|_{\we,p,\cA}$
and the corresponding spaces $\BL_{\cA}^p (G)$ and $\WL_{\cA}^p (G)$ of locally integrable functions for which these seminorms are finite.
Additionally, we study the relationship and properties of these two seminorms. We show that the space $(BL^p_{\cA}(G), \| \cdot \|_{\be,p,\cA})$ is complete and that the translation action induces a unitary map on the completion $BC^p_{\cA}(G)$ of $\Cu(G)$ in $(BL^p_{\cA}(G), \| \cdot \|_{\be,p,\cA})$.
We complete the chapter by introducing the notions of $\cN$-representations and $\cA$-representations and their diffraction theory.

\section{Besicovitch and Weyl seminorms and the associated
spaces}\label{sub:semi-norms} One basic tool in our considerations
will be two seminorms and the corresponding spaces discussed in this
section.

\medskip

\begin{definition}[Besicovitch seminorm]
Let $1 \leq p < \infty$, and let $\cA$ be a F\o lner sequence. For $f
\in \mathcal{L}^p_{loc}(G)$, define the \textit{Besicovitch $p$-seminorm} $\| f \|_{\be,p,\cA}$ \index{seminorm!Besicovitch~seminorm~for~functions} via \nomenclature{$\Vert f \Vert_{\be,p,\cA}$}{Besicovitch seminorm for functions}
\[
\| f \|_{\be,p,\cA}:= \limsup_{n\to\infty} \left(\frac{1}{|A_n|} \int_{A_n} |f(t)|^p \dd t \right)^\frac{1}{p}
\]
and the \textit{Weyl $p$-seminorm} $\| f \|_{\we,p,\cA}$ \index{seminorm!Weyl~seminorm~for~functions} via \nomenclature{$\Vert f \Vert_{\we,p,\cA}$}{Weyl seminorm for functions}
\[
\| f \|_{\we,p,\cA}:=  \limsup_{n\to\infty} \left(\sup_{x \in G} \frac{1}{|A_n|} \int_{x+A_n} |f(t)|^p \dd t \right)^\frac{1}{p} \,.
\]
We set
\[
\BL_{\cA}^p (G) :=\{ f \in \mathcal{L}^p_{loc} (G) : \| f \|_{\be,p,\cA} < \infty\}
\]
\nomenclature{$BL_{\cA}^p (G)$}{space of locally $p$-integrable functions with finite Besicovitch seminorm} and \nomenclature{$WL_{\cA}^p (G)$}{space of locally $p$-integrable functions with finite Weyl seminorm}
\[
\WL_{\cA}^p (G):=\{ f \in \mathcal{L}^p_{loc} (G) : \|f \|_{\we,p,\cA} < \infty\} \,.                     \tag*{$\Diamond$}
\]
\end{definition}

We will often use the following obvious estimate.

\begin{lemma} For all $f \in L^\infty(G)$, all $1 \leq p < \infty$ and all F\o lner sequence $\cA$, we have
\[
\| f \|_{\be,p,\cA} \leq \|f\|_{\we,p,\cA}  \leq \| f\|_\infty \,.
\]
In particular,
\[
L^\infty(G) \subseteq \WL_{\cA}^p (G) \subseteq \\BL_{\cA}^p (G) \,.    \tag*{$\qed$}
\]
\end{lemma}

\smallskip

Next, we will show that $\| \cdot \|_{\be,p,\cA}$ and $\| \cdot \|_{\we,p,\cA}$ are, indeed, seminorms.

\begin{lemma}[Basic properties of $\|\cdot \|_{\be,p,\cA}$ and $\| \cdot \|_{\we,p,\cA}$]
\label{L1} Let $1 \leq p < \infty$, and let a F\o lner sequence $\cA$ on
$G$ be given.
\begin{itemize}
\item[(a)] The maps  $\| \cdot \|_{\be,p,\cA}$ and $\| \cdot \|_{\we,p,\cA}$
define seminorms on $\BL_{\cA}^p (G)$ and $\WL_{\cA}^p (G)$,
respectively.
\item[(b)] For all $f \in \BL_{\cA}^p (G)\cap L^\infty(G)$ and for all $t \in G$, we have
\[
\| f \|_{\be,p,\cA} =\| \tau_t  f \|_{\be,p,\cA} \,.
\]
\item[(c)]  For all $f \in \WL_{\cA}^p (G)$ and all $t \in G$, we have
\[
\| f\|_{\we,p,\cA} =\| \tau_t  f \|_{\we,p,\cA}\,.
\]
\end{itemize}
\end{lemma}
\begin{proof}
(a) The only thing which is not obvious is the triangle inequality.
It follows immediately from the triangle inequality for $\mathcal{L}^p(G)$ as
we have, for each $f,g \in \mathcal{L}^p_{loc}(G)$ and all $n\in\NN,$
\[
\left(  \int_{A_n} |f(t)+g(t) |^p\, \dd
t \right)^\frac{1}{p} \leq  \left(
\int_{A_n} |f(t) |^p\, \dd t \right)^\frac{1}{p}
+ \left(  \int_{A_n} |g(t) |^p\, \dd t
\right)^\frac{1}{p} \,.
\]
Dividing every term by $|A_n|^\frac{1}{p}$ and taking the limsup gives the desired inequality. The proof for the uniform limit is identical.

\medskip

\noindent (b) A short computation gives
\begin{align*}
{}
    &\left|  \frac{1}{|A_m|} \int_{A_m} | f(z) |^p\, \dd z
      - \frac{1}{|A_m|} \int_{A_m} | \tau_t f(s) |^p \,
      \dd s \right| \\
    &\phantom{======}= \left|  \frac{1}{|A_m|}  \int_{A_m} | f(s)
      |^p \,\dd t - \frac{1}{|A_m|}  \int_{A_m} |f(s -t)
       |^p\, \dd t \right| \\
    &\phantom{======}= \left|  \frac{1}{|A_m|}  \int_{A_m} | f(z)|^p\,
      \dd z - \frac{1}{|A_m|}  \int_{t+A_m} |f(z)|^p\, \dd z \right| \\
    &\phantom{======}= \left|  \frac{1}{|A_m|}  \int_{A_m\, \triangle\, (t+A_m)}
      | f(z) |^p \, \dd z \right| \leq \frac{|A_m\, \triangle\, (t+A_m)|}{|A_m|} \| f\|^p_\infty \,.
\end{align*}
Therefore, by the F\o lner condition, we get
\[
\lim_{n\to\infty} \left( \frac{1}{|A_n|} \int_{A_n} | f(z) |^p\, \dd
z - \frac{1}{|A_n|} \int_{A_n}  | \tau_t f(s) |^p\, \dd s\right) =0
\]
and, hence,
\[
\limsup_{n\to\infty}  \frac{1}{|A_n|} \int_{A_n} | f(z) |^p\, \dd z
=\limsup_{n\to\infty} \frac{1}{|A_n|} \int_{A_n}  | \tau_t f(s) |^p
\, \dd s \,.
\]



\medskip

\noindent (c) follows immediately from the definition. Indeed,
\begin{align*}
\| \tau_t f \|_{\we,p,\cA}^p
    &= \limsup_{m\to\infty} \sup_{x \in G} \frac{1}{|A_m|}\int_{x+A_m}
      |\tau_t f(s)|^p\, \dd s\\
    &= \limsup_{m\to\infty} \sup_{x \in G} \frac{1}{|A_m|}\int_{t+x+A_m}
      |f(s)|^p \, \dd s=\|f \|_{\we,p,\cA}^p \,.
\end{align*}
This finishes the proof.
\end{proof}

Let us now cover an example which shows that, in general, $BL^p_{\cA}(G)$ is not invariant under taking translates.

\begin{example}[$BL^1_{\cA}(\RR)$ is not translation invariant]\label{rem:not-translation-invariant}

Choose two strictly increasing sequences $(a_n)$ and $(b_n)$ which satisfy the following three properties:
\begin{itemize}
\item $b_{n+1} > b_n +n$ for all $n$,
\item $\lim_{N \to \infty} \frac{\sum_{k=1}^N k \cdot a_k }{b_{N+1}}=0$,
\item $\lim_{N \to \infty} \frac{a_{N}}{b_{N}} = \infty$.
\end{itemize}
Let $A_n = [-b_n,b_n]$, and let $f: \RR \to \RR$ be defined by
\[
f(x)=
\left\{
\begin{array}{cc}
  a_n & \mbox{ if }b_n\leq |x| \leq b_n + n , \\
  0& \mbox{ otherwise} \,.
\end{array}
\right.
\]
Then,
\begin{align*}
  \| f\|_{\be,1,\cA} &=\limsup_{n\to\infty} \frac{1}{2b_n} \sum_{k=1}^{n-1}k a_{k}=0 \,.
\end{align*}
Moreover, for all $t \geq  1$, we have
\begin{align*}
\| \tau_t f\|_{\be,1,\cA}
&=\limsup_{n\to\infty} \frac{1}{2b_n} \int_{-b_n}^{b_n} |f(x-t)| \dd x = \limsup_{n\to\infty} \frac{1}{2b_n} \int_{-b_n+t}^{b_n+t} |f(u)| \dd u \\
&\geq  \limsup_{n\to\infty} \frac{1}{2b_n} \int_{b_n}^{b_n+1} |f(u)| \dd u= \limsup_{n\to\infty} \frac{a_n}{2b_n}= \infty \,.
\end{align*}
Finally, for all $0 < t \leq 1$, we have
\begin{align*}
\| \tau_t f\|_{\be,1,\cA}
&= \limsup_{n\to\infty} \frac{1}{2b_n} \int_{-b_n+t}^{b_n+t} |f(u)| \dd u
\geq  \limsup_{n\to\infty} \frac{1}{2b_n} \int_{b_n}^{b_n+t} |f(u)| \dd u\\
  &= \limsup_{n\to\infty} \frac{t \cdot a_n}{2b_n}= \infty \,.
\end{align*}
Similarly, one can show that $\| \tau_t f \|_{\be,1,\cA}=\infty$ for all $t<0$.

Therefore, $f \in BL^1_{\cA}(\RR)$ and
\[
\tau_t f \notin BL^1_{\cA}(G) \qquad \text{ for all } t \neq 0 \,.    \tag*{$\Diamond$}
\]
\end{example}

\begin{remark}
For the usual van Hove sequence $A_n=[-n,n]^d$ in $G=\RR^d$ or $G=\ZZ^d$, the phenomenon from Example~\ref{rem:not-translation-invariant} does not occur.
Indeed, in this case, for each $t \in G$, there exists a sequence $\big(k_n(t)\big)$ such that
\begin{equation}
t+A_n \subseteq A_{n+k_n(t)}  \qquad \forall n \label{van-str1}
\end{equation}
and
\begin{equation}
\lim_{n\to\infty} \frac{|A_{n+k_n(t)}|}{|A_n|} = 1 \,. \label{van-str2}
\end{equation}
 Hence, for all $f \in \cL^p_{loc}(G)$, all $t \in G$ and all $n$, we have
\begin{align*}
 \frac{1}{|A_n|} \int_{A_n} | \tau_t f(z) |^p\, \dd z &= \frac{1}{|A_n|} \int_{A_n} |  f(z-t) |^p\, \dd z  = \frac{1}{|A_n|} \int_{t+A_n} |  f(z) |^p\, \dd z \\
    &\leq  \frac{1}{|A_n|} \int_{A_{n+k_n(t)}} |  f(z) |^p\, \dd z \\
    &=  \frac{|A_{n+k_n(t)}|}{|A_n|}  \cdot \frac{1}{|A_{n+k_n(t)}|} \int_{A_{n+k_n(t)}} | f(z)|^p\, \dd z  \,.
\end{align*}
This immediately implies
\[
\| \tau_t f \|_{\be,p,\cA} \, \leq \, \| f \|_{\be,p,\cA}  \,.
\]
Replacing $t$ by $-t$ and $f$ by $\tau_t f$, we also find the reverse inequality.
Therefore, if the van Hove sequence $(A_n)$ satisfies \eqref{van-str1} and \eqref{van-str2}, the seminorm $\| \cdot \|_{\be,p,\cA} $ is invariant under translation.

The van Hove sequence we used in Example~\ref{rem:not-translation-invariant} fails \eqref{van-str2}.
\exend
\end{remark}

\begin{remark}
For $G = \RR$, the spaces $BL^p_{\cA}(G)$ were studied by Marcinkiewicz \cite{Mar} under the name `Besicovitch space'. Later, they were called Marcinkiewicz spaces, see e.g. \cite{CohLos}.       \exend
\end{remark}

Next, we give some standard inequalities involving the mean and
square mean. Such estimates were used in \cite{Eb,ARMA} for weakly
almost periodic functions.

\begin{lemma}[Inequalities for Besicovitch seminorms] \label{lemma C-S} Let $1 < p <\infty$ be arbitrary, and let
 $q$ be the conjugate exponent of $p$ (i.e. $1/p + 1/q = 1$). Let $\cA$ be a F\o lner
 sequence on $G$.
\begin{itemize}
\item [(a)]  For each $f\in BL^p_{\cA}(G)$ and $g\in BL^q_{\cA}(G) (G)$,
the H\"older inequality \index{H\"older~inequality!for~Besicovitch~seminorm}
\[
\| f g \|_{\be,1,\cA}\leq \| f \|_{\be,p,\cA} \|
g\|_{\be,q,\cA}
\]
holds. In particular, for each $f,g \in BL^2_{\cA}(G)$, the Cauchy--Schwarz inequality \index{Cauchy--Schwarz~inequality!for~Besicovitch~seminorm}
\[
\| f g \|_{\be,1,\cA}\leq \| f \|_{\be,2,\cA} \|
g\|_{\be,2,\cA}
\]
holds.
\item[(b)] For all $1 \leq p <\infty$ we have:
\[
\| f \|_{\be,1,\cA}\leq \| f \|_{\be,p,\cA} \ .
\]
\item [(c)]  For each $f \in BL^1_{\cA}(G)$ and $g\in L^\infty (G)$,
we have
\[
\| f g \|_{\be,1,\cA}\leq \| f \|_{\be,1,\cA} \|g\|_\infty \,.
\]
\item [(d)]  For each $f \in \mathcal{L}^1_{loc}(G) \cap L^\infty(G)$, we have
\[
\| f \|_{\be,p,\cA}^p \leq \|f \|_\infty^{p-1}\,\| f \|_{\be,1,\cA} \,.
\]
\end{itemize}
\end{lemma}
\begin{proof}
 (a) By  H\"older's inequality, we have
\begin{displaymath}
\frac{1}{|A_n|} \int_{A_n} \left| f(x) g(x) \right|\ \dd x  \leq
\left(  \frac{1}{|A_n|} \int_{A_n} \left| f(x) \right|^p\ \dd x
\right)^\frac{1}{p} \left(  \frac{1}{|A_n|} \int_{A_n} \left| g(x) \right|^q\ \dd x
\right)^\frac{1}{q}  \,.
\end{displaymath}
Taking $\limsup_{n\to\infty}$ yields the first
statement.

The Cauchy--Schwarz inequality follows by setting $p=q=2$.

\noindent (b) follows from (a) by setting $g=1$.

\noindent (c) is obvious.

\noindent (d) The inequality
\begin{displaymath}
\frac{1}{|A_n|} \int_{A_n} \left| f(x) \right|^p\ \dd x  \leq
\|f \|_\infty^{p-1}  \left( \frac{1}{|A_n|} \int_{A_n} \left| f(x)
\right|\ \dd x \right)
\end{displaymath}
is obvious. Taking $\limsup_{n\to\infty}$ on both sides completes
the proof.
\end{proof}

Exactly the same proofs carry over when $\| \cdot \|_{\be,p,\cA}$ is
replaced by $\| \cdot \|_{\we,p,\cA}$, giving us the same statements for the Weyl norm.

\begin{lemma}[Inequalities for Weyl seminorms] \label{lemma C-S2} Let $1 < p <\infty$ be arbitrary, and let
 $q$ be the conjugate exponent of $p$ (i.e. $1/p + 1/q = 1$). Let $\cA$ be a F\o lner
 sequence on $G$.
\begin{itemize}
\item [(a)]  For each $f\in WL^p_{\cA}(G)$ and $g\in WL^q_{\cA}(G) (G)$,
the H\"older inequality \index{H\"older~inequality!for~Weyl~seminorm} holds:
\[
\| f g \|_{\we,1,\cA}\leq \| f \|_{\we,p,\cA} \|
g\|_{\we,q,\cA} \,.
\]
In particular, for each $f,g \in WL^2_{\cA}(G)$, the Cauchy--Schwarz inequality \index{Cauchy--Schwarz~inequality!for~Weyl~seminorm}
\[
\| f g \|_{\we,1,\cA}\leq \| f \|_{\we,2,\cA} \|
g\|_{\we,2,\cA} \,.
\]
holds.
\item[(b)] For all $1 \leq p <\infty$ we have:
\[
\| f \|_{\we,1,\cA}\leq \| f \|_{\we,p,\cA} \ .
\]
\item [(c)]  For each $f \in WL^1_{\cA}(G)$ and $g\in L^\infty (G)$,
we have
\[
\| f g \|_{\we,1,\cA}\leq \| w \|_{\be,1,\cA} \|g\|_\infty \,.
\]
\item [(d)]  For each $f \in \mathcal{L}^1_{loc}(G) \cap L^\infty(G)$, we have
\[
\| f \|_{\we,p,\cA}^p \leq \|f \|_\infty^{p-1}\,\| f \|_{\we,1,\cA} \,.    \tag*{$\qed$}
\]
\end{itemize}
\end{lemma}

We note the following consequence of Lemma~\ref{lemma C-S} and Lemma~\ref{lemma C-S2}.

\begin{lemma}[Inclusion of spaces]\label{lemma norm inequality}
Let $1 \leq p < q <\infty$ be arbitrary, and
 let $\cA$ be a F\o lner sequence on $G$. Then, for all $f \in \mathcal{L}^q_{loc} (G)$, we have
 \begin{align*}
 \| f \|_{\be,p,\cA} \, &\leq \,  \| f \|_{\be,q,\cA}   \,, \\
 \| f \|_{\we,p,\cA} \,&\leq\,  \| f \|_{\we,q, \cA}  \,.
 \end{align*}
 In particular,
 \begin{align*}
 \BL_{\cA}^q (G) \,& \subseteq \, \BL_{\cA}^p (G)  \,, \\
 \WL_{\cA}^q (G) \, &\subseteq \, \WL_{\cA}^p (G) \,.
 \end{align*}
\end{lemma}
\begin{proof}
From Lemma~\ref{lemma C-S} we obtain
\begin{displaymath}
 \| f \|_{\be,p,\cA}^p=  \| |f|^p  \|_{\be,1,\cA}  \leq  \| |f|^p  \|_{\be,\frac{q}{p}, \cA} = \|f\|_{\be,q,\cA}^p \,.
\end{displaymath}
The statement for $\|\cdot\|_{\we,p,\cA}$ follows analogously.
\end{proof}

We now come to the crucial completeness result. The result is
certainly known. As we could not find it in the form stated here, we
include a proof.

\begin{theorem}[Completeness of $BL^p_{\cA}(G)$]\label{thm:completeness}
For each $p\geq 1$ and every van Hove sequence $\cA$ on $G$, the
space  $(\BL_{\cA}^p (G), \| \cdot   \|_{\be,p,\cA})$  is complete.
\end{theorem}
\begin{proof} We distinguish two cases.

\smallskip

\noindent Case 1: $G$ is compact. It suffices to show that
$BL^p_{\cA}(G)$ is the space $\mathcal{L}^p (G)$ of measurable $f :
G\longrightarrow \CC$ with $\int_G |f|^p\, \dd t <\infty$ and
$\|f\|_{\be,p,\cA} = \left(\int_G |f|^p\, \dd t\right)^{1/p}$.

Via a standard renormalisation, we can assume without loss of
generality that $|G|=1$. Setting $K=G$ in the definition of the van
Hove sequence, we see that $G \backslash A_n \subseteq
\partial^K(A_n)$, and hence, by the definition of the van Hove
sequence, we get
\begin{displaymath}
\limsup_{n\to\infty} |G \backslash A_n| \leq \limsup_{n\to\infty}
\frac{|G \backslash A_n|}{|A_n|}\leq \limsup_{n\to\infty}
\frac{\partial^K(A_n)}{|A_n|}  =0 \,.
\end{displaymath}
From here, it follows immediately that, for all $f \in \mathcal{L}^p(G)$, we
have
\begin{displaymath}
\| f \|_{\be,p,\cA}= \lim_{n\to\infty} \left( \frac{1}{|A_n|}
\int_{A_n} |f(t)|^p\, \dd t \right)^\frac{1}{p} = \| f \|_ p\,.
\end{displaymath}
We directly deduce that $(\BL_{\cA}^p(G), \| \cdot
\|_{\be,p,\cA}) =(\mathcal{L}^p(G), \| \cdot \|_p)$ and the claim
follows.

\medskip

\noindent  Case 2: $G$ is not compact. Hence, $|G| = \infty$ holds.
 By \cite[Prop.~2.2]{CohLos}, it suffices to find, for each $f\in
BL^p_{\cA} (G)$ with $\|f\|_{\be,p,\cA}>0$, an $f^* \in BL^p_{\cA}(G)$
satisfying
\begin{align*}
\|f - f^*\|_{\be,p,\cA} \,&=0 \, \qquad \mbox{ and  } \\
\sup_{n\in\NN}
\frac{1}{|A_n|} \left( \int_{A_n} |f^*(t)|^p\, \dd t\right)^{1/p} \,&\leq\, 2\,
\|f\|_{\be,p,\cA} \,.
\end{align*}
To do so, choose an $N\in\NN$ large enough
such that
\[
\left(\frac{1}{|A_n|}\int_{A_n} |f(t)|^p\, \dd t\right)^{1/p}
\leq 2 \|f\|_{\be,p,\cA} \qquad \text{ for all } n\geq N \,.
\]
Next, consider $f^*$ with $f^* = 0$ on the relatively compact set $A_1\cup
A_2 \cdots \cup  A_N$ and $f^* = f$ else. By construction, $f^*$ has
the second desired property. Thus, it remains to show $\|f -
f^*\|_{\be,p,\cA} = 0$. This in turn follows immediately once we show
$\lim_{n\to\infty} |A_n|= \infty$.  Let  $m>0$ be given. Since $G$ is not
compact, $|G|=\infty$. Therefore, we can find a compact set $K$ such that $0 \in
K$ and  $|K| >2 m$. Since $0 \in K$, we immediately obtain
$A_n \subseteq A_n+K \subseteq A_n \cup
\partial^{K}(A_n)$ for all $n\in\NN$.
The van Hove property then gives $\lim_{n\to\infty} \frac{|A_n+K| }{
|A_n|} =1$. Therefore, there exists some $M\in\NN$ such
that, for all $n
>M$, we have $\frac{|A_n+K| }{ |A_n|}<2$. Since $A_n$ is non-empty,
there exists some $t_n \in A_n$. Then, we have
\begin{displaymath}
|A_n| > \frac{1}{2}\,|A_n+K| > \frac{1}{2}\, |t_n+K| = \frac{|K|}{2}
> m
\end{displaymath}
for each $n\geq M$, which finishes the proof.
\end{proof}

\begin{remark} The completeness of the so-called generalized Marcinkiewicz spaces is usually easy to establish when dealing with averaging sequences, see \cite[Examples 2.3. (i)]{CohLos}. The situation is more complicated when one uses averaging nets. Indeed, \cite[Example 2.4]{CohLos} and \cite[Sect~5]{Davis} provide examples of van Hove nets on $\NN$ and $\RR$, respectively, for which the corresponding Marcinkiewicz spaces is not complete.
\end{remark}

\medskip
Next, we extend the Besicovitch and Weyl seminorms to measures.

\begin{definition}[Besicovitch and Weyl seminorm for measures]
Let $\cA$ be a van Hove sequence. For a measure $\mu$, we define the \textit{Besicovitch seminorm} $\|  \mu \|_{\be,\cA}$ \index{seminorm!Besicovitch~seminorm~for~measures} via
\[
\| \mu \|_{\be,\cA}:= \limsup_{n\to\infty} \frac{|\mu|(A_n)}{|A_n|}
\]\nomenclature{$\Vert \mu \Vert_{\be,\cA}$}{Besicovitch seminorm for measures}
and the \textit{Weyl seminorm} $\| \mu \|_{\we,p,\cA}$ \index{seminorm!Weyl~seminorm~for~measures} via\nomenclature{$\Vert \mu \Vert_{\we,\cA}$}{Weyl seminorm for measures}
\[
\| \mu \|_{\we,\cA}:= \limsup_{n\to\infty}\sup_{t \in G} \frac{|\mu|(t+A_n)}{|A_n|} \,.     \tag*{$\Diamond$}
\]
\end{definition}

It is easy to see that these are indeed seminorms on the subspaces of measures where they are finite. In particular, these are seminorms on $\cM^\infty(G)$ \cite{Martin2}.

The following result is immediate.

\begin{lemma}\label{lem-bes-ine-measures}
Let $\cA$ be a van Hove sequence.
\begin{itemize}
  \item[(a)] For all $f \in \mathcal{L}^1_{loc}(G)$, we have
  \begin{align*}
  \| f\|_{\be,1,\cA} \,&=\, \| f \theta_G \|_{\be,\cA} \,,\\
    \| f\|_{\we,1,\cA} \,&=\, \| f \theta_G \|_{\we,\cA} \,.
\end{align*}
  \item[(b)] For all $\mu \in \cM^\infty(G)$ and $\varphi \in \Cc(G)$, we have
  \begin{align*}
    \| \mu*\varphi \|_{\be,1\,\cA} \,&\leq\, \| \varphi\|_1 \|\mu\|_{\be,\cA} \,,\\
     \| \mu*\varphi \|_{\we,1\,\cA} \, &\leq\, \| \varphi\|_1 \|\mu\|_{\we,\cA} \,.
  \end{align*}
\end{itemize}
\end{lemma}
\begin{proof}
(a) is obvious.

(b) One has
\begin{align*}
  \frac{1}{|A_n|} \int_{x+A_n} \big|(\mu*\varphi)(t) \big|\, \dd t
  & =  \frac{1}{|A_n|} \int_{x+A_n} \left|\int_{G} \varphi(y-s)\, \dd \mu(s) \right| \dd t \\
  &\leq  \frac{1}{|A_n|} \int_{G} \int_{G} 1_{x+A_n}(t)  \big|\varphi(t-s)\big|\, \dd |\mu|(s)\, \dd t \\
   &=  \frac{1}{|A_n|} \int_{G} \int_{G} 1_{x+A_n}(t)  \big|\varphi(t-s)\big|\,  \dd t\, \dd |\mu|(s) \,.
\end{align*}
Now, a standard van Hove argument gives
\begin{align*}
\lim_{n\to\infty} \left(  \frac{1}{|A_n|} \int_{G} \int_{G} \right. &1_{x+A_n}(t)  |\varphi(t-s)|\,  \dd t\, \dd |\mu|(s)   \\
&\left.- \frac{1}{|A_n|} \int_{G} \int_{G} 1_{x+A_n}(s)  |\varphi(t-s)|\,  \dd t\, \dd |\mu|(s) \right) =0 \,.
\end{align*}
This implies
\begin{align*}
\| \mu*\varphi \|_{\be,1,\cA}
&= \limsup_{n\to\infty} \frac{1}{|A_n|} \int_{A_n} |(\mu*\varphi)(t)|\, \dd t \\
&\leq \limsup_{n\to\infty} \frac{1}{|A_n|} \int_{G} \int_{G} 1_{A_n}(t)  |\varphi(t-s)|\,  \dd t\, \dd |\mu|(s)\\
&=\limsup_{n\to\infty} \frac{1}{|A_n|} \int_{G} \int_{G} 1_{A_n}(s)  |\varphi(t-s)|\,  \dd t\, \dd |\mu|(s)\\
&=\|\varphi\|_1 \|\mu\|_{\be, \cA} \,
\end{align*}
and
\begin{align*}
\| \mu*\varphi \|_{\we,1,\cA}
&= \limsup_{n\to\infty} \sup_{x \in G} \frac{1}{|A_n|} \int_{x+A_n} |(\mu*\varphi)(t) |\, \dd t \\
&\leq \limsup_{n\to\infty} \sup_{x \in G}  \frac{1}{|A_n|} \int_{G} \int_{G} 1_{x+A_n}(t)  |\varphi(t-s)|\,  \dd t\, \dd |\mu|(s)\\
&=\limsup_{n\to\infty} \sup_{x \in G} \frac{1}{|A_n|} \int_{G} \int_{G} 1_{x+A_n}(s)  |\varphi(t-s)|\,  \dd t\, \dd |\mu|(s)\\
&=\|\varphi\|_1 \|\mu\|_{\we, \cA} \,.
\end{align*}
This finishes the proof.
\end{proof}

We complete this subsection discussing the relationship between Fourier--Bohr coefficients and the Besicovitch seminorm. The first result is immediate.

\begin{lemma}\label{lemma-FB-bound-bes} Let $\cA$ be a van Hove sequence.
\begin{itemize}
  \item[(a)] Let $p \geq 1$. If $f \in BL^p_{\cA}(G)$ is such that $a_{\chi}^\cA(f)$ exists, then
  \[
  \left| a_{\chi}^\cA(f)\right| \leq \|f \|_{\be,p, \cA} \, .
  \]
  \item[(b)] If $\mu \in \cM^\infty(G)$ is such that $a_{\chi}^\cA(\mu)$ exists, then
  \[
  \left| a_{\chi}^\cA(\mu)\right| \leq \| \mu \|_{\be, \cA} \, .
  \]
\end{itemize}
\end{lemma}
\begin{proof}
(a) follows from Lemma~\ref{lemma C-S} applied with $g= \chi$, and (b) is obvious.
\end{proof}

We can next prove the following result on continuity of the Fourier--Bohr coefficients, which we will use often in the book.

\begin{prop}[Continuity of Fourier--Bohr coefficients]\label{prop-aprox-FB} Let $\cA$ be a F\o lner sequence on $G$.
\begin{itemize}
  \item[(a)] Let $f_n, f \in BL^1_\cA(G)$ be such that
  \[
  \lim_{n\to\infty} \|f_n-f \|_{\be,1,\cA} =0 \,.
  \]
  If the Fourier--Bohr coefficients $a_\chi^\cA(f_n)$ exist, then $a_\chi^\cA(f)$ exists and
  \[
  a_\chi^\cA(f)= \lim_{n\to\infty} a_\chi^\cA(f_n) \,.
  \]
  \item[(b)]  Let $f_n, f \in WL^1_\cA(G)$ be such that
  \[
  \lim_{n\to\infty} \|f_n-f \|_{\we,1,\cA} =0 \,.
  \]
  If the Fourier--Bohr coefficients $a_\chi(f_n)$ exist uniformly, then $a_\chi(f)$ exist uniformly and
  \[
  a_\chi(f)= \lim_{n\to\infty} a_\chi(f_n) \,.
  \]
\end{itemize}
\end{prop}
\begin{proof}
(a) By Lemma~\ref{lemma-FB-bound-bes}, we have
\[
\left|a_\chi^\cA(f_n)-a_\chi^\cA(f_m) \right| \leq \| f_m-f_n \|_{\be,1,\cA}
\]
for all $m,n$. It follows that $a_\chi^\cA(f_n)$ is a Cauchy sequence, and hence convergent to some $a \in \CC$. We show that $a_{\chi}^\cA(f)$ exists and is equal to $a$.

Let $\eps >0$. Then, there exists some $N$ such that
\begin{align*}
 \left|a_\chi^\cA(f_n)-a \right| \, &< \, \frac{\eps}{3} \\
  \| f_n -f \|_{\be,1,\cA} \, &< \, \frac{\eps}{4}
\end{align*}
hold for all $n>N$. Fix one $n > N$.
Since
\[
\limsup_{m\to\infty} \frac{1}{|A_m|} \int_{A_m} \left| f_n(t) -f(t) \right| \dd t < \frac{\eps}{4} \,,
\]
there exists some $M_1$ such that, for all $m >M_1$, we have
\[
\frac{1}{|A_m|} \int_{A_m} \left| f_n(t) -f(t) \right| \dd t < \frac{\eps}{3} \,.
\]
Also, since $a_\chi^\cA(f_n)$ exists, there exists some $M_2$ such that
\[
  \left| \frac{1}{|A_m|} \int_{A_m} \overline{\chi(t)}f_n(t) \dd t -a_\chi^\cA(f_n) \right| < \frac{\eps}{3}
\]
holds for all $m >M_2$. Then, for all $m >M:= \max\{ M_1, M_2\}$, we have
\begin{align*}
\left| \frac{1}{|A_m|} \int_{A_m} \overline{\chi(t)}f(t) \dd t -a \right|
&\leq  \left| \frac{1}{|A_m|} \int_{A_m} \overline{\chi(t)}f(t) \dd t -\frac{1}{|A_m|} \int_{A_m} \overline{\chi(t)}f_n(t) \dd t \right|  \\
&\phantom{XX}+ \left| \frac{1}{|A_m|} \int_{A_m} \overline{\chi(t)}f_n(t) \dd t -a_\chi^\cA(f_n)\right| +\left| a_\chi^\cA(f_n)-a \right| \\
   &\leq  \frac{1}{|A_m|} \int_{A_m} \left| \overline{\chi(t)}(f(t)  - \overline{\chi(t)}f_n(t))  \right| \dd t + \frac{2 \eps}{3} \\
   &= \frac{1}{|A_m|} \int_{A_m} \left| f(t)  - f_n(t)  \right| \dd t + \frac{2 \eps}{3} =\eps \,.
\end{align*}
(b) is proven exactly as (a), with all limits being replaced by uniform limits.
\end{proof}

The following result is a consequence of Proposition~\ref{prop-aprox-FB} and Corollary~\ref{FB measure relations}.

\begin{coro}\label{coro-FB-bound-bes-measure} Let $\mu_n$, $n\in\NN$,  and $ \mu$ be translation bounded measures on $G$,  and let $\cA$ be a van Hove sequence.
 \begin{itemize}
  \item[(a)] If
  \[
  \lim_{n\to\infty} \|\mu_n-\mu \|_{\be, \cA} =0
  \]
  and the Fourier--Bohr coefficients $a_\chi^\cA(\mu_n)$ exist, then $a_\chi^\cA(\mu)$ exists and
  \[
  a_\chi^\cA(\mu)= \lim_{n\to\infty} a_\chi^\cA(\mu_n) \,.
  \]
  \item[(b)] If
  \[
  \lim_{n\to\infty} \|\mu_n-\mu \|_{\we,\cA} =0
  \]
 and the Fourier--Bohr coefficients $a_\chi(\mu_n)$ exist uniformly, then $a_\chi(\mu)$ exists uniformly and
  \[
  a_\chi(\mu)= \lim_{n\to\infty} a_\chi(\mu_n) \,.      \tag*{$\qed$}
  \]
\end{itemize}
\end{coro}

\section{Continuous translation action and invariant subspaces}\label{sub:trans}
\label{sub:trans-sub}
In this section, we establish a (continuous) translation action on certain spaces with the Besicovitch seminorm.
As shown in Example~\ref{rem:not-translation-invariant}, the
Besicovitch seminorm is far from being  invariant under translations
$\tau_t$, $t\in G$,  and it may even be that translates of a
function with finite Besicovitch seminorm do not belong to any
Besicovitch space. Thus, we cannot hope to find a translation
action on the whole space. To remedy this, we will restrict our
attention to a subspace. The various spaces of almost periodic functions appearing in later chapters naturally belong to this subspace.  This part of the theory seems to be new.

\medskip

For $1\leq p <\infty$, we define \nomenclature{$BC^p_{\cA} (G)$}{closure of $\Cu (G)$ in $BL^p_{\cA}(G)$ with respect to the Besicovitch seminorm}
\[
BC^p_{\cA} (G) \,:=\, \text{Closure of $\Cu (G)$ in $BL^p_\cA
(G)$ with respect to $\|\cdot\|_{\be,p,\cA}$} \,.
\]
Elements in
$BC^p_{\cA} (G)$ can naturally be approximated by their cut-off
functions. To make this precise, we define, for $L \in (0, \infty)$,
the cut-off $c_L$\index{cut-off~function} at $L$ by
\begin{equation}\label{eq: cut off}
c_L :\CC\longrightarrow \CC, \qquad c_L (z) =
\begin{cases}
z, & \text{ if } |z|\leq L \,, \\
L \frac{z}{|z|}, & \text{otherwise} \,.
\end{cases}
\end{equation}
Then, $|c_L (z) -
c_L (w)|\leq |z -w|$ for all $z,w\in\CC$, together with $c_L(\Cu(G)) \subseteq \Cu(G)$ imply
\[
\|c_L
(f) - c_L (g)\|_{\be,p,\cA} \leq \|f - g\|_{\be,p,\cA}
\]
for all $f,g\in BC^p_{\cA}(G)$, where $c_L(f)(t):= (c_L \circ f) (t) =c_L(f(t))$.

\smallskip

Let us start with the following simple result.

\begin{prop}\label{prop:approxiimation by bounded functions} For
$f\in BC^p_\cA (G)$, we have $c_n (f)\to f$ in $(BC^p_\cA (G),
\|\cdot\|_{\be,p,\cA})$ as $n\to \infty$.
\end{prop}
\begin{proof} Let $\varepsilon >0$ be arbitrary. We can choose
$g\in \Cu (G)$ with $\|f - g\|_{\be,p,\cA}<\varepsilon/2$. Then, for
all $n\in\NN$ with $n\geq \|g\|_\infty$, we have $c_n(g) =
g$. Hence,
\[
\|f - c_n (f)\|_{\be,p,\cA} \leq\|f - g\|_{\be,p,\cA} + \|c_n (g) - c_n
(f)\|_{\be,p,\cA}  < \varepsilon \,,
\]
where we used $ \|c_n (f) - c_n(g)\|_{\be,p,\cA} \leq  \|f-g\|_{\be,p,\cA}$.
\end{proof}

\begin{remark}[Normal contraction]
A map $c : \CC\longrightarrow \CC$ satisfying $c(0) = 0$ together with $|c(z)  -
c(w)|\leq |z - w|$ for all $z,w \in \CC$, is called a normal contraction. For such a $c$, we
clearly  have
\[
\|c(f) - c(g)\|_\infty \leq \|f - g\|_\infty
\]
for all
$f,g\in\Cu (G)$ as well as
\[
\| c(f) - c(g)\|_{\be,p,\cA} \leq \|f -
g\|_{\be,p,\cA}
\]
for all $f,g\in BL^p_{\cA} (G)$. This implies
that $\sap(G)$, as well as the sets
of almost periodic functions considered below, are closed under
taking normal contractions.     \exend
\end{remark}

On $BC^p_{\cA} (G)$, we introduce the equivalence relation $\equiv$\nomenclature{$\equiv$}{equivalence relation}
with $f\equiv g$  whenever $\|f - g\|_{\be,p,\cA} =0$. In this case,
$\|\cdot\|_{\be,p,\cA}$ becomes a norm on the quotient
 $BC^p_{\cA} (G)/\equiv $ turning it into a complete
space. It is not hard to establish the following crucial feature of
this space: For each $t\in G$, the map $\tau_t  : \Cu
(G)\longrightarrow \Cu (G)$ can be uniquely extended to a continuous
map $T_t$ on $BC^p_{\cA} (G)/\equiv$. The map $T_t$ is isometric for
each $t\in G$, and for each $[f]\in BC^p_{\cA}(G)/\equiv$\label{equiv-class}, the map
\begin{equation}\label{Tt}
G\longrightarrow BC^p_{\cA} (G)/\equiv\,,\qquad t\mapsto T_t [f] \,,
\end{equation}
is continuous.

For $\varphi \in\Cc (G)$, we define the operator $T(\varphi)$\label{T-conv} as the
convolution of $\varphi$ on $BC^p_{\cA} (G)/\equiv$ by setting
\[
T(\varphi)[f]:=\int_G \varphi (s)\, T_s [f]\,\dd s \,
\]
for $[f] \in BC^p_{\cA} (G)$. Here, the integral is defined via
Riemann sums (which is possible as   $G\longrightarrow BC^p_{\cA}
(G)$, $s\mapsto \varphi (s) T_s [f]$, is continuous with compact
support). As each $T_t$, $t\in G$, is an isometry, the inequality
$\|T(\varphi)\| \leq \|\varphi\|_1$ holds for all $\varphi \in\Cc
(G)$. It is not hard to see that $T(\varphi)$ agrees on $\Cu (G)$
with the convolution by $\varphi$, i.e.
$$T(\varphi) [f] = [f*\varphi]$$
holds for all $\varphi \in \Cc (G)$ and $f\in \Cu (G)$. Indeed, the
function $G\times G\longrightarrow \CC, (t,s)\mapsto \varphi (s)
f(t-s)$ is bounded and uniformly continuous, and this shows that the
approximation of  $\int_G \varphi (s)\, T_s [f]\, \dd s$  by
Riemann sums is close to $f\ast \varphi$ in supremum norm and,
hence, also in $BC^p_{\cA}(G)$.
In particular, if $(\varphi_\alpha)$ is an approximate identity, we
find that
\[
T(\varphi_\alpha) [f]\to [f] \,,
\]
first for all $f\in \Cu (G)$ and, then, by the uniform boundedness of
the $T(\varphi_\alpha)$, for all $f\in BC^p_{\cA} (G)$.
Moreover, we easily see by a direct computation that $(\tau_t f_n)$
converges to $\tau_t f $ for all $t\in G$, and $(\varphi \ast f_n)$
converges to $\varphi
*f$ with respect to $\|\cdot\|_{\be,p,\cA}$ whenever $f$ is  bounded
and measurable, $f_n$ belongs to $\Cu (G)$ and $(f_n)$ converges to $f$
with respect to $\|\cdot\|_{\be,p,\cA}$. So, we obtain
\begin{equation}\label{eq-tra}
T_t [f] = [\tau_t f]\qquad \mbox{ as well as }\qquad T(\varphi) [f] = [\varphi *f]
\end{equation}
for all
bounded $f\in BC^p_{\cA}(G)$.

\begin{prop}[Compatibility with translations]\label{prop:compatibility}
 Let $\cA$ be a van Hove sequence. Let $f\in BC^p_{\cA} (G)$ be
 given.  If $\tau_t f\in BC^p_{\cA} (G)$ for some $t\in G$, then
 \[
 T_t [f] = [\tau_t f] \,.
 \]
\end{prop}
\begin{proof}
By Proposition~\ref{prop:approxiimation by bounded
functions}, we have $c_n (f)\to f$ as well as $c_n (\tau_t f) \to
\tau_t f$.
Since $(c_n (f))$ is bounded by construction, we have
\[
T_t [c_n (f)] = [\tau_t c_n (f)] = [c_n (\tau_t f)]
\]
and, consequently,
\[
 T_t [f]= \lim_{n\to\infty} T_t [c_n (f)] = \lim_{n\to\infty} [c_n (\tau_t f)] =  [\tau_t f] \,.
\]
\end{proof}

In  subsequent parts of the book, we will consider the situation
that we  are given a  subspace $\fS'$ of $\Cu (G)$ which is
invariant under translation and closed in $\|\cdot\|_\infty$. Hence,
this subspace is  also invariant under taking convolutions with
elements from $\Cc (G)$. We will be interested in the closure
$\fS$ of this subspace in $BC^p_{\cA}(G)$ equipped with
$\|\cdot \|_{\be,p,\cA}$. Clearly, the translation action and the convolution
then descend from $BC^p_{\cA}(G)$ to $\fS$, and the above
considerations hold for $\fS$ as well. Specifically, we
find the following property.

\begin{prop}[Translation action and convolution]\label{prop:translation}
Given a F\o lner sequence $\cA$ and $p\geq 1$. Let
$\fS'$ be a subspace of $\Cu (G)$ which is invariant under
translation and closed in $\|\cdot\|_\infty$, and let $\fS$ be
its closure in $BL^p_{\cA} (G)$.
\begin{itemize}
\item[(a)] For each $t\in G$, there exists a (unique) continuous map
\[
T_t : \fS/\equiv \longrightarrow \fS/\equiv
\]
extending the translation $\tau_t$  on $\fS'$.  Each $T_t $
is an isometry.
\item[(b)] The map $G\longrightarrow \fS/\equiv$, $t\mapsto
T_t [f]$, is continuous for each $f\in \fS$.
\item[(c)]  $T_t \circ T_s = T_{t+ s}$ and $T_0 = \mbox{Id}$ hold for all
$t,s\in G$.
\item[(d)] For each $\varphi \in \Cc (G)$, there exists a unique continuous  map
\[
T(\varphi) : \fS/\equiv \longrightarrow \fS/
\equiv
\] with
\[
T(\varphi)[f] = [f * \varphi]
\] for all $f\in\fS'$.
\item[(e)] If $(\varphi_\alpha)$ is an approximate identity, then
$T(\varphi_\alpha) [f] \to [f]$ for all $f\in\fS$.
\item[(f)] For all $f\in \fS\cap L^\infty(G)$, we have $\tau_t  f\in \fS$ and $T_t [f] = [\tau_t  f]$ for all $t\in G$.
\item[(g)]  For all $f\in \fS\cap L^\infty(G)$, we have $ f*\varphi \in \fS$ and $T(\varphi) [f] = [f*\varphi]$ for all $\varphi \in \Cc (G)$.   \qed
\end{itemize}
\end{prop}


\begin{coro}[Translation action and convolution]\label{prop:translation2}
Let a F\o lner sequence $\cA$ and $p\geq 1$ be given.

\begin{itemize}
\item[(a)] For each $t\in G$, there exists a (unique) continuous map
\[
T_t : BC^p_{\cA} (G)/\equiv \longrightarrow BC^p_{\cA} (G)/\equiv
\]
extending the translation $\tau_t$  on $\Cu(G)$.  Each $T_t $
is an isometry.
\item[(b)] The map $G\longrightarrow BC^p_{\cA} (G)/\equiv$, $t\mapsto
T_t [f]$, is continuous for each $f\in \Cu(G)$.
\item[(c)]  $T_t \circ T_s = T_{t+ s}$ and $T_0 = \mbox{Id}$ hold for all
$t,s\in G$.
\item[(d)] For each $\varphi \in \Cc (G)$, there exists a unique continuous  map
\[
T(\varphi) :  BC^p_{\cA} (G)/\equiv \longrightarrow  BC^p_{\cA} (G)/
\equiv
\] such that
\[
T(\varphi)[f] = [f * \varphi]
\] for all $f\in \Cu(G)$.
\item[(e)] If $(\varphi_\alpha)$ is  an approximate identity, then
$T(\varphi_\alpha) [f] \to [f]$ for all $f\in  BC^p_{\cA} (G)$.
\item[(f)] For all $f\in BC^p_{\cA} (G) \cap L^\infty(G)$, we have $\tau_t  f\in BC^p_{\cA} (G)$
and $T_t [f] = [\tau_t  f]$ for all $t\in G$ .
\item[(g)] For all $f\in BC^p_{\cA} (G)\cap L^\infty(G)$, we have $ f*\varphi \in BC^p_{\cA} (G)$
 and $T(\varphi) [f] = [f*\varphi]$ for all $\varphi \in \Cc (G)$. \qed
\end{itemize}
\end{coro}

One can even extend the validity of $T(\varphi) [f] = [f*\varphi]$ to
all $f\in BC^1_\cA (G)$ with $f\theta \in \cM^\infty(G)$. To show
this (and more), we need a preparatory lemma.

\begin{lemma}\label{lemm:conv ineq}
Let $\cA$ be a van Hove sequence and $p\geq 1$ be given. For any function $f \in
\BL_{\cA}^p(G)$ which satisfies $f \theta_G \in
\cM^\infty(G)$, one has
\[
\| f *\varphi \|_{\be,p,\cA} \leq 2\, \| f \|_{\be,p,\cA}\, \|\varphi \|_1
\]
for all $\varphi \in \Cc(G)$.
\end{lemma}
\begin{proof}
For each $n$, Young's convolution inequality \cite[Prop.~2.39]{Foll} implies
\begin{align*}
\int_{G} \left| \int_{G} 1_{A_n}(t-s)\, f(t-s)\, \varphi (s)\, \dd s
\right|^p \dd t \leq \| \varphi \|_1^p \int_{A_n} | f(t)|^p\, \dd t  \,.
\end{align*}
We also have
\[
\left| \int_{G} 1_{A_n}(t)\, f(t-s)\, \varphi(s)\, \dd s \right| \leq
\left| \int_{G} 1_{A_n}(t-s)\, f(t-s) \,\varphi (s)\, \dd s
\right|+1_{\partial^K(A_n)}(t)\,\|f*\varphi \|_\infty  \,.
\]
By the standard inequality $(a+b)^p \leq 2^p(a^p+b^p)$, we get
\begin{align*}
\left| \int_{G} 1_{A_n}(t)\, f(t-s)\, \varphi (s)\, \dd s \right|^p
    &\leq 2^p \left| \int_{G} 1_{A_n}(t-s)\, f(t-s)\, \varphi (s)\,
     \dd s \right|^p  \\
    &\phantom{==}+2^p \left(1_{\partial^K(A_n)}(t)\|f*\varphi
      \|_\infty\right)^p \,.
\end{align*}
Therefore, we obtain
\begin{align*}
\int_{G} 1_{A_n}(t) \left| \int_{G} f(t-s)\, \varphi (s)\, \dd s
\right|^p \dd t
    &\leq 2^p \int_{G}\left| \int_{G} 1_{A_n}(t-s)\, f(t-s)\,
       \varphi (s)\, \dd s \right|^p \dd t \\
    &\phantom{==}+\int_{G} 1_{\partial^K(A_n)}(t)\,\|f*\varphi \|^p_\infty\, \dd t
      \,.
\end{align*}
Invoking the consequence of Young's inequality given above and using
the van Hove property and the boundedness of $f*\varphi$ (which follows
from $f \theta_G \in \cM^\infty(G)$), the claim follows.
\end{proof}

\begin{prop}[$f$ versus $f\theta_G$]\label{prop:measure-vs-function}
Let $\fS'$ be a subspace of  $\Cu (G)$ which is invariant
under translation and closed with respect to $\|\cdot\|_\infty$. Let
$\fS$ be its closure in $BL^p_{\cA} (G)$ with respect to
$\|\cdot\|_{\be,p,\cA}$. Then, the following assertions hold:
\begin{itemize}
\item[(a)] One has $f*\varphi \in \fS'$ for all $f\in\fS'$ and
$\varphi \in \Cc (G)$.
\item[(b)] Let $f\in \mathcal{L}^p_{loc} (G)$ be such that  $f \theta_G$ is a
translation bounded measure, and assume that $(f_n)$ is a sequence in
$\fS'$ with $f_n \to f$ with respect to
$\|\cdot\|_{\be,p,\cA}$. Then, $f*\varphi$ belongs to $\fS$
and  $(f_n*\varphi)$ converges to $f*\varphi$ with respect to
$\|\cdot\|_{\be,p,\cA}$.
\item[(c)] For $f\in BL^p_{\cA} (G)$ with $f\theta_G \in \cM^\infty(G)$,
the following assertions are equivalent:
\begin{itemize}
\item[(i)] $f$ belongs to $BC^p_{\cA} (G)$ and $f* \varphi\in \fS$ for all $\varphi \in \Cc
(G)$.
\item[(ii)] $f$ belongs to $\fS$.
\end{itemize}
In particular, for $f\in\Cu (G)$, we have  $f\in \fS$ if and
only if $f* \varphi\in \fS$ for all $\varphi \in \Cc (G)$.
\end{itemize}
\end{prop}
\begin{proof} (a) This follows easily as $\fS'$ is closed
under translations and with respect to $\|\cdot\|_\infty$.

\medskip

\noindent (b) As $\fS$ is closed, it suffices to show that
$(f_n*\varphi)$ converges to $f*\varphi$ with respect to
$\|\cdot\|_{\be,p,\cA}$.  Now, the desired statement follows from Lemma~\ref{lemm:conv ineq} and
\[
\|f *\varphi - f_n*\varphi\|_{\be,p,\cA} \leq 2\, \|f - f_n\|_{\be,p,\cA}\,
\|\varphi\|_1 \,.
\]

\smallskip

\noindent (c) The implication (ii)$\Longrightarrow$ (i) follows from (b). It
remains to show  (i)$\Longrightarrow$(ii): By (i), the function
$f*\varphi$ belongs to $\fS$ for all $\varphi \in \Cc (G)$
and $\fS$ is closed. Hence, it suffices to show that $\|f -
f*\varphi\|_{\be,p,\cA}$ becomes arbitrarily small. Now, for $\varphi
\in \Cc (G)$ with $\|\varphi\|_1\leq 1$ and $g\in \Cu (G)$, we find
\begin{eqnarray*}
\| f - f*\varphi\|_{\be,p,\cA}  &\leq &\|f - g\|_{\be,p,\cA} + \|g -
g*\varphi\|_{\be,p,\cA} + \|g*\varphi - f*\varphi\|_{\be,p,\cA}\\
&\leq & \|f  - g\|_{\be,p,\cA} + \|g - g*\varphi\|_\infty + 2\, \|g -
f\|_{\be,p,\cA},
\end{eqnarray*}
where we used Lemma~\ref{lemm:conv
ineq} and $\|\cdot\|_{\be,p,\cA} \leq\|\cdot\|_\infty$. Now, the right
hand side can be made arbitrarily small, first by choosing $g$
sufficiently close to $f$ (which is possible due to the assumption
$f\in BC^p_{\cA}(G)$) and then choosing $\varphi \in \Cc (G)$ with
$\|g - g*\varphi\|_\infty$ sufficiently small (which is possible due
to $g\in \Cu (G)$). This proves (ii).
\end{proof}

All spaces of almost periodic functions considered in the remainder
of this book will be subspaces of $BC^p_{\cA} (G)$. Consequently, they
can and will be equipped with a continuous  translation.

\section{Spectral theory of $\mathcal{N}$-representations}
In this section, we develop an abstract version of diffraction
theory based on the concept of $\mathcal{N}$-representation.
In the next section we  will combine this with  our considerations on the spaces coming from the Besicovitch- and Weyl seminorms. This will naturally lead to the concept of $\mathcal{A}$-representation. These are at the core of our subsequent considerations (and cover the diffraction theory developed above for translation bounded measures).  A discussion of these two concepts in the context of model sets is given in Example \ref{ex-cp-a-rep-vs-n-rep}. As far as methods go the material of this section and the next section is a crucial new ingredient for our solution of the three problems discussed in the introduction.

\medskip

Let $\mathcal{H}$ be a vector space with semi-inner product $\langle
\cdot,\cdot\rangle$ and associated seminorm $\|\cdot\|$. A
continuous action $T$  of $G$ on $\mathcal{H}$ via isometries $T_t$,
$t\in G$, will also  be referred to as \textit{unitary
representation}\index{unitary~representation} of $G$ on $\mathcal{H}$. Given such a representation,
a direct computation shows that $g: G\longrightarrow \CC$, $t\mapsto
\langle f, T_t f\rangle$, satisfies
\[
\sum_{j,k=1}^n c_i\, \overline{c_j} g(t_i - t_j) =
\Big\|\sum_{j=1}^n c_j T_{t_j} f\Big\|^2 \geq 0
\]
for $n\in\NN$ and arbitrary
$c_1,\ldots, c_n\in \CC$ and $t_1,\ldots, t_n\in G$. Hence, this
function is positive definite for each $f\in\mathcal{H}$, and it is clearly continuous.
Thus, by Bochner's theorem, for each $f\in\mathcal{H}$, there exists a unique,
positive and finite measure $\sigma_f$ on $\widehat{G}$ with
\[
\langle f, T_t f\rangle = \int_{\widehat{G}}  \chi (t)\, \dd\sigma_f
(\chi)
\]
for all $t\in G$. This measure is called the
\textit{spectral measure}\index{spectral~measure} of $f$.

When we have a continuous action of $G$ by isometries $T_t$ on
$\mathcal{H}$, we can define the operator
\[
T(\varphi) : \mathcal{H}\longrightarrow \mathcal{H}, \qquad f\mapsto \int_G
\varphi (s) T_s f\, \dd s,
\]
for $\varphi \in \Cc (G)$. Then, $T(\varphi)$ will be a bounded
operator with
\[
\|T(\varphi)\|\leq  \| \varphi\|_1 \,.
\]
The spectral
measure is compatible with taking convolutions in the following
sense.

\begin{prop}\label{prop-spectralmeasure-convolution}
Let $\mathcal{H}$ be a vector space with semi-inner product $\langle
\cdot,\cdot\rangle$. Let $G$ act continuously  on $\mathcal{H}$ via
isometries $T_t$, $t\in G$. Then,
\[
\sigma_{T(\varphi)f} =
|\widehat{\varphi}|^2 \sigma_{f} \, .
\]
\end{prop}
\begin{proof} A direct computation shows that both  $\sigma_{T(\varphi)f}$ and $|\widehat{\varphi}|^2
\sigma_{f}$ have the same (inverse) Fourier transform:
\begin{align*}
\langle T(\varphi)f, T_t T(\varphi) f\rangle
    &= \int_G \int_G \varphi(s)\, \overline{\varphi(r)}\, \langle f,
       T_{t-s+r} f\rangle\, \dd s\, \dd r \\
    &=\int_G \int_G \varphi(s)\, \overline{\varphi(r)} \int_{\widehat{G}}
         \chi(t-s+r)\, \dd\sigma_{f}(\chi)\,  \dd s\, \dd r \\
    &= \int_{\widehat{G}} \chi (t)\, |\widehat{\varphi} (\chi)|^2 \,
        \dd\sigma_{f} (\chi) \,.
\end{align*}
The desired claim follows by the uniqueness property of Fourier transforms.
\end{proof}

We will be particularly interested in the situation that all
spectral measures are pure point measures. To put this in context,
consider a  unitary representation $T$ of $G$ on a Hilbert space $\mathcal{H}$.
An $f\in\mathcal{H}$ with $f\neq 0$ is called an
\textit{eigenfunction}\index{eigenfunction} of $T$ to the \textit{eigenvalue}\index{eigenvalue} $\chi
\in\widehat{G}$ if $T_t f = \chi(t) f$ for all $t\in G$.

$T$ is said to have \textit{pure point spectrum}\index{spectrum!pure~point~spectrum~for~unitary~representations} if $\cH$ has an
orthonormal basis of eigenfunctions. As is well known (and not hard
to see), pure point spectrum of $T$ is equivalent to all measures
$\sigma_f$ being pure point measures. This suggests to look for
criteria ensuring that a spectral measure is a pure point measure.
The following characterization is well known for Hilbert spaces, see e.g. \cite{LS}.
It applies to pre-Hilbert spaces as well. We include a proof for completeness reasons.

\begin{prop}\label{prop:Hilbert} Let $\mathcal{H}$ be a vector space with
semi-inner product $\langle \cdot,\cdot\rangle$ and associated
seminorm $\|\cdot\|$. Let  $G$ act continuously on $\mathcal{H}$ via
isometries $T_t$, $t\in G$. For $f\in \mathcal{H}$, the following
assertions are equivalent:
\begin{itemize}
\item[(i)] The spectral measure $\sigma_f$ is pure point.
\item[(ii)] The function $G\longrightarrow \CC$, $t\mapsto \langle  f,
 T_t f\rangle$, is Bohr almost periodic.
\item[(iii)] The function $G\longrightarrow \mathcal{H}$, $t\mapsto T_t
f$, is almost periodic in the sense that, for each $\varepsilon >0$,
the set
\[
P_{\cH}(f;\eps):=\{t\in G: \|f - T_t f\|< \varepsilon\}
\]
is relatively dense.
\end{itemize}
\end{prop}
\begin{proof} The equivalence between (i) and (ii) is a classical
result of Wiener \cite{Wie} (see \cite{Ebe2} for a proof in arbitrary LCAG's).

\medskip

\noindent (ii)$\Longrightarrow$(iii): Using that  $T_t$ is an isometry, we have
\[
\|f - T_t f\|^2 = 2 (\langle f, f\rangle - \mbox{Re} \langle  f, T_t f\rangle) \leq 2|
\langle f, f\rangle - \langle f, T_t f \rangle|
\]
for each $t\in G$. The claim follows immediately.

\medskip

\noindent (iii)$\Longrightarrow$(ii): Set $F(t) := \langle  f, T_t  f\rangle$. Then, for each $t\in G$, we have
\begin{align*}
|F (t + s) - F(s)| &= |\langle f, T_{t+s} f\rangle - \langle f, T_s f
\rangle| \leq \|T_{t+s} f - T_s f\|\, \|f\|\\
& = \|T_{t} f - f\|\, \|f\| \,,
\end{align*}
and the claim follows.
\end{proof}

We will now consider the  special representations to be defined next.

\begin{definition}[$\mathcal{N}$-representation]
 An \textit{$\mathcal{N}$-representation}\index{$\mathcal{N}$-representation} of $G$ is a quadruple
$(N,\mathcal{H},\langle\cdot,\cdot\rangle,T)$ consisting of a complex vector space $\mathcal{H}$ with semi-inner product $\langle \cdot,\cdot\rangle$ together with a continuous  action $T_t$, $t\in G$, of $G$ on $\mathcal{H}$ by isometries and a linear $G$-equivariant map
\[
N : \Cc (G)\longrightarrow \mathcal{H} \,.  \tag*{$\Diamond$}
\]
\end{definition}

Whenever $(N,\mathcal{H},\langle\cdot,\cdot\rangle,T)$  is an
$\mathcal{N}$-representation, we denote the spectral measure of
$N(\varphi)$ by $\sigma_{\varphi}$ (instead of $\sigma_{N(\varphi)}$). We will be interested in the situation that all $\sigma_{\varphi}$ come about by one measure on $\widehat{G}$ as discussed in the subsequent definition.

\begin{definition}[Diffraction measure of an $\mathcal{N}$-representation]
A measure $\sigma$ on $\widehat{G}$ is called the \textit{diffraction measure}\index{diffraction!$\mathcal{N}$-representation} of the $\mathcal{N}$-representation  if
\[
|\widehat{\varphi}|^2 \sigma = \sigma_{\varphi}
\]
holds for all $\varphi \in\Cc (G)$. \exend
\end{definition}

We will often refer to $N: \Cc
(G)\longrightarrow \mathcal{H}$ or even just $N$  as an
$\mathcal{N}$-representation.
For us, the situation where $\mathcal{H}$ is a Hilbert space and $N(\Cc (G))$ is dense in $\mathcal{H}$ will be particularly relevant. In this case, we speak about an $\mathcal{N}$-representation \textit{on a Hilbert space with dense range}. Note that this is not so much an assumption  but rather a matter of convenience. Indeed, whenever $N :\Cc (G)\longrightarrow \mathcal{H}$ is an
$\mathcal{N}$-representation, we can always  factor out elements of $\mathcal{H}$  with vanishing seminorm and take the completion of $N(\Cc (G))$.

We call the $\mathcal{N}$-representation $N$ \textit{intertwining}\index{intertwining!$\mathcal{N}$-representation}
if
\[
T(\varphi) N(\psi) = T(\psi) N(\varphi) \qquad \mbox{ for all } \varphi,\psi\in \Cc (G) \,.
\]

\begin{prop}[Intertwining follows from continuity]
\label{prop:continuity-implies-intertwining} If the
$\mathcal{N}$-representation $N$ is continuous, then for all $\varphi, \psi \in \Cc(G)$ we have
\[
T(\varphi) N(\psi) = N(\varphi \ast \psi) \,.
\]
In particular, $N$ is intertwining.
\end{prop}
\begin{proof} A short computation using the continuity of $N$ gives
\[
T(\varphi) N(\psi) = \int_G \varphi (s)\, T_s N(\psi)\, \dd s = \int_G N(\varphi(s) \tau_s \psi)\, \dd s = N(\varphi \ast \psi) \,.
\]
From this, we directly see that $N$ is intertwining as $\varphi \ast \psi = \psi \ast \varphi$.
\end{proof}

An  intertwining $\mathcal{N}$-representation may not admit an autocorrelation measure. However, it does admit a weak analogue of an autocorrelation measure  that is sufficient to allow for a diffraction measure (which is the main concern for us). To discuss this we need some preparation.

Let us recall the following notation of \cite{ARMA}  \nomenclature{$K_2(G)$}{span of the set of convolutions of two continuous and compactly supported functions on $G$}
\[
K_2(G) := \operatorname{span}\{\varphi*\psi : \varphi,\psi\in\Cc(G)
\} \,.
\]

\begin{definition}\label{semi-measure}
  A linear map  $\vartheta : K_2(G) \to \CC$ is called a
\textit{semi-measure}\index{semi-measure}. A semi-measure $\vartheta$ is
\textit{Fourier transformable}\index{Fourier~transformable!semi-measure} if there exists a measure
$\widehat{\vartheta}$ on $\widehat{G}$ such that, for all $\varphi
\in \Cc(G)$, we have
\[
\left| \widecheck{\varphi} \right|^2 \in
\mathcal{L}^1(|\widehat{\vartheta}|) \qquad \text{ and } \qquad \vartheta(\varphi*\widetilde{\varphi})=\widehat{\vartheta}\big( |
\widecheck{\varphi}|^2  \big) \,.
\]
In this case, we call the measure $\widehat{\vartheta}$ the
\textit{Fourier transform}\index{Fourier~transform!semi-measure} of $\vartheta$.     \exend
\end{definition}

Given a semi-measure $\vartheta$, for all $\psi \in K_2(G)$, we can
define the \textit{convolution}\index{convolution!function~and~semi-measure}
\[
(\vartheta*\psi)(t):= \vartheta( \tau_t\psi^{\dagger}) \,. \label{semi-conv}
\]

\smallskip

We say that $N$ possesses the  \textit{semi-autocorrelation}\index{semi-autocorrelation} $\eta$
if $\eta$ is a  semi-measure with
\begin{equation}\label{eq:autoc def}
\langle N(\varphi), N(\psi)\rangle = (\eta*\varphi
*\widetilde{\psi})(0)
\end{equation}
for all $\varphi,\psi\in\Cc (G)$. We say that $N$ possesses an
\textit{autocorrelation}\index{autocorrelation!$\mathcal{N}$-representation} if $\eta$ is a measure.
Note that, when a semi-measure satisfying
\eqref{eq:autoc def} exists, it is Fourier transformable by
Remark~\ref{rem:pd implies FT}.

\begin{lemma}[Intertwining and diffraction measure]\label{lem:char-intertwining}
Let $N:\Cc (G)\longrightarrow \mathcal{H}$ be an
$\mathcal{N}$-representation. Then, the following assertions are
equivalent:
\begin{itemize}
\item[(i)] $N$ is intertwining.
\item[(ii)]  $N$ possesses a semi-autocorrelation $\eta$.
\item[(iii)] $N$ possesses a diffraction measure $\sigma$.
\end{itemize}
If one of these equivalent conditions holds, then $\widehat{\eta} =
\sigma$.
\end{lemma}
\begin{proof}  (i)$\Longrightarrow$(iii): By Lemma~\ref{spectral lemma} in
 Appendix \ref{appendix:semi},
it suffices to show that
\[
|\widehat{\varphi}|^2\, \sigma_\psi = |\widehat{\psi}|^2 \,
\sigma_\varphi
\]
holds for all $\varphi,\psi \in \Cc (G)$. However, this is immediate from
(i) and Proposition~\ref{prop-spectralmeasure-convolution}.

\medskip

\noindent (iii)$\Longrightarrow$(i): From (iii) and polarisation, we find
\[
\langle
N(\varphi), T_t N(\psi)\rangle = \int_{\widehat{G}} \chi(t)\, \widehat{\varphi}
(\chi)\, \overline{\widehat{\psi} (\chi)}\, \dd\sigma(\chi)
\]
for all
$t\in G$ and $\varphi,\psi \in \Cc (G)$. Given this, a  direct
computation similar to the one in the proof of Proposition~\ref{prop-spectralmeasure-convolution}, shows
\[
\langle T(\varphi) N(\psi), T(\varrho) N(\xi)\rangle = \int_{\widehat{G}}
\widehat{\varphi}(\chi)\, \widehat{\psi}(\chi)\, \overline{
\widehat{\varrho}(\chi)\,\widehat{\xi}(\chi)}\, \dd\sigma(\chi)
\]
for all
$\varphi,\psi,\varrho,\xi\in\Cc (G)$. This easily gives
\[
\|T(\varphi) N(\psi) - T(\psi)N(\varphi)\|^2 = 0 \,.
\]

\smallskip

\noindent (ii)$\Longrightarrow$ (iii): Let $\sigma$ be the Fourier transform
of $\eta$. From the defining properties of $\sigma$, $\eta$ and the
spectral measure, we find
\[
\int_{\widehat{G}} \chi(t)\, |\widehat{\varphi}|^2\, \dd\sigma (\chi) = (\eta*\varphi \ast
\widetilde{\tau_t \varphi})(0) = \langle N(\varphi), T_t
N(\varphi)\rangle =\int_{\widehat{G}} \chi (t)\, \dd\sigma_{\varphi}(\chi)\,.
\]
As this
holds for all $t\in G$, we conclude (iii).

\medskip

\noindent (iii)$\Longrightarrow$(ii): By (iii), the measures
$|\widehat{\varphi}|^2 \sigma$ agree with $\sigma_\varphi$ and, for all $\varphi \in\Cc (G)$, they are finite. Hence, $\sigma$ is
weakly admissible in the sense of Appendix \ref{appendix:semi}.
Then, Proposition~\ref{FT of semi-measures} guarantees the existence of a
semi-measure $\eta$ whose Fourier transform is $\sigma$, and (ii)
follows as $\eta$ satisfies
\[
(\eta*\varphi \ast \widetilde{\psi})(0)
= \eta\big((\varphi \ast \widetilde{\psi})^\dagger\big)
=\int_{\widehat{G}} \widehat{\varphi}(\chi)\,\overline{\widehat{\psi}(\chi)}\, \dd\sigma (\chi)
= \langle N(\varphi),N(\psi)\rangle
\]
for all $\varphi, \psi \in\CC (G)$.
Here, the last equality follows from (iii) and polarisation.

\medskip

\noindent The last statement has been shown in the proofs of
(iii)$\Longrightarrow$(ii) and (ii)$\Longrightarrow$(iii).
\end{proof}

\medskip

In the situation of Lemma~\ref{lem:char-intertwining}, $L^2 (\widehat{G},\sigma)$
admits a natural continuous action of $G$ by isometries via
\[
(t\cdot f) (\chi) = \chi
(t) f(\chi) \,.
\]
We will always think about $L^2 (\widehat{G},\sigma)$
as equipped with this action.

\begin{theorem}[Spectral theorem for $\mathcal{N}$-representation]\label{thm:intertwining-U}
Let $N:\Cc (G)\longrightarrow \mathcal{H}$ be an intertwining
$\mathcal{N}$-representation on a Hilbert space with dense range.
Then, there exists a unique unitary map
\[
U  : L^2 (\widehat{G},\sigma) \longrightarrow \mathcal{H}
\]
with  $\widehat{\varphi}\mapsto N(\varphi)$ for all $\varphi \in\Cc
(G)$. This map is $G$-equivariant.
\end{theorem}
\begin{proof} For all $\varphi \in \Cc(G)$ we have
\[
\|\widehat{\varphi}\|_{L^2(\widehat{G},\sigma)}^2 = \int_{\widehat{G}} |\widehat{\varphi}(\chi)|^2 \,
\dd\sigma(\chi) = \int_{\widehat{G}} \dd\sigma_{\varphi}(\chi) = \|N(\varphi)\|^2 \,.
\]
Thus, $U$ is well defined and isometric on the subspace
$L:=\{\widehat{\varphi} : \varphi \in\Cc (G)\}$ of $L^2
(\widehat{G},\sigma)$. Hence, it can be extended to an isometric map
on the closure of $L$. This closure is $L^2 (\widehat{G},\sigma)$.
The map has dense range as $N$ has dense range. As it is an
isometry, it must be unitary. Finally, note that the map is
$G$-equivariant on $L$ as $N$ is $G$-equivariant.
\end{proof}

As unitary maps completely preserve spectral features, we
immediately obtain the following result.

\begin{coro}\label{coro:pointspectrum}
Let $N:\Cc (G)\to \mathcal{H}$ be an intertwining
$\mathcal{N}$-representation on a Hilbert space with dense range.
 Then, $\{T_t :t \in G \}$ has pure point spectrum if and only if $\sigma$ is a
pure point measure. In this case, each eigenvalue has multiplicity
one and the functions $c_\chi := U(1_\chi)$, for $\chi\in\widehat{G}$
with $\sigma (\{\chi\}) \neq 0$, form an orthogonal system with
dense span in $\mathcal{H}$ satisfying $\langle
c_{\chi},c_{\chi}\rangle = \sigma(\{\chi\})$.\qed
\end{coro}

Now, we consider the situation of Corollary~\ref{coro:pointspectrum}. Furthermore, assume that we are
given additionally an orthonormal basis of eigenfunctions $e_{\chi}$
 to the corresponding eigenvalue $\chi \in\widehat{G}$ of $\mathcal{H}$. For $\chi \in \widehat{G}$, we can define the \textit{Fourier coefficient
$A_\chi $ of $N$ (with respect to $(e_\chi)$)}\index{Fourier~coefficient} as the unique factor
with $A_{\chi} e_{\chi} = c_{\chi}$. Then, the following will be
true
\begin{align*}
|A_{\chi}|^2 \,&=\, \|c_{\chi}\|^2 = \sigma(\{\chi\})
\qquad \mbox{ and } \\
\langle N(\varphi), e_\chi \rangle
       \,& =\,\Big\langle \widehat{\varphi},\frac{1_\chi}{A_{\chi}} \Big\rangle
        = \widehat{\varphi}(\chi)\, A_{\chi}
\,,
\end{align*}
where we used $\sigma(\{\chi\}) = A_\chi
\cdot\overline{A_{\chi}}$ in the last step. This gives
$$N(\varphi) = \sum_{\chi} \langle N(\varphi),e_\chi\rangle e_\chi = \sum_{\chi}  \widehat{\varphi}(\chi)\, c_\chi$$
for all $\varphi \in\Cc (G)$. Using the functionals
\[
(\chi, \cdot): \Cc (G)\longrightarrow \CC \qquad \mbox{ with } (\chi,\varphi) = \widehat{\varphi}(\chi)  \,,
\]
we can thus formally expand $N$ as
\[
N = \sum_\chi (\chi, \cdot) \,c_\chi \,.
\]

\section[Diffraction of $\mathcal{A}$-representations]{Where invariant subspaces and $\mathcal{N}$-representations meet: diffraction theory of $\mathcal{A}$-representations}
In this section we put forward  special $\mathcal{N}$-representations which are  induced by van Hove sequences and take  values in the space $BC^2_{\cA} (G)$. As far as methods go these representations are the  crucial ingredient in our approach in subsequent sections.

\medskip

 Let $\cA$ be a van Hove sequence. Recall from
Section \ref{sub:trans} that $BC^2_{\cA} (G)$ is the closure of
$\Cu (G)$ with respect to $\|\cdot\|_{\be,2,\cA}$ and that
$(BC^2_{\cA}(G)/\equiv)$ allows for a continuous  action of $G$ by
translations $T_t, t\in G$. A linear map $N: \Cc (G)\longrightarrow
\mathcal{L}^1_{loc}(G)$ is called an \textit{$\cA$-representation}\index{$\mathcal{A}$-representation} if it
satisfies the following properties:
\begin{itemize}
\item $ N(\Cc (G))\subset BC^2_{\cA} (G)$,
\item  $ M_{\cA}(f\overline{g})$ exists for all $f,g\in N(\Cc (G))$,
\item $N(\tau_t (\varphi)) = \tau_t N(\varphi)$ for all $t\in G$ and
$\varphi \in \Cc (G)$.
\end{itemize}
Any $\cA$-representation $N$ gives naturally rise to an
$\mathcal{N}$-representation on a Hilbert space with dense range.
Specifically, we define
\[
\mathcal{H}:=\mbox{Closure of $\{[N(\varphi)]$, $\varphi \in\Cc
(G)\}$, in $BC^2_{\cA}(G)/\equiv$ } \,.
\]
Note that
\[
\langle [f], [g]\rangle :=M_{\cA}(f\overline{g})
\]
is well defined and gives an inner product on $\mathcal{H}$ whose
associated norm $\|\cdot\|$ satisfies
\[
\|[f]\| = M_{\cA}(|f|^2)^{1/2}  =\|f\|_{\be,2,\cA} \,.
\]
By construction, $\mathcal{H}$ is a Hilbert space. The defining
properties of $N$ and Proposition~\ref{prop:compatibility} imply that this
Hilbert space is invariant under the translation action and the map
\[
\overline{N}: \Cc (G)\longrightarrow \mathcal{H}\,,\qquad \varphi \mapsto
[N(\varphi)] \,,
\]
is $G$-equivariant with dense range. Hence,
$(\overline{N},\mathcal{H},\langle \cdot,\cdot\rangle, T)$ is an
$\mathcal{N}$-representation on a Hilbert space with dense range. If
this $\mathcal{N}$-representation is intertwining, we say that
the $\mathcal{A}$-representation $N$ is \textit{intertwining}\index{intertwining!$\mathcal{A}$-representation}.

We will be mostly interested in $\mathcal{A}$-representations
induced by measures. More specifically, for a measure $\mu$ on $G$,
we consider the map
\begin{align*}
N_\mu : \Cc(G)&\longrightarrow \mathcal{L}^1_{loc} (G)\,,  \\
 N_\mu (\varphi)\, &:=\, \mu \ast \varphi \,.
\end{align*}
In order for this to give an $\mathcal{A}$-representation, $\mu \ast
\varphi$ must belong to $BC^2_\cA (G)$ for all $\varphi \in\Cc (G)$,
and $M_{\cA} (\mu \ast \varphi \cdot \overline{\mu\ast \psi})$ must
exist for all $\varphi,\psi \in\Cc (G)$. The property \[
\tau_t
N_\mu(\varphi) = N_\mu(\tau_t \varphi)  \qquad \forall \, t \in G\,, \, \varphi
\in\Cc (G)
\]
is automatically satisfied. Finally, in order to apply
the above theorems, we also need that $N_\mu$ is intertwining. Next, we
gather some classes of measures for which all these assumptions
are satisfied.

\begin{prop}[Translation bounded measures with autocorrelation]
\label{prop:translation-bounded-admissible} Let $\mu$ be a
translation bounded measure on $G$ whose autocorrelation $\gamma$
exists with respect to the van Hove sequence $\cA$. Then, $N_\mu$
is an intertwining $\mathcal{A}$-representation.
\end{prop}
\begin{proof} As $\mu$ is translation bounded, $\mu \ast \varphi $
belongs to $\Cu (G)\subseteq BC^2_{\cA} (G)$ for all $\varphi \in\Cc
(G)$. Moreover, by Proposition~\ref{prop-compute-autocorrelation},
we have
\[
M_{\cA}( \mu*\varphi\cdot \overline{\mu*\psi}) = (\gamma
\ast \varphi \ast \tilde{\psi})(0) \qquad \text{ for all } \varphi, \psi \in\Cc
(G)\,.
\]
So, $N_\mu$ is indeed an $\mathcal{A}$-representation.
Clearly, $N_\mu$ possesses the autocorrelation $\gamma$. Hence, it
is intertwining by Lemma~\ref{lem:char-intertwining}.
\end{proof}

\begin{remark}
From Proposition~\ref{prop-compute-autocorrelation}, we see that
there is a converse of sorts to this proposition: If $\mu$ is
translation bounded such that $N_\mu$ is an intertwining
$\mathcal{A}$-representation, then $\mu$ possess an
autocorrelation.
 Let us also note that, in this case, $N_\mu$ is continuous (as a short computation shows).    \exend
 \end{remark}

\begin{prop}[$\mathcal{A}$-representation derived from continuity]
\label{prop:a-cont}
Let $\cA$ be a van Hove sequence on $G$ and  $\mu$ a  measure on $G$
with $\mu \ast \varphi \in BC^2_{\cA} (G)$ for all $\varphi \in \Cc
(G)$ such that $M_\cA (\mu*\varphi \cdot \overline{\mu*\psi})$
exists for all $\varphi, \psi \in\Cc (G)$. If $\Cc
(G)\longrightarrow BC^2_{\cA}(G)$, $\varphi \mapsto \mu*\varphi$, is
continuous, then $N_\mu$ is an intertwining
$\mathcal{A}$-representation. In particular, $N_\mu$ possesses a
semi-autocorrelation.
\end{prop}
\begin{proof} By Proposition~\ref{prop:continuity-implies-intertwining}, $N_\mu$ is an
intertwining $\mathcal{A}$-representation. The last statement
follows from Lemma~\ref{lem:char-intertwining}.
\end{proof}

\begin{prop}\label{prop:general-admissible}
Let $\cA$ be a van Hove sequence on $G$ and   $\mu$  a measure on
$G$ such that $\mu \ast \varphi \in BC^2_{\cA} (G)$ for all $\varphi \in
\Cc (G)$ and $|\mu|*\psi \in BL^2_{\cA}(G)$ for all $\psi \in
\Cc(G)$ with $\psi \geq 0$. If $M_\cA (\mu*\varphi \cdot
\overline{\mu*\psi})$ exists for all $\varphi, \psi \in\Cc (G)$,
then $N_\mu$ is an intertwining $\mathcal{A}$-representation. In
particular, $N_\mu$ possesses a semi-autocorrelation.
\end{prop}
\begin{proof}
This follows from Proposition~\ref{prop:a-cont} as the map $\Cc
(G)\longrightarrow BC^2_{\cA}(G)$, $\varphi \mapsto \mu*\varphi$, is
continuous. To see this, consider a sequence $(\varphi_n)$  in $\Cc
(G)$ converging to $\varphi \in \Cc (G)$ in the inductive limit
topology on $\Cc (G)$. Then, there exists a compact set $K\subseteq G$
such that the support of all $\varphi_n$ (and then the support of
$\varphi$ as well) is contained in $K$. Let $\psi\in\Cc(G)$ be a
nonnegative function with $\psi =1$ on $K$. Then, one has
\[
|\varphi
-\varphi_n|\leq \|\varphi-\varphi_n\|_\infty \psi
\]
and
\[
|\mu \ast \varphi - \mu \ast \varphi_n| =|\mu \ast (\varphi -\varphi_n)|\leq \|\varphi -\varphi_n\|_\infty\, \big(|\mu| \ast \psi\big) \,.
\]
This easily gives the desired continuity.
\end{proof}

As consequences, we get the following results.

\begin{coro}[$\mathcal{A}$-representation derived from positive measures]
\label{prop:positive-admissible} Let $\cA$ be a van Hove sequence on $G$, and let $\mu$ be a positive measure on $G$  with  $\mu \ast \varphi \in BC^2_{\cA} (G)$ for all $\varphi \in \Cc (G)$ such that $M_\cA
(\mu*\varphi \cdot \overline{\mu*\psi})$ exists for all $\varphi,
\psi \in\Cc (G)$. Then, $N_\mu$ is an intertwining
$\mathcal{A}$-representation. In  particular, $N_\mu$ possesses a
semi-autocorrelation. \qed
\end{coro}

\begin{coro}[$\mathcal{A}$-representation derived from measures with uniformly discrete support]
\label{prop:unif disc-admissible}
 Let $\cA$ be a van Hove sequence on $G$ and  $\mu$  a measure on $G$ with  $\mu \ast \varphi \in BC^2_{\cA}
(G)$ for all $\varphi \in \Cc (G)$ such that  $M_\cA (\mu*\varphi
\cdot \overline{\mu*\psi})$ exists for all $\varphi, \psi \in\Cc
(G)$. If $\supp(\mu)$ is uniformly discrete, then $N_\mu$ is an intertwining
$\mathcal{A}$-representation. In  particular, $N_\mu$ possesses a
semi-autocorrelation.
\end{coro}
\begin{proof}
Let $U$ be any open set such that $0 \in U$ and $\supp(\mu)$ is
$U$-uniformly discrete. Choose $t \in G$ and  $\varphi \in \Cc(G)$
with $\varphi \geq 0$ and $\supp(\varphi) \subset t+U$. Since
$\supp(\mu)$ is $t+U$-uniformly discrete, we have
\cite[Lem.~5.8.3]{NS11}
\begin{displaymath}
\left| \mu *\varphi \right|=\left| \mu\right| *\left|\varphi \right|= \left| \mu \right| * \varphi \,.
\end{displaymath}
Therefore, since $\mu \ast \varphi \in BC^2_{\cA}(G)$, we get
$|\mu|\ast\varphi = \left| \mu \ast \varphi\right| \in
BC^2_{\cA}(G)\subseteq BL^2_{\cA}(G)$. Since every function $\psi
\in \Cc(G)$ can be written as a finite linear combination of
positive function $\varphi \in \Cc(G)$ supported inside translates
of $U$, the claim follows from
Proposition~\ref{prop:general-admissible}.
\end{proof}

\chapter[Mean almost periodicity]{Mean almost periodicity and characterization of pure-point
diffraction}\label{sec-mean} In this chapter, we introduce the
notion of  mean  almost periodicity for functions and  measures
based on the seminorms $\|\cdot\|_{\be,p,\cA}$ and $\|\cdot\|_{\we,p,\cA}$, which arise  from averaging.
We use it to characterize pure-point diffraction.

For bounded functions $ f: \NN \to \CC$, mean almost periodicity was defined in \cite[Def.~6.6]{Que}, and it was used to characterize pure-point diffraction \cite[Lem.~6.6]{Que}. The approach in \cite{Que} uses the fact that the dual group $\widehat{\ZZ}=\RR/\ZZ$ is compact, and it does not seem to be easy to extend mean almost periodicity beyond non-discrete groups.

Systematic study for mean almost periodicity in arbitrary $\sigma$-compact LCAG seems not to have been undertaken before.

\section{Mean almost periodic functions and measures}

\subsection{Mean almost periodic functions}

\begin{definition}[Mean almost periodic classes of functions]
Let $\cA =(A_n)$ be a F\o lner sequence on $G$.
A class $[f] \in
BC^p_{\cA}(G)/\equiv$ is called \textit{mean $p$-almost periodic}\index{almost~periodic!mean~$p$-almost~periodic~class} with respect to
$\cA$ if, for each $\eps >0$, the set
\[
\AP^{\equiv}_{\be,p,\cA}([f], \eps):= \{ t \in G\, :\, \big\| [f]-T_t [f] \big\|_{\be,p,\cA} < \eps \}
\]
of \textit{$\eps$-mean almost periods of $[f]$}\index{almost~periods!mean~$p$-almost~periods~for~classes} is relatively dense. \nomenclature{$\AP^{\equiv}_{\be,p,\cA}([f], \eps)$}{mean $p$-almost periods for classes}
We denote the set of mean $p$-almost periodic classes in
$BC^p_{\cA}(G)/\equiv$ by $\mean^p_{\cA,\equiv} (G)$. \nomenclature{$\mean^p_{\cA,\equiv} (G)$}{set of mean $p$-almost periodic classes}

A function $f \in BC^p_{\cA}(G)$ is called \textit{mean $p$-almost periodic}\index{almost~periodic!mean~$p$-almost~periodic~function} if
$[f]$ is a mean $p$-almost periodic class.
We denote by $\mean^p_{\cA}(G)$ the set of mean $p$-almost periodic functions on $G$.
We will also set \nomenclature{$\mean^p_{\cA}(G)$}{space of mean $p$-almost periodic functions}
\[
\Mean_{\cA}(G):= \left(\mean^1_{\cA} (G) \right) \cap \Cu(G) \,.           \tag*{$\Diamond$}
\]
\end{definition}

Below, we will show that in the definition of $\Mean_{\cA}(G)$, the choice $p=1$ can be replaced by any $1 \leq p < \infty$. First, for $f \in \mean^p_{\cA} (G)$, let us set
\[
\AP_{\be,p,\cA}(f,\eps):= \{t \in G : \| f- \tau_t f \|_{\be,p,\cA} < \varepsilon \}
\]
and refer to this as the set of \textit{$\eps$-mean almost periods of $f$}\index{almost~periods!mean~$p$-almost~periods~for~functions}. \nomenclature{$\AP_{\be,p,\cA}(f, \eps)$}{mean $p$-almost periods for functions}

The following result is an immediate consequence of Corollary~\ref{prop:translation2}. We will prove a more general version of this below in Proposition~\ref{prop:independence}.

\begin{lemma}[Equivalence of mean almost periodicity for classes of bounded functions]
For any F\o lner sequence $\cA$ and any $f \in L^\infty(G) \cap  BC^p_{\cA}(G)$, one has
\[
\AP_{\be,p,\cA}(f,\eps)= \AP^{\equiv}_{\be,p,\cA}([f], \eps)\,.
\]
Therefore, $f$ is mean almost periodic if and only if,
for each $\eps >0$, the set $\AP_{\be,p,\cA}(f,\eps)$ is relatively dense.

In particular, the equivalence holds for all $f \in \Cu(G)$. \qed
\end{lemma}

In the spirit of Proposition~\ref{prop:SAPchar}, one can characterize mean $p$-almost periodicity in the following way.

\begin{theorem}[Characterization of mean almost periodicity via the hull]\label{thm-hull-char-map} Let $\cA$ be a F\o lner sequence on $G$, and let $1 \leq p < \infty$. For $[f] \in BC^p_{\cA}(G)/\equiv$ the following are equivalent:
\begin{itemize}
  \item[(i)] $[f]$ is mean $p$-almost periodic.
  \item[(ii)] $\{ T_t [f] : t \in G \}$ has a compact closure in $( BC^p_{\cA}(G)/\equiv , \| \cdot \|_{\be, p, \cA})$.
\item[(iii)] The closure $\overline{\{ T_t [f] : t \in G \}}$ of $\{ T_t [f] : t \in G \}$ in $( BC^p_{\cA}(G)/\equiv , \| \cdot \|_{\be, p, \cA})$ becomes a compact Abelian group with the operation induced by
\[
    T_t[f] \oplus T_s [f]= T_{t+s} [f] \qquad \text{ for all } t,s \in G \,.
\]
\item[(iv)] There exists a continuous function
\[
\Psi : G_{\mathsf{b}} \to \overline{\{ T_t [f] : t \in G \}}
\]
satisfying
\[
\Psi(i_\mathsf{b}(t))=T_t[f] \qquad \text{ for all } t \in G \,.
\]
\end{itemize}
Moreover, in this case, the mapping $\Psi$ is an onto group homomorphism.
\end{theorem}
\begin{proof}
By Theorem~\ref{thm:completeness}, we know that $( BC^p_{\cA}(G)/\equiv , \| \cdot \|_{\be, p, \cA})$ is complete. Moreover, by \eqref{Tt}, the action $T_t$ is continuous at each $[f] \in BC^p_{\cA}(G)/\equiv$. Finally, for each $t$, the mapping $T_t$ is an isometry on $( BC^p_{\cA}(G)/\equiv , \| \cdot \|_{\be, p, \cA})$.
Therefore, the conditions of \cite[Thm.~4.12]{LSS4} hold. The claim follows from \cite[Thm.~4.12]{LSS4}.
\end{proof}


Next, let us show that $\mean^p_{\cA,\equiv} (G)$ is a closed subspace of $BC^p_{\cA}(G)/\equiv$. This will imply that $\mean^p_{\cA} (G)$ is a closed subspace of $BC^p_{\cA}(G)$.

\begin{prop}\label{MAPisclosed} The space $\mean^p_{\cA,\equiv} (G)$ is a closed $G$-invariant subspace of $( BC^p_{\cA}(G)/\equiv , \| \cdot \|_{\be, p, \cA})$.

In particular, $(\mean^p_{\cA,\equiv} (G), \| \cdot \|_{\be,p,\cA})$ is a Banach space.
\end{prop}
\begin{proof}
We start by showing that $\mean^p_{\cA,\equiv} (G)$ is a subspace.
The set is clearly closed under multiplication with scalars. We show that it is closed under addition.

Let $\varepsilon >0$ be given. Let us note first that if $t,s \in\AP^{\equiv}_{\be,p,\cA}([f], \eps)$, one has
\begin{align*}
\|[f] - T_{t-s} [f]\|_{\be,p,\cA}
	&\leq \|[f] - T_t  [f]\|_{\be,p,\cA} + \|T_t[f] - T_{t-s} [f]\|_{\be,p,\cA} \\
	&<\varepsilon + \|T_s [f] - [f]\|_{\be,p,\cA} < 2 \varepsilon \,.
\end{align*}
Hence, $t-s \in \AP^{\equiv}_{\be,p,\cA}([f], 2\eps)$. Here, we used the fact that $T_t$ is an isometry on $BC^p_{\cA}(G)/\equiv$.

Next, let $[f],[g] \in \mean^p_{\cA,\equiv} (G) $ and
$\varepsilon >0$ be given. By \eqref{Tt}, we can find an open neighborhood
$U$ of $0\in G$ such that, for all $t \in U$, we have
\begin{align*}
\|[f] - T_t [f]\|_{\be,p,\cA} \,&<\, \varepsilon \,,\\
\|[g] - T_t[g]\|_{\be,p,\cA} \,&<\, \varepsilon \,.
\end{align*}
We obtain the inclusions
\[
\AP^{\equiv}_{\be,p,\cA}([f],\eps)
- \AP^{\equiv}_{\be,p,\cA}([f],\eps) + U \subseteq \AP^{\equiv}_{\be,p,\cA}([f],3\varepsilon)
\]
and
\[
(\AP^{\equiv}_{\be,p,\cA}([g],\varepsilon) - \AP^{\equiv}_{\be,p,\cA}([g],\varepsilon)) + U \subseteq \AP^{\equiv}_{\be,p,\cA}([g],3\varepsilon)
\,.
\]
Moreover, as both $\AP^{\equiv}_{\be,p,\cA}([f],\varepsilon) $ and
$\AP^{\equiv}_{\be,p,\cA}([g],\varepsilon)$ are relatively dense by assumption, the set
\[
\Big(\big( \AP^{\equiv}_{\be,p,\cA}([f],\varepsilon) - \AP^{\equiv}_{\be,p,\cA}([f],\varepsilon)\big) + U\Big)\cap  \Big(\big(\AP^{\equiv}_{\be,p,\cA}([g],\varepsilon) - \AP^{\equiv}_{\be,p,\cA}([g],\varepsilon)\big) +
U\Big)
\]
is relatively dense by standard arguments, see e.g. \cite[App.]{LSS}. Hence, there is a relatively dense set of
$3\varepsilon$-almost periods of both $[f]$ and $[g]$, and this easily
gives  that $[f] + [g]=[f+g]$ has a relatively dense set of $6\varepsilon$-
periods. As $\varepsilon>0$ was arbitrary, we infer that $[f]+[g]$ is
mean almost periodic.

Next, since for each $t$, $T_t$ is an isometry on $( BC^p_{\cA}(G)/\equiv , \| \cdot \|_{\be, p, \cA})$, $\mean^p_{\cA,\equiv} (G)$ is clearly $G$-invariant.

Finally, it is immediate that $\mean^p_{\cA,\equiv} (G)$ is closed in $( BC^p_{\cA}(G)/\equiv , \| \cdot \|_{\be, p, \cA})$.
Indeed, let $[f_n] \in \mean^p_{\cA,\equiv} (G)$ be such that it converges to some $[f] \in ( BC^p_{\cA}(G)/\equiv , \| \cdot \|_{\be, p, \cA})$.
Let $\varepsilon >0$. There exists some $n$ such that $\| [f]-[f_n] \|_{\be, p, \cA} < \varepsilon$. An immediate computation shows that
\[
\AP^{\equiv}_{\be,p,\cA}\big([f_n],\varepsilon\big) \subseteq \AP^{\equiv}_{\be,p,\cA}\big([f],3\varepsilon\big) \,.
\]
Since $\AP^{\equiv}_{\be,p,\cA}([f_n],\varepsilon)$ is relatively dense, so is $\AP^{\equiv}_{\be,p,\cA}([f],3\varepsilon)$. This implies that $[f] \in  \mean^p_{\cA,\equiv} (G)$.
\end{proof}

\begin{remark}
The fact that $\mean^p_{\cA,\equiv} (G)$ is closed under addition also follows immediately from Theorem~\ref{thm-hull-char-map}.
Indeed, by Theorem~\ref{thm-hull-char-map}, for $[f],[g] \in \mean^p_{\cA,\equiv} (G)$, the sets $\overline{\{ T_t [f] : t \in G \}}$ and $\overline{\{ T_t [g] : t \in G \}}$ are compact in $( BC^p_{\cA}(G)/\equiv , \| \cdot \|_{\be, p, \cA})$. The continuity of the addition implies that the range $K$ of the mapping
\[
+: \overline{\{ T_t [f] : t \in G \}} \times \overline{\{ T_t [g] : t \in G \}} \to BC^p_{\cA}(G)/\equiv
\]
must be compact. Since this range contains $\{ T_t ([f]+[g]) : t \in G \}$, Theorem~\ref{thm-hull-char-map} implies that $[f]+[g] \in \Map^p_{\cA,\equiv} (G)$.      \exend
\end{remark}

There are two immediate consequences of Proposition~\ref{MAPisclosed}.

\begin{coro}
The space $\mean^p_{\cA} (G)$ is a closed $G$-invariant subspace of $( BC^p_{\cA}(G) , \| \cdot \|_{\be, p, \cA})$.

In particular, $(\mean^p_{\cA}(G), \| \cdot \|_{\be,p,\cA})$ is complete.
\qed
\end{coro}

\begin{coro}\label{prop:map-subspace}
The space $\Mean_{\cA} (G)$ is a closed $G$-invariant subspace of $( \Cu(G), \| \cdot \|_\infty)$.
\end{coro}
\begin{proof}
This follows from $\| \cdot \|_{\be,p,\cA} \leq \| \cdot \|_\infty$.
\end{proof}

The following is an immediate consequence of Lemma~\ref{lemma norm inequality} and Lemma~\ref{lemma C-S}.

\begin{prop}\label{prop-inclus-map}
Let $1 \leq p \leq q < \infty$.
\begin{itemize}
\item[(a)] The inclusions
\begin{align*}
\mean^q_{\cA,\equiv} (G) \,&\subseteq\, \mean^p_{\cA,\equiv} (G) \,,\\
\mean^q_{\cA} (G) \,&\subseteq\, \mean^p_{\cA} (G)
\end{align*}
hold with continuous embedding.
\item[(b)] We have
\[
\left( \mean^p_{\cA} (G) \right) \cap L^\infty(G) = \left( \mean^1_{\cA} (G) \right) \cap L^\infty(G) \,.
\]
In particular, one has
\begin{align*}
\left(   \mean^p_{\cA} (G) \right) \cap L^\infty(G) \,&=\, \left(  \mean^q_{\cA} (G)\right)  \cap L^\infty(G) \,,\\
  \Mean_{\cA}(G) \, &=\,  \left( \mean_{\cA}^p(G) \right) \cap \Cu(G) \,.      \tag*{$\qed$}
\end{align*}
\end{itemize}
\end{prop}

Next, let us discuss the space $\mean^p_{\cA} (G) \cap L^\infty(G)$ in more detail. We can now show that, for bounded functions, mean almost periodicity does not depend on $p \geq 1$.


\begin{prop}[Independence of $p\geq 1$ for bounded functions]\label{prop:independence}
For any $f \in BC^1_{\cA}(G) \cap L^\infty(G)$. The following statements are equivalent:
\begin{itemize}
  \item[(i)] The function  $f $ belongs to $\mean^1_{\cA} (G)$.
  \item[(ii)] There exists some $1 \leq p < \infty$ such that $f \in \mean_{\cA}^p(G)$.
  \item[(iii)] For all $1 \leq p < \infty$, we have $f \in \mean_{\cA}^p(G)$.
  \item[(iv)] For all $\varepsilon >0$, the set $\AP_{\be,1,\cA}(f,\eps)$  is relatively dense.
 \item[(v)] There exists some $1 \leq p < \infty$ such that, for all $\varepsilon >0$, the set $\AP_{\be,p,\cA}(f,\eps)$  is relatively dense.
  \item[(vi)] For all $1 \leq p < \infty$ and for all $\varepsilon >0$, the set $\AP_{\be,p,\cA}(f,\eps)$  is relatively dense.
\end{itemize}
\end{prop}
\begin{proof}
The equivalence of (i),(ii) and (iii) follows from Proposition~\ref{prop-inclus-map}.
The equivalences (i)$\Longleftrightarrow$ (iv), (ii)$\Longleftrightarrow$ (v), (iii)$\Longleftrightarrow$ (vi) follow from
\[
\| T_t[f] -[f] \|_{\be,p,\cA}= \| \tau_t f -f \|_{\be,p,\cA}  \,,
\]
which holds for all $f \in L^\infty(G)$ by \eqref{eq-tra}.
\end{proof}

\begin{remark} For $G=\RR$, a function satisfying Proposition~\ref{prop:independence}(iv) is called \textit{almost-periodic} in the sense of Doss in \cite[Def.~5.16]{ABG}.      \exend
\end{remark}

The next result is an immediate consequence.

\begin{coro} Let $f \in \Cu(G)$ and $\cA$ a F\o lner sequence. Then, the following statements are equivalent:
\begin{itemize}
  \item[(i)] The function $f$ belongs to $\mean_{\cA}(G)$.
  \item[(ii)] There exists some $1 \leq p < \infty$ such that $f \in \mean^p_{\cA}(G)$.
 \item[(iii)] For all $1 \leq p < \infty$, we have $f \in \mean^p_{\cA}(G)$.
   \item[(iv)] For all $\varepsilon >0$, the set $\AP_{\be,1,\cA}(f,\eps)$  is relatively dense.
 \item[(v)] There exists some $1 \leq p < \infty$ such that, for all $\varepsilon >0$, the set $\AP_{\be,p,\cA}(f,\eps)$  is relatively dense.
  \item[(vi)] For all $1 \leq p < \infty$ and for all $\varepsilon >0$, the set $\AP_{\be,p,\cA}(f,\eps)$  is relatively dense.\qed
\end{itemize}
\end{coro}

It is now possible state the following characterization for $\mean_{\cA,\equiv}^p(G)$, which can be used as an equivalent definition.

\begin{lemma}\label{lem-intrinsic-char-map}
Let $\cA$ be a van Hove sequence. The space $\mean_{\cA,\equiv}^p(G)$ is the closure of $\Mean_{\cA}(G)/\equiv$ in
$(BC_{\cA}^p(G)/\equiv, \| \cdot \|_{\be,p,\cA})$.
\end{lemma}
\begin{proof}
Let $[f]\in \mean_{\cA,\equiv}^p(G)$ be given. Without loss of generality, we can assume that $f$ is bounded as otherwise we could use
Proposition~\ref{prop:approxiimation by bounded functions} to replace it by a cut-off version of it.
Now,  we have  $T(\varphi )[f] = [f*\varphi]$ for any $\varphi \in\Cc (G)$ and $f*\varphi$ belongs to $\Cu (G)$. Hence,  a short computation gives
\[
\| \tau_t (f*\varphi) -  f*\varphi\|_{\be,p,\cA} = \|T(\varphi) ( T_t[f] - [f])\|_{\be,p,\cA} \leq \|\varphi\|_{1}\, \| T_t [f]  - [f]\|_{\be,p,\cA}
\]
by the results above.  From the assumption on $f$, we easily infer that  $f*\varphi$ is mean almost
periodic. The desired statement now follows by
choosing an approximate identity $(\varphi_\alpha)$ in $\Cc(G)$ and
noting that $T(\varphi_\alpha)[f] =[f*\varphi_\alpha] \to [f]$.
\end{proof}

\subsection{Mean almost periodic measures}

As usual, we extend the definition of mean almost periodicity from functions to measures via convolutions with compactly supported continuous functions.

\begin{definition}[Mean almost periodic measures]
Let $\cA =(A_n)$ be a van Hove sequence on $G$, and let $1 \leq p
<\infty$.
A measure $\mu$ on $G$ is called \textit{mean $p$-almost
periodic}\index{almost~periodic!mean~$p$-almost~periodic~measure} with respect to $\cA$ if, for all $\varphi \in \Cc(G)$, we
have $\mu*\varphi \in \mean^p_{\cA}(G)$.
The  space of mean $p$-almost
periodic measures is denoted by $\Map^p_{\cA}(G)$. \nomenclature{$\Map^p_{\cA} (G)$}{set of mean $p$-almost periodic measure}  \exend
\end{definition}

As noted in Proposition~\ref{prop-inclus-map}, $\mean_{\cA}(G)$
contains $\mean_\cA^p(G)$ for all $1 \leq p$. Hence, $\Map_{\cA}^1
(G)$ contains $\Map_{\cA}^p(G)$ for all $1\leq p$. For this reason,
we often drop the superscript $1$ when referring to $\Map_{\cA}^1
(G)$ and call its elements just \textit{mean almost periodic
measures}\index{almost~periodic!mean~almost~periodic~measure}.

The following is an immediate consequence of Proposition~\ref{prop-inclus-map}.

\begin{coro}
For a van Hove sequence $\cA$ and $\mu \in \cM^\infty(G)$, the following assertions are equivalent:
\begin{itemize}
  \item[(i)] The measure $\mu$ belongs to $\Map_{\cA}^1 (G)$.
  \item[(ii)] There exists some $p \geq 1$ such that $\mu \in \Map_{\cA}^p (G)$.
  \item[(iii)] For all $p \geq 1$ we have $\mu \in \Map_{\cA}^p (G)$.
  \item[(iv)] For each $\eps >0$ and each $\varphi \in \Cc(G)$, the set
  \[
\AP_{\be,1,\cA}(\mu*\varphi,\eps)= \{ t  \in G: \| (\tau_t \mu-\mu)*\varphi \|_{\be,1,\cA} < \eps \}
\]
is relatively dense. \qed
\end{itemize}
\end{coro}

Mean almost periodic functions and mean almost periodic, absolutely continuous
measures are related in the following way.

\begin{prop}\label{prop-map-measure-vs-function}
Let $\cA$ be a van Hove sequence on $G$. Assume that
$f\in BC^1_{\cA} (G)$ is such that $f\theta_G$ is translation bounded.
Then, $f \in \mean_{\cA}(G)$ if and only if $f \theta_G \in
\Map_{\cA}(G)$.
\end{prop}
\begin{proof} This follows from Proposition~\ref{prop:measure-vs-function} applied to $\fS' = \mean_{\cA}
(G)$. The assumptions of that proposition are satisfied by
Corollary~\ref{prop:map-subspace}.
\end{proof}

\begin{coro} Let $\cA$ be a van Hove sequence. A function
$f\in BC^1_{\cA} (G) \cap L^\infty(G)$ is an element of $\mean_{\cA}(G)$ if and only if $f \theta_G \in
\Map_{\cA}(G)$.

In particular, $f\in\Cu (G)$ is mean almost periodic
if and only if $f\theta_G$ is mean almost periodic. \qed
\end{coro}

\section{Characterization of pure-point diffraction}
In this subsection, we characterize pure-point
diffraction via mean almost periodicity. This provides  our solution
to Question 1.
We start by solving the more general question of characterizing pure-point spectra for $\cA$-representations.

\begin{theorem}[Characterization of $\cA$-representations with pure-point spectrum]
\label{thm:a-representation-mean} Let $N: \Cc (G)\longrightarrow
BC^2_{\cA}(G)$ be an intertwining  $\cA$-representation with  diffraction measure
$\sigma$. Then, the following assertions are equivalent:
\begin{itemize}
\item[(i)] The diffraction measure $\sigma$  is a pure-point measure.
\item[(ii)] One has $N(\Cc(G))\subseteq \mean^2_\cA (G)$.
\end{itemize}
\end{theorem}
\begin{proof} Clearly, the measure $\sigma$ is a pure-point measure if and only if
$|\widehat{\varphi}|^2 \sigma$ is a pure-point measure for all
$\varphi \in\Cc (G)$. By Lemma \ref{lem:char-intertwining}, we have
$|\widehat{\varphi}|^2\, \sigma = \sigma_\varphi$ for all $\varphi
\in\Cc (G)$. From Proposition~\ref{prop:Hilbert}, we see that
$\sigma$ is a pure-point measure if and only if, for all $\varepsilon
>0$ and $\varphi \in\Cc (G)$, the set
\[
\{t\in G : \|T_t [ N(\varphi)] - [N(\varphi)]\|_{\be,2, \cA}< \varepsilon\}
\]
is relatively dense. This is equivalent to $N(\varphi) \in
\mean^2_{\cA} (G)$ by Lemma~\ref{lem-intrinsic-char-map}.
\end{proof}

Our first main result (Result 1) is a rather direct consequence of
Theorem~\ref{thm:a-representation-mean}.

\begin{theorem}[Characterization of pure point diffraction] \label{single element}
Let $\mu \in \cM^\infty(G)$, and let $\gamma$ be its autocorrelation with respect to some van Hove sequence $\cA$. Then, $\widehat{\gamma}$ is a pure point
measure if and only if $\mu$ is mean almost periodic with respect to $\cA$. \index{almost~periodic!mean~almost~periodic~measure} \index{diffraction!pure~point}
\end{theorem}
\begin{proof} As discussed in Proposition~\ref{prop:translation-bounded-admissible},
a translation bounded measure $\mu$ with autocorrelation $\gamma$
gives rise to the intertwining  $\mathcal{A}$-representation $N_\mu$ (defined by
$N_\mu (\varphi) = \mu \ast \varphi$) with autocorrelation $\gamma$.
Now, the claim is a direct consequence of Theorem~\ref{thm:a-representation-mean} and
the definition of mean almost periodicity for measures.
\end{proof}

We complete the section by relating the mean almost periods of $\mu$ with the strong almost periods on the autocorrelation.
In particular, this will provide a second independent proof of Theorem~\ref{single element}.

\begin{theorem}\label{rel-AP}
Let $\mu \in \cM^\infty(G)$, and let $\gamma$ be its autocorrelation with respect to some van Hove sequence $\cA$. Then,
\begin{equation}\label{EQ33}
\|\mu*\varphi-\tau_{t}(\mu*\varphi)\|_{\be,2,\cA}^2  = [
(\gamma*\varphi*\widetilde{\varphi})(0)  - \re
(\gamma*\varphi*\widetilde{\varphi})(t)]
\end{equation}
holds for all $\varphi \in \Cc(G)$. In particular, we have
\begin{align*}
\AP_{\infty}(\gamma*\varphi*\widetilde{\varphi},\frac{\eps^2}{2}) \,& \subseteq \,\AP_{\be,2,\cA}(\mu*\varphi,\eps) \,,\\
\AP_{\be,2,\cA}(\mu*\varphi, \frac{\eps^2}{1+(\gamma*\varphi*\widetilde{\varphi}) (0) \| \mu*\varphi \|_\infty}) \,& \subseteq\,  \AP_{\infty}(\gamma*\varphi*\widetilde{\varphi},\eps) \,.
\end{align*}
\end{theorem}
\begin{proof}
Since $\gamma$ is the autocorrelation of $\mu$ with respect to $\cA$,
Proposition~\ref{prop-compute-autocorrelation} implies that
\begin{align*}
\lim_{n\to\infty} \frac{1}{|A_{n}|} \int_{A_{n}} (\mu*\varphi)(x)\,
          (\widetilde{\mu*\varphi})(-x)\ \dd x
       &= (\gamma*\varphi*\widetilde{\varphi})(0)\,,     \\
\lim_{n\to\infty} \frac{1}{|A_{n}|} \int_{A_{n}} \tau_t(\mu*\varphi)(x)\, (\widetilde{\mu*\varphi})(-x)\ \dd x
       &=  (\gamma*\varphi*\widetilde{\varphi})
          (-t)\,, \\
 \lim_{n\to\infty} \frac{1}{|A_{n}|} \int_{A_{n}}  (\mu*\varphi)(x)\, (\widetilde{\tau_t(\mu*\varphi)})(-x)\ \dd x  &=(\gamma*\varphi*\widetilde{\varphi})(t)\,, \\
 \lim_{n\to\infty} \frac{1}{|A_{n}|} \int_{A_{n}}  \tau_t(\mu*\varphi)(x)\, (\widetilde{T_t(\mu*\varphi)})(-x)\ \dd x  &= (\gamma*\varphi*
    \widetilde{\varphi})(0)
\end{align*}
holds for all $\varphi \in \Cc(G)$ and all $t \in G$. Hence, we obtain
\begin{align*}
{}
    &\|\mu*\varphi-\tau_{t}(\mu*\varphi)\|_{\be,2,\cA}^2 \nonumber\\
    &\phantom{====}= \limsup_{n\to\infty} \frac{1}{|A_n|} \int_{A_n}
      \left(\mu*\varphi-T_{t}(\mu*\varphi)\right)(x)\, \widetilde{(\mu*\varphi-
      T_{t}(\mu*\varphi))}(-x)\ \dd x    \nonumber\\
    &\phantom{====}= \lim_{n\to\infty} \frac{1}{|A_{n}|} \int_{A_{n}}
      (\mu*\varphi)(x)\,(\widetilde{\mu*\varphi})(-x)\ \dd x \nonumber \\
    &\phantom{=======}-
      \lim_{n\to\infty} \frac{1}{|A_{n}|} \int_{A_{n}} T_t(\mu*\varphi)(x)\,
       (\widetilde{\mu*\varphi})(-x)\ \dd x \nonumber \\
    &\phantom{=======}-  \lim_{n\to\infty} \frac{1}{|A_{n}|} \int_{A_{n}}
         (\mu*\varphi)(x)\, (\widetilde{T_t(\mu*\varphi)})(-x)\ \dd x\nonumber \\
    &\phantom{=======}+  \lim_{n\to\infty} \frac{1}{|A_{n}|} \int_{A_{n}}
         T_t(\mu*\varphi)(x)\, (\widetilde{T_t(\mu*\varphi)})(-x)\ \dd x  \nonumber \\
     &\phantom{====}=  (\gamma*\varphi*\widetilde{\varphi})(0) - (\gamma*\varphi*\widetilde{\varphi})
          (-t) - (\gamma*\varphi*\widetilde{\varphi})(t) + (\gamma*\varphi*
          \widetilde{\varphi})(0)   \\
     &\phantom{====}= 2[ (\gamma*\varphi*\widetilde{\varphi})(0)  - \re (\gamma*\varphi*\widetilde{\varphi})(t)]
\end{align*}
with the last equality following from the fact that $\gamma*\varphi*\widetilde{\varphi}$ is a positive definite function and satisfies
\[
(\gamma*\varphi*\widetilde{\varphi})(-t) =\overline{ (\gamma*\varphi*\widetilde{\varphi})(t)} \,.
\]
This proves \eqref{EQ33}. The inclusions of the sets of almost periods follows immediately from Krein's
inequality \cite[p.~12]{BF}.
\end{proof}

\begin{remark} As it is instructive, we include a second proof of Theorem~\ref{single element}:
By Theorem~\ref{rel-AP}, $\mu$ is mean almost periodic if and only if $\gamma*\varphi*\widetilde{\varphi}$ is Bohr almost periodic for all $\varphi \in \Cc(G)$.
By the definition of Fourier transformability, for all $\varphi \in \Cc(G)$, the measure $\widehat{\gamma} \left| \check{\varphi} \right|^2$ is finite  and
\[
(\gamma*\varphi*\widetilde{\varphi})(t)= \left( \widehat{\gamma} \left| \reallywidecheck{\varphi} \right|^2 \right)\!\reallywidecheck{\phantom{i}} (t) \,.
\]
By \cite[Thm.~4.8.11]{MoSt}, for each $\varphi \in \Cc(G)$, $\gamma*\varphi*\widetilde{\varphi}$ is Bohr almost periodic if and only if $\widehat{\gamma} \left| \reallywidecheck{\varphi} \right|^2 $ is a pure point measure. The claim follows immediately.   \exend
\end{remark}

%



\section{Delone and Meyer sets with pure point diffraction}
In the context of Aperiodic Order, particular cases of interest are
translation bounded measures arising from Delone sets and Meyer
sets, see \cite{Meyer,MOO,Moody2} for definitions and properties. Specifically, for a uniformly discrete set $\varLambda\subseteq G$,
we define its \textit{Dirac comb}\index{Dirac~comb} to be the measure \nomenclature{$\delta_{\varLambda}$}{Dirac comb of the set $\varLambda$}
\[
\delta_\varLambda:=\sum_{t\in \varLambda} \delta_t
\]
with $\delta_s$ being the unit point mass at $s\in G$. For such
measures, sufficient conditions for pure point diffraction have
been given earlier. Here, we use Theorem~\ref{single element} to
recover these conditions and even show that they are necessary by
characterizing mean almost periodicity of Dirac combs of Delone
sets and measures with Meyer set support.

\medskip

In the spirit of \cite{Gouere-2}, we define \nomenclature{$\vL\, \triangle_U\, \vG$}{$U$-symmetric difference of $\vL$ and $\vG$}
\[
\vL\, \triangle_U\, \vG := \left( \vL \backslash
(\vG+U) \right) \cup \left( \vG \backslash (\vL+U) \right)
\,,
\]
when $\vL$ and $\vG$ are Delone sets in $G$ and $U$
is an open neighborhood of $0\in G$.

With this notation, we have the following equivalent characterization of mean almost periodicity for Dirac combs.

\begin{theorem}[Characterizing mean almost periodicity for Delone sets]\label{gou}
 If $\vL$ is a Delone set, then $\delta_\vL$ is mean almost periodic if and only if, for each open neighbourhood $U \subseteq G$ of $0$ and each $\eps >0$, the set
\[
\Big\{ t\in G\, :\, \limsup_{n\to\infty} \frac{\card \left(
\vL\, \triangle_U\, (t+\vL) \right) \cap A_n }{|A_n|} < \eps
\Big\}
\]
  is relatively dense. \index{almost~periodic!mean~almost~periodic~measure}
\end{theorem}
\begin{proof}
 $\Longrightarrow$: Let $U$ be an open neighbourhood of $0$, and
let $\eps >0$. Choose some open set $V$ such that $0 \in V$ and
$\overline{V} \subset U$. Let $\varphi \in \Cc(G)$ be such that
$\varphi \geq 1_V$ and $\supp(\varphi) \subset U$. A simple
computation shows that
\begin{displaymath}
\left| (\delta_\vL*\varphi)(x)- \tau_t
(\delta_\vL*\varphi)(x) \right| \geq 1
\end{displaymath}
for all $t \in G$ and for all $x \in \left( \vL\, \triangle_U\,
(t+\vL) \right) +V$. Therefore, for all $t \in G$, we have
\begin{displaymath}
\theta_G(V) \limsup_{n\to\infty} \frac{\card \left( \vL\,
\triangle_U\, (t+\vL) \right) \cap A_n }{|A_n|} \leq \|
\delta_\vL*\varphi- \tau_t  (\delta_\vL*\varphi)\|_{\be,1,\cA} \,.
\end{displaymath}
The claim follows.

\medskip

\noindent $\Longleftarrow$: Fix some open neighbourhood $U=-U$ of $0$ such that $\vL$ is $U-U$-uniformly discrete. Let $\varphi \in \Cc(G)$, $\eps >0$, and let $K:=\supp(\varphi)$.
As $\varphi$ is uniformly continuous, there exists some open neighbourhood $V \subseteq U$ of $0$ such that, for all $x,y \in G$ with $x-y \in V$, we have
\[
|\varphi(x)-\varphi(y)| < \frac{\eps}{2\cdot |K|\cdot \| \delta^{}_{\vL} \|_{\be,\cA} \, +1} \,.
\]
Now, for each $x \in \vL \backslash \left( \vL\, \triangle_V\, (t+\vL) \right)$, there is a unique $y_x \in (t+\vL) \backslash \left( \vL\, \triangle_V\, (t+\vL) \right)$ such that $x-y_x \in V$.
We know that the set
\[
P:= \Big\{ t\in G\, :\, \limsup_{n\to\infty} \frac{\card \left( \vL\, \triangle_V\, (t+\vL) \right) \cap A_n }{|A_n|} < \frac{\eps}{2 \| \varphi \|_1  +1} \Big\}
\]
is relatively dense.
It follows from a standard Fubini and van Hove type argument that, for all $t \in P$, we have
\begin{align*}
{}
	&\| \delta_\vL*\varphi- \tau_t  \delta_\vL*\varphi \|_{\be,1,\cA}   \\
        &\phantom{XXXX}\leq \limsup_{n\to\infty} \frac{\card \left( \vL\, \triangle_V\,
      (t+\vL) \right) \cap A_n }{|A_n|}\,\int_G |\varphi(t)|\ \dd t \\
        &\phantom{XXXX}\phantom{===}+ \limsup_{n\to\infty} \frac{1}{|A_n|}\sum_{x \in \vL
      \backslash \left( \vL \triangle_V (t+\vL) \right)\cap A_n}
      \int_G \left| \varphi(x-t) -\varphi(y_x-t) \right|\ \dd t   \\
        &\phantom{XXXX}\leq \frac{\eps}{2} + \limsup_{n\to\infty} \frac{1}{|A_n|}
      \sum_{x \in \vL \backslash \left( \vL \triangle_V (t+\vL)
       \right)\cap A_n} \frac{\eps}{2\cdot |K|\cdot \| \delta^{}_{\vL} \|_{\be,\cA} \, +1}|K|  \\
        &\phantom{XXXX}\leq \frac{\eps}{2} +\| \delta^{}_{\vL} \|_{\be,\cA} \frac{\eps}{2\cdot |K|\cdot \| \delta^{}_{\vL} \|_{\be,\cA} \, +1}\, |K| \, < \, \eps \,.
\end{align*}
This finishes the proof.
\end{proof}

Point sets $\vL$, which are such that $\vL - \vL$ is uniformly discrete,
are particularly relevant, see Remark~\ref{rem-ba-gou} below. For
such sets (and even measures supported on such sets),  we can
characterize pure point diffraction as follows.


\begin{theorem}[Characterizing mean almost periodicity for measures supported inside Meyer sets]
\label{ba} Let $\mu$ be a translation bounded measure supported
inside the uniformly discrete set $\vL$.
 \begin{itemize}
   \item[(a)] If, for each $\eps >0$, the set
 \[
 \AP_{\be,\cA}(\mu,\eps):=
\Big\{ t\in G\, :\, \| \mu -\tau_t \mu \|_{\be,\cA} < \eps
\Big\}
 \]
is relatively dense, then $\mu \in \Map_{\cA}$. \index{almost~periodic!mean~almost~periodic~measure}
   \item[(b)] If $\mu \in \Map_{\cA}(G)$ and $\vL -\vL$ is uniformly discrete, then, for all $\eps >0$, the set $ \AP_{\be,\cA}(\mu,\eps)$ is relatively dense.
 \end{itemize}
\end{theorem}
\begin{proof}
(a) Let $\varphi \in \Cc(G)$. The claim follows from Lemma~\ref{lem-bes-ine-measures} because
\[
\| \mu*\varphi - \tau_t \mu*\varphi \|_{\be,1,\cA} \leq \|\varphi\|_1 \| \mu -\tau_t \mu \|_{\be,\cA} \,.
\]

\smallskip

\noindent (b) If $\| \mu \|_{\be,\cA}=0$, the claim is trivial. So, we can assume that $\| \mu \|_{\be,\cA} >0$. Select an open precompact set $U=-U \subset G$ such that $0\in U$ and $\vL -\vL$ is $U-U-U$-uniformly discrete. Choose a $\varphi \in \Cc(G)$ such that  $0 \leq \varphi \leq 1, \varphi(0)=1$ and $\supp(\varphi) \subset U$.
Let $0< \eps < \frac{24\| \varphi\|_1 \| \mu \|_{\be,\cA}}{|U|}$. By the uniform continuity of $\mu*\varphi$, there exists some open set $W \subseteq U$ such that $0\in W$ and
\[
\| \mu*\varphi -\tau_t \mu*\varphi \|_\infty < \frac{\eps}{ 6 \| \delta^{}_{\vL} \|_{\be,\cA}+1}
\]
for all $t \in W$. Since $\mu$ is translation bounded and mean almost periodic, there exists some compact set $K \subseteq G$ such that
\[
\Big\{ t \in G : \| \mu*\varphi - \tau_t \mu*\varphi \|_{\be,1,\cA} < \frac{\eps|W|}{12} \Big\} +K =G \,.
\]
We show that
\begin{equation}\label{eq rel d}
\AP_{\be,\cA}(\mu,\eps)+K+\overline{U}-\overline{U}= G \,.
\end{equation}
Select $g \in G$ arbitrary and write $g=t+k$ for $k \in K$ and
\[
t \in \Big\{s \in G : \| \mu*\varphi - \tau_s \mu*\varphi \|_{\be,1,\cA} < \frac{\eps|W|}{12} \Big\} \,.
\]
Since
\[
\| \mu*\varphi - \tau_t \mu*\varphi \|_{\be,1,\cA} < \frac{\eps|W|}{12}< 2 \| \varphi\|_1 \|\mu\|_{\be,\cA} \,,
\]
a simple argument shows that $\supp(\varphi*\mu)\cap \supp(\tau_t \varphi *\mu) \neq \varnothing$. Let $x$ be any point in this intersection. Then, there exists some $y_1,y_2 \in \vL$ and $u_1,u_2 \in U$ such that
\[
x=y_1+u_1=-t+y_2+u_2 \,.
\]
Define $s:=y_1-y_2 \in \vL -\vL$. Then, $t=s+(u_1-u_2)$. We show that $s \in \AP_{\be,\cA}(\mu,\eps)$, which proves \eqref{eq rel d} and completes the argument.

Note first that, by the choice of $\varphi$, we have
\[
| \mu*\varphi(x) -\tau_t \mu*\varphi(x) | = \Big| \mu(\{x\})- \sum_{z \in \vL} \varphi(x-t-z) \mu(\{z \}) \Big|
\]
for all $x \in \vL$. Moreover, at most one term in the second sum can be nontrivial. If there exists $z \in \vL$ such that $\varphi(x-t-z) \mu(\{z \}) \neq0$, we must have $x-z=s$. Then,
\[
| \mu*\varphi(x) -\tau_t \varphi(x) | = | \mu(\{x\})-  \varphi(s-t) \mu(\{x-s \}) |  \,,
\]
which trivially also holds when $\varphi(x-t-z) \mu(\{z \}) =0$ for all $z \in \vL$. Similarly,
\[
| \mu*\varphi(x-s+t) -\tau_t \varphi(x-s+t) | = | \mu(\{x\})\varphi(s-t)-  \mu(\{x-s \}) | \,.
\]
A simple computation, compare \cite[Lem.~8]{bm}, yields
\begin{equation}\label{eq1}
\big| \mu\big(\{x\}\big)-  \mu\big(\{x-s \}\big) \big| \leq \left| \mu*\varphi- \tau_t \mu*\varphi \right|(x)+2 \left| \mu*\varphi- \tau_t \mu*\varphi  \right|(x-s+t)
\end{equation}
for all $x \in \vL$. Likewise, for all $x \in s+ \vL$, we have
\begin{equation}\label{eq2}
\big| \mu\big(\{x\}\big)-  \mu\big(\{x-s \}\big) \big| \leq \left| \mu*\varphi- \tau_t \mu*\varphi\right|(x-s+t)+2 \left| \mu*\varphi- \tau_t \mu*\varphi \right|(x)  \,.
\end{equation}
Now, note that
\[
\left| \mu- \tau_s \mu \right|
= \sum_{x \in \vL \cup (s+\vL)} \big| \mu\big(\{x\}\big) - \mu\big(\{x-s\}\big) \big| \delta_x \,.
\]
Set
\[
A:= \Big\{ x \in \vL \cup (s+\vL) : \big| \mu\big(\{x\}\big) -
\mu\big(\{x-s\}\big) \big|< \frac{\eps}{ 6 \| \delta^{}_{\vL} \|_{\be,\cA}+1} \Big\}\]
and define
\[
\omega_1\,=\, \sum_{x \in A} \big| \mu\big(\{x\}\big) - \mu\big(\{x-s\}\big) \big|
\delta_s
\]
as well as
\[
\omega_2\,:=\, \left| \mu- \tau_s \mu \right| -\omega_1 \,.
\]
It is obvious that
\[
\|\omega_1\|_{\be,\cA}< \frac{\eps}{3} \,.
\]
Let $\vG := \supp(\omega_{2})$. Then,  by \eqref{eq1}, \eqref{eq2} and the choice of $W$, we get
\begin{align*}
\int_{x+W} \left| \mu*\varphi- \tau_t \mu*\varphi \right|(r)\, \dd r &+\int_{x+s-t+W} \left| \mu*\varphi- \tau_t \mu*\varphi \right|(r)\, \dd r \\
&\geq |W| \left( \frac{1}{2}\big| \mu\big(\{x\}\big)-  \mu\big(\{x-s \}\big) \big|  - \frac{2\eps}{ 6 \| \delta^{}_{\vL} \|_{\be,\cA}+1} \right)
\end{align*}
for all $x \in \vG$.
Since $\vL- \vL$ is $W$-uniformly discrete, the set $x+W$ with $x \in \vG \subset \vL$ are pairwise disjoint, and so are the sets $x+s-t+W$. Summing the above relation for $x \in \vG \cap A_n$, and using the van Hove property, we get $\| \omega_{2}\|_{\be,\cA} < \frac{2 \eps}{3}$. Therefore,
\[
\| \mu -\tau_s\mu\|_{\be,\cA} \leq \| \omega_{1}\|_{\be,\cA}+\| \omega_{2} \|_{\be,\cA} < \eps \,.
\]
This finishes the proof. \end{proof}

\begin{remark}\label{rem-ba-gou}
\phantom{XX}
\begin{itemize}
\item[(a)] Theorem~\ref{gou} shows that mean almost periodicity for Delone sets
is equivalent to the conditions in \cite[Thm.~3.3(5)]{Gouere-2}.
Hence, \cite[Thm.~3.3]{Gouere-2} is a special case of
Theorem~\ref{single element}.
\item[(b)]  If $G$ is metrisable, it is easy to show that the condition in
Theorem~\ref{gou} can be replaced by the following statement:  For
each $\eps >0$, the set
\begin{displaymath}
 \Big\{ t\in G\, :\, \limsup_{n\to\infty} \frac{\card \left( \vL\, \triangle_{B_\eps(0)}\, (t+\vL) \right) \cap A_n }{|A_n|} < \eps \Big\}
\end{displaymath}
is relatively dense, compare \cite{Gouere-2}. Here, $B_r(0)$ is
the ball around $0\in G$ with radius $r > 0$.
\item[(c)] Using the well-known equivalence between pure point diffraction
and dynamical spectrum \cite{LMS,Gouere-2,BL}, it is easy to see
that Theorem~\ref{gou} generalizes \cite[Thm.~2.2]{SOL}.
\item[(d)]  A Delone set $\vL\subseteq G$ is called a \textit{Meyer set} \index{Meyer~set} if $\vL
- \vL-\vL$ is uniformly discrete.

Any Meyer set satisfies $\vL-\vL$ is uniformly discrete.
Moreover, if $G$ is compactly generated, the Meyer condition is equivalent to
$\vL-\vL$ being uniformly discrete (and even weaker conditions \cite{BLM,LAG1,NS11}). Therefore,
Theorem~\ref{ba} applies to Meyer sets.
\item[(e)] For Dirac combs $\delta_\vL$ of point sets $\vL$ Theorem~\ref{ba} generalizes Theorem~5 of \cite{bm}.  \exend
\end{itemize}
\end{remark}

\chapter{Besicovitch almost periodicity and the consistent phase property}\label{sec-Besicovitch}
In this chapter, we first study Besicovitch almost periodic
functions. This is done in two steps. In the first step, we develop
a general theory. In the second step, we focus on
a certain Hilbert space of Besicovitch almost periodic functions.
This Hilbert space structure will allow us to set up a Fourier
expansion type theory. Having this theory at hand, we can then turn
to $\mathcal{A}$-representations with values in the  Besicovitch
almost periodic functions and characterize them by pure-point
diffraction together with existence of the Fourier coefficients.
From this, we obtain our solution to Question 2.

\section{Besicovitch almost periodic functions: general theory}

\begin{definition}[Besicovitch almost periodic functions]
Let a F\o lner sequence  $\cA =(A_n)$ on $G$ and  $1 \leq p <\infty$
be given.
A function $f \in \mathcal{L}^p_{loc}(G)$ is called
\textit{Besicovitch $p$-almost periodic}\index{almost~periodic!Besicovitch~$p$-almost~periodic~function} with respect to $\cA$ if, for each $\eps >0$, there exists a trigonometric polynomial $P$ with \nomenclature{$\bap^p(G)$}{set of Besicovitch $p$-almost periodic functions}
\[
\| f -P \|_{\be,p,\cA} < \eps \,.
\]
The space of Besicovitch $p$-almost
periodic functions is denoted by $\bap^p(G)$.
Moreover, we set
\[
\bap(G):= \bap^1(G)\,.
\]\exend
\end{definition}

\begin{remark}\label{rem:sap in bap}
\phantom{XX}
\begin{itemize}
  \item[(a)]
A function is Besicovitch $p$-almost periodic if and only if, for
each $\eps >0$, there exists a Bohr almost periodic function $g$
such that $ \| f -g \|_{\be,p,\cA} <\eps$. In particular, all
trigonometric polynomials and all Bohr almost periodic functions are
Besicovitch almost periodic (for any $p\geq 1$). In fact, it is not
hard to see that every weakly almost periodic function is
Besicovitch $p$-almost periodic (for any $p\geq 1$).
  \item[(b)] For  $G=\RR$ a review of Besicovitch almost periodic functions and
  their properties can be found in \cite[Sect.~5]{ABG}.    \exend
\end{itemize}
\end{remark}

\begin{prop}[Inclusions of spaces] \label{ap inclusions 1}
Let $1 \leq p \leq q < \infty$.
\begin{itemize}
  \item [(a)] We have
  \[
  \bap^p(G) \subseteq  \mean^p_{\cA}(G)
  \]
  with a continuous inclusion map.
  \item [(b)] We have
  \[
  \bap^q(G) \subseteq   \bap^p(G) \subseteq   \bap(G)
  \]
  with a continuous inclusion map.
  \item[(c)] $\bap^p(G)$ is the closure of $\sap(G)$ in $BL^p_{\cA}(G)$.
\end{itemize}
\end{prop}
\begin{proof}
(a)  Consider $f\in \bap^p (G)$. Let $\varepsilon >0$ be arbitrary.
Then, there exists a trigonometric
polynomial $P$ with $\|f - P\|_{\be,p,\cA} < \varepsilon$.  Clearly,
the trigonometric polynomial $P$  is Bohr almost periodic and,
hence, also mean almost periodic. Thus, we can approximate $f$
arbitrarily well by mean almost periodic functions in $\Cu (G)$, and
$f\in  \mean^p_{\cA}(G)$ follows.

\medskip

\noindent (b) follows from Lemma~\ref{lemma norm inequality}.

\medskip
\noindent(c) Since trigonometric polynomials belong to $\sap(G)$, any Besicovitch almost periodic function belongs to the closure of $\sap(G)$ in $BL^p_{\cA}(G)$.
Let now $g$ be any function in the closure of $\sap(G)$ in $BL^p_{\cA}(G)$, and let $\eps >0$. There exists some $f \in \sap(G)$ such that
\[
\|g-f\|_{\be,p,\cA} < \frac{\eps}{2}\,.
\]
Since $f \in \sap(G)$, there exists a trigonometric polynomial $P$ such that \cite{MoSt}
\[
\| f-P \|_\infty < \frac{\eps}{2} \,.
\]
Then,
\[
\|g-P\|_{\be,p,\cA} \leq \|g-f\|_{\be,p,\cA} +\|f-P\|_{\be,p,\cA} \leq  \|g-f\|_{\be,p,\cA} +\|f-P\|_{\infty} < \eps \,.
\]
\end{proof}

Let us now give an example of a mean almost periodic function which is not Besicovitch almost periodic.

\begin{example}
\label{rem:map vs bap}
In general, the space $\mean_{\cA}^p (G)$ is strictly bigger than $\bap^p (G)$.
To see this, we consider $G = \RR$ with the van Hove sequence $A_n =
[-n, n]$, and let $f:\RR\longrightarrow \RR$ be the function
\begin{displaymath}
f(x):= \left\{ \begin{array}{cc}
1\,, &\mbox{ if  } x>1 \,, \\
x\,, &\mbox{ if } 0 \leq x \leq 1 \,, \\
 0\,, & \mbox{ if } x < 0 \,.
\end{array}
\right.
\end{displaymath}
Clearly, $f$ is uniformly continuous and $\|f - \tau_t f\|_{\be,1,\cA}
=0$ for all $t\in \RR$. So, $f$ belongs to $ \mean_{\cA}^p (G)$.

However, $f$ does not belong to $\bap^p (G)$. Indeed, let us note that if $g$ is any Bohr almost periodic function, an application of the triangle inequality gives
\begin{align*}
\frac{1}{2n} \int_{-n}^n \left| f(t) -g(t) \right| \dd t
&= \frac{1}{2} \left( \frac{1}{n}\int_{-n}^0\left| f(t) -g(t) \right|\dd t + \frac{1}{n}\int_{0}^n\left| f(t) -g(t) \right|\dd t  \right)\\
  &=\frac{1}{2} \left( \frac{1}{n}\int_{-n}^0 \left|g(t) \right|\dd t + \frac{1}{n}\int_{0}^n \left| f(t)-g(t) \right|\dd t  \right) \\
  &\geq \frac{1}{2} \left( \frac{1}{n}\int_{-n}^0 \left|g(t) \right|\dd t + \frac{1}{n}\int_{0}^n \left| f(t)\right| -\left|g(t) \right|\dd t  \right) \\
  &\geq \frac{1}{2} \left( \frac{1}{n}\int_{-n}^0 \left|g(t) \right|\dd t-\frac{1}{n}\int_{0}^n \left|g(t) \right|\dd t \right) + \frac{1}{2n}\int_{0}^n \left| f(t)\right| \dd t \,.
\end{align*}
Now, since $g$ is Bohr almost periodic, the mean of $|g|$ does not depend on the averaging F\o lner sequence \cite[Cor.~4.3.6 and Prop.~4.5.9]{MoSt}. Therefore,
\[
\lim_{n\to\infty} \frac{1}{n}\int_{-n}^0 \left|g(t) \right|\dd t= \lim_{n\to\infty} \frac{1}{n}\int_{0}^n \left|g(t) \right|\dd t \,.
\]
It follows immediately that we have
\begin{align*}
  \| f-g\|_{\be,1,\cA} &\geq \limsup_{n\to\infty} \frac{1}{2n}\int_{0}^n \left| f(t)\right| \dd t  \\
  &= \limsup_{n\to\infty} \frac{1}{2}\left( \int_0^1 t \dd t +\int_1^n 1 \dd t \right) =\frac{1}{2}
\end{align*}
for all Bohr almost periodic functions $g$ and (in particular)
\[
\| f -P \|_{\be,1,\cA} \geq \frac{1}{2}
\]
for all trigonometric polynomials $P$. Consequently, $f$ cannot be Besicovitch almost periodic.  \exend
\end{example}

For bounded functions, Besicovitch $p$-almost
periodicity is independent of $p$, see also \cite[Thm.~2.1]{LinOlsen}. We start with the following
preliminary lemma.

\begin{lemma}\label{BAP bounded approx}
Let $f \in L^\infty(G) \cap \bap^p(G)$. Then, for each $\eps, \eps' >0$, there exists some trigonometric polynomial $P$ satisfying
\begin{align*}
\| f -P \|_{\be,p,\cA} \,&<\, \eps \qquad \mbox{ and } \\
\| P \|_\infty \, &< \, \|f \|_{\infty} +\eps' \,.
\end{align*}
\end{lemma}
\begin{proof}
It suffices to prove the claim for $0 < \eps <1$.  Fix such an
$\eps$ and choose some trigonometric polynomial $Q$ such that $\|f -Q
\|_{\be,p,\cA} < \frac{\eps}{2}$. Set $L:= \| f \|_\infty+\frac{\eps'}{2}$,
and define $g:= c_L(Q)$, where $c_L$ is the normal contraction
from \eqref{eq: cut off}. Since $|f(x)|\leq L$ for all
$x\in G$, we have
\[
\| f -g \|_{\be,p,\cA} \,=\, \| c_L(f) -c_L(g)
\|_{\be,p,\cA} \leq  \| f -Q \|_{\be,p,\cA}
\]
and
\[
\| g \|_\infty \,\leq\, L \,.
\]
It is easy to see that
\[
\left| \tau_t  g(x)- g(x) \right| \leq
\left| \tau_t Q(x)-Q(x) \right| \qquad \text{ for all } t,x \in G \,.
\]
As a
trigonometric polynomial, $Q$ is Bohr almost periodic and, hence, so
is $g$. Therefore, there exists a trigonometric polynomial $P$ such
that $\| g -P \|_\infty < \min\{ \frac{\eps}{2},\frac{\eps'}{2} \}$.
We obtain
\[
  \| P \|_\infty  \leq \|P-g \|_\infty + \| g \|_\infty < L +\frac{\eps'}{2} <  \| f \|_{\infty} + \eps'
\]
and
\[
 \| f -P \|_{\be,p,\cA}  \leq \| f -g \|_{\be,p,\cA} +\| g -P \|_{\be,p,\cA} \leq \| f -Q \|_{\be,p,\cA} + \| g-P \|_\infty <\eps \,.
\]
This finishes the proof.
\end{proof}

\begin{prop}\label{BAP equality}
For each $1 \leq p < \infty$, we have
\[
\bap^p(G) \cap L^\infty(G)=
\bap(G) \cap L^\infty(G) \,.\]
\end{prop}
\begin{proof} The inclusion
$\subseteq$ follows from Proposition~\ref{ap inclusions 1}.
To show
the inclusion $\supseteq$, let $f \in \bap(G) \cap L^\infty(G)$, and
let $\eps >0$. By Lemma~\ref{BAP bounded approx}, we can find
some trigonometric polynomial $P$ satisfying
\[
\| f -P \|_{\be,1,\cA} \, < \, \frac{\eps^p}{\big(2 \| f \|_\infty\big)^{p-1} +1 }
\]
and
\[
\| P \|_\infty \,<\, \| f \|_{\infty} +1 \,.
\]
Therefore, by Lemma~\ref{lemma C-S}, we have
\begin{displaymath}
 \| f-P \|_{\be,p,\cA}^p
 \leq \|f-P \|_\infty^{p-1}\,\| f \|_{\be,1,\cA}
 \leq \big(\| f \|_\infty + \| P \|_\infty \big)^{p-1} \frac{\eps^p}{\big(2 \| f \|_\infty\big)^{p-1} +1} < \eps^p \,.
\end{displaymath}
This finishes the proof.
\end{proof}



Next, we review basic properties of $\bap^p(G)$, see
\cite[Thm.~2.4]{LinOlsen} for (a) as well as
\cite[Thm.~2.7]{LinOlsen} for (g).

\begin{prop}\label{Bap props}
Let $\cA$ be a F\o lner sequence in $G$.
\begin{itemize}
\item [(a)] If $f \in \bap(G)$, the mean
\[
  M_{\cA}(f)= \lim_{n\to\infty} \frac{1}{|A_n|} \int_{A_n} f(t)\, \dd t
\]
of $f$ exists with respect to $(A_n)$ and satisfies
\[
\left|M_{\cA}(f)\right|\leq
\|f\|_{\be,1,\cA} \,.
\]
\item [(b)] For all $f \in \bap^p(G)$ the seminorm $\| f \|_{\be,p, \cA}$ is a limit
\[
\|f\|_{\be,p,\cA}= \lim_{n\to\infty} \sqrt[p]{\frac{1}{|A_n|} \int_{A_n} |f(t)|^p \dd t }\,.
\]
\item [(c)] If $f, g \in \bap^p(G)$ for some $1\leq p<\infty$, and if $\chi \in \widehat{G}, c \in \CC$, one has $f \pm g, cf, \overline{f}, \chi f, |f| \in \bap^p(G)$.
\item [(d)] If $f \in \bap (G) \cap L^\infty(G)$, we have $\tau_t f\in \bap (G)$ for all $t\in G$.
\item [(e)] If $f, g \in \bap^p(G)$ and $f \in L^\infty(G)$, then $fg \in \bap^p(G)$.
\item [(f)] If $1< p,q <\infty$ are conjugate, $f,P \in BL^p_{\cA}(G)$ and $g,Q \in BL^q_{\cA}(G)$, one has
\begin{equation}\label{eq:pq}
\| fg-PQ \|_{\be,1,\cA} \leq \|f \|_{\be,p,\cA} \| g-Q \|_{\be,q,\cA}+ \| f-P  \|_{\be,p,\cA} \| Q \|_{\be,q,\cA} \,.
\end{equation}
\item[(g)] If $1< p,q <\infty$ are conjugate, and $f \in \bap^p(G), g \in \bap^q(G)$, then $fg \in \bap(G)$.
\end{itemize}
\end{prop}
\begin{proof}
(a) Since the mean exists for all $f \in \sap(G)$ \cite{MoSt}, (a) follows from Proposition~\ref{prop-aprox-FB} and Lemma~\ref{lemma-FB-bound-bes}.

\smallskip
\noindent (b) Let $\eps >0$. Choose some $g \in SAP(G)$ such that
\[
\|f -g \|_{\be,p,\cA} < \frac{\eps}{4} \,.
\]
Then, there exists some $N$ such that
\[
\sqrt[p]{\frac{1}{|A_n|}\int_{A_n} |f(t)-g(t)|^p\, \dd t} < \frac{\eps}{2}
\]
holds for all $n>N$. Since $g \in SAP(G)$, the following limit
\[
\lim_{n\to\infty} \sqrt[p]{\frac{1}{|A_n|}\int_{A_n} |g(t)|^p\, \dd t}
\]
exists \cite{MoSt} and is obviously equal to $\| g \|_{\be,p,\cA}$. Therefore, there exists some $N_2$ such that, for all $n >N_2$, we have
\[
\left| \sqrt[p]{\frac{1}{|A_n|}\int_{A_n} |g(t)|^p\, \dd t}- \| g \|_{\be,p,\cA} \right| < \frac{\eps}{4} \,.
\]
Finally, the triangle inequality for $\mathcal{L}^p(A_n)$ implies that
\[
\left| \sqrt[p]{\frac{1}{|A_n|}\int_{A_n} |f(t)|^p\, \dd t} - \sqrt[p]{\frac{1}{|A_n|}\int_{A_n} |g(t)|^p\, \dd t}\right| < \sqrt[p]{\frac{1}{|A_n|}\int_{A_n} |f(t)-g(t)|^p\, \dd t} < \frac{\eps}{2}
\]
holds for all $n$. Therefore, for all $n > N:= \max\{N_1, N_2 \}$, we have
\begin{eqnarray*}
& & \left| \sqrt[p]{\frac{1}{|A_n|}\int_{A_n} |f(t)|^p\, \dd t} - \|f  \|_{\be,p,\cA} \, \dd t \right| \\ &\leq &
\left| \sqrt[p]{\frac{1}{|A_n|}\int_{A_n} |f(t)|^p\, \dd t}
-   \sqrt[p]{\frac{1}{|A_n|}\int_{A_n} |g(t)|^p\, \dd t}\right| \\
& + & \left| \sqrt[p]{\frac{1}{|A_n|}\int_{A_n} |g(t)|^p\, \dd t}- \| g \|_{\be,p,\cA} \right|
+ \|f -g \|_{\be,p,\cA}   \\
&\leq & \frac{\eps}{2}+\frac{\eps}{4}+\frac{\eps}{4} \\
&=&\eps \,.
\end{eqnarray*}

\medskip

\noindent (c) follows from the (in)equalities
\begin{align*}
\| f+g-(P+Q) \|_{\be,p,\cA} &\leq \| f-P \|_{\be,p,\cA} +\| g-Q \|_{\be,p,\cA} \,,\\
\| cf-cP \|_{\be,p,\cA}&=|c|\, \| f-P\|_{\be,p,\cA}\,,  \\
\| \chi f -\chi P\|_{\be,p,\cA} &=\| \overline{f}-\overline{P}
\|_{\be,p,\cA}=\| f - P\|_{\be,p,\cA} \,,
\end{align*}
with trigonometric polynomials $P,Q$.

For $|f|$, recall first that, for all $a,b \in \CC$, we have
\[
\big| |a|-|b| \big| \,\leq \, |a-b| \,.
\]
This implies that for all trigonometric polynomials $P$ we have
\[
\big| |f|-|P| \big|^p \,\leq \, |f-P|^p \,.
\]
It follows that
\[
\big\| |f|-|P| \big\|_{\be,p, \cA} \, \leq \, \| f-P \|_{\be,p, \cA}
\]
This proves (c).

\medskip
\noindent (d) Let $\eps >0$. Choose a trigonometric polynomial $P$ such that
\[
\| f -P \|_{\be,1,\cA} < \eps \,.
\]
Since $f-P \in L^\infty(G)$, we have
\[
\| \tau_t f - \tau_t P \|_{\be,1,\cA} < \eps
\]
by Lemma~\ref{L1}(b). This proves the claim.
\smallskip

\noindent (e) Let $\eps >0$. Choose some trigonometric polynomials $P$ and $Q$ satisfying
\begin{align*}
\| f -P \|_{\be,p,\cA} \,&<\, \frac{\eps}{2 \| g \|_\infty+1} \,, \\
\| g-Q\|_{\be,p, \cA} \, &<\,  \frac{\eps}{2 \| P \|_{\infty} +1} \,, \\
\| P\|_\infty \, &< \, \|f \|_\infty +1 \,.
\end{align*}
Then, we have
\begin{align*}
  \| fg -PQ \|_{\be,p,\cA} &\leq \| fg-Pg \|_{\be,p,\cA}+ \| Pg-PQ \|_{\be,p,\cA} \\
  & \leq \| f-P \|_{\be,p,\cA}\,\| g \|_\infty+ \| g-Q \|_{\be,p,\cA}\, \|P \|_\infty < \eps  \,.
\end{align*}

\smallskip

\noindent (f) follows immediately from the H\"older inequality, Theorem~\ref{lemma C-S}.

\smallskip

\noindent (g) Select sequences $(P_n)$ and $(Q_n)$ of trigonometric polynomials satisfying
\[
  \| f-P_n \|_{\be,p,\cA} \, < \, \frac{1}{n}
\]
and
\[
  \| g-Q_n \|_{\be,q,\cA}  \, <\,  \frac{1}{n} \,.
\]
By \eqref{eq:pq}, the claim follows from
\begin{align*}
\| fg-P_nQ_n \|_{\be,1,\cA} &\leq \|f \|_{\be,p,\cA} \| g-Q_n \|_{\be,q,\cA}+ \| f-P_n  \|_{\be,p,\cA} \| Q_n \|_{\be,q,\cA} \\
&\leq \frac{ \|f \|_{\be,p,\cA} }{n}+ \frac{1}{n} (\|g\|_{\be,q,\cA}+1) \to 0 \,.
\end{align*}
\end{proof}

\smallskip
We can talk about Fourier coefficients on $\bap (G)$, see
\cite[Thm.~2.5]{LinOlsen} for the existence of the corresponding limits.

\begin{coro}[Existence and continuity of the Fourier coefficients]
\label{FB for BAP fct} For  $\chi \in \widehat{G}$, the map
\[
a_{\chi}^{\cA} : \bap (G)\longrightarrow \CC\,, \qquad f\mapsto
\lim_{n\to\infty} \frac{1}{|A_n|} \int_{A_n} \overline{\chi(t)}\,
f(t)\, \dd t \,,
\]
is well defined, linear, continuous and satisfies
\[
\left|  a_\chi^\cA(f) \right| \leq \|f
\|_{\be,1,\cA} \,.
\]
In particular, for all $f,g \in \bap(G)$ satisfying $\|f -g\|_{\be,1,\cA} = 0$, we have
\[
a_{\chi}^\cA (f) = a_{\chi}^{\cA}(g) \,.
\]
Moreover,
\[
a_{\chi}^{\cA} (f*\varphi) = \widehat{\varphi} (\chi) \,
a_{\chi}^{\cA} (f)
\]
holds for all $f \in \bap$ with $f \theta_G \in \cM^\infty(G)$ and all $\varphi \in \Cc(G)$. \index{Fourier--Bohr~coefficient!Fourier--Bohr~coefficient~of~function}
\end{coro}
\begin{proof} By (c) of Proposition~\ref{Bap props}, $\overline{\chi} f$ belongs to
$\bap (G)$ for any $f\in\bap (G)$. The existence of the
limit follows then from Proposition~\ref{Bap props}(a). Moreover, the map $ a_\chi^\cA : \bap(G) \to \CC$ is linear. Finally, by Lemma~\ref{lemma-FB-bound-bes}, we have
\[
\left|  a_\chi^\cA(f) \right| \leq \|f
\|_{\be,1,\cA} \,.
\]
The last statement follows from Corollary~\ref{FB
measure relations} applied to $\mu=f \theta_G \in \cM^\infty(G)$.
\end{proof}

It is possible to turn $\bap^p(G)$ into a normed space and show that
it is a Banach space. To do this, we need to factor out all the
elements of norm $0$. As before, we use the equivalence relation
$\equiv$ on $\bap^p(G)$ defined via
\[
f \equiv g \iff \|f-g \|_{\be,p,\cA} =0\,.
\]
Note that if $h \in \mathcal{L}^p_{loc}(G)$ satisfies
$\|h\|_{\be,p,\cA}=0$, then $h \in \bap^p(G)$ and $h \equiv 0$. As
usual, we denote by $[f]_p$ the equivalence class of $f$. When there
is no possibility of confusion, we will use the shorter notation
$[f]$. In this case, $\| \cdot \|_{\be,p,\cA}$ becomes a norm on the space
\[
\bappe \, := \{ [f] : f \in \bap^p(G)  \}
\]
of equivalence classes, and $(\bappe , \| \cdot \|_{\be,p,\cA})$ is a
normed space. Corollary~\ref{FB for BAP fct} allows us to
define the Fourier--Bohr coefficient of $[f] \in \bap(G)/\equiv$ in
$\chi \in \widehat{G}$  via
\[
a_{\chi}^{\cA}\big([f]\big):=a_{\chi}^{\cA}(f)\,.
\]
The subsequent completeness result is certainly known, see
e.g. \cite[Rem.~2]{CohLos}. However, we could not find a discussion
with a proof in the literature.

\begin{theorem}[Completeness of Besicovitch spaces]\label{Bap is complete}
For each $1 \leq p < \infty$, the space $(\bappe, \| \cdot \|_{\be,p,\cA})$ is a Banach space.
\end{theorem}
\begin{proof}
Recall that $\bap^p(G)$ is the closure of the subspace spanned by trigonometric polynomials in $(BL_{\cA}^p (G), \| \cdot \|_{\be,p,\cA})$.
By Theorem \ref{thm:completeness}, the former space is closed, and the claim follows.
\end{proof}

It is possible to extend the translation action from $\sap(G)$ as
well as the convolution on $\sap(G)$ to all of $\bap^p
(G)/\equiv$. Indeed, with $\fS' = \sap(G)$ and $\fS =
\bap^p (G)$, we are exactly in the situation discussed
 in Section \ref{sub:trans-sub}. In particular, by Proposition~\ref{prop:translation2}, for each $t\in G$, there is a
 unique isometric  map
 \[
 T_t : \bap^p (G) /\equiv\, \longrightarrow\,
 \bap^p (G)/\equiv \]
 which satisfies
 \[
 T_t [f] = [\tau_t f] \qquad \text{ for all } f\in \sap(G)\,.
\]
One can also argue that, by definition, we have $\bap^p(G) \subseteq BC^p_{\cA}(G)$ and, hence, $\bap^p(G)/\equiv \subseteq BC^p_{\cA}(G)/\equiv$.
The mapping $T_t$ is then just the restriction to $\bap^p(G)/\equiv$ of the mapping from \eqref{Tt}.

\smallskip
Proposition \ref{prop:translation2} gives the following result.

\begin{prop}[Translation]\label{prop translation clases} \phantom{X}
\begin{itemize}
\item[(a)] For each $s,t \in G$, we have
\[
T_t \circ T_s= T_{t+s} \qquad \text{ and }  \qquad
 T_{0} = \mbox{\rm Id} \,.
\]
\item[(b)] For each $[f] \in \bappe$, the function $G\longrightarrow
\bappe$, $t \to T_t [f]$, is continuous and satisfies
\[
\|T_t [f]\|_{\be,p,\cA}= \|[f]\|_{\be,p,\cA} \qquad \text{ for all } t\in G \,.
\]
\item[(c)] If $[f] \in (\bappe)$ and $f \in L^\infty(G)$, we have $T_t[f]=[\tau_tf]$.\qed
\end{itemize}
\end{prop}

We finish this section by providing an alternative view on $\bappe$
via the Bohr compactification of $G$.
Recall first that we have a natural continuous, injective group
homomorphism $i_{\mathsf{b}} : G \to G_{\mathsf{b}}$ of $G$ into its
Bohr compactification $G_{\mathsf{b}}$. Under this embedding, a
function $f\in \Cu(G)$ is Bohr almost periodic if and only if
there exists $f_{\mathsf{b}} \in C(G_\mathsf{b})$ such that $f=
f_{\mathsf{b}}\circ i_{\mathsf{b}}$. In this case, the function
$f_{\mathsf{b}}$ is unique, and the mapping $f \to f_{\mathsf{b}}$\label{bohr-map}
is called the \textit{Bohr mapping}\index{Bohr~mapping}, see \cite{MoSt} for the
details. Moreover, we have
\[
M_\cA(f)= \int_{G_{\mathsf{b}}} f_{\mathsf{b}}(t)\, \dd
t \,,
\]
where we use the probability Haar measure on
$G_{\mathsf{b}}$  on the right hand side, see again \cite{MoSt}. These considerations
immediately yield the following.

\begin{lemma}\label{norm preserving} For each van Hove sequence $\cA$ on $G$, all $f\in \sap(G)$ and all $1 \leq p < \infty$, we have
\[
\| f\|_{\be,p,\cA}= \left(\int_{G_{\mathsf{b}}} |f_{\mathsf{b}}(t)|^p\,
\dd t \right)^\frac{1}{p}= \| f_{\mathsf{b}}\|_p \,.    \tag*{$\qed$}
\]
\end{lemma}

Therefore, we obtain the next result, compare \cite{Fol} for $G=\RR^d$.

\begin{theorem}\cite[p.12]{Pan}\label{Bap completion}
Fix a van Hove sequence $\cA$. Then, for each $1 \leq p < \infty$, the Bohr mapping $( \cdot )_{\mathsf{b}}
: \sap(G) \to C(G_{\mathsf{b}}) \subseteq \mathcal{L}^p(G_{\mathsf{b}})$
extends uniquely to an isometric isomorphism
\begin{displaymath}
  ( \cdot )_{\mathsf{b},p } : (\bappe, \| \cdot \|_{\be,p,\cA})  \to ( \mathcal{L}^p(G_{\mathsf{b}}), \| \cdot \|_p) \,.
\end{displaymath}
Moreover, for all $t \in G$ and $[f] \in \bappe$, we have
 \[
  \big( T_t[f] \big)_{\mathsf{b},p}= \tau^{}_{i_{\mathsf{b}}(t)}  \big( [f] \big)_{\mathsf{b},p } \,.
 \]
Furthermore, if $f \in \bap^p(G) \cap L^\infty(G)$, then there exists a representative $( f )_{\mathsf{b},p}$ such that
 \[
 \big\|( f )_{\mathsf{b},p} \big\|_{\infty} \leq \|f \|_\infty \,.
 \]
\end{theorem}
\begin{proof}
Let us think of $\sap(G)$ as a subspace of $\bappe$.
For each $1 \leq p < \infty$, by Lemma~\ref{norm preserving}, the
Bohr mapping is a norm preserving isometry from $(\sap(G), \| \cdot
\|_{\be,p,\cA})$ into the Banach space $\mathcal{L}^p(G_{\mathsf{b}})$.
Since $(\sap(G), \| \cdot \|_{\be,p,\cA})$ is dense in $(\bappe ,  \| \cdot
\|_{\be,p,\cA})$, the Bohr mapping has a unique extension to an
isometry $( \cdot )_{\mathsf{b},p} : (\bappe, \| \cdot \|_{\be,p,\cA}) \to
( \mathcal{L}^p(G_{\mathsf{b}}), \| \cdot \|_p)$. Since the range contains
$C(G_{\mathsf{b}})$, as the image of $\sap(G)$, the extension $( \cdot
)_{\mathsf{b},p}$ has dense range. Hence, as an isometry, it is
onto.

The `moreover' claim follows trivially from
\[
(\tau_t f)_{\mathsf{b}}= \tau^{}_{i_{\mathsf{b}}(t)}( f_{\mathsf{b}}) \qquad \text{ for all } t \in G, f \in SAP(G) \,.
\]

Finally, let $f \in \bap^p(G) \cap L^\infty(G)$, and let $\eps >0$.
By Lemma~\ref{BAP bounded approx}, we can find a sequence $(P_n)$ of trigonometric polynomials such that
\begin{align*}
\| f -P_n \|_{\be,p,\cA} \,&<\, \frac{1}{n} \\
\| P_n \|_\infty \,&<\, \|
f \|_{\infty} + \eps \,.
\end{align*}
It follows that
\begin{align*}
\Big\| \big([ f ]\big)_{\mathsf{b},p} -(P_n)_{\mathsf{b}} \Big\|_{p} \,&<\, \frac{1}{n} \\
\big\| (P_n)_{\mathsf{b}} \big\|_\infty \,=\, \|P_n \|_\infty \,&<\, \|
f \|_{\infty} + \eps \,.
\end{align*}
If we choose any representative $g$ in the class $ \big([ f ]\big)_{\mathsf{b},p} \in \mathcal{L}^p(G_\mathsf{b})$, the sequence $\big((P_n)_{\mathsf{b}}\big)$ converges in $\| \cdot \|_p$ to $g$.
Then, there exists a subsequence $\big((P_{n_k})_{\mathsf{b}}\big)$ which converges pointwise almost everywhere to $g$. This implies that
\[
| g(x)| \leq  \| f \|_{\infty} + \eps
\]
for almost all $x \in G_{\mathsf{b}}$. The claim follows.
\end{proof}

\begin{remark}
\phantom{XXX}
\begin{itemize}
\item[(a)] If  $P= \sum_{k=1}^n c_k\, \chi_k$, then
$P_{\mathsf{b}}= \sum_{k=1}^n c_k\, (\chi_k)_{\mathsf{b}}$, where
$\chi_{\mathsf{b}}$ denotes the character $\chi \in
\widehat{G}=\widehat{G_{\mathsf{b}}}$ viewed as a character on
$G_{\mathsf{b}}$.
\item[(b)]  Let $f \in \bap^p(G)$ and let $(P_n)$ be a sequence of trigonometric polynomials
such that $\lim_{n\to\infty} \| f -P_n\|_{\be,p,\cA}=0$. Then,
Theorem~\ref{Bap completion} implies
\[
[f]_{\mathsf{b},p}= \lim_{n\to\infty} (P_n)_{\mathsf{b}} \qquad \text{ in $(\mathcal{L}^p(G_\mathsf{b}), \|\cdot\|_p)$}  \,.
\]
In particular, for all $1 \leq p <q < \infty$ and all $f \in
\bap^q(G) \subseteq \bappe $, we have $[f]_{\mathsf{b},p} =
[f]_{\mathsf{b},q}$ in $\mathcal{L}^p(G_{\mathsf{b}})$. Indeed, if we choose trigonometric polynomials $P_n$ such that $\|
f -P_n \|_{\be,q,\cA} \to 0$, we obtain $\| f -P_n \|_{\be,p,\cA} \to 0$. This
gives
\[
\lim_{n\to\infty} \| [f]_{\mathsf{b},q} - (P_n)_{\mathsf{b}} \|_q = 0 \qquad \text{ and } \qquad
\lim_{n\to\infty} \| [f]_{\mathsf{b},p} - (P_n)_{\mathsf{b}} \|_p = 0 \,.
\]
Finally, the claim follows from $\mathcal{L}^q(G_\mathsf{b}) \subseteq \mathcal{L}^p(G_\mathsf{b})$ and $\| \cdot \|_q \leq \| \cdot \|_p$.
\item[(c)] Let $f \in \bap^p(G)$, and let $(P_n)$ be a sequence of trigonometric polynomials
such that
\[
\sum_{n=1}^{\infty} \| f -P_n\|_{\be,p,\cA}< \infty \,.
\]
Then, for any representative $g$ in the class $ \big([ f ]\big)_{\mathsf{b},p} \in \mathcal{L}^p(G_\mathsf{b})$, the sequence $\big((P_n)_{\mathsf{b}}\big)$ converges pointwise almost everywhere to $g$.      \exend
\end{itemize}
\end{remark}

Using Theorem~\ref{Bap completion}, we obtain the following
complementary view on averaging and taking Fourier--Bohr
coefficients on $\bap(G)$.

\begin{lemma}
Let $\cA$ be a F\o lner sequence on $G$. For all $f \in
\bap(G)$, we have
\begin{enumerate}
\item[(a)]
\[
\lim_{m\to\infty}\frac{1}{|A_m|}\int_{A_m} f(t)\, \dd t  = \int_{G_\mathsf{b}} ([f])_{\mathsf{b},1}(t)\, \dd t \,.
\]
\item[(b)]  $a_{\chi}^\cA(f) = \reallywidehat{ [f]_{\mathsf{b},1}} (\chi_\mathsf{b})$.
\end{enumerate}
\end{lemma}
\begin{proof} In both (a) and (b),
both sides are continuous functionals on $\bap(G)$, which agree on
the dense subspace $\sap(G)$.
\end{proof}

If $1<p,q<\infty$ are conjugates, then $\mathcal{L}^p(G_{\mathsf{b}})$ and
$\mathcal{L}^q(G_{\mathsf{b}})$ are dual spaces, with the duality given by
$(f,g) := \int_{G_{\mathsf{b}}} f(t)\, g(t)\, \dd t$. This leads to the following observation (compare \cite[Theorem 4]{Fol} for $G=\RR^d$).

\begin{theorem}[Dual of $\bap^q(G)/\equiv$] \label{them:dual}
Let $1<p<\infty$, and let $q$ be the conjugate of $p$, together with $f \in \bap^p(G), g \in \bap^q(G)$. Then,
\begin{itemize}
  \item[(a)] $fg \in \bap(G)$.
  \item[(b)] If $[f]_p=[f']_p$ and $[g]_q=[g']_q$ then $[fg]_1=[f'g']_1$.
\end{itemize}
In particular,
\begin{equation}\label{c4-e1}
( [g]_q, [f]_p ) := M_{\cA}(gf)
\end{equation}
does not depend on the choice of the representative.

Moreover, via \eqref{c4-e1} we can identify $(\bap^q(G)/\equiv, \| \, \|_{\be,q,\cA})$ as the dual space of  $(\bappe,, \| \, \|_{\be,p,\cA})$.
\end{theorem}
\begin{proof}
(a) and (b) follow Proposition~\ref{Bap props} (f) and (g).
As stated above, the duality follows trivially from the duality between $\mathcal{L}^p(G_{\mathsf{b}})$ and
$\mathcal{L}^q(G_{\mathsf{b}})$.
\end{proof}

Finally, we note that one can also understand the translation action
via the Bohr map.
 Let $f \in \bap^p(G)$ and $t \in G$. Then, $T_t [f]$ is
the only class $[g] \in \bap^p(G)$ such that
\[
\big( [g] \big)_{\mathsf{b},p } =
\big([f] \big)_{\mathsf{b},p } \big( \cdot - i_{\mathsf{b}}(t)  \big)\,.
\]

\section{Besicovitch almost periodic functions: Fourier
expansion} In this section, we focus on $\bap^2 (G)$. This space
has a natural Hilbert space structure, and we use it to develop a
Fourier expansion theory.

\medskip

Proposition~\ref{Bap props}
implies that the mean
\[
M_{\cA}(f \overline{g})= \lim_{n\to\infty} \frac{1}{|A_n|}
\int_{A_n} f(t) \overline{g(t)}\, \dd t
\]
exists for all $f,g\in\bap^2 (G)$. Moreover, by Theorem~\ref{them:dual} we can define the map $\langle \cdot , \cdot \rangle : \bapte \times \bapte \to \CC$ via
\[
\langle [f] , [g] \rangle:= M_{\cA}(f \overline{g}) \,.
\]
Given this, one can exhibit the
desired Hilbert space structure on $\bapte$ as follows.

\begin{theorem}[$\bapte$ as Hilbert space]\label{thm hilbert}
Let $\cA$ be a van Hove sequence.
\begin{itemize}
\item[(a)]  The map $\langle\cdot,\cdot\rangle : \left(\bapte\right) \times \left(
\bapte\right) \longrightarrow \CC$ defined by
\begin{displaymath}
\langle [f], [g] \rangle_{\cA} = \lim_{n\to\infty} \frac{1}{|A_n|}
\int_{A_n} f(t) \,\overline{g(t)}\, \dd t
\end{displaymath}
is an inner product on $\bapte$. The norm defined by it is $\| \cdot  \|_{\be,2,\cA}$.
\item[(b)] The space $(\bapte, \langle \cdot , \cdot \rangle_{\cA})$ is a Hilbert space.
\item[(c)] The group $\widehat{G}$ is an orthogonal basis in $\bapte$.
\item[(d)] For all $f \in \bap^2(G)$ and $\chi \in \widehat{G}$, we have
$ a_{\chi}^{\cA}([f])= \langle [f] , [\chi] \rangle$.
\item[(e)] For all $f \in \bap^2(G)$, one has $a_{\chi}^\cA(f) \neq 0$ for at most a countable set  of characters, and we have the \textit{Parseval identity}\index{Parseval~identity!Parseval~identity~for~$\bap^{2}(G)$}
\begin{equation}
\|f \|_{\be,2,\cA} ^2  = \sum_{\chi\in \widehat{G}} \big|
a_{\chi}^{\cA}(f) \big|^2 \label{eqn-parseval}
\end{equation}
and
\[
f = \sum_{\chi
\in\widehat{G}} a_{\chi}^{\cA}(f) \chi \, \mbox{ in } (\bapte, \|
\cdot \|_{\be,2,\cA})\,.
\]
\end{itemize}
\end{theorem}
\begin{proof}
(a) For  $f,g \in \bap^2(G)$, Proposition~\ref{Bap props} gives
 $f \overline{g} \in \bap(G)$, and $M_\cA (f\overline{g})$ exists. It
follows immediately from the Cauchy--Schwarz inequality that $\langle [f], [g]
\rangle_{\cA}$ does not depend on the choice of the representative; thus it is well defined. Clearly, the associated norm is
$\|\cdot\|_{\be,2,\cA}$.

\medskip

\noindent (b) follows from Theorem~\ref{Bap is complete} and the fact,
established in (a), that the norm induced from $\langle
\cdot,\cdot \rangle$ agrees with $\|\cdot\|_{\be,2,\cA}$.

\medskip

\noindent (c) It is well known that $M_\cA (\chi \overline{\xi}) =0$ when
$\chi, \xi \in \widehat{G}$ do not agree. This shows that the
characters form an orthonormal system in $\bapte$. Moreover, linear
combinations of characters are dense in
$(\bapte,\|\cdot\|_{\be,2,\cA})$  by the very definition of the
Besicovitch space. This gives (c) as the Hilbert space  norm on
$\bapte$ agrees with $\|\cdot\|_{\be,2,\cA}$ due to (a).

\medskip

\noindent (d) follows from the definition of the inner product
and the Fourier--Bohr coefficient. The Fourier--Bohr coefficient does not depend on the
representative we choose from the class $[f]$ by (a).

\medskip

\noindent (e) is immediate from (b) and (c).
\end{proof}

As an immediate consequence, we obtain the so called Riesz--Fischer property.

\begin{coro}[Riesz--Fischer Property]\index{Riesz--Fischer~Property!Riesz--Fischer~Property~for~$\Bap^{2}_{\cA}(G)$}
Let $a: \widehat{G} \to \CC$, and let $\cA$ be a van Hove sequence. Then, there is some $f \in \bap^2(G)$ such that $a_{\chi}^\cA(f)=a(\chi)$ if and only if
\[
\sum_{\chi \in \widehat{G}} |a(\chi)|^2 < \infty\,.
\]
Moreover, in this case, $[f]$ is unique.
\end{coro}
\begin{proof}
$\Longrightarrow$: This follows from \eqref{eqn-parseval}.

\medskip

\noindent $\Longleftarrow$: The series $\sum_{\chi}  a(\chi) \chi$ is convergent in $\bap^2(G)$ to some $f$, which satisfies the given relation by the orthonormality of $\widehat{G}$.
\end{proof}

Another consequence of Theorem~\ref{thm hilbert}(e) is that any Besicovitch 2-almost periodic function can be approximated by trigonometric polynomials which are truncations of its Fourier--Bohr series.

\begin{coro}
Let $f \in \bapte$. Then, there exists a sequence of distinct characters $(\chi_n)$ in $\widehat{G}$ and some $c_n \in \CC$ such that
\[
\lim_{N\to\infty} \Big\| f- \sum_{n=1}^N c_n \chi_n \Big\|_{\be,2,\cA} =0 \,.
\]
Moreover, for all $\chi \in \widehat{G}$, we have
\[
a^{\cA}_\chi(f)=
\begin{cases}
  c_n\,, & \mbox{ if there exists } n \mbox{ such that } \chi=\chi_n \,, \\
  0\,, & \mbox{ otherwise} \,.          \tag*{$\qed$}
\end{cases}
\]
\end{coro}

We can next give the following characterization of $\bap^2 (G)$, which is important for diffraction theory.

\begin{coro}[Intrinsic characterization $\bap^2 (G)$] \label{bap2 char}
Let $f \in \mathcal{L}^2_{loc}(G)$, and let $\cA$ be a van Hove sequence. Then, $f \in
\bap^2(G)$ if and only if the following three conditions hold: \index{almost~periodic!Besicovitch~$p$-almost~periodic~function}
\begin{itemize}
  \item[(a)]  For each $\chi \in \widehat{G}$, the Fourier--Bohr coefficient $a_{\chi}^\cA(f)$ exists.
  \item[(b)]  The mean $M_{\cA}(|f|^2)$ exists.
  \item[(c)]  The Parseval equality
  \[
M_{\cA}(|f|^2) = \sum_{\chi \in \widehat{G}} \big| a_{\chi}^\cA(f)
\big|^2
\]
  holds.
\end{itemize}
\end{coro}
\begin{proof} This follows from Theorem~\ref{thm hilbert}. Indeed,
the `only if' part  is immediate from Theorem~\ref{thm hilbert}. As
for the `if' statement, let  $\eps >0$ be given. We can find
characters $\chi_1,\ldots, \chi_N \in \widehat{G}$ such that
\[
\sum_{k=1}^N \left| a_{\chi_k}^\cA(f) \right|^2 \geq M_{\cA}(|f|^2)  -\eps \,.
\]
Let $P := \sum_{k=1}^N  a_{\chi_k}^\cA(f) \chi_k$. Using the fact that $M_\cA
\big(|f|^2\big)$ and $a_{\chi}^{\cA} = M_{\cA}(
f\overline{\chi})$ exist, for all $\chi\in \widehat{G}$, together with
$M_{\cA}(\chi) =1$ for $\chi = 1$ and $M_{\cA} (\chi) = 0$, we easily
compute
\begin{align*}
M_{\cA}( |f-P|^2)
    &=M(|f|^2)- \sum_{k=1}^N \left| a_{\chi_k}^\cA(f) \right|^2  -
      \sum_{k=1}^N \left| a_{\chi_k}^\cA(f) \right|^2 + \sum_{k=1}^N
      \big| a_{\chi_k}^\cA(f) \big|^2  \\
    &= M_{\cA}(|f|^2)- \sum_{k=1}^N \left| a_{\chi_k}^\cA(f) \right|^2 \,.
\end{align*}
Putting this together, we see that $\|f - P\|_{\be,2,\cA}^2\leq
\varepsilon$. As $\varepsilon>0$ was arbitrary, this finishes the
proof.
\end{proof}

As consequence of Proposition \ref{prop translation clases}, we
obtain that the translation action of $G$ on $\bapte$ is a unitary
representation.

\begin{prop}[Translation]\label{prop_unitary-represenatation}
Let $\cA$ be a van Hove sequence.
\begin{itemize}
\item[(a)]
 For each $t\in G$, the map $T_t : \bapte \longrightarrow \bapte$ is a unitary map.
\item[(b)]  For each $[g] \in \bapte$, the function $t \mapsto T_t[g]$ is continuous, and so is then
$t\mapsto \langle [g], T_t [g]\rangle$.
\item[(c)] For each $s,t \in G$, we have $
T_t \circ T_s= T_{t+s}$ and $T_{0}= \mbox{\rm Id}$.
\end{itemize}
In particular, $T_t$ defines a unitary representation of the group $G$ on the Hilbert space $\bapte$.
\qed
\end{prop}

\begin{remark} Considering the Bohr compactification is instructive
in this situation as well. The map
 $( \cdot )_{\mathsf{b}} : \sap(G) \to C(G_{\mathsf{b}}) \subseteq
L^2(G_{\mathsf{b}})$ extends uniquely to a unitary map between $
(\bapte, \langle \cdot , \cdot \rangle_{\cA})$ and   $
(L^2(G_{\mathsf{b}}), \langle \cdot,\cdot \rangle)$, compare Theorem
\ref{Bap completion} as well.     \exend
\end{remark}

We complete the section by defining an involution as well as the
Eberlein convolution for $\bapte$ and discussing some of their
properties.

\begin{prop}[Involution]\label{prop tilde extension}
There exists a unique isometric involution $\widetilde{\cdot}: \bapte \to \bapte$\nomenclature{$\widetilde{[f]}$}{complex conjugate and involuted version of the class $[f]$} satisfying
\[
\widetilde{[f]} = [\widetilde{f}] \qquad \mbox{ for all } f \in \sap(G)
\]
and
\[
\big\langle [\widetilde{f}], [\widetilde{g}] \big\rangle_{\cA} = \overline{\big\langle [f], [g] \big\rangle_{\cA}}  \mbox{ for all } f,g \in \bapte  \,.
\]
\end{prop}
\begin{proof}
The mapping $f \mapsto \widetilde{f}$ is an involution on $\sap(G)$, and we have
\begin{align*}
\big\langle [\widetilde{f}], [\widetilde{g}] \big\rangle_{\cA}
    &= \lim_{n\to\infty} \frac{1}{|A_n|} \int_{A_n} \widetilde{f}(t)\,
       \overline{\widetilde{g}(t)}\, \dd t = \lim_{n\to\infty} \frac{1}{|A_n|}
       \int_{A_n}\overline{f(-t)}\, g(-t)\, \dd t \\
    &= \lim_{n\to\infty} \frac{1}{|A_n|} \int_{-A_n}\overline{f(t)} \, g(t)
       \, \dd t = \overline{\big\langle [f], [g] \big\rangle_{\cA}} \,,
\end{align*}
for all $f,g \in \sap(G)$, with the last equality following from the fact that, for $h \in \sap(G)$, the mean is independent of the F\o lner
sequence \cite{Eb,MoSt}. In particular, for all $f \in \sap(G)$, we
have $\big\| [\widetilde{f}] \big\|_{\be,2,{\cA}} = \big\| [f] \big\|_{\be,2,{\cA}}$. The
claim follows immediately from the denseness of $\sap(G)$ in
$\bapte$.
\end{proof}

\begin{remark}
\phantom{XX}
\begin{itemize}
\item[(a)]  It is easy to see that $(\widetilde{f})_{\mathsf{b}}= \widetilde{ f_{\mathsf{b}}}$ for all $f \in \sap(G)$. It follows immediately that, for all
$f \in \bapte$, we have
$\widetilde{[f]_{\mathsf{b},p}}=\big(\widetilde{[f]}\big)_{\mathsf{b},p}$.
\item[(b)] In the same way as in Proposition~\ref{prop tilde extension}, it can be shown that the involution $\widetilde{\cdot}$ on $\sap(G)$ can be uniquely extended to
an isometric involution $\widetilde{\cdot}$ on $(\bappe,  \| \cdot \|_{\be,p,{\cA}})$.     \exend
\end{itemize}
\end{remark}

\medskip

\begin{definition}[Reflected Eberlein convolution of classes]
We define the \textit{reflected Eberlein
convolution of classes}\index{Eberlein~convolution!reflected~Eberlein~convolution~on~$\bapte$} $[f],[g] \in \bapte$ via
\[
\lb [f], [g] \rb_{\cA} : G \to \CC\,, \qquad \lb [f], [g] \rb_{\cA}(t) := \big\langle [f], T_t [g] \big\rangle \,.     \tag*{$\Diamond$}
\]\nomenclature{$\lb [f], [g] \rb_{\cA}$}{reflected Eberlein
convolution of Besicovitch 2-almost periodic classes}
\end{definition}

By the properties of the inner product on $\bapte$, we have
\begin{equation}\label{eq ebe conv}
\lb [f], [g] \rb_{\cA} (t) = \lim_{n\to\infty}
\frac{1}{|A_n|} \int_{A_n} f(s)\, \widetilde{h}(s)\, \dd s
\end{equation}
for $[f], [g] \in \bapte$, $t\in G$ and $h \in T_t[g]$. In particular, whenever $f$ and $g$ are bounded functions with
$[f],[g]\in\bapte$, Proposition \ref{prop
translation clases} (c) implies
\[
\lb [f], [g] \rb_{\cA} (t) = \lim_{n\to\infty}
\frac{1}{|A_n|} \int_{A_n} f(s)\, \overline{g(s-t)}\, \dd s.
\]
In this case, we just recover the usual definition of the reflected
Eberlein convolution between $f$ and $g$, see Section
\ref{sec-key-player}. This is the reason for our notation. Also
note that
\begin{equation}\label{EQ: ebe ref}
\lb [f], [g] \rb_{\cA}(t) = \langle [f], T_t [g] \rangle = \langle T_{-t} [f], g \rangle= \widetilde{\lb [g], [f] \rb_{\cA}} (t)  \,,
\end{equation}
compare \cite[Lem.~3.6]{LSS3}.

\begin{prop}[Continuity of reflected Eberlein
convolution]\label{prop:continuity-eberlein} Let $\cA$ be a van Hove
sequence on $G$. Let $(f_n)$, $(g_n)$ be sequences in $\bap^2
(G)$ converging to $f$ and $g$ respectively.  Then,
$\lb [f_n], [g_n] \rb_{\cA} \to \lb [f], [g] \rb_{\cA}$ with respect to
 $\|\cdot\|_\infty$.
\end{prop}
\begin{proof} This is a straightforward computation, compare \cite[Lem.~3.13]{LSS3}: For each
$t\in G$, we find
\begin{align*}
|\lb [f_n], [g_n] \rb_{\cA} (t) - \lb [f], [g] \rb_{\cA} (t)|
    &= |\langle [f_n], T_t [g_n] \rangle - \langle[f], T_t [g]\rangle| \\
    &\leq  \|f_n\|\, \|T_t ([g_n] - [g])\| + \|[f] - [f_n]\| \, \|T_t[g]\| \\
    & = \|f_n\|\, \|[g_n] - [g]\| + \|[f] - [f_n]\| \,\| [g]\| \to 0 \,.
\end{align*}
Here, we used that $T_t$ is an isometry. As  the convergence to zero
in the last line is clearly independent of $t\in G$, the claim follows.
\end{proof}

\begin{theorem}[Properties of reflected Eberlein convolution]\label{eber conv Bap2 functions}
Let $f,g \in \bap^2(G)$. Then, the following statements hold.
\begin{itemize}
\item[(a)] $\lb [f], [g] \rb_{\cA} \in \sap(G)$.
\item[(b)] For all $\chi \in \widehat{G}$, we have
$ a_{\chi} \big( \lb [f], [g] \rb_{\cA} \big) = a_{\chi}^\cA\big([f]\big)
\, \overline{  a_{\chi}^\cA\big([g]\big) }$.
\item[(c)] If $t \in G$ is such that $T_t [g] = [\tau_t  g]$, then
\[
 \lb [f], [g] \rb_{\cA}(t) = \lim_{n\to\infty}
\frac{1}{|A_n|} \int_{A_n} f(s)\, \overline{g(s-t)}\, \dd s \,.
\]
\item[(d)]  If $g \in \bap^2(G) \cap L^\infty(G)$, then the classical reflected Eberlein convolution $\lb f, g \rb_{\cA}$ exists
and $ \lb f, g \rb_{\cA}=\lb [f], [g] \rb_{\cA}$.
\end{itemize}
\end{theorem}

\begin{remark}
We note that the reflected Eberlein convolution of $[f]$ and $[g]$
on $\bap^2(G)/\equiv$ can also be understood as the usual
convolution of the functions $([f])_{\mathsf{b},2}$ and $\widetilde{([g])_{\mathsf{b},2}}$ on
the Bohr compactification $G_{\mathsf{b}}$. Indeed, this is clear if $f$ and $g$ are
trigonometric polynomials. It follows for general $f,g \in
\bap^2(G)$ by the continuity of the Eberlein convolution, see Proposition~\ref{prop:continuity-eberlein}.  \exend
\end{remark}

\begin{proof} (a) The claim trivially holds when $f,g \in \sap(G)$. In particular, it holds for trigonometric polynomials.
Now, choose trigonometric polynomials $P_n, Q_n$ such that $\| f-P_n\|_{\be,2,\cA} \to 0$ and $\| g-Q_n\|_{\be,2,\cA} \to 0$. By
Proposition \ref{prop:continuity-eberlein} we then have
\[
\| \lb [f], [g] \rb_{\cA} - \lb [P_n], [Q_n] \rb_{\cA} \rb \|_\infty \to 0 \,.
\]
As $\lb [P_n], [Q_n] \rb_{\cA} \in \sap(G)$ and $\sap(G)$ is a closed subspace of $\Cu(G)$, we get $\lb [f], [g] \rb_{\cA} \in \sap(G)$.

\medskip

\noindent (b) Again, the claim trivially holds when $f,g \in \sap(G)$. In particular, it holds for trigonometric polynomials. Choose
 trigonometric polynomials $P_n, Q_n$ such that $\| f-P_n\|_{\be,2,\cA} \to 0$ and $\| g-Q_n\|_{\be,2,\cA} \to 0$. By
Proposition \ref{prop:continuity-eberlein}, we have
\[
\big\| \lb [f], [g] \rb_{\cA} - \lb [P_n], [Q_n] \rb_{\cA} \rb \big\|_\infty \to 0
\]
and
\[
\big\| \lb [f], [g] \rb_{\cA} - \lb [P_n], [Q_n] \rb_{\cA} \rb \big\|_{\be,1,\cA} \to 0 \,.
\]
Corollary~\ref{FB for BAP fct} then implies
\[
a_{\chi}\big(\lb [f], [g] \rb_{\cA}\big)
  =\lim_{n\to\infty} a_{\chi}\big(\lb [P_n], [Q_n] \rb_{\cA}\big)
  =\lim_{n\to\infty}   a_{\chi}\big([P_n]\big) \overline{ a_{\chi}\big([Q_n]\big)}
\]
as well as
\[
a_{\chi}^\cA\big([f]\big) \,=\, \lim_{n\to\infty}   a_{\chi}\big([P_n]\big)
\]
and
\[
a_{\chi}\big([g]\big) \,=\, \lim_{n\to\infty} a_{\chi}\big([Q_n]\big) \,.
\]
The claim follows.

\medskip

\noindent (c) follows immediately from Eq.~\eqref{eq ebe conv} with
$h=\tau_t g \in T_t[g]$.

\medskip

\noindent (d) By Proposition~\ref{prop translation clases}, we have
$T_t[g]=[\tau_t  g]$, and the claim follows from (c).
\end{proof}

\section{Besicovitch almost periodic measures}
Having studied Besicovitch almost periodic functions in the last
section, we now turn to Besicovitch almost periodic measures.

\begin{definition}[Besicovitch almost periodic measures] \label{defi: bap} Let a van Hove sequence $\cA$  on $G$, and let $1 \leq p < \infty$ be given.
A measure $\mu$ on $G$  is called \textit{Besicovitch
$p$-almost periodic} (with respect to $\cA$)\index{almost~periodic!Besicovitch~$p$-almost~periodic~measure} if the function
$\varphi*\mu$ is Besicovitch $p$-almost periodic for all $\varphi \in\Cc(G)$. \nomenclature{$\Bap^p(G)$}{set of Besicovitch $p$-almost periodic measures}
The  space of Besicovitch $p$-almost periodic measures is
denoted by $\Bap^p_{\cA}(G)$. In the case $p=1$, we drop the
superscript $1$.      \exend
\end{definition}

In the next remarks, we will list a few basic properties of
Besicovitch almost periodic measures, which are simple analogues of the
corresponding features of mean almost periodic measures.

\begin{remark}[Independence of $p$ for translation bounded measures]
Proposition~\ref{BAP equality} implies that a measure in
$\mathcal{M}^{\infty}(G)$ is Besicovitch $p$-almost periodic if
and only if it is Besicovitch 1-almost periodic, i.e. for all $1 \leq p < \infty $ we have
\[
 \left( \Bap^p_{\cA}(G) \right) \cap \cM^\infty(G) = \left( \Bap_{\cA}(G) \right)\cap \cM^\infty(G) \qquad \,.  \tag*{$\Diamond$}
\]
\end{remark}

\begin{remark}[Inclusion of spaces] \label{rem:ap inclusions1}
From the definition and
Proposition~\ref{ap inclusions 1}, we immediately obtain the following properties.
\begin{itemize}
  \item [(a)] For each $1 \leq p <\infty$, we have
  \begin{displaymath}
 \Bap^p_{\cA}(G) \subseteq \Map^p_{\cA}(G) \,.
  \end{displaymath}
  Moreover, in general the inclusion is strict.
  \item [(b)] For each $1 \leq p \leq q < \infty$, we have
  \begin{align*}
    \Bap^q_{\cA}(G) &\subseteq   \Bap^p_{\cA}(G) \subseteq   \Bap_{\cA}(G) \,.        \tag*{$\Diamond$}
  \end{align*}
\end{itemize}
\end{remark}

As  in the case of mean almost periodic measures, we can use (c) of
Proposition \ref{prop:measure-vs-function}, with $\mathcal{S}' =
\sap(G)$ instead of $\mathcal{S}'= MAP_{\cA}(G)$, to show that, for
a function $f \in \Cu(G)$, almost periodicity as function and as a
measure coincide.

\begin{prop}\label{prop:bap for functions vs measures}
Let $\cA$ be a van Hove sequence on $G$, and let $1\leq  p < \infty$ be
given. Assume that $f\in \mathcal{L}^1_{loc} (G)$ is such that $f \theta_G$ is a
translation bounded measure and that there exists a sequence $(f_n)$ in
$\Cu (G)$ with $f_n \to f$ with respect to $\|\cdot\|_{\be,p,\cA}$.
Then, $f \theta_G$ belongs to $\Bap^p_{\cA}(G)$ if and only if $f$
belongs to $\bap^p (G)$. In particular, $f\in \Cu (G)$ belongs to $
\bap^p(G)$ if and only if $f \theta_G \in \Bap^p_{\cA}(G)$.\qed
\end{prop}

A characteristic feature of the $2$-Besicovitch almost
periodic functions is that they can be expanded with respect to the orthogonal basis $[\chi], \chi \in\widehat{G}$, compare Theorem~\ref{thm hilbert}.
Our next goal is to show that a similar expansion result holds for translation bounded, Besicovitch almost periodic measures.
We start with an intermediate result.

\begin{prop}[Expansion of $\mu\ast \varphi$]\label{prop-bapappr}
For a van Hove sequence $\cA$ and $\mu \in  \Bap^2_{\cA}(G)\cap \cM^\infty(G)$, the Fourier Bohr coefficients $a^\cA_\chi (\mu)$ exist for each $\chi \in\widehat{G}$. Moreover, in $(\bapte, \| \, \|_{\be,2,\cA})$, the identity
\[
[\mu \ast \varphi] = \sum_{\chi \in \widehat{G}} a^\cA_\chi(\mu)\widehat{\varphi}(\chi)  [\chi]
\]
holds for all $\varphi\in \Cc (G)$.
\end{prop}
\begin{proof} As $\mu \in  \Bap^2_{\cA}(G)\cap \cM^\infty(G)$, the function $\mu \ast \varphi$ is Besicovitch almost periodic and bounded for any $\varphi \in\Cc (G)$. In particular, the Fourier Bohr coefficients of $\mu \ast \varphi $  exist for any $\varphi \in \Cc (G)$. By Corollary~\ref{FB measure relations}, the Fourier Bohr coefficients $a^\cA_\chi (\mu)$ exist, and
\[
a^\cA_\chi (\mu\ast \varphi) = a^\cA_\chi (\mu) \widehat{\varphi}(\chi)
\]
holds for any $\chi\in\widehat{G}$.  Given this, this result follows from Theorem \ref{thm hilbert}.
\end{proof}

Before proceeding, let us briefly discuss how large the Fourier--Bohr spectrum of a Besicovitch almost periodic function or measure can be.

\begin{prop}\label{FB-spe}
\phantom{XX}
\begin{itemize}
  \item[(a)] For all $f \in \bapte$, the set
  \[
  \BB(f):= \{ \chi \in \widehat{G} : a_\chi^\cA(f) \neq 0 \}
  \]
  is countable.
  \item[(b)] For all $\mu \in  \Bap^2_{\cA}(G)\cap \cM^\infty(G)$, the set
  \[
  \BB(\mu):= \{ \chi \in \widehat{G} : a_\chi^\cA(\mu) \neq 0 \}
  \]
  is locally countable\footnote{Recall that a set is called locally countable if its intersection with each compact set is countable.}. In particular, if $G$ is metrisable, $\BB(\mu)$ is countable.
\end{itemize}
\end{prop}
\begin{proof}
(a) follows from the Parseval identity \eqref{eqn-parseval}
\[
\sum_{\chi\in \widehat{G}} \big|
a_{\chi}^{\cA}(f) \big|^2 = \| f \|_{\be,2,\cA}^2 < \infty \,.
\]
(b) For each compact set $K \subseteq \widehat{G}$, there exists some $\varphi \in \Cc(G)$ whose Fourier transform does not vanish on $K$ \cite[Cor.~4.9.12]{MoSt}. Next,
\[
\BB(\mu) \cap K \subseteq \BB(\mu*\varphi)       \,,
\]
follows from
\[
a_{\chi}^{\cA}(\mu*\varphi)=\widehat{\varphi}(\chi) a_{\chi}^{\cA}(\mu)   \,,
\]
which proves the first claim. Finally, if $G$ is metrisable, $\widehat{G}$ is $\sigma$-compact \cite[Th.~4.2.7]{ReiSte}, which implies the second claim.
\end{proof}

Let us also note that Proposition~\ref{FB-spe} (b) also immediately follows from Theorem~\ref{Bap and BT} below.

Proposition~\ref{prop-bapappr} can be extended to give the desired Fourier--expansion result. To formulate this, we use  convergence along the net of all finite subsets of $\widehat{G}$. Specifically, for a function $u $ from the finite subsets of $\widehat{G}$ to $\CC$  and $c\in \CC$, we write
\[
\lim_F u(F) = c
\]
if, for any $\varepsilon >0$, there exists a finite subset $F_\varepsilon$ of $\widehat{G}$ with $|u(F) - c|\leq\varepsilon$ for all finite subsets $F$ of $\widehat{G}$ with $F_\varepsilon \subseteq F$.

\begin{theorem}[Expansion of an Besicovitch almost periodic measure]\label{thm-expa}
Let $\cA$ be a van Hove sequence, and let $\mu$ be a translation bounded measure. Then, the following assertions are equivalent: \index{almost~periodic!Besicovitch~$p$-almost~periodic~measure}
\begin{itemize}
\item[(i)] The measure $\mu$ is Besicovitch almost periodic.
\item [(ii)]  There exists a map $a : \widehat{G}\longrightarrow \CC $ such that, for all $\varphi \in \Cc (G)$, we have
\[
\lim_F  \Big\|\mu\ast \varphi - \Big(\sum_{\chi \in F} a_\chi \chi\Big)\ast \varphi \Big\|_{\be,2,\cA} = 0 \,.
\]
\end{itemize}
If the equivalent conditions (i) and (ii) hold, one has
\[
a_\chi = a_\chi^\cA (\mu) \qquad \text{ for all } \chi \in\widehat{G} \,.
\]
\end{theorem}
\begin{proof}
(i)$\Longrightarrow$(ii): We first note that
\begin{equation} \label{eq:diamond}
\chi \ast\varphi =\widehat{\varphi} (\chi)  \chi
\end{equation}
holds for all $\varphi \in \Cc(G)$ and all $\chi \in \widehat{G}$. Next, by Eq.~\eqref{eqn-parseval}, we obtain the identity
\[
 \varphi*\mu=\sum_{\chi \in \widehat{G}} a_{\chi}^{\cA}(\varphi*\mu) \chi   \qquad \text{ in $(\bapte, \| \cdot \|_{\be,2,\cA})$} \,,
\]
for all $\varphi \in \Cc(G)$. Hence, for each $\eps>0$, there exists some finite set $F \subseteq \widehat{G}$ such that
\[
\| \varphi*\mu- \sum_{\chi \in \widehat{G}} a_{\chi}^{\cA}(\varphi*\mu) \chi  \|_{\be,2,\cA} < \eps \,.
\]
Now
\[
a_{\chi}^{\cA}(\varphi*\mu) \chi =a_{\chi}^{\cA}(\mu) \widehat{\varphi}(\chi) \chi=a_{\chi}(\mu) \chi \ast\varphi
\]
implies (ii).

\medskip

\noindent (ii)$\Longrightarrow$(i): (ii) together with Eq.~\eqref{eq:diamond} implies that $\mu \ast \varphi$ is Besicovitch almost periodic for all $\varphi \in \Cc (G)$. This shows that $\mu$ is Besicovitch almost periodic measure.
\end{proof}

\begin{remark}
Loosely speaking, one can summarize (ii) as saying that $\mu$ can be expanded as $ \sum_{\chi \in \widehat{G}} a_\chi \chi$. So, the Fourier transform of $\mu$ formally equals $\sum_{\chi \in \widehat{G}} a_\chi \delta_\chi$. Theorem~\ref{thm-expa} can be seen as part of the solution to  the problem of Lagarias discussed  on Page \pageref{problem-lagarias} of the introduction. The proper solution to the problem, i.e. the answer to Question 2,  will be given in the next section.     \exend
 \end{remark}

 Let us complete this section by observing that Theorem~\ref{thm-expa} and Proposition~\ref{FB-spe} (b) immediately give the following result.

\begin{coro}
Let $G$ be a metrisable LCAG, and let $\cA$ be a van Hove sequence in $G$. For any translation bounded measure $\mu$, the following assertions are equivalent:
\begin{itemize}
\item[(i)] The measure $\mu$ is Besicovitch almost periodic.
\item [(ii)]  There exists a sequence $ (a_n)$ in  $\CC$ and a sequence $(\chi_n)$ in $ \widehat{G}$ such that, for all $\varphi \in \Cc(G)$, we have
\[
\lim_{N\to\infty}  \Big\|\Big(\mu - \sum_{n=1}^N a_n \chi_n\Big)\ast \varphi \Big\|_{\be,2,\cA} = 0 \,.
\]
\end{itemize}
If the equivalent conditions (i) and (ii) hold, then  $a_{\chi_n}^\cA (\mu)=a_n$ holds for all $n$, and all the remaining Fourier--Bohr coefficients of $\mu$ are zero.  \qed
\end{coro}

\section[Complete point diffraction]{Complete point diffraction  with respect to a given  van Hove sequence}
In this section, we characterize $\mathcal{A}$-representations that have
 pure-point diffraction together with  Fourier--Bohr
coefficients satisfying a certain completeness relation. We will summarize these properties under the name of complete point diffraction (see below).  As a consequence, we obtain a solution to Question 2.
Throughout the section, we denote the characteristic
function of $\{\chi\}$, $\chi\in\widehat{G}$, by $1_\chi$.

\begin{theorem}[Characterization of $\cA$-representation with Fourier--Bohr
coefficients]\label{thm:a-representation-besicov}  \index{almost~periodic!Besicovitch~$p$-almost~periodic~function}
Let $N$ be an $\cA$-representation with dense range which possesses a semi-autocorrelation $\eta$, and let
$\mathcal{H}$ be the associated Hilbert space. Then, the following
assertions are equivalent:
\begin{itemize}
\item[(i)] The Fourier transform of $\eta$ is a
pure-point measure $\sigma$, and there exist (necessarily unique)
complex numbers $A_\chi$, for $\chi \in \widehat{G}$, satisfying
\[
M_{\cA}(N(\varphi) \overline{\chi}) =A_\chi
\widehat{\varphi}(\chi)
\]
for all $\varphi \in\Cc (G)$, as well as
\[
|A_\chi|^2 = \sigma(\{\chi\}) \,.
\]
\item[(ii)] $N(\Cc (G))\subseteq \bap^2(G)$.
\item[(iii)] The space $\mathcal{H}$ has a dense subspace consisting
of trigonometric polynomials.
\end{itemize}
If these equivalent conditions are satisfied, the equation
\[
[N(\varphi)] = \sum_{\chi\in\widehat{G}} A_\chi \widehat{\varphi}(\chi) [\chi]
\]
holds in $Bap^2_\cA (G)/\equiv$, for all $\varphi \in\Cc (G)$, and the
(unique) unitary map
\[
U  : L^2 (\widehat{G},\sigma) \longrightarrow \mathcal{H} \,,
\]
with  $\widehat{\varphi}\mapsto N(\varphi)$ for all $ \varphi \in\Cc
(G)$, satisfies $U(1_\chi) = A_\chi \chi$ for all $\chi \in
\widehat{G}$. Moreover, in this case, the trigonometric polynomials
in (iii) are exactly the linear span of the set $ \{\chi \in\widehat{G}$
\,:\, $\sigma (\{\chi\})>0 \}$.
\end{theorem}

\begin{remark}[Fourier--Bohr coefficients]\label{rem-FBC}
\phantom{XX}
\begin{itemize}
\item[(a)] In the discussion following Corollary~\ref{coro:pointspectrum},
 we introduced the concept of Fourier--Bohr
coefficients of an $\mathcal{N}$-representation with respect to an
orthonormal basis. Theorem~\ref{thm:a-representation-besicov} states that the $A_\chi$,
$\chi \in\widehat{G}$, are exactly the Fourier--Bohr coefficients of
$N$ with respect to the orthonormal basis given by the characters.
For this reason, we think of the $N$ in Theorem~\ref{thm:a-representation-besicov} as
$\mathcal{N}$-representation with Fourier coefficients. In fact, the formula for $[N(\varphi)]$ given in the theorem can be understood to give an expansion of $N$ in the form
$$N = \sum_{\chi \in\widehat{G}} A_\chi (\chi,\cdot) \chi,$$
where $(\chi, \cdot) :\Cc(G)\longrightarrow \CC$ is given by $(\chi,\varphi):=\widehat{\varphi}(\chi)$ (compare also  the corresponding discussion for $\mathcal{N}$-representations following Corollary~\ref{coro:pointspectrum}).

\item[(b)] From Theorem~\ref{thm:a-representation-besicov} (i), we see that the family $A_\chi$,
$\chi\in\widehat{G}$, has the property that
\[
\sum_{\chi \in \widehat{G}} |\widehat{\varphi}(\chi)|^2\, |A_\chi|^2
= \sum_{\chi \in \widehat{G}} |\widehat{\varphi}(\chi)|^2\,
\sigma\big(\{\chi\}\big) = \|\widehat{\varphi}\|_{L^2 (\sigma)} <\infty  \,,
\]
for all $\varphi \in\Cc (G)$. Conversely, when $A_\chi\in\CC$,
$\chi\in\widehat{G}$, are given with the summability property
$\sum_{\chi \in \widehat{G}} |\widehat{\varphi} (\chi) |^2
|A_\chi|^2 <\infty$, for all $\varphi \in \Cc (G)$, we can define an
$\mathcal{N}$-representation
\[
N : \Cc (G)\longrightarrow \bap^2 (G)/\equiv \,,\qquad  N(\varphi):= \sum_{\chi\in\widehat{G}} \widehat{\varphi}(\chi)\, A_{\chi} [\chi] \,.
\]
This $\mathcal{N}$-representation has a semi-autocorrelation, whose
Fourier transform is the pure-point measure
\[
\sigma \, := \, \sum_{\chi \in \widehat{G}} |A_\chi|^2 \delta_\chi\,.
\]
In this sense, there is a one-to-one correspondence between $A_\chi$,
$\chi \in \widehat{G}$, satisfying this summability property and
$\mathcal{N}$-representations with Fourier coefficients.
\item[(c)] A measure $\mu$ satisfying
\[
\int_{\widehat{G}} \left| \widehat{\varphi} (t) \right| \dd \mu(t) < \infty \qquad \text{ for all } \varphi \in \Cc(G)
\]
is called \textit{weakly admissible}\index{weakly~admissible~measure}. It follows from (b) that a pure-point measure $\mu$ on
$\widehat{G}$ is the diffraction of an intertwining $\cN$-representation if and only if it is positive and weakly admissible.      \exend
\end{itemize}
\end{remark}

\begin{proof} (iii)$\Longrightarrow$(ii): Recall that the norm on $\mathcal{H}$
agrees with $\|\cdot\|_{\be,2,\cA}$. Hence, (iii) and the definition of Besicovitch almost periodicity clearly imply (ii).

\medskip

\noindent (ii)$\Longrightarrow$(i): As $\bap^2(G)\subseteq \mean^2_{\cA} (G)$
and $N$ possesses a semi-autocorrelation, we infer from (ii) and
Theorem~\ref{thm:a-representation-mean} that the Fourier transform
of the semi-autocorrelation is a pure-point measure $\sigma$. By
Theorem~\ref{thm:intertwining-U}, there exists a
(unique) unitary $G$-map
\[
U  : L^2 (\widehat{G},\sigma) \longrightarrow \mathcal{H}
\]
with $\widehat{\varphi}\mapsto N(\varphi)$ for all $\varphi \in\Cc
(G)$.  As $1_{\chi}\in L^2 (\widehat{G},\sigma)$ is an eigenfunction
for each  $\chi\in \widehat{G}$ with $\sigma\big(\{\chi\}\big)
>0$, we obtain, for each such $\chi$, an eigenfunction $ U(1_\chi)$ in $\mathcal{H}$.
As the Besicovitch space has an orthonormal basis consisting of
characters and these characters are eigenfunctions to different
eigenvalues, an eigenfunction in the Besicovitch space
must be a multiple of the character. Hence, each $U(1_\chi)$ must be
a (multiple of a) character, i.e. there exist $A_\chi$, $\chi
\in\widehat{G}$, with $U(1_\chi) = A_\chi [\chi]$. Moreover, we have
\[
[N(\varphi)] = U(\widehat{\varphi})\, U\Big(\sum_{\chi\in\widehat{G}} \widehat{\varphi}(\chi) 1_{\chi} \Big) = \sum_{\chi \in\widehat{G}}
\widehat{\varphi}(\chi)\, A_\chi\, [\chi]
\]
for all $\varphi \in\Cc (G)$. This gives
\[
\widehat{\varphi}(\chi)\, A_\chi = \langle [N(\varphi)],
[\chi]\rangle  = M_{\cA}(N(\varphi) \overline{\chi})
\]
for all $\chi\in\widehat{G}$ and $\varphi \in \Cc (G)$.

Finally, since $|\widehat{\varphi}|^2 \sigma= \sigma_{N(\varphi)}$, we obtain
\[
\left| \widehat{\varphi}(\chi) \right|^2 \sigma( \{ \chi \})= \sigma_{N(\varphi)}( \{ \chi \})= \left| \widehat{\varphi}(\chi)A_\chi \right|^2
\]
for all $\chi \in \widehat{G}$ and $\varphi \in \Cc (G)$.

\medskip

\noindent (i)$\Longrightarrow$(iii):  By (i),  $\cN$ possesses a
semi-autocorrelation whose Fourier transform $\sigma$ is a pure
point measure. Hence, we infer from Theorem~
\ref{thm:intertwining-U} that there exists a (unique) unitary
$G$-map
\[
U  : L^2 (\widehat{G},\sigma) \longrightarrow \mathcal{H}
\]
with  $U(\widehat{\varphi}) =  N(\varphi)$ for all $\varphi \in\Cc
(G)$. As $U$ is unitary, we infer from (i) that
\[
\|N(\varphi)\|^2 = \|\widehat{\varphi}\|^2 = \sum_{\chi\in
\widehat{G}} |\widehat{\varphi}(\chi)|^2\, \sigma(\{\chi\})
=\sum_{\chi\in \widehat{G}} |\widehat{\varphi}(\chi)|^2\, |A_\chi|^2 =
\sum_{\chi\in\widehat{G}} |M_{\cA}(N(\varphi) \overline{\chi})|^2 \,.
\]
Moreover, a short direct computation, invoking $M_{\cA} (\chi
\overline{\varrho}) =0$ for $\chi,\varrho \in \widehat{G}$ with
$\chi\neq \varrho$, gives
\[
M_{\cA}( | N(\varphi) - \sum_{\chi \in F}
M_{\cA}(N(\varphi)\chi)|^2) = \|N(\varphi)\|^2 - \sum_{\chi\in F}
|M_{\cA}(N(\varphi)\overline{\chi})|^2
\]
for all finite sets $F\subseteq \widehat{G}$. Putting everything together, we arrive at (iii).

\medskip

\noindent The last statements of the theorem were shown in the proof
of the equivalence.
\end{proof}

From Theorem~\ref{thm:a-representation-besicov}, we easily obtain the solution to Question 2.

\begin{theorem}\label{Bap and BT}
Let $\mu \in \cM^\infty(G)$. Then, $\mu \in
\Bap^2_\cA(G)$ if and only if the following three conditions hold:
\begin{itemize}
  \item[(a)] The autocorrelation $\gamma$ of $\mu$ exists with respect to $\cA$, and $\widehat{\gamma}$ is a pure-point measure.
  \item[(b)] The Fourier--Bohr coefficients $a_{\chi}^\cA(\mu)$ exist for all
$\chi \in \widehat{G}$.
  \item[(c)] For all $\chi \in \widehat{G}$, we have the consistent phase property
  \[
  \widehat{\gamma}(\{\chi\}) = \left| a_{\chi}^\cA(\mu) \right|^2 \,.
  \]
\end{itemize}
\end{theorem} \index{Fourier--Bohr~coefficient!Fourier--Bohr~coefficient~of~function} \index{almost~periodic!Besicovitch~$p$-almost~periodic~measure}
\begin{proof}
First, we note that, for $\mu \in \Bap^2_{\cA} (G)$, the limit
$M_{\cA} ((\mu \ast \varphi) \cdot \overline{(\mu\ast \psi)})$ exists
for all $\varphi,\psi \in\Cc (G)$ due to the existence of means for
products of functions from $\bap^2(G)$. Hence, for such $\mu$,
the autocorrelation exists by
Proposition~\ref{prop-compute-autocorrelation}. Thus, we infer from
Proposition \ref{prop:translation-bounded-admissible} that  the map
\[
N_\mu : \Cc (G)\longrightarrow \mathcal{L}^1_{loc} (G) \,,\qquad
N_\mu(\varphi) = \mu \ast \varphi  \,,
\]
is an intertwining $\mathcal{A}$-representation.
We also note that, for $\mu \in \cM^\infty(G)$, the existence of the
Fourier--Bohr coefficients $a_{\chi}^{\cA}(\mu)$ is equivalent to the
existence of the Fourier--Bohr coefficients $a_{\chi}^{\cA}(\mu*\varphi) =
M_\cA ((\mu\ast\varphi) \cdot \overline{\chi})$ for all $\varphi
\in\Cc (G)$ due to Corollary \ref{FB measure relations}.
Now, we can easily infer the statement of the theorem by an
application of Theorem~\ref{thm:a-representation-besicov} to the
$\mathcal{A}$-representation  $N_\mu$ with the Fourier--Bohr
coefficients $A_\chi$ appearing in
Theorem~\ref{thm:a-representation-besicov} being the
Fourier--Bohr coefficients $a_{\chi}^\cA(\mu)$ of $\mu$.
\end{proof}

\begin{remark}
\phantom{XX}
\begin{itemize}
\item[(a)] We could replace the assumption that $\mu$ is translation bounded by
the assumption that  $\mu \ast \varphi$ belongs to $BC^2_{\cA} (G)$ for
all $\varphi \in\Cc (G)$ and that either $\mu$ is positive or
$\supp(\mu)$ is uniformly discrete. The proof  proceeds along the
very same lines with Proposition \ref{prop-compute-autocorrelation}
and Proposition \ref{prop:translation-bounded-admissible} replaced
by Corollary \ref{prop:positive-admissible} or
Corollary~\ref{prop:unif disc-admissible}.
\item[(b)] In the situation of Theorem~\ref{Bap and BT}, the Fourier--Bohr coefficients
$a_{\chi}^\cA(\mu)$ are exactly the abstract Fourier--Bohr coefficients
appearing in Theorem \ref{thm:a-representation-besicov}. Hence, they
satisfy the following square summability type condition: For all
$\varphi \in \Cc(G)$, we have
\begin{displaymath}
\sum_{\chi \in \widehat{G}}\left| \widehat{\varphi}(\chi)
a^\cA_\chi(\mu) \right|^2 =\| \mu*\varphi \|_{\be,2,\cA}^2 \,.
\end{displaymath}
In the particular case $\mu \in \SAP(G)$, this was proven in
\cite[Prop.~8.3]{ARMA}.
\item[(c)] For $\mu \in \Bap_{\cA}(G)$, we can define a discrete Fourier
transform, see \cite[Sect.~6.3]{NS18}, by
\[
\cF_{\mathsf{d}}(\mu):= \sum_{\chi \in \widehat{G}}
a_{\chi}^\cA(\mu) \delta_{\chi} \,.
\]
This coincides with the classical
Fourier transform in the case of Fourier transformable strongly almost periodic
measures \cite{NS18}. Theorem~\ref{Bap and BT} tells us that, for
 $\mu \in \Bap_{\cA}(G) \cap \cM^\infty(G)$, we have a
commutative Wiener diagram

\begin{center}
\begin{tikzpicture}[scale=.35]
  \matrix (m) [matrix of math nodes,row sep=3em,column sep=4em,minimum width=2em]
  {
    \mu & \gamma \\
     \cF_{\mathsf{d}}(\mu) & \widehat{\gamma} \\};
  \path[-stealth]
    (m-1-1) edge node [left] {$\cF_{\mathsf{d}}$} (m-2-1)
            edge node [above] {$\lb \, , \, \rb_{\cA}$} (m-1-2)
    (m-2-1.east|-m-2-2) edge node [above] {$| \cdot |^2$} (m-2-2)
    (m-1-2) edge node [right] {$\widehat{\cdot}$} (m-2-2);
\end{tikzpicture}
\end{center}
where we define the square $|\cdot|^2$ of the point measure $\nu = \sum c_\chi \delta_\chi$ by $|\nu|^2 = \sum |c_\chi|^2 \delta_\chi$.   \exend
\end{itemize}
\end{remark}

The previous result suggests to single out the diffraction features given by  (a), (b), (c) with  a specific name.

\begin{definition}[Complete point spectrum with respect to $\mathcal{A}$]
Let $\cA$ be a van Hove sequence. A  measure $\mu$ is said to have \emph{complete point spectrum with respect to $\cA$} if the following three conditions hold:
\begin{itemize}
  \item[(a)] The autocorrelation $\gamma$ of $\mu$ exists with respect to $\cA$, and $\widehat{\gamma}$ is a pure-point measure.
  \item[(b)] The Fourier--Bohr coefficients $a_{\chi}^\cA(\mu)$ exist for all
$\chi \in \widehat{G}$.
  \item[(c)] For all $\chi \in \widehat{G}$, we have the consistent phase property
  \[
  \widehat{\gamma}(\{\chi\}) = \left| a_{\chi}^\cA(\mu) \right|^2 \,.  \tag*{$\Diamond$}
  \]
\end{itemize}
\end{definition}

With this definition the main result of this section reads that a translation bounded measure has complete point spectrum with respect to $\cA$ if and only if it is Besicovitch almost periodic.

\medskip

The proof of the preceding theorem has the following consequence.

\begin{coro}[Expansion of a Besicovitch almost periodic measure]  Any  $\mu \in \cM^\infty(G) \cap
\Bap^2_\cA(G)$ has an expansion
$$\mu = \sum_{\chi \in\widehat{G}} a_{\chi}^\cA (\mu)\, \chi$$
in the sense that
$$\mu \ast \varphi \equiv  \sum_{\chi \in \widehat{G}} a_{\chi}^\cA (\mu) \, \chi \ast \varphi$$
holds for all $\varphi \in \Cc (G)$. Here, $\equiv$ refers to equality in the $\bap^2 (G)$.
\end{coro}
\begin{proof} A direct computation gives
$$\chi \ast \varphi = \widehat{\varphi}(\chi) \chi$$
 for all $\chi \in\widehat{G}$.  Also,  the proof of the preceding theorem gives that $N_\mu$ with $N_\mu (\varphi) = \mu\ast \varphi$ is an $\mathcal{A}$-representation with coefficients $A_\chi = a_{\chi}^\cA (\mu)$.
Therefore,  Theorem~\ref{thm:a-representation-besicov} applied to this $\mathcal{A}$-representation gives
$$\mu \ast \varphi \equiv  \sum_{\chi \in\widehat{G}} a_{\chi}^\cA (\mu)  \chi \ast \varphi$$
for all $\varphi \in\Cc (G)$.
\end{proof}

As it is instructive, we discuss here  an example to see the
difference between mean almost periodic measures and Besicovitch
almost periodic measures.

\begin{example}
Consider the function $f$ and the van
Hove sequence $\mathcal{A}$ from Example~\ref{rem:map vs bap}. Then, $f \in
C_u (\RR) \cap \mean(\RR)$. Let $\mu =f \lm$ (with the Lebesgue
measure $\lm$). Everything can be computed explicitly:
\[
\gamma = \frac{1}{2}\,\lm \qquad \mbox{ and } \qquad \widehat{\gamma} = \frac{1}{2} \delta_{\underline{1}} \,,
\]
where we write $\underline{1}$ for the character which maps
everything to $1$. This character could also be denoted as $0$ if
we identify $\widehat{\RR}$ with $\RR$.

Now, let us consider $\mathcal{H}_\mu$:  Clearly, the
$\mathcal{H}_\mu$ norm is just the Besicovitch $2$-norm, which is
always true in the mean almost periodic case. Moreover, we note
that
\[
\mu \ast \varphi = f \ast \varphi = \left(\int_{\RR} \varphi(x)\, \dd x\right) \cdot
f \in \mathcal{H}_\mu
\]
for all $\varphi \in C_c (\RR)$. Here the first equality holds in
the sense of honest functions, and the last equality holds in
$BL_{\cA}^1 (G)$, i.e. after factoring out functions which vanish in the
Besicovitch seminorm. Therefore, $\mathcal{H}_\mu$ is one dimensional with
\[
\mathcal{H}_\mu = \{c f : c\in \CC\} \,.
\]
In particular,  $f$ is an eigenfunction to the eigenvalue $1$. This
can also directly be seen as $\| \tau_t f - f\|_{\be,2,\cA}=0$. Now, the character $\underline{1} $  does
\textit{not} belong to $\mathcal{H}_\mu$. However, the mean
\[
M_{\mathcal{A}} (f \cdot \overline{\chi}) = \lim_{n\to\infty}
\frac{1}{|A_n|} \int_{A_n} f(s) \, \overline{\chi (s)}\, \dd s
\]
exists for all $\chi$ in the dual group of $\RR$, and it is zero for
$\chi \neq \underline{1}$. For $\chi = \underline{1}$, we find
$a_{\underline{1}} (f) = \frac{1}{2}$. So, in this example, all
Fourier coefficients exist but the characters are not elements of the Hilbert
space $\mathcal{H}_\mu$ and
\[
|a_{\underline{1}}|^2  =\frac{1}{4} \neq \frac{1}{2} =
\widehat{\gamma}(\{0\}) \,.       \tag*{$\Diamond$}
\]
\end{example}

%
%

\section{Weak model sets of maximal density}\label{weak ms}
In this section, we apply the preceding considerations to the study
of weak model sets of maximal density. This will allow us to recover
various recent results.  For more details on cut-and-project schemes (CPS)
as well as for the notation used below, we refer to Appendix~\ref{appendix:cp}.

\medskip

When a CPS $(G,H, \cL)$ and a compact set $W \subseteq H$ are
given, we say that $\oplam(W)$ is a \textit{weak model set of
maximal density} with respect to $\cA$, see Definition~\ref{wms def}, if
\[
\dens(\oplam(W)) = \dens(\cL)\, \theta_H(W) \,.
\]
We now show that each weak model set of maximal density leads to a Besicovitch almost periodic
Dirac comb. This explains the pure-point diffractivity of this class. In fact, below, we recover all known results for
this class of point sets.
We start with the following preliminary result.

\begin{lemma}\label{lem-appr}
Let $(G, H, \cL)$ be a CPS with compact window $W \subseteq H$, and let $\cA$ be a van Hove sequence in $G$. Assume that $h_n \in \Cc(H)$ satisfies $h_n \geq 1_W$ and
\[
\lim_{n\to\infty} \int_{H} h_n(t) \dd t = \theta_H(W)\,.
\]
Then, $\oplam(W)$ is a weak model set of maximal density with respect to $\cA$ if and only if
\[
\lim_{n\to\infty} \| \omega_{h_n}-\delta_{\oplam(W)} \|_{\be,\cA} =0 \,.
\]
\end{lemma}
\begin{proof}
Note first that, by Theorem~\ref{thm-dens}, we have
\begin{equation}\label{eq1z}
\lim_{m\to\infty} \frac{\omega_{h_n}(A_m)}{|A_m|} = \dens(\cL) \int_{H} h_n(t) \dd t
\end{equation}
for all $n$. Also, we have
\[
\omega_{h_n} \geq \delta_{\oplam(W)} \qquad \text{ for all } n \,.
\]
$\Longrightarrow$: Since $\oplam(W)$ is a model set of maximal density, we have
\[
\lim_{m\to\infty} \frac{\delta_{\oplam(W)}(A_m)}{|A_m|} = \dens(\cL)\theta_H(W) \,.
\]
Now, $\omega_{h_n} \geq \delta_{\oplam(W)}$ implies
\begin{align*}
  \lim_{m\to\infty}  \frac{1}{|A_m|} \left| \omega_{h_n} -\delta_{\oplam(W)}\right|(A_m)
  &=\lim_{m\to\infty}\left( \frac{\omega_{h_n}(A_m)}{|A_m|}  - \frac{\delta_{\oplam(W)}(A_m)}{|A_m|} \right)\\
  &= \dens(\cL) \left( \int_{H} h_n(t) \dd t  - \theta_H(W) \right) \,.
\end{align*}
Therefore, we obtain
\[
\| \omega_{h_n} - \delta_{\oplam(W)} \|_{\be,\cA} = \dens(\cL) \left( \int_{H} h_n(t) \dd t  - \theta_H(W) \right) \stackrel{\ n \to \infty\ }{\longrightarrow} 0 \,.
\]

\noindent $\Longleftarrow$: Again, using $\omega_{h_n} \geq \delta_{\oplam(W)}$ and the fact that the limit in \eqref{eq1z} exits, we get
\begin{align*}
  \| \omega_{h_n} - \delta_{\oplam(W)} \|_{\be,\cA}
  & =\limsup_{m\to\infty} \left(\frac{\omega_{h_n}(A_m)}{|A_m|}  - \frac{\delta_{\oplam(W)}(A_m)}{|A_m|}\right) \\
  &= \left(\lim_{m\to\infty} \frac{\omega_{h_n}(A_m)}{|A_m|}\right)  - \left(\liminf_{m\to\infty} \frac{\delta_{\oplam(W)}(A_m)}{|A_m|}\right) \\
  &= \dens(\cL) \int_{H} h_n(t) \dd t  - \liminf_{m\to\infty} \frac{\delta_{\oplam(W)}(A_m)}{|A_m|} \,.
\end{align*}
This gives
\[
\liminf_{m\to\infty} \frac{\delta_{\oplam(W)}(A_m)}{|A_m|} =\left( \dens(\cL) \int_{H} h_n(t) \dd t \right)- \| \omega_{h_n} - \delta_{\oplam(W)} \|_{\be,\cA} \,.
\]
Taking the limit as $n \to \infty$, we get
\[
\liminf_{m\to\infty} \frac{\delta_{\oplam(W)}(A_m)}{|A_m|} = \dens(\cL)\theta_H(W) \,.
\]
The compactness of $W$ and Proposition~\ref{prop-dens-bounds} imply
\[
\liminf_{m\to\infty} \frac{\delta_{\oplam(W)}(A_m)}{|A_m|} = \limsup_{m\to\infty} \frac{\delta_{\oplam(W)}(A_m)}{|A_m|} =\dens(\cL)\theta_H(W)\,.
\]
Therefore, $W$ is a weak model set of maximal density.
\end{proof}

\begin{prop}\label{wms}
Let $\vL$ be a weak model set of maximal density with respect to
$\cA$. Then, $\delta_{\vL} \in \Bap_{\cA}(G) \cap \cM^\infty(G)$.
\end{prop}
\begin{proof}
Let $(G,H, \cL)$ be the
CPS, and let $W$ be the window which satisfies $\vL=\oplam(W)$ as a model set of maximal
density.
Select $h_n \in \Cc(H)$ such that $1_{W} \leq h_n $ and
\[
\int_H \big(h(t)-1_{W}(t)\big)\ \dd t < \frac{1}{n} \,.
\]
Such a sequence $(h_n)$ exists by the regularity of the Haar measure $\theta_H$ and the compactness of $W$.
Since $\oplam(W)$ is a weak model set of maximal density, we have
\[
\lim_{n\to\infty} \| \omega_{h_n}-\delta_{\oplam(W)} \|_{\be,\cA} =0 \,.
\]
Also, by Theorem~\ref{thm-dens}, we have $\omega_{h_n} \in \SAP(G)$.
Lemma~\ref{lem-bes-ine-measures} and Proposition~\ref{ap inclusions 1}(c) complete the proof.


\end{proof}

Now, we can give an alternative proof to the following result first
shown in \cite{BHS,KR}.

\begin{coro}\label{coro:weak-model-set}
Let $(G,H, \cL)$ be a CPS with compact window $W \subseteq H$ such that
$\vL=\oplam(W)$ is a weak model set of maximal density with
respect to $\cA$.
\begin{itemize}
\item [(a)] The autocorrelation $\gamma$ of $\vL$ exists with respect to $\cA$, and $\widehat{\gamma}$ is pure point.
\item [(b)] For each $\chi \in \widehat{G}$, the Fourier--Bohr coefficient exists with respect to $\cA$ and satisfies
\[
a_\chi^{\cA}(\delta_{\vL})=
\begin{cases}
\dens(\cL)\, \reallywidecheck{1_W}(\chi^\star), & \mbox { if } \chi \in \pi_{\widehat{G}}(\cL^0)\,, \\
  0, &\mbox{ otherwise }\,.
\end{cases}
\]
\item [(c)] For all $\chi \in \widehat{G}$, we have $\widehat{\gamma}(\{ \chi \})= \big|  a_\chi^{\cA}(\delta_{\vL}) \big|^2$.
\item [(d)] We have
\begin{align*}
\gamma \,&=\, \dens(\cL) \omega_{1_W*\widetilde{1_W}} \\
\widehat{\gamma} \,&=\, \big(\dens(\cL)\big)^2\, \omega^{}_{\big|\reallywidecheck{1_W}\big|^2} \,.
\end{align*}
\end{itemize}
\end{coro}
\begin{proof}
(a) follows from $\delta_{\vL} \in \Bap_{\cA}(G)$.

\medskip

\noindent (b) The Fourier--Bohr coefficients exist by Besicovitch almost
periodicity. Fix some $\chi \in \widehat{G}$ and $\varphi \in \Cc(G)$ such
that $\widehat{\varphi}(\chi)=1$.
Assume that $h_n \in \Cc(H)$ satisfies $1_{W} \leq h_n$ and $\int_{H} h_n(t) \dd t \leq \theta_{H}(W) +\frac{1}{n}$ for each $n$. Then, Lemma~\ref{lem-appr} gives
\[
\lim_{n\to\infty} \|\omega_{h_n}-\delta_{\vL} \|_{\be,\cA} =0 \,,
\]
and Corollary~\ref{coro-FB-bound-bes-measure} implies
\[
a_\chi^\cA(\delta_{\vL})= \lim_{n\to\infty} a_\chi(\omega_{h_n}) \,.
\]
Finally, by \cite{NS11,TAO}, we have
\[
a_\chi(\omega_{h_n})=
\begin{cases}
\dens(\cL)\, \reallywidecheck{h_n}(\chi^\star), & \mbox { if } \chi \in \pi_{\widehat{G}}(\cL^0) \,, \\
  0, &\mbox{ otherwise }\,.
\end{cases}
\]
The claim follows.

\smallskip

\noindent (c) follows from Theorem~\ref{thm:a-representation-besicov}.

\medskip

\noindent (d) is now immediate. Indeed, by the above, we have
\[
\widehat{\gamma}=  \big(\dens(\cL)\big)^2 \sum_{ \chi \in
\pi_{\widehat{G}}(\cL^0)} \big| \reallywidecheck{1_W}(\chi^\star)
\big|^2\delta_{\chi}= \big(\dens(\cL)\big)^2\, \omega_{\big|\reallywidecheck{1_W}\big|^2}  \,.
\]
Moreover, $ \dens(\cL)\, \omega_{1_W*\widetilde{1_W}}$ is positive
definite, thus Fourier transformable \cite{CRS} with
\[
\reallywidehat{\dens(\cL)\, \omega_{1_W*\widetilde{1_W}}} = \widehat{\gamma} \,,
\]
which gives the claim.
\end{proof}

\begin{remark}
\phantom{XX}
\begin{itemize}
\item[(a)] The corollary contains the main results of \cite{BHS}. The only
result from \cite{BHS} which is missing above, namely that every weak model set of
maximal density is generic for an ergodic measure, follows from Theorem~\ref{bap gives ergodic}, which
we prove in Section \ref{sec:Dynamics}.
\item[(b)] Recall that given a CPS $(G,H, \cL)$, a compact set $W \subseteq
H$ and a tempered\footnote{Recall that $(A_n)$ is called tempered if there exists
a constant $C$ such that $|A_n-A_n| \leq C |A_n|$ for all $n$.} van Hove sequence
$\cA$, for almost all $(s,t)+\cL \in \TT=(G\times H)/\cL$, the set
$-s+\oplam(t+W)$ has maximal density with respect to $\cA$, see
\cite{Moody}. It is therefore Besicovitch almost periodic. This
explains the pure-point spectrum of the extended hull of weak model
sets \cite{KR}.
\item[(c)] For every weak model set of maximal density $\oplam(W)$, we have
\[
\cF_{\mathsf{d}}\big(\delta_{\oplam(W)}\big) = \dens(\cL)\, \omega_{\reallywidecheck{1_W}} \,,
\]
i.e. the discrete Fourier transform of the underlying Besicovitch
almost periodic measure corresponds to the inverse Fourier transform
of the window. Note that  this holds in particular  for regular
model sets.    \exend
\end{itemize}
\end{remark}

\begin{example}\label{ex-cp-a-rep-vs-n-rep}
Consider a cut-and-project scheme $(G,H,\cL)$ with second countable $G,H$. Let $W \subseteq H$ be a compact set, and let
\[
\gamma:= \sum_{(x,x^\star)\in \cL} 1_{W}*\widetilde{1_{W}}(x^\star)\, \delta_x
\]
be the weighted Dirac comb given by the covariogram $c(W)=1_{W}*\widetilde{1_{W}}$ of $W$. Then, $\gamma$ is Fourier transformable \cite{CRS} and
\[
\widehat{\gamma}=\sum_{(\chi, \chi^\star) \in \cL^0} \Big| \reallywidecheck{1_W}(\chi^\star) \Big|^2 \delta_{\chi} \,.
\]
It follows that $N: \Cc(G) \to {\mathcal{H}}:= L^2(\widehat{\gamma})$ defined by
\[
N(\varphi):= \reallywidecheck{\varphi}
\]
and $U_t f(\chi):=\chi(t) f(\chi)$ define an intertwining $\mathcal{N}$-representation with autocorrelation $\gamma$ and diffraction $\widehat{\gamma}$.

Next, if $\cA=(A_n)$ is a tempered van Hove sequence, the set
$\vL:=-s+\oplam(t+W)$ has maximal density with respect to $\cA$ for almost all $(s,t)+\cL \in \TT$. Choosing one such set, we can define an intertwining $\mathcal{\cA}$ representation $N': \Cc(G) \to \mathcal{H}' \subseteq \Bap^2_{\cA}(G)$ via
\[
N'(\varphi):= \varphi*\delta_{\vL}\, ,
\]
and $T'_t$ being the translation operator. Here, $\mathcal{H}'$ is the Hilbert subspace of $\Bap^2_{\cA}(G)$ such that $N'$ has dense range.
Corollary~\ref{coro:weak-model-set} implies that this representation has the autocorrelation $\gamma$ and the diffraction $\widehat{\gamma}$.
It is easy to see that the mapping
\[
L^2(\widehat{\gamma})\supseteq \{ \widehat{\varphi}: \varphi \in \Cc(G) \} \ni \widehat{\varphi} \to \varphi*\delta_{\vL}
\]
induces a Hilbert space isometric isomorphism
\[
S: L^2(\widehat{\gamma}) \to \mathcal{H}'
\]
such that $U \circ S =S \circ (T')$ and $S \circ N=N'$.    \exend
\end{example}

\chapter[Weyl almost periodicity]{Weyl almost periodicity and the uniform consistent phase property}\label{sec:Weyl}
In this chapter, we study Weyl almost periodicity. By its
definition, it is a very uniform form of almost periodicity, and we
will show that, under suitable boundedness assumptions, Weyl almost
periodicity is the same as Besicovitch almost periodicity with
respect to any van Hove sequence. This will allow us to solve Question 3.

\section{Weyl almost periodic functions and measures}
Let us discuss Weyl almost periodic functions and
measures.

\begin{definition}[Weyl almost periodic functions and measures]
Let $\cA =(A_n)$ be a F\o lner sequence on $G$, and let $1 \leq p<\infty$ be given.
A function $f \in \mathcal{L}^p_{loc}(G)$ is called
\textit{Weyl $p$-almost periodic}\index{almost~periodic!Weyl~$p$-almost~periodic~function} with respect to $\cA$ if, for each $\eps >0$, there exists a trigonometric polynomial $P= \sum_{k=1}^n
c_k \chi_k$ with $c_k \in \CC$ and $\chi_k \in \widehat{G}$ such that  \nomenclature{$\wap^p(G)$}{set of Weyl $p$-almost periodic functions}
\[
\| f -P \|_{\we,p,\cA} < \eps \,.
\]
We denote the space of Weyl $p$-almost periodic functions by
$\wap^p_{\cA}(G)$\label{wap}. When $p=1$, we will simply set
\[
\wap_\cA(G) \, := \, \wap_\cA^1(G) \,.    
\]
A measure $\mu$ on $G$ is called \textit{Weyl
$p$-almost periodic}\index{almost~periodic!Weyl~$p$-almost~periodic~measure} if the function $\varphi*\mu$ is Weyl
$p$-almost periodic for all $\varphi \in \Cc(G)$. The  space of
Weyl $p$-almost periodic measures is denoted by $\Wap_{\cA}^p(G)$. \nomenclature{$\Wap^p(G)$}{set of Weyl $p$-almost periodic measures}
When $p=1$, we will simply set
\[
\Wap_\cA(G) \, := \, \Wap_\cA^1(G) \,.       \tag*{$\Diamond$}
\]
\end{definition}

\begin{remark}\label{rem:Weyl-ap}\phantom{X}
\begin{itemize}
\item[(a)] A function $f$ is Weyl $p$-almost periodic if and only if, for each $\eps >0$, there exists a Bohr almost periodic function $g$ with $\| f -g \|_{\we,p,\cA} <\eps$.
\item[(b)] If $f \in \mathcal{L}^p_{loc}(G)$ is so that 
\[
\|f \|_{\we,p,\cA}=0 
\]
then $f \in \wap^p_{\cA}(G)$. 

In particular, When $h \in \wap^p_{\cA}(G)$ and $f :G\longrightarrow \CC$ is measurable with $\|f \|_{\we,p,\cA} = 0$,
 then $f + h \in \wap^p_{\cA}(G)$, and
 \[
 \| f+h \|_{\we,p,\cA}= \| h \|_{\we,p,\cA} \,.
 \]
\item[(c)] As is clear from (a), all Bohr almost periodic functions are Weyl
almost periodic.

In fact, it is not hard to see that every weakly
almost periodic function is Weyl almost periodic. Indeed, any such
$f$ can be decomposed as 
\[
f = g + h
\]
with $g$ Bohr almost periodic
and $\| h\|_{\we,p,\cA}=0$, see \cite[Prop.~4.5.9]{MoSt},  and the statement follows from (b).
\item[(d)] In general, the space $(\wap_{\cA}^p(G), \| \cdot \|_{\we,p,\cA})$ is not complete \cite{BoF}.
\item[(e)]  In general, one has $ \wap^p_{\cA}(G) \subsetneq \bap^p(G)$, see Example~\ref{remark Bap but not Wap}.
\item[(f)] If $G$ is compact, then for all van Hove sequences $
\cA$  we have 
\begin{align*}
\bap^p(G) \, &= \, \wap^p_{\cA}(G)=\mathcal{L}^p(G)\\
\| \cdot \|_{\be,p,\cA}\, &= \, \| \cdot \|_{\we,p,\cA}=\| \cdot \|_p \,.
\end{align*}
\item[(g)] We observe $\Wap^p_{\cA}(G) =  \Bap^p_{\cA}(G)=\cM(G)$ if $G$ is compact. 
\item[(h)]  The function $f$ from Example \ref{rem:map vs bap}  can easily be seen  to satisfy 
\[
\| f - f_t\|_{W,q,\cA} = 0 \qquad \text{ for all } t\in \RR\,.
\]
However, as it is not Besicovitch almost periodic it can not be Weyl almost periodic. This shows that we can not define Weyl almost periodicity based on denseness of $\varepsilon$-periods.   \exend
\end{itemize}
\end{remark}

\begin{prop}[Inclusions of spaces] \label{ap inclusions} \phantom{X}
\begin{itemize}
\item [(a)] For each $1 \leq p <\infty$, we have
\[
\wap_{\cA}^p(G) \subseteq \bap^p(G)
\]
with continuous inclusion map. Moreover, the inclusion can be strict.
\item [(b)] For each $1 \leq p \leq q < \infty$, we have
\[
\wap_{\cA}^q(G) \subseteq   \wap_{\cA}^p(G) \subseteq
\wap_{\cA}(G)
\]  with continuous inclusion map.
\item [(c)] For each $1 \leq p <\infty$, we have 
\[
\Wap^p_{\cA}(G) \subseteq \Bap^p_{\cA}(G) \,.
\]
Moreover, the inclusion can be strict.
\item [(d)] For each $1 \leq p \leq q < \infty$, we have
\[
\Wap^q_{\cA}(G) \subseteq   \Wap^p_{\cA}(G) \subseteq
\Wap_{\cA}(G)\,.
\]
\end{itemize}
\end{prop}
\begin{proof} (a) and (b) follow from Lemma~\ref{L1}, Lemma~\ref{lemma C-S} and Lemma~\ref{lemma norm inequality}. (c) and (d) follow from (a) and (b).
\end{proof}

Next, we will recollect a few results for Weyl almost periodic functions
that follow easily, when one replaces $\| \cdot \|_{\be,p,\cA}$ by $\|
\cdot \|_{\we,p,\cA}$ in the corresponding statements and proofs for
Besicovitch almost periodic functions. We begin with an analogue of
Proposition~\ref{prop:bap for functions vs measures}.

\begin{prop}
Let $\cA$ be a van Hove sequence on $G$, let $1\leq  p < \infty$, and let $f \in \Cu(G)$ be arbitrary. Then, $f \in
\wap_{\cA}^p(G)$ if and only if $f \theta_G \in \Wap_{\cA}^p(G)$. \qed
\end{prop}

The next result is the
analogue of Proposition~\ref{Bap props}. (a) and (e) follow trivially from Proposition~\ref{prop-aprox-FB}.
Again, since the proofs of (b), (c), (d) and (f) are almost identical to Proposition~\ref{Bap props}, we skip it.

\begin{prop}[Basic properties Weyl almost periodic functions]\label{spi lemma}
Let $\cA$ be a F\o lner sequence on $G$.
\begin{itemize}
\item [(a)] Any $f\in  \wap_{\cA}(G)$ is amenable, i.e.
\[
\lim_{n\to\infty} \frac{1}{|A_n|} \int_{s+A_n} f(t)\, \dd t
\]
exists uniformly in $s\in G$.
\item[(b)] For all $f \in \wap^p_{\cA}(G)$ the limit
\[
\|f \|_{\we,p,\cA}= \lim_{n\to\infty} \left(\frac{1}{|A_n|} \int_{t+A_n} |f(t)|^p \dd t \right)^{\frac{1}{p}}
\]
exists uniformly in $t$.
\item[(c)] When $f,g$ belong to $\wap_{\cA}^p(G)$ for some $p\geq 1$, so do $f\pm g$, $cf$, $\chi f$ and $|f|$
for all $c \in \CC$ and $\chi \in \widehat{G}$.
\item[(d)] If $f,g\in \wap^p_{\cA} (G)$ and $f$ is bounded, then $f g \in \wap^p_{\cA} (G)$.
\item[(e)] For all $ 1 \leq p <\infty$, all $f\in \wap^p_{\cA}  (G)\subseteq \wap^1_{\cA} (G)$ and all $\chi \in\widehat{G}$, the Fourier--Bohr coefficient
\[
a_\chi(f)=\lim_{n\to\infty} \frac{1}{|A_n|} \int_{s+A_n}
\overline{\chi(t)}\, f(t)\, \dd t
\]
exists uniformly in  $s\in G$.
\item[(f)] If $1< p,q <\infty$ are conjugate, and $f \in \wap^p(G), g \in \wap^q(G)$, then $fg \in \wap(G)$.\qed
\end{itemize}
\end{prop}

\smallskip


As an immediate consequence, we get the following result.

\begin{coro}\label{coro-wap-compl} Let $1 \leq p < \infty$, and let $\cA$ be a F\o lner sequence.
Then, for all $f \in \wap^{p}(G)$, we have $f \in \bap^p(G)$ and
\[
\| f \|_{\we,p,\cA}= \| f \|_{\be,p,\cA} \,.
\]
In particular, the mapping
\[
\wap_{\cA}^p(G) \ni f \to [f] \in (\bappe, \| \, \|_{\be,p,\cA})
\]
is a completion map. \qed
\end{coro}


\smallskip

Analogously to Proposition~\ref{BAP equality}, one obtains the following identity.

\begin{prop}\label{WAP equality}
For each $1 \leq p < \infty$, we have
\[
\wap^p_{\cA}(G) \cap L^\infty(G)= \wap_{\cA}(G) \cap L^\infty (G) \,.     \tag*{$\qed$}
\]
\end{prop}

\medskip

\begin{coro}
For each $1 \leq p < \infty$, we have
\[
\cM^\infty(G) \cap \Wap^p_{\cA}(G) = \cM^\infty(G) \cap \Wap_{\cA}(G) \,.     \tag*{$\qed$}
\]
\end{coro}

The next statement is immediate from the definition of the Weyl
seminorm. It does not have an analogue for Besicovitch almost
periodic functions.

\begin{prop}
Let $1 \leq p < \infty$, and let $ f \in \wap^p_{\cA} (G)$. Then, for all $t \in G$ we have
$\tau_t  f  \in \wap^p_{\cA} (G)$, and
\[
\|f\|_{\we,p,\cA} =
\|\tau_t f\|_{\we,p,\cA} \,.        \tag*{$\qed$}
\]
\end{prop}

For bounded functions, Weyl almost periodicity does not depend of
the choice of the van Hove sequence. For this reason, for the rest of the chapter, we will restrict to van Hove sequences $\cA$.

\begin{prop}\label{prop:wap indep}
Let $f : G\longrightarrow \CC$ be a bounded and measurable function.
Let  $\cA$ and $\cB$ be van Hove sequences. Then, $f$ belongs to
$\wap^p_{\cA} (G)$ if and only if it belongs to $\wap^p_{\cB} (G)$. If
$f$ belongs to $\wap^p_{\cA} (G)$ and $\wap^p_{\cB} (G)$, we have
\[
\| f
\|_{\we,p,\cA} = \| f \|_{\we,p,\cB} \,.
\]
\end{prop}
\begin{proof}
This follows from Proposition~\ref{prop:mother-of-uniform-van-Hove-results}:  Assume $f\in
\wap^p_{\cA} (G)$. Let $\varepsilon >0$ be given. Then, there exist
an $N\in\NN$ and a trigonometric polynomial $P$ with
\[
\frac{1}{|A_N|} \int_{A_N+ s} |f(t) - P(t)|^p\, \dd t <\varepsilon
\]
for all $s\in G$.  With $h = |f - P|$, $A= A_N$ and $r = \varepsilon$, Proposition~\ref{prop:mother-of-uniform-van-Hove-results} implies that
\[
\frac{1}{|B_n|} \int_{B_n +u} |f(t) - P(t)|^p\, \dd t <2 \varepsilon
\]
holds for all $u\in G$ and $n$ sufficiently large. As $\varepsilon>0$ was
arbitrary, we deduce that $f\in \wap^p_{\cB} (G)$.

Similarly, there exist an $N'\in\NN$ with
\[
\frac{1}{|A_{N'}|} \int_{A_{N'}+ s} |f(t)|^p\, \dd t < \| f \|_{\we,p,\cA}^p  +  \varepsilon
\]
for all $s\in G$. With $h = |f|^p$, $A = A_{N'}$ and $r = \| f
\|_{\we,p,\cA}^p  +  \varepsilon$ we infer, again from
Proposition \ref{prop:mother-of-uniform-van-Hove-results}, that
\[
\frac{1}{|B_n|} \int_{B_n+ s} |f(t)|\, \dd t < \| f \|_{\we,p, \cA}  + 2\,\varepsilon
\]
for all $s\in G$ and $n$ sufficiently large. As $\varepsilon>0$ was arbitrary this gives
\[
\|f\|_{\we,p,\cB} \leq \|f\|_{\we,p,\cA} \,.
\]
Reversing the roles of $\cA$ and $\cB$, we obtain the remaining statement.
\end{proof}

In fact, it is even  possible to think about bounded Weyl almost
periodic functions as functions which are Besicovitch almost
periodic for every van Hove sequence.

\begin{prop}[Characterization of bounded elements of $\wap_{\cA}(G)$]
\label{prop-char-wap} Let $f : G\longrightarrow \CC$ be a bounded
measurable function, and let $\cA$ a van Hove sequence on $G$. Then, the
following assertions are equivalent:
\begin{itemize}
\item[(i)] The function $f$ belongs to $\wap_{\cA}(G)$.
\item[(ii)] The Fourier--Bohr coefficient
\[
a^{\cA}_{\chi}(f) = \lim_{n\to\infty} \frac{1}{|A_n|} \int_{A_n + s} f(t)\, \chi (t)\, \dd t
\]
and the limit
\[
\lim_{n\to\infty} \frac{1}{|A_n|}\int_{A_n+s} |f(t)|^2\, \dd t=M\big(|f|^2\big)
\]
exist uniformly in $s\in G$. Moreover, they satisfy the Parseval identity
\[
\sum_{\chi \in\widehat{G}} |a^{\cA}_\chi|^2 = M\big(|f|^2\big) \,.
\]
\item[(iii)] The Fourier--Bohr coefficient
\[
a^{\cB}_{\chi}(f) = \lim_{n\to\infty} \frac{1}{|B_n|} \int_{B_n} f(t)\, \overline{ \chi(t)}\, \dd t
\]
and the limit
\[
\lim_{n\to\infty} \frac{1}{|B_n|}\int_{B_n}
|f(t)|^2\, \dd t=M_{\cB}\big(|f|^2\big)
\]
exist for each van Hove sequence $\cB$. Moreover, they satisfy the Parseval identity
\[
\sum_{\chi \in\widehat{G}} |a^{\cB}_\chi|^2 = M_{\cB}\big(|f|^2\big) \,.
\]
\item[(iv)]
The function $f$ belongs to $B\hspace*{-1pt}es^2_{\cB} (G)$ for every van Hove
sequence $\cB$.
\end{itemize}
In particular, any bounded function in $\wap_{\cA}(G)$ is amenable.
\end{prop} \index{Fourier--Bohr~coefficient!Fourier--Bohr~coefficient~of~function} \index{almost~periodic!Weyl~$p$-almost~periodic~function}
\begin{proof}
By Proposition~\ref{BAP equality}, a bounded measurable function $f$ belongs
to $B\hspace*{-1pt}es_{\cB}(G)$ if and only if it belongs to $B\hspace*{-1pt}es_{\cB}^2(G)$.
This will be used throughout the proof.

\medskip

\noindent The equivalence between (iii) and (iv) follows from Corollary
\ref{bap2 char}.

\medskip

\noindent The equivalence between (iii) and (ii) follows easily from
Proposition \ref{prop: amenable}.

\medskip

\noindent The equivalence between (i) and (ii) is just a uniform (in $s\in G$)
version of the characterization of $\bap^2 (G)$ in Corollary~\ref{bap2 char}. It can be shown by mimicking the proof of that
corollary.

\medskip

\noindent The last claim follows as  the mean $M(f)=c_{1}(f)$ exists uniformly
in translates by (ii).
\end{proof}

Let us next discuss the Fourier--Bohr expansion of an $f \in \wap^2_{\cA}(G) \cap L^\infty(G)$. By the Proposition~\ref{prop:wap indep} and Corollary~\ref{coro-wap-compl} , we can embed $\wap_{\cA}^2(G) \cap L^\infty(G)$ into $B\hspace*{-1pt}es_{\cB}^2(G)$ for all van Hove sequences $\cB$, and the norms $\| \cdot \|_{\we,2,\cA}$ and $\| \cdot \|_{\be,2,\cB}$ agree on $\wap_{\cA}^2(G)$. Theorem~\ref{thm hilbert} then leads to the following result.

\begin{theorem}\label{thm-pars-wap}  Let $\cA$ be a van Hove sequence.
\begin{itemize}
  \item[(a)] For all $f, g \in \wap_{\cA}^2(G)$
  \[
  \langle f,g \rangle = \lim_{n\to\infty} \frac{1}{|A_n|} \int_{t+A_n} f(s) \overline{g(s)}\, \dd s
  \]
  exists uniformly in $t$ and defines a semi-inner product on  $\wap_{\cA}^2(G)$.
  \item[(b)] For all van Hove sequences $\cB$ and all $f,g \in \wap_{\cA}^2(G) \cap L^\infty(G)$, we have
  \[
    \langle f,g \rangle = \big\langle [f], [g] \big\rangle_{\cB} \,,
  \]
  where $\langle [f], [g] \rangle_{\cB}$ is the semi-inner product from Theorem~\ref{thm hilbert}.
  \item[(c)] For all van Hove sequences $\cB$, $(B\hspace*{-1pt}es_{\cB}^2(G)/\equiv, \| \cdot \|_{\be,2,\cB})$ is the Hilbert space completion of $(\wap_{\cA}^2(G) \cap L^\infty(G), \langle \cdot , \cdot \rangle)$.
  \item[(d)] For all $f \in \wap^2_{\cA}(G)$, one has $a_{\chi}(f) \neq 0$ for at most a countable set  of characters, and we have the \textit{Parseval identity}\index{Parseval~identity!Parseval~identity~for~$\Wap^{2}(G)$}
\[
\|f \|_{\we,2,\cA} ^2  = \sum_{\chi\in \widehat{G}} \big|
a_{\chi}(f) \big|^2
\]
and
\[
f = \sum_{\chi
\in\widehat{G}} a_{\chi}^{\cA}(f) \chi \qquad \mbox{ in } (\wap^2_{\cA}(G), \|
\cdot \|_{\we,2,\cA})\,.
\]
\end{itemize}
\end{theorem}
\begin{proof}
(a) follows from Proposition~\ref{spi lemma} (a) and (f). (b) and (c) follow from Proposition~\ref{prop:wap indep} and Corollary~\ref{coro-wap-compl}. (d) follows from Corollary~\ref{coro-wap-compl} and Theorem~\ref{thm hilbert}.
\end{proof}

\begin{remark}
Let $\cA$ be a van Hove sequence. If $f \in \wap^2_{\cA}(G) \cap L^\infty(G)$, then for all van Hove sequences $\cB$, we have
\[
f = \sum_{\chi
\in\widehat{G}} a_{\chi}^{\cA}(f) \chi \qquad \mbox{ in } (\wap^2_{\cB}(G), \|
\cdot \|_{\we,2,\cB}) \,.    \tag*{$\Diamond$}
\]
\end{remark}

Weyl almost periodic functions satisfy a stronger version
of Theorem~\ref{eber conv Bap2 functions}.

\begin{prop}\label{eber conv Wap2 functions}
Let $\cA$ be a van Hove
sequence. Let $f,g\in\wap^2_{\cA}(G)$.
\begin{itemize}
  \item[(a)] The reflected
Eberlein convolution $\lb f , g \rb_{\cA} $ exists and belongs to
$\sap(G)$.
  \item[(b)]If $g \in L^\infty(G)$ then Eberlein convolution $f \circledast_{\cA} g $ exists and belongs to
$\sap(G)$.
\end{itemize}
\begin{remark}
In (b) the condition $g \in L^\infty(G)$ can be replaced by $\cA$ being symmetric, meaning
\[
A_{n} =-A_n 
\]
for all $n$. \exend
\end{remark}
\end{prop}
\begin{proof} (a) The space $\wap^2_{\cA} (G)$ is contained in $B\hspace*{-1pt}es^2_{\cA} (G)$.
Moreover,  for any $f\in \wap_{\cA}^2 (G)$ and $t\in G$, its
translate $\tau_t  f$ clearly belongs to $B\hspace*{-1pt}es^2_\cA (G)$ as well.
Hence, it is a representative of $T_t [f]$ by Proposition~\ref{prop:compatibility}.

\medskip

\noindent (b) Since $g$ is bounded, we have $g \in \wap^2_{-\cA}(G)$. This immediately gives $\tilde{g} \in \wap^2_{\cA}(G)$ and $[\tilde{g}]=\widetilde{[g]}$.
Hence, $f\circledast_{\cA} g$ exists and
belongs to $\sap(G)$ by Theorem~\ref{eber conv Bap2 functions}.
\end{proof}

\smallskip
We complete the section by discussing the connection between Weyl almost periodicity and an uniform version of
mean almost periodicity. The following lemma is immediate.

\begin{lemma}\label{lem-wap-implies-unif-map} Let $f \in \wap^p(G)$. Then, for each $\eps>0$, the set
\[
AP_{\we,p,\cA}(f,\eps):= \{ t \in G : \|\tau_tf-f \|_{\we,p,A} < \eps\}
\]
is relatively dense.
\end{lemma}
\begin{proof} Let $\eps>0$. Then, there exists a trigonometric polynomial $P$ such that
\[
\|f-P\|_{\we,p,\cA}<\frac{\eps}{3} \,.
\]
It follows immediately that
\[
\{ t \in G : \|\tau_t P-P \|_{\infty} < \eps\}= AP_{\infty}(P,\frac{\eps}{3}) \subseteq AP_{\we,p,\cA}(f,\eps) \,.
\]
Since $P$ is Bohr almost periodic, the claim follows.
\end{proof}

\medskip
Let us note that the function $f$ from Remark~\ref{rem:map vs bap} has the property that, for all $\eps>0$, the set $AP_{\we,p,\cA}(f,\eps)$
is relatively dense, but it is not Weyl almost periodic (since it is not Besicovitch almost periodic). This shows that the converse of Lemma~\ref{lem-wap-implies-unif-map} does not hold.

We briefly discuss this below: The conclusion of Lemma~\ref{lem-wap-implies-unif-map} is that, for each $\eps>0$, there exists a relatively dense set $P_\eps$ such that, for all $t \in P_\eps$, we have
\[
\|\tau_tf-f \|_{\we,p,A} < \eps \,.
\]
This means that, for each $t \in P_\eps$, there exists some $N_{\eps,t}$ such that
\[
\frac{1}{|A_n|} \int_{s+A_n} \left| f(s-t) - f(s) \right|^p \dd s < (2 \eps)^p
\]
holds for all $n>N_{\eps,t}$ and all $s \in G$. As the next result shows, Weyl almost periodicity is equivalent to the fact that $N_{\eps,t}$ can be chosen independently of $t$. As we do not use this in the sequel, we skip the proof and refer the reader to \cite{Spi}.

\begin{prop}\cite[Prop.~3.13]{Spi}
Let $f \in \mathcal{L}^p_{loc}(G)$. Then, $f \in \wap^p(G)$ if and only if, for each $\eps >0$, there exists a relatively dense set $P_\eps$ and some $N_\eps \in \NN$ such that, for all $t \in P_\eps$ and all $n >N_\eps$, we have
\[
\frac{1}{|A_n|} \int_{s+A_n} \left| f(s-t) - f(s) \right|^p \dd s < \eps  \,.        \tag*{$\qed$}
\]
\end{prop}

\section{Complete point diffraction}
The standard examples of Aperiodic Order do not only exhibit pure-point
diffraction and the existence of amplitudes (Fourier--Bohr coefficients) but rather a uniform
version of the existence of amplitudes and the consistent phase property. We will characterize the validity of these properties by
Weyl almost periodicity and then refer to this as complete point diffraction.

\medskip

In this subsection, we deal with translation bounded Weyl almost
periodic measures.
Let us start with the following obvious consequence of Proposition~\ref{prop:wap indep}.

\begin{prop}
Let $\mu \in \cM^\infty(G)$, and let $\cA$ be van Hove sequence. Then, $\mu \in \Wap_{\cA}(G)$ if and only if
$\mu \in \Wap_{\cB}(G)$ for all van Hove sequences $\cB$.    \qed
\end{prop}

Let us look now at the Fourier--Bohr coefficients of a Weyl
almost periodic measure. The following is an immediate
consequence of Lemma~\ref{spi lemma} and Corollary~\ref{FB measure
relations}.

\begin{lemma}\label{lem:fb}
Let  $\mu \in \Wap(G) \cap \cM^\infty(G)$. Then, for each $\chi \in
\widehat{G}$, the Fourier--Bohr coefficient
\[
a_\chi(\mu)=\lim_{n\to\infty} \frac{1}{|A_n|} \int_{s+A_n}
\overline{\chi(t)}\, \dd \mu(t)
\]
exists uniformly in $s\in G$ and does not depend on the choice of
the van Hove sequence. Moreover, for all $\varphi \in \Cc(G)$, we
have
\[
a_\chi(\mu*\varphi)=a_\chi(\mu)\, \widehat{\varphi}(\chi) \,.       \tag*{$\qed$}
\]
\end{lemma}

Now, we can characterize the space $\Wap(G) \cap \cM^\infty(G)$.

\begin{theorem}[Characterization of $\Wap_{\cA}(G)\cap \cM^\infty(G)$]
\label{thm: wap in bap} Let $\mu \in \cM^\infty(G)$, and let
$\cA$ be a van Hove sequence on $G$. Then, the following assertions are
equivalent: \index{almost~periodic!Weyl~$p$-almost~periodic~measure}
\item[(i)] The measure $\mu$ belongs to $\Wap_{\cA}  (G)$.
\item[(ii)] The measure $\mu$ belongs to $\Bap_{\cA} (G)$ and the following statements hold:
\begin{itemize}
  \item{}  For all $\varphi \in \Cc(G)$, the function $\left|\mu*\varphi\right|^2$ is amenable.
  \item{}  For each $\chi \in \widehat{G}$, the Fourier--Bohr coefficient
\begin{displaymath}
a_\chi(\mu)=\lim_{n\to\infty} \frac{1}{|A_n|} \int_{s+A_n}
\overline{\chi(t)}\, \dd \mu(t)
\end{displaymath}
exist uniformly in $s\in G$.
\end{itemize}
\item[(iii)] The measure $\mu$ belongs to $\Bap_{\cB} (G)$ for all van Hove sequences $\cB$.

Moreover, if one of the equivalent assertions holds, then any finite
product of functions from the set $\{ \mu*\varphi,
\overline{\mu*\varphi}: \varphi \in \Cc(G) \}$ is amenable.
\end{theorem}
\begin{proof} By Corollary~\ref{cor unif FB}, the existence of the Fourier--Bohr
coefficients for a translation bounded measure $\mu$ is equivalent
to the existence of the Fourier--Bohr coefficients for $\mu\ast
\varphi$ for all $\varphi \in \Cc (G)$. Clearly, $\mu \ast \varphi$ is
bounded belongs to $\Cu (G)$ for every $\varphi \in \Cc
(G)$. Now, the characterization follows easily from Proposition
\ref{prop-char-wap}.

The last claim is immediate. Indeed, for each $\varphi \in \Cc(G)$,
the bounded functions $\mu*\varphi$ and $\overline{\mu*\varphi}$ belong
to $\wap_{\cA}^2(G) \subseteq \wap_{\cA}(G)$. Therefore, by
Proposition~\ref{spi lemma}(c), any product of such functions
belongs to $\wap_{\cA}(G)$, i.e. it is amenable.
\end{proof}

Recall from Proposition \ref{prop: amenable} that the
existence of the mean, for each van Hove sequence, actually implies the
independence of the mean of the van Hove sequence. For this reason,
we do not state the independence of the van Hove sequence in the
condition below.

\begin{theorem}\label{theorem-uniform-phase-problem} Let $\mu \in
\cM^\infty(G)$. Then, $\mu \in \Wap(G)$ if and only if the
following three conditions hold:
\begin{itemize}
\item[(a)] The autocorrelation $\gamma$ of $\mu$ exists for each van Hove sequence $\cB$, and $\widehat{\gamma}$ is a pure-point measure.
\item[(b)] The Fourier--Bohr coefficients $a^\cB_{\chi}(\mu)$
exist for all $\chi \in \widehat{G}$ and van Hove sequences $\cB$.
\item[(c)] For all $\chi \in \widehat{G}$, we have the consistent phase property
\[
\widehat{\gamma}(\{\chi\}) = \left| a_{\chi}(\mu) \right|^2 \,.
\]
\end{itemize}
\end{theorem}
\begin{proof}  This is a direct consequence of the characterization of Weyl almost periodic measures in Theorem~\ref{thm: wap in bap} and Theorem~\ref{Bap and BT}.
\end{proof}

The theorem provides an answer to Question 3 from the introduction. It suggests to single out the measures satisfying the conditions (a), (b), (c).

\begin{definition}[Complete point spectrum] A measure $\mu$ is said to have \emph{complete point spectrum} if the 
following three conditions hold:
\begin{itemize}
\item[(a)] The autocorrelation $\gamma$ of $\mu$ exists for each van Hove sequence $\cB$, and $\widehat{\gamma}$ is a pure-point measure.
\item[(b)] The Fourier--Bohr coefficients $a^\cB_{\chi}(\mu)$
exist for all $\chi \in \widehat{G}$ and van Hove sequences $\cB$.
\item[(c)] For all $\chi \in \widehat{G}$, we have the consistent phase property
\[
\widehat{\gamma}(\{\chi\}) = \left| a_{\chi}(\mu) \right|^2 \,.    \tag*{$\Diamond$}
\]
\end{itemize}
\end{definition}

With this definition the main result above says that a translation bounded measure has complete point spectrum if and only if it belongs to $\Wap(G)$.

\medskip

We complete the section by briefly discussing the Fourier expansion of a Weyl almost periodic measure. The result is an immediate consequence of  Theorem~\ref{thm-expa} and Proposition~\ref{prop-char-wap}.

\begin{theorem}[Expansion of an Weyl almost periodic measure]
Let $\cA$ be a van Hove sequence.
For any translation bounded measure $\mu$, the following assertions are equivalent:
\begin{itemize}
\item[(i)] The measure $\mu$ is Weyl almost periodic.
\item [(ii)]  There exists a map $a : \widehat{G}\longrightarrow \CC $ such that, for all $\varphi \in \Cc (G)$, we have
\[
\lim_F  \Big\|\mu\ast \varphi - \Big(\sum_{\chi \in F} a_\chi \chi\Big)\ast \varphi \Big\|_{\we,2,\cA} = 0 \,.
\]
\end{itemize}
If the equivalent conditions (i) and (ii) hold, one has
\[
a_\chi = a_\chi (\mu) \qquad \text{ for all } \chi \in\widehat{G} \,.        \tag*{$\qed$}
\]
\end{theorem}

\section{Meyer almost periodic functions and measures}
Next, let us look at a generalization of almost periodicity which
was introduced by Yves Meyer \cite{Mey2}. We will show that elements of
this large class of measures are Weyl almost periodic.

\smallskip

Recall first that
all Bohr almost periodic functions are amenable.

\begin{definition}[Generalized almost periodicity] \cite[Def.~2.1]{Mey2}.
A function $f : G \to \RR$ is called \textit{generalized almost periodic (g-a-p)}\index{almost~periodic!generalized~almost~periodic~function} if it is measurable and, for each $\eps>0$, there exist Bohr almost periodic functions $g$ and $h$ satisfying 
\begin{align*}
g \leq f \,&\leq\, h \,, \\
M(h-g) \,&<\, \eps \,.
\end{align*}
A complex valued
function $f : G\longrightarrow \CC$ is called \textit{generalized
almost periodic} if both its real and its imaginary part are
generalized almost periodic.

A real valued Borel measure $\mu$ on
$G$ is called \textit{generalized almost periodic (g-a-p) measure}\index{almost~periodic!generalized~almost~periodic~measure}
if, for each $\eps >0$, there exist strongly almost periodic
measures $\nu$ and $\omega$ satisfying
\begin{align*} 
\nu \,\leq\, \mu \,& \leq\, \omega \,, \\
M(\omega-\nu) \,&<\, \eps \,.
\end{align*}
A measure $\mu$ on $G$ is called
\textit{weakly g-a-p measure}\index{almost~periodic!weakly~generalized~almost~periodic~measure} if, for all $\varphi \in \Cc(G)$, the
function $\mu*\varphi$ is g-a-p.    \exend
\end{definition}

\begin{remark}\phantom{X}
\begin{itemize}
\item[(a)]  Note that any g-a-p function must be bounded, since Bohr
almost periodic functions are bounded and the g-a-p function is
bounded by Bohr almost periodic functions from above and below.
\item[(b)] The article of Meyer deals with $G= \RR^n$ and, accordingly,
gives the definition for $\RR^n$.
\item[(c)] Since each $\varphi \in \Cc(G)$ is a linear combination of
non-negative functions in $\Cc(G)$ and  convolution with
non-negative functions preserves inequalities, every g-a-p measure
is weakly g-a-p. The converse is not true \cite[Lem.~2.45]{Mey2}.  
\item[(d)] g-a-p functions have also been studied by Hartman \cite{Har1,Har2} under the name "R-almost periodic functions". \exend
\end{itemize}
\end{remark}

We note the following immediate consequence of the definition.

\begin{prop}[g-a-p implies Weyl almost periodicity]\label{g-a-p-is weyl}\phantom{X}
\begin{itemize}
\item [(a)] If $f \in \mathcal{L}^1_{loc}(G)$ is a g-a-p function, then $f
\in \wap(G)$.
\item [(b)] If $\mu\in \cM^\infty(G)$ is a weakly g-a-p measure, then
$\mu \in \Wap(G)$. In particular, any g-a-p measure is Weyl almost periodic.
\end{itemize}
\end{prop}
\begin{proof} It suffices to show (a) for real valued functions $f$. Let
$\varepsilon >0$, and let $h,g\in \sap(G)$ with
\begin{align*}
g \leq f \,&\leq\, h \,, \\
M(h-g) \,&<\, \eps \,,
\end{align*} be given. Now, clearly 
\begin{align*}
|f - g| \, &\leq \,  (h-g) \qquad \mbox{ and } \\
\|f - g\|_{\we,1} \, &\leq \, M(h-g) \leq \varepsilon
\end{align*}
follows. As
$\varepsilon>0$ was arbitrary, the desired statement holds.

\medskip

\noindent (b) follows from (a).
\end{proof}

\begin{remark}\phantom{X}
\begin{itemize}
    \item[(a)] If $f$ is g-a-p function, then $f \in \wap(G) \cap L^\infty(G)$ and hence $f \in \wap^p(G)$, for all $p \geq 1$.
    \item[(b)] Meyer \cite[Sect.~2]{Mey2} (compare \cite[Thm.~5.6.]{LRS}) constructed an example of a subset $A \subseteq \ZZ$ such that, $1_{A}$ is not g-a-p but 
    \[
    \| 1_{A} \|_{\we,1,\cA} \, = \, 0 \,.
    \]
    In particular, $1_{A}$ is Weyl almost periodic.  \exend
\end{itemize}
\end{remark}

A main merit of generalized almost periodicity is that the class of
regular model sets can be seen to have  this property. So,
generalized almost periodicity covers the arguably most important
examples of Aperiodic Order. This was already shown by Meyer in the
Euclidean setting, compare \cite[Thm.~3.3]{Mey2}. Our
more general situation can be treated similarly.
For the definition
of regular model sets and notation used below, see Appendix
\ref{appendix:cp}.

First, we will show that regular model sets satisfy a strong type of approximation by almost periodic measures, which was observed by M. Baake, which implies g-a-p.

\begin{lemma}\label{model is gwap}
\phantom{XX}
\begin{itemize}
   \item[(a)] Let $\vL$ be a regular model set on $G$. Then, for each $\eps >0$, there exist a locally finite set $\vG$ and measures $\mu, \nu \in \SAP(G)$ supported inside $\vG$ such that $0 \leq \mu \leq \delta_{\vG} \leq \nu \leq \delta_{\vG}$ and
   \[
   \limsup_{n\to\infty} \frac{1}{|A_n|} \card( \{ x \in \vG\, :\, \mu(\{x\}) \neq \nu(\{x\})\} \cap A_n ) < \eps \,.
   \]
   \item[(b)] Let $\omega$ be a measure such that, for each $\eps >0$, there exist a locally finite set $\vG$ and measures $\mu, \nu \in \SAP(G)$ supported inside $\vG$ such that $0 \leq \mu \leq \delta_{\vG} \leq \nu \leq \delta_{\vG}$ and
   \[
   \limsup_{n\to\infty} \frac{1}{|A_n|} \card( \{ x \in \vG\, :\, \mu(\{x\}) \neq \nu(\{x\})\} \cap A_n ) < \eps \,.
   \]
   Then, $\omega$ is a g-a-p measure.
 \end{itemize}
In particular, if $\vL$ is a regular model set on $G$, then $\delta_\vL$ is
a g-a-p measure.
\end{lemma}

\begin{remark}
Recently, in \cite{LRS}, it was shown that a Delone set $\vL$ is a regular model set if and only if $\vL$ is a Meyer set and $\delta_{\vL}$ is g-a-p.     \exend
 \end{remark}

\begin{proof}
(a) Let $(G,H, \cL)$ be a CPS, and let $W$ the regular window which
produces the regular model set. Let $\eps >0$. By the regularity of $\theta_H$ and \cite[Thm.~2.7]{rud2}, there exist an open set $U$ and compact set $K$ such that
\[
U \subseteq \overline{U} \subseteq W^\circ \subseteq W \subseteq  \overline{W} \subseteq K^\circ \subseteq K
\]
and $\theta_{H}(K \backslash U) < \frac{\eps}{\dens(\cL)}$. Choose $h,g \in \Cc(H)$ such that
\[
1_{\overline{U}} \leq h \leq 1_W \leq g \leq 1_{K^\circ}\,.
\]
Then,
\[
\mu:=\omega_h \leq \delta_\vL\leq \omega_g=:\nu
\]
and $\vG := \oplam(K)$ satisfy
\[
0 \leq \mu \leq \delta_{\vG} \leq \nu \leq \delta_{\vG}\,.
\]
Moreover, we have
\[
\{ x \in \vG\, :\, \mu(\{x\}) \neq \nu(\{x\})\} \subseteq \oplam(K \backslash U) \,.
\]
Since $K \backslash U$ is a compact set, we get
\[
\limsup_{n\to\infty} \frac{1}{|A_n|} \card( \{ x \in \vG : \mu(\{x\}) \neq \nu(\{x\})\} \cap A_n ) \leq \dens(\cL) \theta_{H}(K \backslash U) <\eps \,.
\]

\smallskip

\noindent (b) The given conditions show that, for all $x \in \vG$, we have $0 \leq (\omega-\nu)(\{x \}) \leq 1$. It follows immediately that
\[
M(\omega-\nu) \leq \limsup_{n\to\infty} \frac{1}{|A_n|} \card( \{ x \in \vG\, :\, \mu(\{x\}) \neq \nu(\{x\})\} \cap A_n )  \,.
\]
\end{proof}

It is possible to characterize  g-a-p functions via the Bohr
compactification. This is already hinted at in Meyer's work, but
details are not given. A weaker version of this result can be found in \cite[Page~68]{Har2}. We include the proof for completeness.

\begin{lemma}[Characterization of g-a-p]\label{lem:char-gap}
Let $f : G\longrightarrow \RR$ be a bounded and
measurable real valued function. Then, the following assertions are equivalent:
\begin{itemize}
\item[(i)] The function $f$ is g-a-p.
\item[(ii)] There exist two Riemann integrable functions $k_{<}',  k_{>}'$
on $G_\mathsf{b}$
with $k_{<}' \leq k_{>}'$ and
\[
\int_{G_\mathsf{b}}  \left(k_{>}' (x)- k_{<}'(x)\right)(x)\, \dd x =0
\]
such that
\[
g'\circ i_\mathsf{b} \leq f \leq h'\circ i_\mathsf{b}
\]
holds for all
$g',h'\in C(G_b)$ with $g'\leq k_{<}'$ and $k_{>}'  \leq h'$.
\end{itemize}
If the equivalent conditions (i) and (ii) hold, then $f$ belongs to
$B\hspace*{-1pt}es_{\cA} (G)$ for every van Hove sequence $\cA$, and
\[
(f)_{\mathsf{b},1} = [k_{<}']_1=[k_{>}']_1
\]
holds.
\end{lemma}
\begin{proof}
(i)$\Longrightarrow$(ii): By (i), there exist $g_n, h_n \in \sap
(G)$, $n\in \NN$, with $g_n \leq f \leq h_n$ and
\[
\lim_{n\to\infty} M(h_n - g_n) = 0 \,.
\]
By eventually replacing $g_n$ and $h_n$ by
\begin{align*}
G_n \, &:= \, \max\{ g_1,g_2, \ldots ,g_n \} \\
  H_n \, &:= \, \min \{ h_1, h_2, \ldots , h_n \} \,,
\end{align*}
we can assume without loss of generality
that $g_n \leq g_{n+1}$ and $h_{n+1} \leq h_n$ for all $n\in \NN$.
By \cite[Thm.~4.3.5]{MoSt}, there exist unique $g_n',h_n'\in
C(G_\mathsf{b})$ with $h_n = h_n'\circ i_\mathsf{b}$ and $g_n =
g_n'\circ i_\mathsf{b}$ for each $n\in\NN$.
Then, since $i_\mathsf{b}(G)$ is dense in $G_\mathsf{b}$, we have
\[
g_1' \leq g_2' \leq \ldots \leq g_n' \leq  \ldots \leq h_n' \leq \ldots \leq h_1' \leq h_1' \,.
\]
Define
\begin{align*}
 k_{>}'(x) \, &:=\, \lim_{n\to\infty} h_n'(x) \,, \\
k_{<}'(x) \,&:=\, \lim_{n\to\infty} g_n'(x) \,.
\end{align*}
Then, $k_{<}'$ and $k_{>}'$ are Riemann integrable, and by construction, we have
\[
\int_{G_\mathsf{b}} (k_{>}' - k_{<}')(x)\, \dd x = 0 \,.
\]
Now, let $h'\in C(G_b)$ with $k_{>}' \leq h'$ be given. Then, for
each $\varepsilon>0$ and each $s\in G$, there exists some $n$ such that
$$f(s) \leq  h_n(s) = h_n' \circ i_\mathsf{b} (s) \leq k_{>}'\circ i_\mathsf{b} (s)
+\varepsilon \leq h' \circ i_\mathsf{b} (s)  + \varepsilon \,.
$$
This shows $f \leq h'\circ i_\mathsf{b}$.
The statement for $g'\in C(G_\mathsf{b})$ with $g'\leq k_{<}'$ can be shown
similarly. This shows (ii).

\medskip

\noindent (ii)$\Longrightarrow$(i): Let $\varepsilon >0$ be given. By
definition of Riemann integrability, we can find $h',g'\in C(G_\mathsf{b})$
with $h'\geq k_{>}'$ and $ k_{<}' \geq g'$ and $\int_{G_\mathsf{b}}  (h' -
g')(x)\, \dd x <\varepsilon$. Then, $g'\circ i_\mathsf{b}, h'\circ i_\mathsf{b}$ belong to $\sap
(G)$. By (ii), they satisfy
\[
g'\circ i_\mathsf{b} \leq f \leq h'\circ
i_\mathsf{b}\,.
\]
Moreover, we have
\[
M(h'\circ i_\mathsf{b} - g'\circ i_\mathsf{b}) = \int_{G_\mathsf{b}} (h' - g')(x)\,
\dd x < \varepsilon \,.
\]
This proves (i).

\medskip

\noindent We already showed in Lemma~\ref{g-a-p-is weyl} that any g-a-p
function belongs to $\wap(G)$ and, hence, to $B\hspace*{-1pt}es_{\cA} (G)$. The
last part of the statement follows from the construction in the
proof of the equivalence.
\end{proof}

As noted in Remark~\ref{rem:Weyl-ap} (c), any
perturbation of a Weyl almost periodic function by a function with
vanishing uniform absolute mean is also Weyl almost periodic. This
suggests to consider the class of functions which are  g-a-p up to
such a perturbation. This natural class will contain all g-a-p
functions. At the same time, it will also contain all weakly almost
periodic functions. So, the induced class of measures
will contain all regular model sets as well as all weakly almost
periodic measures.  Thus, it contains all  `smooth' examples studied
so far in the context of Aperiodic Order, see
\cite{bm,LS2,Martin2}. We will refer to this type of almost
periodicity as Meyer almost periodicity. A precise definition is
provided next.

\begin{definition}[Meyer almost periodicity]
A function $f : G \to \CC$ is called \textit{Meyer
almost periodic}\index{almost~periodic!Meyer~almost~periodic~function} if
$$f = g + h$$
with a g-a-p function $g$ and a bounded function $h$ with
$\|h\|_{\we,1,\cA}=0$.

A measure $\mu \in \cM^\infty(G)$ is called  \textit{Meyer almost
periodic}\index{almost~periodic!Meyer~almost~periodic~measure} if, for all $\varphi \in \Cc(G)$, the function
$\mu*\varphi$ is a Meyer almost periodic function.   \exend
\end{definition}

Since $h$ is bounded, the condition $\|h\|_{\we,1,\cA}=0$  can easily be seen to be
equivalent to $|h|$ being amenable with $M(|h|) =0$.  In particular,
Meyer almost periodicity is independent of the choice of van Hove
sequence, compare Proposition \ref{prop: amenable}.

\begin{remark} \phantom{X}
\begin{itemize}
   \item[(a)] If $f$ is Meyer almost periodic and $h$ is bounded with $\|h\|_{\we,1,\cA}=0$, then $f+h$ is Meyer almost periodic.
  \item [(b)] It is obvious that g-a-p functions are Meyer almost periodic functions. Similarly, g-a-p measures and weakly g-a-p measures are Meyer almost periodic measures.
  \item[(c)] A weakly almost periodic function or measure, respectively, is a Meyer almost periodic function or Meyer almost periodic measure, respectively.
  \item[(d)] A function $f$ is Meyer almost periodic if and only if $\re(f)$ and $\imm(f)$ are Meyer
almost periodic functions.
  \item[(e)] It is easy to see that a linear combination of Meyer almost periodic functions or measures, respectively, is a Meyer almost periodic function or measure, respectively.    \exend
\end{itemize}
\end{remark}

From Proposition~\ref{g-a-p-is weyl}, we
immediately obtain the following result.

\begin{coro}[Meyer almost periodicity entails Weyl almost
periodicity]\label{cor12345} \phantom{X}  \vspace*{-.5cm}
\begin{itemize}
\item [(a)] If $f \in \mathcal{L}^1_{loc}(G)$ is a Meyer almost periodic function, then $f
\in \wap(G)$.
\item [(b)] If $\mu\in \cM^\infty(G)$ is a Meyer almost periodic measure, then
$\mu \in \Wap(G)$.\qed
\end{itemize}
\end{coro}

\begin{remark} Note that most of the results in \cite[Sect.~1 and 2]{Mey2}
follow from Corollary~\ref{cor12345} and the
properties of Weyl almost periodic functions discussed above.   \exend
\end{remark}

As a consequence of the preceding results, we can show that $\Wap
(G)$ contains an ample supply of examples.

\begin{coro} \label{cor21}
The set $\Wap (G)$ contains all Dirac
combs of regular model sets as well as all weakly almost periodic
measures.\qed
\end{coro}

It is well known that weakly almost periodic measures and model sets
are uniquely ergodic, have pure-point spectrum and continuously representable
eigenfunctions \cite{LS2,Martin2}. When combining
Corollary~\ref{cor21} with Theorem~\ref{thm:wap} below, we obtain an alternative
proof for this statement.

\smallskip

We complete the section with a characterization of  Meyer almost
periodicity.

\begin{theorem}[Characterization of Meyer almost periodicity]
Let $f \in \mathcal{L}^1_{loc}(G)$. Then, the following assertions are equivalent:
\begin{itemize}
  \item[(i)]  The function $f$ is Meyer almost periodic.
  \item[(ii)] $f \in \wap(G)\cap L^\infty(G)$ and the class of $(f)_{\mathsf{b},1} \in
\mathcal{L}^1(G_{\mathsf{b}})$ contains a Riemann integrable function.
\end{itemize}
\end{theorem}
\begin{proof} (i)$\Longrightarrow$(ii): Corollary~\ref{cor21}
implies that any Meyer almost periodic function belongs to $\wap(G)$. Moreover, if $f = g  +h$ with $g$
being g-a-p and $h$ having uniform mean zero, then
$(f)_{\mathsf{b,1}} = (g)_{\mathsf{b,1}}$ contains an Riemann
integrable representative by Lemma  \ref{lem:char-gap}. Finally, since $g,h$ are bounded, so is $f$.

\medskip

\noindent (ii)$\Longrightarrow$(i): It suffices to prove the
claim for real valued functions.
Fix a van Hove sequence $\cA$, and let $k \in \mathcal{L}^1(G_{\mathsf{b}})$ be Riemann integrable such that
$(f)_{\mathsf{b},1, \cA}=[k]_1$ in $\mathcal{L}^1(G_{\mathsf{b}})$.
Since $f \in \wap(G) \cap L^\infty(G)$, we have
\[
(f)_{\mathsf{b},1, \cB}=[k]_1
\] for any other van Hove sequence $\cB$.
By Riemann integrability, we can then find an increasing sequence $(g_n')$ in $\Cc(G_{\mathsf{b}})$ and a decreasing sequence $(h_n')$ in $\Cc(G_{\mathsf{b}})$ satisfying
\[
g_n' \leq k \,\leq\,  h_n' 
\]
and
\[
\int_{G_{\mathsf{b}}}   (h_n'-  g_n')\, \dd
\theta_{G_{\mathsf{b}}}\, \leq \, \frac{1}{n} \,.
\]
Let $g_n =g_n' \circ i_{\mathsf{b}}, h_n =h_n' \circ i_{\mathsf{b}}$. Then, $(g_n)$ is an increasing sequence, $(h_n)$ is a decreasing sequence, $g_n,h_n \in \sap(G)$ and
\[
M(h_n-g_n) = \int_{G_{\mathsf{b}}}  (h_n'-  g_n')\, \dd
\theta_{G_{\mathsf{b}}} \leq \frac{1}{n} \,.
\]
Define $g(x)= \sup \{ g_n(x): n\in\NN \}$. Then, we have
\[
g_1\leq   g_2 \leq \ldots \leq   g_n \leq \ldots\leq g   \leq \ldots  h_n
\leq \ldots \leq  h_2 \leq h_1
\]
from where we get that $g$ is a g-a-p function. In particular, $g \in \wap(G)$.

Let $h=f-g$.
Note first that 
\[
|h|=|f-g| \leq |f|+|g| \leq |f|+|h_1| \,,
\]
and thus $h$ is bounded. Moreover, since $g_n \to g$ in $\wap(G)$, recalling that
$g_n \in \sap(G)$, for all van Hove sequences $\cB$, we have
$(g_n)_{\mathsf{b}} \to [g]_{\mathsf{b},1,\cB}$ in
$\mathcal{L}^1(G_{\mathsf{b}})$ and hence
\[[g]_{\mathsf{b},1,\cB}= [k]_1 =
[f]_{\mathsf{b},1,\cB} \]
in $\mathcal{L}^1(G_{\mathsf{b}})$. Therefore,
$[h]_{\mathsf{b},1,\cB}= 0$ and thus
\[\| h \|_{\mathsf{b}, 1,\cB}=0
\]
holds for all van Hove sequences $\cB$. Therefore, by Proposition~\ref{prop:
amenable} we get $\| h \|_{\we, 1}=0$. This completes the
proof.
\end{proof}

\section[Unique mean on Weyl almost periodic functions] {Unique mean on the class of uniformly continuous Weyl almost periodic functions.}

In this section, we briefly discuss the class
\[
\wap^p(G) \cap \Cu(G) \, = \, \wap(G) \cap \Cu(G) \,.
\]
This is a vector space which, as we discussed above, does not depend on the choice of $p$ or of the van Hove sequence $\cA$. We show that there exists a unique mean in this class, extending this well-known result from $\weak(G)$ to $\wap^p(G) \cap \Cu(G)$.
Moreover, we will see that there exists an alternate way of defining this mean without any reference to F\o lner sequence, and that the elements in this class can characterized by using this mean. This gives an alternative way to understand the independence of Weyl almost periodicity of the choice of  F\o lner sequence.

Let us briefly review the following results; for proofs, generalizations and a more details we refer the reader to \cite{MoSt}.
For $f \in \Cu(G)$ we denote by $C_f$ the closed convex hull of $\{ \tau_t f : t \in G\}$ in $(\Cu(G), \| \, \|_\infty)$.
Let us now recall the following result of Eberlein \cite{Eb}.

\begin{lemma}\cite[Prop.~4.5.6.]{MoSt}\label{prop-amen}
\begin{itemize}
  \item[(a)] For each $f \in \Cu(G)$, the set $C_f$ contains at most a constant function.
  \item[(b)] For each $f \in \Cu(G)$, the following are equivalent:
  \begin{itemize}
    \item[(i)] $C_f$ contains a constant function.
    \item[(ii)] The mean $M_{\cA}(f)$ exists uniformly for one F\o lner sequence.
     \item[(iii)] The mean $M_{\cB}(f)$ exists uniformly for all F\o lner sequences, and is the same.
  \end{itemize}
  Moreover, in this case, the unique constant function in $C_f$ is $M_{\cA}(f)$.\qed
\end{itemize}
\end{lemma}

With this in mind, we define  \nomenclature{$\amen$}{subset of $\Cu(G)$ consisting of amenable functions}
\[
\amen:= \{ f  \in \Cu(G)  : C_f \mbox{ contains a constant function} \} \,.
\]
Since condition (ii) (or (iii)) in Proposition~\ref{prop-amen} is linear, it follows that $\amen$ is a subspace of $(\Cu(G), \| \cdot \|_\infty)$, which is closed
by Proposition~\ref{prop-aprox-FB} (ii) applied to the trivial character.

\smallskip

Since, for $f \in \amen$, the closed convex hull $C_f$ contains exactly one constant function, we can define a mapping
\[
m: \amen \to \CC
\]
by setting $m(f)$ to be the only complex number such that $m(f) 1_G \in C_f$. Using again condition (ii) (or (iii)) in Proposition~\ref{prop-amen}, we see that $m$ is linear.

\medskip

Recall next that, given a $G$-invariant subspace $X$ of $L^\infty(G)$ such that $1_G \in X$, a \textit{mean}\index{mean} on $X$ is a mapping $ L : X \to \CC$
which is linear, positive, $G$-invariant and satisfies $L(1_G)=1$.

\smallskip

It is easy to see that the mapping $m$ we defined above is a mean. In fact, the following follows immediately from the definition, compare \cite[Prop.~4.5.9]{MoSt}.

\begin{prop}\label{uni-mea}
\phantom{XX}
\begin{itemize}
  \item[(a)] The mapping $m$ defined above is a mean on $\amen$.
  \item[(b)] Let $X \subseteq \amen$ be a closed $G$-invariant subspace such that $1_G \in X$ and $X$ is closed under complex conjugation. If $L$ is a mean on $X$, one has
\[
L(f) =m(f) \qquad  \text{ for all } f \in X \,.
\]
\end{itemize}
\end{prop}
\begin{proof}
(a) follows immediately from condition (ii) in Proposition~\ref{prop-amen}(b).

\medskip

\noindent(b) If $f \in  X$ is real valued, we have
\[
-\| f \|_\infty 1_G \leq f \leq \| f\|_\infty 1_G\,.
 \]
The positivity of $L$ gives $-\| f\|_\infty \leq L(f) \leq \|f\|_\infty$, i.e.
\begin{equation}
|L(f)| \leq \| f\|_\infty \,. \label{eq22xz}
\end{equation}
Next, let $f \in X$ be arbitrary. Since $X$ is closed under complex conjugation, we get $\mbox{Re}(f)=\frac{1}{2}(f+\bar{f}) \in X $ and $\mbox{Im}(f)=\frac{1}{2i}(f-\bar{f})\in X$. Then,
\[
  |L(\mbox{Re}(f))| \leq \| \mbox{Re}(f)\|_{\infty} \leq  \| f\|_\infty
\]
and
\[
  |L(\mbox{Im}(f))| \leq \| \mbox{Im}(f)\|_{\infty} \leq  \| f\|_\infty
\]
follow from Eq.~\eqref{eq22xz}. This gives
\[
|L(f)| \leq 2 \|f \|_\infty \,,
\]
which implies the continuity of $L$.

Finally, let $f \in X$ be arbitrary. By linearity, $L$ is constant on the convex hull $\mbox{Conv}\big( \{ \tau_t f :t \in G \}\big)$. Since $L$ is continuous, it is constant on
$C_f$. Using the fact that $m(f)1_G , f \in C_f$, we obtain
\[
L(f)=L(m(f)1_G)=m(f) \,.
\]
\end{proof}

Keeping in mind that the definition of $m$ does not depend on the choice of the F\o lner sequence, the following results are immediate.

\begin{lemma} \label{lem-mean-wap}
\phantom{XX}
\begin{itemize}
  \item[(a)] $\wap(G) \cap \Cu(G)$ is an algebra which is closed under translation, absolute value and multiplication by characters.
  \item[(b)] $\wap(G) \cap \Cu(G) \subseteq \amen$ and for all $f \in \wap(G) \cap \Cu(G)$, we have
  \[
  \| f\|_{\we,\cA} = m( |f|) \,.
  \]
  \item[(c)] The space $\wap(G) \cap \Cu(G)$ is closed in $(\Cu(G),  \| \cdot \||_\infty)$.
  \item[(d)] The mean $m$ is the only mean on $\wap(G) \cap \Cu(G)$.
\end{itemize}
\end{lemma}
\begin{proof}
(a) and (b) follow from  Proposition~\ref{spi lemma}.

\medskip

\noindent (c) Let $f_n \in \wap(G) \cap \Cu(G)$, and let $f \in \Cu(G)$ be such that $\| f_n -f \|_\infty \to 0$. Then, one has
\[
\| f_n -f \|_{\we,\cA} \leq \| f_n -f \|_\infty \to 0 \,.
\]
Since $f_n \in \wap(G)$, we get $f \in \wap(G)$.

\medskip

\noindent (d) follows from Proposition~\ref{uni-mea}.
\end{proof}

Finally, the definition, Lemma~\ref{lem-mean-wap} and Proposition~\ref{prop-char-wap} give the next result.

\begin{prop}
Let $f \in \Cu(G)$. Then, the following assertions are equivalent:
  \item[(i)] One has $f \in \wap(G)$.
  \item[(ii)] For each $\eps >0$, there is some $g \in \sap(G)$ such that $|f-g| \in \amen$ and
  \[
  m\big(|f-g|\big) < \eps \,.
  \]
  \item[(iii)] One has,
  \begin{itemize}
    \item{} $|f|^2 \in \amen$,
    \item{} $\chi f \in \amen$ for all $\chi \in \widehat{G}$,
    \item{} 
    \[
    m\big(|f|^2\big)= \sum_{\chi \in \widehat{G}} \big| m(\chi f) \big|^2 \,.      \tag*{$\qed$}
    \]
  \end{itemize}
\end{prop}

\chapter{Unavoidability of Besicovitch and Weyl almost periodicity}\label{sec:unavoidable}

In the preceding chapters, we have discussed how Besicovitch and
Weyl almost periodic functions allow us to solve Questions 2 and 3. Now, we discuss how,
under a mild additional regularity condition, these are actually
the only solutions.

\smallskip

In the article \cite{LAG}, Lagarias outlines some conditions that a
vector space $C$  of almost periodic functions should satisfy in
order to give a good theory. These conditions include the following:

\begin{itemize}
\item Expansion in a Fourier series, i.e. each  $f\in C$ has a formal
Fourier series $f \sim \sum c_\chi \chi$.
\item The Riesz--Fischer property  holds, i.e. for each square summable $(c_\chi)$ there is an element $f\in C$ with $f\sim\sum c_\chi \chi$.
\item The Parseval equality holds, i.e.  $\|f\|^2 = \sum |c_\chi|^2$ for all $f\in C$.
\end{itemize}

While it is not explicitly stated, two further requirements seem to
be natural. First, the Fourier expansion is linear. Second, with the
choice that the coefficients $c_\chi$ vanish for all but one $\chi$,
one obtains that the space $C$ contains the characters. Now, the
basic idea is that the measures (or distributions) $\mu$, with $\mu
\ast \varphi \in C$ for all $\varphi\in \Cc (G)$, have the desired
diffraction properties. This suggests to add another assumption, namely
that the elements of $C$ themselves also have the desired
diffraction properties. Making this additional assumption
(assumption (d) in Lemma~\ref{lem:appearance}), one ends up with
Besicovitch almost periodic functions.

\begin{lemma}[Appearance of  $B\hspace*{-1pt}es^2_{\cA} (G)$]  \label{lem:appearance}
Let $C$ be a subspace of $L^2_{loc}(G)$ with a seminorm $\| \cdot \|$,
and let $\cA$ be a van Hove sequence with the following properties:
\begin{itemize}
\item[(a)] For all $f \in C$, there exist $c_{\chi} \in \CC$
and a formal expansion $  f \sim \sum_{\chi} c_{\chi} \chi$ such that
\[
\| f \|^2 = \sum_{\chi\in\widehat{G}} |c_\chi |^2 \,.
\]
\item[(b)] We have $\widehat{G} \subseteq C$, and each $\chi \in \widehat{G}$ has a formal expansion $\chi \sim \chi$.
\item[(c)] The formal Fourier expansion in (a) is linear.
\item[(d)] For all $f \in C$, the reflected Eberlein convolution
$\lb f,f \rb_\cA$ exists with respect to $\cA$, is
continuous, and the measure $\gamma_f:=\lb f,f \rb_\cA \theta_G$ satisfies
\[
\reallywidehat{\gamma_f}=\sum_{\chi\in\widehat{G}} |c_\chi|^2\, \delta_{\chi} \,.
\]
\end{itemize}
Then, $C \subseteq \bap^2(G)$ and, for each $f\in C$, we have
\[
\| f\| \, = \, \| f \|_{\be, 2,\cA} 
\]
and 
\[
c_\chi  \,=\, a_\chi^{\cA}(f)   \,,
\]
for all $\chi \in \widehat{G}$.

Moreover, if the Riesz--Fischer condition holds in $C$, then $C=\Bap_{\cA}^2(G)$.
\end{lemma}
\begin{proof}
Let $f \in C$, and let $\{\chi_n: n\in\NN\}$ be an enumeration such that 
\[
f \sim \sum_{n=1}^{\infty} c_{\chi_n} \chi_n \,.
\] 
Note that this is possible, since $\{ \chi :c_\chi \neq 0\}$ is at
most countable, and, if finite, we can add some $\chi_n$ such that
$c_{\chi_n}=0$.
For each $N\in\NN$, set
\[
P_N:= \sum_{n=1}^{N} c_{\chi_n} \chi_n \,.
\]
Let $g_N=f-P_N \in C$. Then, by (c), $g_N$ has the formal Fourier
series
\[
g_N \sim \sum_{n=N+1}^{\infty} c_{\chi_n} \chi_n 
\]
and 
\[
\|g_N\|^2  = \sum_{n=N+1}^{\infty} |c_{\chi_n}|^2 \,.
\]
We also know that
\[
\reallywidehat{\gamma_{g_N}}=\sum_{n=N+1}^\infty |c_{\chi_n}|^2
\delta_{\chi_n}
\]
is a finite measure. Let $h$ be the inverse Fourier transform of
this finite measure. By \cite{ARMA1,MoSt}, we have 
\[
\widehat{\gamma_{g_N}}=\widehat{h \theta_G} \,,
\] 
which gives 
\[
\gamma_{g_N} = h\theta_G \,.
\] 
We clearly also have
$\gamma_{g_N}=\lb g_N, g_N \rb \theta_G$. This
shows that
\[
h\theta_G= \lb g_N, g_N \rb\, \theta_G \,,
\]
and therefore $h=\lb g_N, g_N \rb$ almost everywhere. Since $h$ and $\lb g_N, g_N \rb$ are continuous, they are equal
everywhere. In particular, 
\[
h(0)=\lb g_N, g_N \rb(0)=M_{\cA}(|g_N|^2) \,.
\] 
and hence,
\[
\|g_N\|^2 =\sum_{n=N+1}^{\infty} |c_{\chi_n}|^2=h (0)=
M_{\cA}(|g_N|^2)  \,.
\]
Note that, for $N=0$, this yields
\[
\|f\|^2=M_{\cA}(|f|^2) = \| f \|_{\be, 2, \cA}^2 \,.
\]
for all $f \in C$.

\noindent We also have
\[
\lim_{N \to \infty} \| f -P_N\|_{\be, 2, \cA}^2=\lim_{N \to \infty} M_{\cA}(|f- P_N|^2) = \lim_{N \to \infty}
\sum_{n=N+1}^{\infty} |c_{\chi_n}|^2=0  \,.
\]
This shows that $f \in \bap^2(G)$ and that $P_N \to f$ in
$\bap^2(G)$. Therefore,
\[
f= \sum_{\chi} c_{\chi} \chi \qquad \mbox{ in }
\bap^2(G)\,,
\]
which implies $c_\chi= a_\chi^\cA(f)$. The last
claim is obvious.
\end{proof}

If one assumes uniform existence of the autocorrelation, one ends up
with Weyl almost periodic functions as follows by a variant of the
preceding considerations.

\begin{lemma}[Appearance of $\wap^2 (G)$]
Let $C$ be a subspace of $L^2_{loc}(G)$ with a seminorm $\|\cdot \|$. Assume that the following properties hold:
\begin{itemize}
\item[(a)] For all $f \in C$, there exist $c_{\chi} \in \CC$
and a formal expansion $f \sim \sum_{\chi} c_{\chi} \chi$ with
\[
\| f \|^2 = \sum_{\chi\in\widehat{G}} |c_\chi |^2 \,.
\]
\item[(b)] We have $\widehat{G} \subseteq C$, and each $\chi \in \widehat{G}$ has a formal expansion $\chi \sim \chi$.
\item[(c)] The formal Fourier expansion in (a) is linear.
\item[(d)]  For all $f \in C$, the reflected Eberlein convolution $\lb f, f \rb$ exists, is independent of the choice of the van Hove sequence, is continuous, and the measure $\gamma_f:=\lb f,f \rb \, \theta_G$ satisfies
\[
\reallywidehat{\gamma_f}=\sum_{\chi} |c_\chi|^2 \delta_{\chi}  \,.
\]
\end{itemize}
Then, $C \subseteq \wap^2(G)$ and, for each $f\in C$, we have  
\[
\|f\| \, = \, \| f \|_{\we,2} 
\]
and 
\[
c_\chi \,=\, a_\chi(f)  \,,
\]
for all $\chi \in\widehat{G}$.
\end{lemma}
\begin{proof}
%
Follow the lines of the previous proof until the line: 
In particular
\[
h(0)=\lb g_N, g_N \rb(0)=M(|g_N|^2) \,.
\]

Here, note that $\lb g_N, g_N \rb(0)$ is independent
of the choice of the van Hove sequence, and so is $M(|g_N|^2)$. In
particular, the mean exists uniformly in translates.
Therefore,
\[
\|g_N\|^2 =\sum_{n=N+1}^{\infty} |c_{\chi_n}|^2=h (0)= M(|g_N|^2)  \,.
\]
For $N=0$, this yields
\[
\|f\|^2=M(|f|^2)=\| f\|_{\we,2}^2  \,,
\]
since the mean exists uniformly in translates.
We also have
\[
\lim_{N \to \infty} \| f-P_N\|_{\we,2}^2  =\lim_{N \to \infty} M(|f- P_N|^2) = \lim_{N \to \infty}
\sum_{n=N+1}^{\infty} |c_{\chi_n}|^2=0 \,.
\]
Since, by the above observations, the mean exists uniformly in translates, we get $f
\in \wap^2(G)$ and $P_N \to f$ in $\wap^2(G)$. Therefore,
\[
f= \sum_{\chi} c_{\chi} \chi \qquad \mbox{ holds  in } \wap^2(G) \,,
\]
which implies $c_\chi= M(f \bar{\chi})=a_\chi(f)$.
\end{proof}

\begin{remark} The preceding two lemmas contain in (d) the
requirement of continuity of the Eberlein convolution. This may seem
like an extra condition. However, for a translation
bounded measures $\mu$, the existence of the Eberlein convolution
$h:=\lb \mu\ast \varphi, \mu \ast \psi\rb$, for
$\varphi,\psi \in\Cc (G)$, automatically entails that $h$ is
continuous and even uniformly continuous, see Proposition
\ref{prop:Eberlein-functions}. So, as far as our application goes,
this is not a restriction.   \exend
\end{remark}

\chapter[Pure point diffraction and TMDS]{Pure point diffraction, almost periodicity and
translation bounded measure dynamical systems}\label{sec:Dynamics}
In this chapter, we have a look at pure-point diffraction and almost
periodicity from the point of dynamical
systems. As discussed in the introduction, dynamical systems play a
key role in the investigation of Aperiodic Order.
In the  companion
articles \cite{LSS,LS3}, we study related aspects for general dynamical
systems.

\section{Dynamical systems}
Suitable dynamical systems provide a convenient and most  used
way to deal with diffraction. The necessary background is discussed
in this section. We follow \cite{BL}, to which we refer for further
details and background.
We start by briefly reviewing general results about topological dynamical systems.

\subsection{Topological dynamical systems and ergodic measures}

Recall that a topological dynamical system $(X,G)$ consists of a compact topological space $X$ and a continuous group action
\[
\alpha : G \times X \to X \,.
\]
We will further assume that the topology on $X$ is metrisable and that $G$ is $\sigma$-compact. The metrisability condition on $X$ implies that $C(X)$ is separable, which will be important in many proofs below.
A dynamical system is called \textit{transitive}\index{transitive!dynamical~system} if there exists some $x \in X$ such that the orbit
\[
O_x:=\{ \alpha(t,x) : t \in G \}
\]
of $x$ is dense in $X$. In this case, we say that $x$ is a \textit{transitive point}\index{transitive!point}. The action $\alpha$ induces a translation operation first on functions $f :X \to \CC$, and then on measures via
\begin{align*}
  \tau_t f (x) \, &:=\, f\big( \alpha(-t,x)\big)   \\
  \tau_t \mathfrak{m} (f) \, &:=\, \mathfrak{m}(\tau_{-t}f) \,.
\end{align*}
Given a function $f \in C(X)$ and some $x \in X$, we define the function $f_x : G \to \CC$ via
\[
f_x(t):=  f\big( \alpha(t,x)\big)  \,.
\]
Then, the translation action is compatibly withe mapping $f \mapsto f_x$. Indeed, for all $f \in C(X), t \in G$ and $ x \in X$, we have
\[
\tau_t f_{x}(s) = f_{x}(s-t)=  f\big( \alpha(s-t,x)\big)= (\tau_t f)\big( \alpha(s,x)\big)= (\tau_t f)_x(s) \,.
\]

We now recall the following standard definitions.

\begin{definition}
A measure $\mathfrak{m}$ on $X$ is called \textit{$G$-invariant}\index{$G$-invariant~measure} if
\[
\tau_t \mathfrak{m} =\mathfrak{m} \qquad \text{ for all } t \in G \,.
\]
A $G$-invariant probability measure $\mathfrak{m}$ on $G$ is called \textit{ergodic}\index{ergodic~measure} if, for all measurable $G$-invariant subsets
$A \subseteq X$, we have $\mathfrak{m}(A) \in \{0,1\}$.
The dynamical system $(X, G)$ is called  \textit{uniquely ergodic}\index{unique~ergodicity} if there exists an unique ergodic measure on $X$.   \exend
\end{definition}

It is well known that for metrisable $X$ and $\sigma$-compact $G$, there exists ergodic measures on $X$.
Next, recall that a F\o lner sequence $\cA$ is called \textit{tempered}\index{F\o lner~sequence!tempered}
if there exists some constant $C>0$ such that
\begin{equation}\label{eq:temp}
\bigg| \bigcup_{k=1}^{n-1} (A^{}_n-A^{}_k) \bigg| \,<\, C |A^{}_n|
\end{equation}
for all $n\in\NN$.

The importance of ergodic measures is given by the following result.

\begin{theorem}[Birkhoff ergodic theorem]\cite{Lin}\label{thm:pointet}
Let $(X,G)$ be a topological dynamical system with metrisable $X$ and $\sigma$-compact $G$, and let $\mathfrak{m}$ be an ergodic measure on $X$.
\begin{itemize}
  \item[(a)] Any F\o lner sequence $\cA$ admits a subsequence which satisfies \eqref{eq:temp}.
  \item[(b)] If $\cA$ is a F\o lner sequence satisfying \eqref{eq:temp}, then, for each $f \in \cL^1(X, \mathfrak{m})$, there exists a set $X_f \subseteq X$ with $\mathfrak{m}(X_f)=1$, such that, for all
  $x \in X_f$, we have
  \begin{equation}\label{eq:pet}
\lim_{n\to\infty} \frac{1}{|A_n|} \int_{A_n} f\big(\alpha(t,x)\big)\, \dd t = \int_{X} f(y)\, \dd \mathfrak{m}(y) \,.  \tag*{$\qed$}
\end{equation}
\end{itemize}
\end{theorem}

\begin{remark}
If $f,g \in \cL^1(X, \mathfrak{m})$ are such that $\| f-g\|_1=0$, then
\[
 \int_{X} f(y)\, \dd \mathfrak{m}(y)=  \int_{X} g(y)\, \dd \mathfrak{m}(y) \,,
\]
but $\frac{1}{|A_n|} \int_{A_n} f(\alpha(t,x)) \dd t$ and $\frac{1}{|A_n|} \int_{A_n} g(\alpha(t,x)) \dd t$ could have different limits for some $x$.
In particular, the set $X_f$ can depend on the representative $f$ we choose in its $L^1$ class. For this reason, below, we work with functions in $\cL^1(X, \mm)$ and not with classes in $L^1(X, \mm)$.    \exend
\end{remark}
\medskip

The above result suggest the following definition.

\begin{definition}
Let $(X, G)$ be a topological dynamical system, let $\mathfrak{m}$ be a $G$-invariant probability measure on $X$, and let $\cA$ be a F\o lner sequence.
We say that $x \in X$ is \textit{generic}\index{generic~point} for $\mathfrak{m}$ and $\cA$ if, for all $f \in C(X)$, we have
\[
\lim_{n\to\infty} \frac{1}{|A_n|} \int_{A_n} f\big(\alpha(t,x)\big)\, \dd t = \int_{X} f(y)\, \dd \mathfrak{m}(y) \,.
\]
The set of generic points is denoted by $X_{\mathsf{g}}$\nomenclature{$X_{\mathsf{g}}$}{set of generic points for an ergodic measure}.    \exend
\end{definition}

Since $C(X)$ is separable, the Birkhoff ergodic theorem has the following consequence.

\begin{coro}\label{cor-ergodic-generic}
Let $(X, \mathfrak{m}, G)$ be an ergodic dynamical system and let with metrisable $X$ and $\sigma$-compact $G$, and let $\cA$ be a F\o lner sequence satisfying \eqref{eq:temp}.
\begin{itemize}
  \item[(a)] One has $\mathfrak{m}(X_{\mathsf{g}})=1$.
  \item[(b)] If $(X, G)$ is uniquely ergodic with unique ergodic measure $\mathfrak{m}$, then all elements $x \in X$ are generic for $\mathfrak{m}$ and all F\o lner sequences $\cB$.
\end{itemize}
\end{coro}
\begin{proof}
(a) follows from the separability of $C(X)$ and Birkhoff's ergodic theorem.

\smallskip
\noindent (b) is just the unique ergodic theorem. Alternatively, this follows immediately from Proposition~\ref{prop-gen} below.
\end{proof}

Next, we address the other direction, showing that each element $x \in X$ is generic for some $G$-invariant measure.

\begin{prop}\label{prop-gen}
Let $(X,G)$ be topological dynamical system with $X$ metrisable and $G$ $\sigma$-compact, and let $x \in X$.
Then, for each F\o lner sequence $\cA$, there exists a $G$-invariant probability measure $\mathfrak{m}$ on $X$ and some subsequence $\cB$ of $\cA$ such that
$x$ is generic for $\mm$ and $\cB$.
\end{prop}
\begin{proof}
For each $n$ define $\mathfrak{m}_n :C(X) \to \CC$ via
\[
\mathfrak{m}_n(f) := \frac{1}{|A_n|} \int_{A_n} f\big(\alpha(t,x)\big)\, \dd t \,.
\]
It is immediate that each $\mathfrak{m}_n$ is a probability measure on $X$.
Since $X$ is metrisable, $C(X)$ is separable. It follows that the set $\mbox{Prob}(X)$ of probability measures on $X$ is compact and metrisable in
the vague topology. Therefore, $(\mathfrak{m}_n)$ has a subsequence $(\mathfrak{m}_{n_k})$, which converges vaguely to some probability measure $\mathfrak{m}$. By definition, this measure satisfies
\[
\lim_{k\to\infty} \frac{1}{|A_{n_k}|} \int_{A_{n_k}} f\big(\alpha(t,x)\big)\, \dd t = \int_{X} f(y)\, \dd \mathfrak{m}(y) \,,
\]
showing that $x$ is generic for $\mathfrak{m}$ and $\cB= (A_{n_k})$.

Finally, the $G$-invariance of $\mathfrak{m}$ follows trivially from the fact that $\cB$ is a F\o lner sequence.
\end{proof}

We complete this subsection by recalling the following
characterization of unique ergodicity. For $G = \ZZ$, this is given in \cite{Wal}.
For more general $G$, this follows by simple
adaption of the argument. We include the details for the convenience of the reader.

\begin{theorem}\label{uniq ergod char}
Let $(X, G)$ be a transitive dynamical system, and let $x \in X$
be any transitive element. Then, $(X, G)$ is uniquely
ergodic if and only if the set
\[
\AAA:=\{ f \in C(X) \,:\,f_x \mbox{ is amenable} \}
\]
is dense in $C(X)$. Moreover, in this case, $f_x $
is amenable for all $f \in C(X)$.
\end{theorem}
\begin{proof}
$\Longrightarrow$: By the unique ergodic theorem, for all $f \in C(X)$, the function $f_x$ is amenable.

\medskip

\noindent $\Longleftarrow$: Fix an arbitrary F\o lner sequence $(A_n)$. Let $\mathfrak{m}_1,\mathfrak{m}_2$ be any two ergodic measures on $X$.
Let $f \in C(X)$ and let $\eps >0$. Then, there exists a $g \in \AAA$ such that $\|f - g \|_\infty <\frac{\eps}{8}$.

Since $g \in \AAA$, there exists some $N$ such that, for all $n >N$ and all $s \in G$, we have
\begin{equation}\label{g amen}
\left| \frac{1}{|A_n|} \int_{s+A_n} g(\alpha(t,x))\, \dd t - M(g) \right| < \frac{\eps}{8} \,.
\end{equation}
Now, by the ergodic theorem, there exists some $y \in X$ such that,
\[
\mathfrak{m}_1(f)= \lim_{n\to\infty} \frac{1}{|A_n|} \int_{A_n}  f\big(\alpha(t, y)\big)\, \dd t \,.
\]
Therefore, there exists some $n>N$ such that
\[
\left| \mathfrak{m}_1(f)- \frac{1}{|A_n|} \int_{A_n}f\big(\alpha(t,y)\big)\, \dd t \right| < \frac{\eps}{8} \,.
\]
For the rest of the proof we fix such an $n$.
Since $\|f - g \|_\infty <\frac{\eps}{8}$, we get
\[
\left|\mathfrak{m}_1(f)- \frac{1}{|A_n|}\int_{A_n} g\big(\alpha(t,y)\big)\, \dd y \right| < \frac{\eps}{4} \,.
\]
Next, since $x$ has dense orbit, there is a sequence $(t_{m})$ in $G$ such that $\alpha(t_{m}, x) \to y$. Therefore, by the uniform continuity of $g$, we get
\[
g\big(\alpha(t, y)\big) = \lim_{m\to\infty} g\big(\alpha(t_m+t, x)\big)  \qquad \text{ for all } t \in A_n \,.
\]
Since $g$ is continuous and bounded and $A_n$ is compact, we can apply the dominated convergence theorem to get
\[
 \frac{1}{|A_n|}\int_{A_n} g\big(\alpha(t,y)\big)\, \dd t= \lim_{m\to\infty} \frac{1}{|A_n|}\int_{A_n} g\big(\alpha(t_m+t, x)\big)\, \dd t \,.
\]
Therefore, for each $n$, there exists some $m_n$ such that
\[
 \left| \frac{1}{A_n} g\big(\alpha(t,y)\big)\, \dd t- \int_{A_n} g\big(\alpha(t_{m_n}+t, x)\big)\, \dd t \right| < \frac{\eps}{8}  \,.
\]
This gives
\[
\left| \mathfrak{m}_1(f)-  \int_{-t_{m_n}+A_n} g\big(\alpha(t, x)\big)\, \dd t \right| <
\frac{3\eps}{8}  \,,
\]
and Eqn.~\eqref{g amen} implies
\[
\left| \mathfrak{m}_1(f)-  M(g) \right| < \frac{\eps}{2} \,.
\]
Repeating the argument with $\mathfrak{m}_2$ instead of $\mathfrak{m}_1$, we get
\[
\left| \mathfrak{m}_2(f)-  M(g) \right| < \frac{\eps}{2} \,,
\]
which leads to
\[
\left| \mathfrak{m}_1(f)-  \mathfrak{m}_2(f) \right| < \eps \,.
\]
Since $\eps>0$ was arbitrary, we get $\mathfrak{m}_1(f)=\mathfrak{m}_2(f)$ for all $f \in C(X)$. This gives that $\mathfrak{m}_1=\mathfrak{m}_2$, and therefore, all ergodic measures on $X$ are equal.
\end{proof}

\subsection{Eigenfunctions}

Let us recall the following classical definition of eigenfunctions.

\begin{definition}
Let $(X,G)$ be a  topological dynamical system, and let $\mathfrak{m}$ be a $G$-invariant probability measure on $X$.
A function $f\in L^2 (X,\mathfrak{m})$ with
$f\neq 0$ is called and \textit{eigenfunction}\index{eigenfunction} with \textit{eigenvalue}\index{eigenvalue}
$\chi\in\widehat{G}$ if, for all $t \in G$, we have
\[
\tau_t f = \chi(t) f \qquad \mbox{ in } L^2(X,\mathfrak{m}) \,.
\]
The set of all eigenvalues of $(X, \mathfrak{m})$ is denoted by $\SSS(X,\mathfrak{m})$\nomenclature{$\mathbb{S}(x,\mathfrak{m})$}{set of eigenvalues for an ergodic measure}, and we will denote by $\EE_\chi$\nomenclature{$\mathbb{E}_\chi$}{eigenspace for the eigenvalue $\chi$} the eigenspace\index{eigenspace} for the eigenvalue $\chi$.

The dynamical system $(X, G, \mathfrak{m})$ is said to have a \textit{pure-point dynamical spectrum}\index{spectrum!pure~point~dynamical~spectrum} if  $L^2 (\XX, \mathfrak{m})$ possess  an orthonormal basis
consisting of eigenfunctions.     \exend
\end{definition}

When $X$ is compact and metrisable, and $G$ is second countable, it is easy to show that $\SSS(X,\mathfrak{m})$ is countable for all $\mathfrak{m}$.
Moreover, if $\mathfrak{m}$ is ergodic, $\SSS(X,\mathfrak{m})$ is a subgroup of $\widehat{G}$.

\smallskip

When talking about eigenvalues, we will often denote by $\underline{1}$ the $0$ eigenvalue. We use this notation to avoid confusion, as the
zero element $\underline{1}$ is actually the character $\underline{1} : G \to U(1)$ defined by
\[
\underline{1}(t)=1 \qquad \text{ for all } t \in G \,.
\]

\smallskip

Next, let $(X, G)$ be a transitive dynamical system. For $\chi \in \widehat{G}$, we provide a criterion for
the existence of a continuous eigenfunction $f_\chi$ for all $G$-invariant measures on $X$.

\begin{theorem}\label{thm-ds-cont-eigenfunctions}
Let $(X, G)$ be a dynamical system, with transitive point $x$, and let $\chi \in \widehat{G}$.
Assume that there exists some function $f \in C(X)$ and a F\o lner sequence $\cA$ such that the Fourier--Bohr coefficient
\[
a_{\chi}(f_x)= \lim_{n\to\infty} \frac{1}{|A_n|} \int_{s+A_n} \overline{\chi(t)} \, f\big(\alpha(t,x)\big)\, \dd t
\]
exist uniformly in $s\in G$ and satisfies  $a_\chi(f_x) \neq 0$.
Then, for  each $y \in X$, the Fourier--Bohr coefficient $ a_\chi^\cA(f_y)$ exists and does not vanish. Moreover, the function $e_\chi : X \to \CC$ defined by
\[
e_\chi(y):= \overline{a_{\chi}^\cA (f_y)}
\]
is continuous and satisfies
\[
(\tau_t e_\chi)(y)= \chi(t) \,  e_\chi(y) \qquad \text{ for all } t \in G, y \in X \,.
\]
In particular, if $\mathfrak{m}$ is a $G$-invariant measure on $X$, $e_\chi \in L^2(X, \mathfrak{m})$ is a continuous eigenfunction for $(X, G, \mathfrak{m})$.
\end{theorem}

\begin{remark}
If $(X,G)$ is uniquely ergodic, the converse result also holds: for any continuous eigenfunction $e_\chi$,
there exists some $f \in C(X)$ such that the Fourier--Bohr coefficient $a_{\chi}(f_x)$ exist uniformly in $s\in G$ and satisfies  $a_\chi(f_x) \neq 0$.
In fact, $f=f_\chi$ satisfies  $a_\chi(f_x)=f_\chi(x)$ uniformly in translates.   \exend
\end{remark}
\begin{proof}
For each $n$, define $\AAA^\chi_n: C(X) \to C(X)$ via
\[
\AAA^\chi_n(f) (y):=\frac{1}{|A_n|} \int_{A_n} \overline{\chi(s)}\, f\big( \alpha(s, y)\big)\, \dd s  \,,
\]
compare \cite{DL}. A straightforward computation reveals that $\AAA^\chi_n(f) \in  C(X)$ for each $f \in C(X)$.

Let $\eps >0$. Since the Fourier--Bohr coefficient $a_\chi(f_x)$ exists uniformly in translates, there exists an $N\in\NN$ such that
\[
\left|\frac{1}{|A_n|}\int_{A_n} \overline{\chi(s)}\, f_x(t+s)\, \dd s
-a_{\chi}(\tau_{t} f_x) \right| <\frac{\eps}{2}
\]
holds for all $n>N$ and all $t \in G$.
Therefore, for all $m,n >N$, we have
\begin{displaymath}
\left|\frac{1}{|A_n|}\int_{A_n} \overline{\chi(s)}\, f_x(t+s)\ \dd s -\frac{1}{|A_m|}\int_{A_m} \overline{\chi(s)}\, f_x(t+s)\ \dd s \right| <\eps\,.
\end{displaymath}
This shows that, for each $m,n >N$, we have
\[
\left| \AAA^\chi_n(f) (\alpha(t, x))- \AAA^\chi_m(f)
(\alpha(t, x)) \right|< \eps  \qquad \text{ for all } t \in G \,.
\]
Since the orbit $O_x=\{ \alpha(t,  x) \, :\, t \in G \}$ is dense in
$X$ and since $\AAA^\chi_n(f)- \AAA^\chi_m(f)
\in C(X)$, one has
\[
\| \overline{\AAA^\chi_n(f)} -\overline{\AAA^\chi_m(f)} \|_\infty  = \| \AAA^\chi_n(f) -\AAA^\chi_m(f) \|_\infty  \leq \eps  \,.
\]
Consequently, since $( C(X), \| \cdot \|_\infty)$ is a Banach spaces, there exists some $e_\chi \in  C(X)$ such that $\overline{\AAA^\chi_n(f)} \to e_\chi$ in $( C(X), \| \cdot \|_\infty)$.

We show that $e_\chi$ has the desired properties.
Since $\AAA^\chi_n(f) \in C(X) $ converges uniformly to $e_\chi$, we have $e_\chi \in C(X)$. Moreover, for all $y \in X$, we have
\begin{align*}
e_\chi(y)
     & = \lim_{n\to\infty} \overline{\AAA^\chi_n(f) (y)}=
        \lim_{n\to\infty} \frac{1}{|A_n|} \overline{\int_{A_n} \overline{\chi(s)}\,
        (f_y) (s)\, \dd s} = \overline{a_{\chi}^\cA (f_y)} \,.
\end{align*}
This shows that the Fourier--Bohr coefficient $a_{\chi}^\cA (f_y)$ exists.

Next, for all $y \in X$ and $t \in G$ we have
\begin{align*}
(\tau_t e_\chi)(y)&=e_{\chi}(\alpha(-t,y))=\lim_{n\to\infty} \frac{1}{|A_n|} \overline{\int_{A_n} \overline{\chi(s)}  f (\alpha(s-t,y) \dd s} \\
&=\chi(t) \lim_{n\to\infty} \frac{1}{|A_n|} \overline{\int_{t+A_n} \overline{\chi(u)}  f (\alpha(u,y) \dd u} = \chi(t)  e_\chi(y)\,.
\end{align*}
Finally, we have
\[
0 \neq \overline{a_{\chi}(f_x)}=e_{\chi}(x)=: \alpha
\]
by assumption. Moreover, for all $t \in G$ we have
\[
|e_{\chi}(\alpha(t,x))|= |\overline{\chi(t)} e_{\chi}(x)|=|\alpha| \,.
\]
Therefore, the continuous function $|e_\chi|$ is constant on the dense orbit $O_x$, and hence, by continuity we have
\[
|e_\chi(y)| = |\alpha| \neq 0 \qquad \text{ for all } y \in X \,.
\]
This shows that $e_\chi$ is not vanishing.
\end{proof}

\subsection{Wiener--Wintner points}

In this section, we review the notion of Wiener--Wintner points, as introduced in \cite{LS3}, and discuss the connection between this and Besicovitch almost periodicity. Here we review the definition and basic properties of these points.
Let us start with the definition.

\begin{definition}
Let $(X, G)$ be a dynamical system, and let $\mathfrak{m}$ be $G$-invariant probability measure on $X$.
We say that $x \in X$ is a \textit{Wiener--Wintner point}\index{Wiener--Wintner~point} with respect to $\mathfrak{m}$ and the F\o lner sequence $\cA$ if, for each
$\chi \in \SSS(X,\mathfrak{m})$, we can find a normal basis $B_\chi$ for the eigenspace $\EE_\chi$ such that, for all $f \in C(X)$, all $\chi \in \SSS(X,\mathfrak{m})$  and all $e_\chi \in B_\chi$, we have
\begin{equation}\label{eq-str-gen}
\int_{X} f(y) \overline{e_\chi(y)}\, \dd \mathfrak{m}(y) = \lim_{n\to\infty} \frac{1}{|A_n|} \int_{A_n} \overline{\chi(t)} f\big(\alpha(x,t)\big)\, \dd t = a_{\chi}^\cA(f_x)\,.
\end{equation}
We will denote the set of Wiener--Wintner points in $(X,G, \mm)$ by $X_{\mathsf{ww}}$\nomenclature{$X_{\mathsf{ww}}$}{set of Wiener--Wintner points for a  measure}.    \exend
\end{definition}

Note that the right hand side of \eqref{eq-str-gen} does not depend on the choice of $e_\chi \in B_\chi$. Therefore, all elements $e_\chi \in B_\chi$ define the same functional
\[
C(X) \ni f \to \langle f, e_\chi \rangle = \int_{X} f(y) \overline{e_\chi(y)}\, \dd \mathfrak{m}(y) \in \CC
\]
and hence, all elements in $B_\chi$ are equal. This means that, when a Wiener--Wintner point exists, all non-trivial eigenspaces are $1$-dimensional.
This gives the following consequences.

\begin{coro}
Let $(X, G)$ be a dynamical system, let $\mathfrak{m}$ be $G$-invariant probability measure on $X$, and let $\cA$ be a F\o lner sequence. Then, $x \in X$ is a Wiener--Wintner point
if and only if, for each $\chi \in \SSS(X,m)$, we can find a normalized eigenfunction $e_\chi$ such that \eqref{eq-str-gen} holds and $B=\{ e_\chi :\chi \in \SSS(X,m)\}$ is a basis for $(L^2(X,\mathfrak{m}))_{\mathsf{pp}}$. In this case, $B$ is an orthonormal basis for $(L^2(X,\mathfrak{m}))_{\mathsf{pp}}$. \qed
\end{coro}

\begin{coro}
Let $(X, G)$ be a dynamical system, let $\mathfrak{m}$ be $G$-invariant probability measure on $X$, and let $\cA$ be a F\o lner sequence. Assume that $(X, \mathfrak{m})$ has pure-point spectrum.
Then, $x \in X$ is a Wiener--Wintner point if and only if for each $\chi \in \SSS(X,m)$ we can find an eigenfunction $e_\chi$ such that \eqref{eq-str-gen} holds and $B=\{ e_\chi :\chi \in \SSS(X,m)\}$ is a basis for $L^2(X,\mathfrak{m})$. In this case $B$ is an orthonormal basis for $L^2(X,\mathfrak{m})$. \qed
\end{coro}

Next, let us show that Wiener--Wintner points are generic for ergodic measures.

\begin{theorem}\cite{LS3}
Let $(X, G)$ be a dynamical system, let $\mathfrak{m}$ be $G$-invariant probability measure on $X$, and let $\cA$ be a F\o lner sequence.
If $x \in X$ is a Wiener--Wintner point, then the following holds:
\begin{itemize}
  \item[(a)] The  measure $\mathfrak{m}$ is ergodic.
  \item[(b)] The  point $x$ is generic.
\end{itemize}
\end{theorem}
\begin{proof}
(a) By the above, $\dim(\EE_{\underline{1}})=1$. Therefore, $\mathfrak{m}$ is ergodic.

\medskip

\noindent (b) Let $B_{\underline{1}}=\{ e_{\underline{1}}\}$. Since $m$ is ergodic, there exists some $C \in \CC$ such that
\[
e_{\underline{1}}(y)=C\qquad \mbox{ for almost all } y \in X \,.
\]
Now, using $f = 1_{X} \in C(X)$ in \eqref{eq-str-gen}, we get
\[
\overline{C}= \int_{X} f(y) e_{\underline{1}}(y)\, \dd \mathfrak{m}(y)=\lim_{n\to\infty} \frac{1}{|A_n|} \int_{A_n} 1\, \dd t =1 \,.
\]
This shows that $C=1$. Therefore, for all $f \in C(X)$, the identity $f e_{\underline{1}}=  f$ holds almost surely. By Eq.~\eqref{eq-str-gen}, $x$ is generic.
\end{proof}

Next, let us recall the following result.

\begin{prop}\cite{LS3}\label{prop-sg-ergodic}
Let $(X, G, \mathfrak{m})$ be an ergodic dynamical system with metrisable $X$, and let $\cA$ be any F\o lner sequence such that Birkhoff's ergodic theorem (Theorem~\ref{thm:pointet}) holds. If $G$ is second countable, the set $X_{\mathsf{ww}}$ of Wiener--Wintner points in $X$ satisfies
\[
\mathfrak{m}(X_{\mathsf{ww}})=1 \,.  \tag*{$\qed$}
\]
\end{prop}




\section{Dynamical systems of translation bounded measures (TMDS)}
In this section, we look at a particular class of dynamical systems which was introduced and studied in \cite{BL}. In this case, one can define an autocorrelation for the dynamical system which, in many situations, can be connected via the so called Dworkin argument \cite{Dwo,DM} to the autocorrelation of individual measures in the dynamical system.

\smallskip

Recall that, given a relatively compact open set $V \subseteq G$ and some $C>0$, the set
\[
\cM_{C,V}:= \{ \mu \in \cM^\infty(G) : \| \mu \|_V \leq C \}
\]
is vaguely compact \cite[Thm.~2]{BL}. Moreover, if $G$ is second
countable, the vague topology is metrisable on $\cM_{C,V}$
\cite[Thm.~2]{BL}.
The natural  group action of $G$ on $\cM^\infty
(G)$ leaves $\cM_{C,V}$ invariant and is continuous
\cite[Prop.~2]{BL}. Specifically,
\[
G\times \cM_{C,V} \longrightarrow \cM_{C,V}\,, \qquad \alpha(t,\mu) :=\delta_t
\ast \mu  \,,
\]
is a continuous action on $\cM_{C,V}$.

\smallskip

With this in mind, we now recall the following definition \cite[Def.~2]{BL}.

\begin{definition}[Translation bounded measure dynamical system]
A pair $(\XX, G)$ is called a \textit{dynamical system on
the translation bounded measures}\index{dynamical~system~on~the~translation~bounded~measures} on $G$ (TMDS) if there is a constant $C >0$ and a relatively compact and open set $V\subseteq G$ such that $\XX$ is a closed subset of $\cM_{C,V}$ that is invariant
under the $G$-action.     \exend
\end{definition}

Note that a closed $G$-invariant subset $\XX \subseteq
\cM^\infty(G)$ is vaguely compact if and only if it is contained in
some $\cM_{C,V}$ \cite{SS}. Therefore, $(\XX, G)$ is a TMDS if and
only if $\XX \subseteq \cM^\infty(G)$ is $G$-invariant and vaguely
compact.

Any translation bounded measure $\mu$ gives rise to a TMDS $(\XX
(\mu), G)$, where the \textit{hull} $\XX(\mu)$\nomenclature{$\mathbb{X}(\mu)$}{hull of the translation bounded measure $\mu$} is defined as
\[
\XX(\mu):=\overline{\{ \tau_t  \mu : t\in G\}} \,,
\]
with closure taken in the vague topology.

\smallskip
Given any TMDS $(\XX, G)$, each $\varphi \in \Cc(G)$ induces a continuous function $f_\varphi : \XX \to \CC$\nomenclature{$f_\varphi$}{function on a dynamical system of translation bounded measures induced by $\varphi \in \Cc(G)$} via
\[
f_\varphi(\omega):= (\omega*\varphi)(0)= \int_{G} \varphi(-s)\, \dd \omega(s) \,,
\]
which is compatible with the translation \cite[Lem.~3]{BL}:
\[
f_\varphi(\tau_t  \omega)=f_{\tau_t  \varphi}(\omega) \qquad \text{ for all } t \in G, \varphi \in \Cc(G) \mbox{ and } \omega \in \XX \,.
\]
\smallskip

Next, let us review the notion of an autocorrelation measure.

\begin{theorem}\cite[Prop.~6, Lem.~7]{BL}\label{thm:aut def}
Let $(\XX, G)$ be a TMDS, and let $\mathfrak{m}$ be a $G$-invariant probability measure $\mathfrak{m}$ on $\XX$. Assume that $\phi \in \Cc(G)$ is such that $\int_{G} \phi(t) \dd t=1$. Then, the mapping
\[
\gamma_{\mathfrak{m}}(\varphi):= \int_{\XX} \int_{G} f_{\varphi}(\tau_{-t} \overline{\omega}) \phi(t)\, \dd \omega(t)\, \dd \mathfrak{m}(\omega)
\]
does not depend on the choice of $\phi$, is a positive definite measure, and
\begin{equation}\label{EQ}
 (\gamma_{\mathfrak{m}}*\varphi*\widetilde{\psi})(t)= \langle \tau_t f_\varphi,  f_\psi \rangle := \int_{\XX} (\tau_t f_\varphi) (\omega) \overline{f_\psi (\omega)}\ \dd \mathfrak{m}(\omega) \,.
\end{equation}
holds for all $\varphi,\psi \in \Cc(G)$ and all $t \in G$. \qed
\end{theorem}

We can now recall the following definition.

\begin{definition}\cite[Def.~6]{BL}
Given a TMDS $(\XX, G)$ with a $G$-invariant probability measure $\mathfrak{m}$, the measure $\gamma_{\mathfrak{m}}$ from Theorem~\ref{thm:aut def} is called the \textit{autocorrelation}\index{autocorrelation!autocorrelation~of~TMDS} of $(\XX, G, \mathfrak{m})$.

Its Fourier transform $\reallywidehat{\gamma_{\mathfrak{m}}}$ is called the \textit{diffraction}\index{diffraction!diffraction~of~TMDS} of $(\XX, G, \mathfrak{m})$.
We say that $(\XX, G, \mathfrak{m})$ has a \textit{pure-point diffraction spectrum}\index{spectrum!pure~point~diffraction~spectrum~for~TMDS} if $\reallywidehat{\gamma_{\mathfrak{m}}}$ is a pure-point measure.     \exend
\end{definition}

We will make use of the following result, see \cite{LS} for
generalisations to non-translation bounded measures.

\begin{theorem} \cite[Thms.~7, 8 and 9]{BL}\label{thm:pp-dyn-diff-spect}
Let $(\XX, G)$ be a TMDS with a $G$-invariant probability measure
$\mathfrak{m}$. Then, $(\XX, G, \mathfrak{m})$ has a pure-point diffraction spectrum if and
only if $L^2(\XX,\mathfrak{m})$ has a pure-point dynamical spectrum. \qed
\end{theorem}

The following result is a generalisation of \cite[Thm.~5]{BL}, and connects the dynamical and diffraction spectra.

\begin{theorem}\label{th-ac-same}
Let $(\XX, G )$ be a TMDS, let $\mathfrak{m}$ be a $G$-invariant measure on $\XX$ and let $\cA$ be a van Hove sequence. If $\omega$ is generic for $\mathfrak{m}$ with respect to $\cA$, the autocorrelation
\[
\gamma_\omega =\lim_{n\to\infty} \frac{1}{|A_n|} (\omega|_{A_n})*\widetilde{(\omega|_{A_n})}
\]
exists with respect to $\cA$ and satisfies $\gamma_\omega=\gamma_{\mathfrak{m}}$.
\end{theorem}
\begin{proof}
For $\varphi, \psi \in \Cc(G)$, we have
\[
(\gamma*\varphi*\widetilde{\psi})(0)= \int_{\XX}  f_\varphi (\mu) \overline{f_\psi (\mu)}\ \dd \mathfrak{m}(\omega) \,.
\]
Since $f_\varphi, f_\psi \in C(\XX)$ and $\omega$ is generic, we get
\begin{align*}
(\gamma_{\mathfrak{m}}*\varphi*\widetilde{\psi})(0)
&= \lim_{n\to\infty} \frac{1}{|A_n|} \int_{A_n}  f_\varphi (\tau_t \omega)\overline{f_\psi ( \tau_t \omega)} \dd t \\
&= \lim_{n\to\infty} \frac{1}{|A_n|} \int_{A_n} (\varphi \ast \omega)(t)\overline{(\psi\ast \omega)(t)}\, \dd t = M\big((\varphi \ast \omega)\cdot \overline{(\psi \ast \omega)}\big) \\
&= (\gamma_\omega*\varphi*\widetilde{\psi})(0) \,.
\end{align*}
with the last equality following from Proposition~\ref{prop-compute-autocorrelation}. The claim follows.
\end{proof}

As an immediate consequence, we see that each autocorrelation measure of a translation bounded measure is also an autocorrelation of a TMDS.

\begin{coro}
Let $\mu \in \cM^\infty(G)$, and let $\gamma_\mu$ be an autocorrelation of $\mu$ with respect to some van Hove sequence $\cA$. Then, there exists a $G$-invariant
measure $\mathfrak{m}$ on $(\XX(\mu), G)$ such that $\gamma_\mu=\gamma_{\mathfrak{m}}$.
\end{coro}
\begin{proof}
By Proposition~\ref{prop-gen}, $\mu$ is generic for some $G$-invariant
measure $\mathfrak{m}$ and some subsequence $\cB$ of $\cA$. Since $\gamma_\mu$ is also the autocorrelation of $\mu$ with respect to $\cB$, the claim follows from Theorem~\ref{th-ac-same}.
\end{proof}

\subsection{Unique ergodicity and continuity of eigenfunctions for TMDS}

Let us discuss now the unique ergodicity and continuity of eigenfunction in the case of TMDS.
The following is an immediate consequence of Theorem~\ref{uniq ergod char}.

\begin{coro}\label{cor:char-ue}
Let $\mu \in \cM^\infty(G)$. Then,
$\XX(\mu)$ is uniquely ergodic if and only if any (finite) product of functions in the set $\{ \mu*\varphi , \overline{ \mu*\varphi } \,:\, \varphi \in \Cc(G) \}$ is amenable.
\end{coro}
\begin{proof}
 $\Longrightarrow$: follows from Theorem~\ref{uniq ergod char}.

\smallskip

\noindent $\Longleftarrow$: Let $\AAA$ be the algebra spanned by
\[
\{ 1_{\XX(\mu)}\} \cup \{ f_\varphi, \overline{f_\varphi} : \varphi \in \Cc(G) \} \,.
\]
Clearly, $\AAA \subseteq C(\XX(\mu))$ is an algebra which separates the points and is closed under complex conjugation. Hence, it is dense
in $C(\XX(\mu))$ by the Stone-Weierstra\ss{} theorem, and Theorem~\ref{uniq ergod char}
proves the claim.
\end{proof}

Next, we show that the uniform existence of the Fourier--Bohr
coefficients implies the continuity of the corresponding
eigenfunction, compare \cite{DL}.

\begin{theorem}\label{them cont eigenfunctions}
Let $\mu \in \cM^\infty(G)$, $\chi \in \widehat{G}$ and  $\cA$ be any
van Hove sequence. Assume that
\begin{displaymath}
a_\chi(\mu)= \lim_{n\to\infty} \frac{1}{|A_n|} \int_{s+A_n} \overline{\chi(t)}\,
\dd \mu (t)
\end{displaymath}
exists uniformly in $s\in G$ and satisfies  $a_\chi(\mu) \neq 0$.
Then, for  each $\omega \in \XX(\mu)$, the Fourier--Bohr coefficient
$a_\chi^\cA(\omega)$ exists, does not vanish, and the function
\[
a_\chi^\cA : \XX(\mu)\longrightarrow \CC
\]
is continuous with
\[
a_\chi^\cA (\tau_t  \omega) =\overline{\chi (t)} \, a_\chi^\cA
(\omega) \qquad \mbox{ for all } \omega \in\XX (\mu) , t\in G \,.
\]
In particular,
\[
f_\chi(\omega):= \overline{a_\chi^\cA ( \omega)}
\]
is a continuous eigenfunction with eigenvalue $\chi$.
\end{theorem}
\begin{proof}
Let $\varphi \in \Cc(G)$ be such that $\widehat{\varphi}(\chi)=1$.
Then, by Corollary~\ref{FB measure relations}, we have
$a_\chi(\mu*\varphi)=a_\chi(\mu)\,\widehat{\varphi}(\chi)\neq 0$.

Since the Fourier--Bohr coefficient $a_\chi(\mu)$ exists uniformly in $x$, by Corollary~\ref{FB measure relations}, so does $a_\chi(\mu*\varphi)$.
Therefore, by Theorem~\ref{thm-ds-cont-eigenfunctions},
\[
f_{\chi}(\omega):=\overline{a_{\chi}^\cA ( \omega*\varphi)} = \overline{\widehat{\varphi}(\chi)} \, \overline{ a_{\chi}^\cA(\omega)}=\overline{a_\chi^\cA ( \omega)}
\]
defines a continuous eigenfunction on $\XX(\mu)$ with eigenvalue $\chi$.

\end{proof}

\section[Generic elements with pure-point spectrum]{Generic elements for TMDS with pure-point spectrum and mean almost periodicity}

In this section, we show that mean almost periodic measures are exactly the measures which are generic with
respect to $G$-invariant measures with pure-point spectrum.

\begin{theorem}\label{thm:DS-char-map}
Let $\mu \in \cM^\infty(G)$, and let $\cA$ be a van Hove sequence.
\begin{itemize}
  \item[(a)] Assume that there exists some measure $\mathfrak{m}$ on $\XX(\mu)$ with pure-point spectrum such that $\mu$ is generic for $\mathfrak{m}$ and $\cA$. Then, the autocorrelation $\gamma$ of $\mu$ exists with respect to $\cA$ and $\mu \in \Map_{\cA}(G)$.
  \item[(b)] Assume that the autocorrelation $\gamma$ of $\mu$ exists with respect to $\cA$. Then, the following assertions are equivalent:
  \begin{itemize}
    \item[(i)] One has $\mu \in \Map_{\cA}(G)$.
    \item[(ii)] There exists a $G$-invariant measure $\mathfrak{m}$ on $\XX(\mu)$ with pure-point spectrum, and a subsequence $\cB$ of $\cA$ such that $\mu$ is generic for $\mathfrak{m}$ and $\cB$.
    \item[(iii)] For every $G$-invariant measure $\mathfrak{m}$ on $\XX(\mu)$ and any subsequence $\cB$ of $\cA$ such that $\mu$ is generic for $\mathfrak{m}$ and $\cB$, the dynamical system $(\XX(\mu), G, \mathfrak{m})$ has pure-point spectrum.
  \end{itemize}
\end{itemize}
\end{theorem}
\begin{proof}
(a) By Theorem~\ref{th-ac-same}, the autocorrelation $\gamma$ of $\mu$ exists with respect to $\cA$ and $\gamma=\gamma_{\mathfrak{m}}$.
By Theorem~\ref{thm:pp-dyn-diff-spect}, $\reallywidehat{\gamma_{\mathfrak{m}}}$ is pure point.

\medskip

\noindent (b) (i) $\Longrightarrow$ (iii):
Let $\mathfrak{m}$ be a $G$-invariant measure on $\XX(\mu)$, and let $\cB$ be a subsequence of $\cA$ such that $\mu$ is generic for $\mathfrak{m}$ and $\cB$.
Then, by Theorem~\ref{th-ac-same}, the autocorrelation $\gamma$ of $\mu$ exists with respect to $\cB$ and
\[
\gamma=\gamma_{\mathfrak{m}} \,.
\]
Since $\mu \in \Map_{\cA}(G) \subseteq \Map_{\cB}(G)$, Theorem~\ref{single element} then gives that $\widehat{\gamma}$ is pure point, and hence, so is
$\reallywidehat{\gamma_{\mathfrak{m}}}$. Therefore, by Theorem~\ref{thm:pp-dyn-diff-spect}, $(\XX(\mu), G, \mathfrak{m})$ has pure-point spectrum.

\medskip

\noindent (iii) $\Longrightarrow$ (ii): By Proposition~\ref{prop-gen}, there exists a $G$-invariant measure $\mathfrak{m}$ on $\XX(\mu)$ and a subsequence $\cB$ of $\cA$ such that $\mu$ is generic for $\mathfrak{m}$ and $\cB$. (iii) implies that $\mathfrak{m}$ has pure-point spectrum.

\medskip

\noindent(ii) $\Longrightarrow$ (i): By (ii) and (a), we have $\mu \in \Map_{\cB}(G)$. Since the autocorrelation $\gamma$ exists with respect to $\cA$ and $\cB$ is a subsequence of $\cA$,
$\gamma$ is the autocorrelation of $\mu$ with respect to $\cB$, and hence, by Theorem~\ref{single element}, $\widehat{\gamma}$ is pure point.
Since $\widehat{\gamma}$ is pure point and $\gamma$ is the autocorrelation of $\mu$ with respect to $\cA$, using Theorem~\ref{single element} again, we get $\mu \in \Map_{\cA}(G)$.
\end{proof}

Next, combining Theorem~\ref{thm:DS-char-map} and Corollary~\ref{cor-ergodic-generic}, we obtain the following two results.

\begin{theorem}\label{them: ue}
Let $(\XX, G)$ be a  TMDS, let $\mathfrak{m}$ be an ergodic measure on
$(\XX, G)$, and let $\cA$ be a  van Hove sequence along which Birkhoff's ergodic theorem holds.
Then, $(\XX, G, \mathfrak{m})$ has a pure-point spectrum if and only if
\[
\mm\bigl(\XX \cap \left( \Map_{\cA}(G)\right)\bigr) =1 \,.     \tag*{$\qed$}
\]
\end{theorem}





\begin{coro}
Let $(\XX, G)$ be a uniquely ergodic TMDS. Then, $(\XX, G)$ has a pure-point spectrum if and only if $\XX \subseteq \Map_{\cA}(G)$. \qed
\end{coro}

\section{Wiener--Wintner points for TMDS with pure-point spectrum}

Here, we show that Besicovitch almost periodic measures are exactly the measures which are Wiener--Wintner points with respect
to $G$-invariant measures with pure-point spectrum.

\begin{theorem}\label{bap-sg}
Let $\mu \in \cM^\infty(G)$, and let $\cA$ be a van Hove sequence. Then, $\mu \in \Bap_{\cA}(G)$ if and only if
there exists some measure $\mathfrak{m}$ on $\XX(\mu)$ with pure-point spectrum such that $\mu$ is a Wiener--Wintner point for $\mathfrak{m}$ and $\cA$.

Moreover, in this case, $\mathfrak{m}$ is ergodic.
\end{theorem}
\begin{proof}
$\Longrightarrow$: First, for all $\varphi \in \Cc(G)$, we have
\[
f_\varphi(\tau_t  \mu)=
(\mu*\varphi)(t) \,.
\]
Therefore, the functions $t \mapsto f_\varphi(\tau_t
\mu)$ and $t \mapsto \overline{f_\varphi(\tau_t  \mu)}$ belong to
$\bap(G)$. It follows immediately that, for every $g$ in the
algebra $\AAA$ generated by $\{  f_\varphi , \overline{ f_\varphi} \}$, we have
\[
g_{\mu}\in \bap(G) \cap L^\infty(G) \,.
\]
A standard density argument, compare \cite{LSS}, implies that, for all
$g \in C(\XX)$, we have $g_\mu \in \bap(G)$.
Therefore, by Proposition~\ref{Bap props}, the limit
\[
\mathfrak{m}(f):= \lim_{n\to\infty} \frac{1}{|A_n|} \int_{A_n} f(\tau_t  \mu)\, \dd t
\]
exists for each $f \in
C(\XX(\mu))$. It is obvious that $\mathfrak{m} : C(\XX) \to \CC$ is linear, positive, and therefore a positive measure.
Moreover, for the constant function $1_{\XX}$, we have
\[
\mathfrak{m}(\XX)= \mathfrak{m}(1_{\XX})= \lim_{n\to\infty} \frac{1}{|A_n|} \int_{A_n}1\, \dd t  =1 \,.
\]
Finally, for all $s \in G$ and $f \in C(\XX)$, we have
\begin{align*}
| \mathfrak{m}(f)-\tau_s \mathfrak{m}(f) |
    &= |\mathfrak{m}(f)-\mm(\tau_{-s}f) | \\
    &= \left| \lim_{n\to\infty} \frac{1}{|A_n|} \int_{A_n} f(\tau_t\mu)\,
      \dd t -\frac{1}{|A_n|} \int_{A_n} f(\tau_{z+s} \mu) \,\dd z \right| \\
    &= \left| \lim_{n\to\infty} \frac{1}{|A_n|} \int_{A_n} f(\tau_t  \mu)\,
       \dd t -\frac{1}{|A_n|} \int_{-s+A_m} f(\tau_{t} \mu)\, \dd t \right| \\
    &= \left| \lim_{n\to\infty} \frac{1}{|A_n|} \int_{A_n \triangle (-s+A_n)}
       f(\tau_t  \mu)\, \dd t \right| \\
       &\leq \| f \|_\infty\, \lim_{n\to\infty}\frac{|A_n \triangle (-s+A_n)|}{|A_n|}
       = 0 \,.
\end{align*}
Therefore, $\mathfrak{m}$ is $G$-invariant.

Next, by the above, for all $f \in C(\XX(\mu))$,
we have $f_{\mu} \in \bap^2(G)$.
Therefore, $f_{\mu}$ has well defined Fourier--Bohr coefficients.
Define $F_\chi : C(\XX(\mu)) \to \CC$ via
\[
F_\chi(f)=a_\chi^{\cA} (f_{\mu}) \,.
\]
By the Cauchy-Schwartz inequality, we then have
\[
\left| \frac{1}{|A_n|} \int_{A_n} \overline{\chi(t)}\, f(\tau_t  \mu)\,
\dd t \right|^2 \leq  \frac{1}{|A_n|} \int_{A_n} \left| f(\tau_t
\mu) \right|^2\, \dd t  \,.
\]
Moreover, by the definition of $\mathfrak{m}$, we have
\begin{equation}\label{eq mean}
\lim_{n\to\infty} \frac{1}{|A_n|} \int_{A_n} \left| f(\tau_t  \mu) \right|^2\,
\dd t= \int_{\XX} \left| f(\omega) \right|^2\, \dd \mathfrak{m}(\omega) \,.
\end{equation}
Therefore, for all $f \in C\big(\XX(\mu)\big)$, we have
\[
\left| F_\chi(f) \right| \leq \| f \|_2 \,.
\]
Since $C(\XX(\mu))$ is a dense subspace of $L^2(\XX(\mu), \mm)$, $F_\chi$ can be extended to a continuous functional on the Hilbert space $L^2(\XX(\mu), \mm)$.
By Riesz' lemma, there exists some element $e_\chi \in L^2(\XX(\mu), \mm)$ with $\| e_\chi \| \leq 1$ such that, for all
$f \in L^2(\XX(\mu), \mm)$, we have
\begin{equation}\label{Eq eigenfunct}
F_\chi(f) = \int_{\XX} f(\omega)\, \overline{ e_\chi(\omega)}\, \dd \mathfrak{m}(\omega) \,.
\end{equation}
Next, define $E:= \{ \chi \in \widehat{G} : e_\chi \neq 0 \}$. By
construction, $\chi \in E$ if and only if there exists some $f \in
C(\XX(\mu))$ such that $a_\chi^{\cA} (f_\mu) \neq 0$.
A short computation shows that, for all $f \in C(\XX(\mu))$, we have
\begin{align*}
0
    &= \int_{\XX} f(\omega)\, \overline{ e_\chi(\tau_t  \omega)}\, \dd \mathfrak{m}(\omega)
      - \int_{\XX}  f(\omega)\, \overline{ e_\chi(\tau_t \omega)}\,
       \dd \mathfrak{m}(\omega) \\
    &= \int_{\XX} f(\omega)\, \overline{e_\chi(\tau_t  \omega)}\, \dd \mathfrak{m}(\omega)
      - \int_{\XX}  f(\tau_{-t}\omega)\, \overline{ e_\chi(\omega)}\,
      \dd \mathfrak{m}(\omega) \\
      &= \int_{\XX} f(\omega)\, \overline{e_\chi(\tau_t  \omega)}\, \dd \mathfrak{m}(\omega)
      - F_\chi(\tau_tf) \\
    &= \int_{\XX} f(\omega)\, \overline{e_\chi(\tau_t  \omega)}\, \dd \mathfrak{m}(\omega)
      - \lim_{n\to\infty} \frac{1}{|A_n|} \int_{A_n} \tau_t f(\tau_s \omega)\,
        \overline{\chi(s)}\, \dd s \\
     &= \int_{\XX} f(\omega)\, \overline{e_\chi(\tau_t  \omega)}\, \dd \mathfrak{m}(\omega)
      - \overline{\chi(t)}\lim_{n\to\infty} \frac{1}{|A_n|} \int_{A_n}  f(\tau_{s-t} \omega)\,
        \overline{\chi(s-t)}\, \dd s \\
      &= \int_{\XX} f(\omega)\, \overline{f_\chi(\tau_t  \omega)}\, \dd \mathfrak{m}(\omega)
      - \overline{\chi(t)}\lim_{n\to\infty} \frac{1}{|A_n|} \int_{t+A_n}  f(\tau_{r} \omega)\,
        \overline{\chi(r)}\, \dd r \\
        &= \int_{\XX} f(\omega)\, \overline{e_\chi(\tau_t  \omega)}\, \dd \mathfrak{m}(\omega)
      - \overline{\chi(t)}\lim_{n\to\infty} \frac{1}{|A_n|} \int_{A_n}  f(\tau_{r} \omega)\,
        \overline{\chi(r)}\, \dd r \\
    &=\int_{\XX} f(\omega) \left( \overline{e_\chi(\tau_t  \omega)}
      - \overline{\chi(t)}\, \overline{e_{\chi} (\omega)} \right)
       \dd \mathfrak{m}(\omega)   \,,
\end{align*}
where the second last equality follows as $f$ is bounded from the van Hove condition.
The density of $C(\XX(\mu))$ in $L^2(\XX(\mu), \mm)$ implies that
\[
 \overline{e_\chi(\tau_t  \omega)}=  \overline{\chi(t)} \, \overline{e_{\chi} (\omega)} \,.
\]
It follows that for all $\chi \in E, e_\chi$ is an eigenfunction.

Finally, for each $f \in C(\XX(\mu))$, Eq.~\eqref{eq mean} and Parseval's identity for the
function $f_\mu$ give
\begin{align*}
\int_{\XX} \left| f(\omega) \right|^2 \dd \mathfrak{m}(\omega)
    &=\lim_{n\to\infty} \frac{1}{|A_n|} \int_{A_n} | f(\tau_t  \mu) |^2\, \dd t
      = \sum_{\chi \in \widehat{G}} \big| a_{\chi}^\cA( f_\mu) \big|^2 \\
    &= \sum_{\chi \in \widehat{G}} \left|F_\chi(f)\right|^2
      = \sum_{\chi \in \widehat{G}} \big|\langle f, e_\chi\rangle \big|^2 \,.
\end{align*}
Since the elements in $\{ e_ \chi : \chi \in E \}$ are orthogonal, $\|f_\chi \| \leq 1$ and $C(\XX(\mu))$ is dense in $L^2(\XX(\mu), \mm)$, it follows that $\| e_\chi \|=1$ for all $\chi \in E$ and that $\{ e_ \chi : \chi \in E \}$ is an orthonormal basis in $L^2(\XX(\mu),\mathfrak{m})$. This implies that
\[
E=\SSS(\XX(\mu), \mathfrak{m}) \,.
\]

Finally, by the definition of $F_\chi$, for all $f \in C(\XX(\mu))$ and all $\chi  \in \SSS(\XX(\mu), \mathfrak{m})=E$ we have
\[
\langle f, e_\chi \rangle = F_\chi(f)=a_\chi^{\cA} (f_{\mu}) \,.
\]
This shows that $\mu$ is a Wiener--Wintner point for $\mathfrak{m}$. It follows that $\mathfrak{m}$ is ergodic.

\medskip

\noindent $\Longleftarrow$: Let $B:=\{ e_\chi : \chi \in \SSS(\XX,\mathfrak{m})\}$ be the choice of eigenfunctions which make $\mu$ a Wiener--Wintner point.
Since $(\XX(\mu), G, \mathfrak{m})$ has pure-point spectrum, $B$ is an orthonormal basis for $L^2(\XX(\mu), G, \mathfrak{m})$.
Let $\varphi \in \Cc(G)$.
Since $f_\varphi \in L^2(\XX(\mu), G, \mathfrak{m})$, there exists characters $\chi_1, \ldots , \chi_n \ldots \in \SSS(\XX(\mu), \mathfrak{m})$ such that
\[
\lim_{N \to \infty} \| f_\varphi - \sum_{k=1}^N \langle  f_\varphi , e_{\chi_k} \rangle e_{\chi_k} \|_2^2 =0 \,.
\]
For simplicity, let
\[
c_k:=\langle  f_\varphi , e_{\chi_k} \rangle \,.
\]
Then, a short computation yields
\begin{align*}
  \| f_\varphi - \sum_{k=1}^N c_k e_{\chi_k} \|_2^2 & = \| f_\varphi\|_2^2 - \sum_{k=1}^N  \langle f_\varphi , c_k e_\chi \rangle   \\
   &- \sum_{k=1}^N  \overline{\langle f_\varphi , c_k e_\chi \rangle} + \sum_{k=1}^N |c_k|^2 \,.
\end{align*}
Next, since $f_\varphi$ is continuous and $\mu$ is a Wiener--Wintner point, we get
\[
  \| f_\varphi\|_2^2  = M_{\cA}(|\varphi\ast \mu|^2)
\]
and
\[
 \langle f_\varphi , c_k e_\chi \rangle  = \overline{c_k} \langle f_\varphi , e_{\chi_k} \rangle = \overline{c_k} a_{\chi_k}^\cA(\varphi\ast \mu) \,.
\]
Therefore,
\begin{align*}
\| f_\varphi - \sum_{k=1}^N c_k e_{\chi_k} \|_2^2
   &= M_{\cA}(|\varphi\ast \mu|^2) -\sum_{k=1}^N\overline{c_k} a_{\chi_k}^\cA(\varphi\ast \mu)  \\
   &\phantom{XXX}- \sum_{k=1}^N c_k \overline{ a_{\chi}^\cA(\varphi\ast \mu)} + \sum_{k=1}^N |c_k|^2 \\
   &= \lim_{n\to\infty} \frac{1}{|A_n|} \int_{A_n} \left( (\varphi\ast \mu)(t) \overline{(\varphi\ast \mu)(t)} - \sum_{k=1}^N\overline{c_k} (\varphi\ast \mu)(t) \overline{\chi_k(t)} \right.\\
   &\phantom{XXXX} \left.- \sum_{k=1}^N c_k \overline{(\varphi\ast \mu)(t)} \chi_k(t)+\sum_{k=1}^N c_k\chi_k(t) \overline{c_k} \overline{\chi_k(t)} \right) \dd t \\
    &= \lim_{n\to\infty} \frac{1}{|A_n|} \int_{A_n} \left| (\varphi\ast \mu)(t) - \sum_{k=1}^N c_k \chi_k(t) \right|^2 \dd t \\
    &= \| \varphi\ast \mu - \sum_{k=1}^N c_k \chi_k \|_{\be,2,\cA}^2 \,.
\end{align*}
Therefore, setting $P_N:= \sum_{k=1}^N c_k \chi_k$, we get
\[
\lim_{N\to\infty} \| \varphi\ast \mu - P_N \|_{\be,2,\cA}= \lim_{N\to\infty} \| f_\varphi - \sum_{k=1}^N c_k e_{\chi_k} \|_2 =0 \,
\]
i.e. $\varphi\ast \mu \in \bap^2(G)$. Since $\varphi \in \Cc(G)$ was arbitrary, we get $\mu \in \Bap_{\cA}(G)$.
\end{proof}

As immediate consequences, we get:

\begin{theorem}[Besicovitch almost periodic translation-bounded measures are generic for an ergodic measure with pure-point spectrum]\label{bap gives ergodic}
Let $\mu \in \Bap_{\cA}(G) \cap \cM^\infty(G)$. Then, there exists an ergodic, $G$-invariant probability measure $\mm$ on $\XX:=\XX(\mu)$ with the following properties:
\begin{itemize}
\item [(a)] For all $f \in C(\XX)$, we have
\[
\lim_{n\to\infty} \frac{1}{|A_n|} \int_{A_n} f(\tau_t  \mu)\, \dd t = \int_{\XX}
f(\omega)\, \dd \mm(\omega) \,.
\]
\item[(b)] The autocorrelation $\gamma_\mm$ of $(\XX, \mm, G)$ is also the autocorrelation $\gamma_\mu$ of $\mu$ with respect to $\cA$.
  \item[(c)] The system $(\XX, \mm, G)$ has pure-point dynamical spectrum, which is generated by $\{ \chi : a_{\chi}^\cA(\mu) \neq 0\}$. \qed
\end{itemize}
\end{theorem}

\begin{theorem}\label{them: ergodic}
Let $(\XX, G, \mathfrak{m})$ be an ergodic TMDS with second countable $G$, and let $\cA$ be a van Hove
sequence along which Birkhoff's ergodic theorem holds.
Then, the system $(\XX, G, m)$ has a pure-point spectrum if and only if
\[
\mm\bigl(\XX \cap \left(\Bap_{\cA}(G)\right)\bigr)= 1 \,.
\]
In this case, for each $\chi$ with $\reallywidehat{\gamma_\mm}(\{ \chi \}) \neq 0$, there exists a non-trivial eigenfunction $f_\chi \in \mathcal{L}^1(\XX,m)$ such that
\[
f_\chi(\omega)= \overline{a_\chi^\cA(\omega)} \qquad \mbox{ for all } \omega \in \Bap_{\cA}(G) \cap \XX \,,
\]
and
\[
\reallywidehat{\gamma_\mu}(\{ \chi \})=\left| f_\chi(\omega) \right|^2 = \left| a_\chi^\cA(\omega) \right|^2 \qquad \mbox{ for } \mm\mbox{-almost all } \omega \in \XX \,.
\]
\end{theorem}
\begin{proof}

\noindent $\Longleftarrow$: This follows from $\Bap_{\cA}(G) \subseteq \Map_{\cA}(G)$ and Theorem~\ref{them: ue}.

\medskip

\noindent $\Longrightarrow$: Follows from Proposition~\ref{prop-sg-ergodic} and Theorem~\ref{bap-sg}.

\smallskip

Finally, we will show the last statement.
Choose a $\chi \in \widehat{G}$ with $\widehat{\gamma}(\{\chi \}) \neq 0$. Define
\[
f_\chi(\omega)=
\begin{cases}
 \overline{a_\chi^\cA(\omega)}, & \mbox{ if } \omega \in \Bap_\cA(G) \cap \XX \,, \\
0, &\mbox{ otherwise } \,.
\end{cases}
\]
This is well defined as the Fourier--Bohr coefficients of Besicovitch
almost periodic measures exist by Theorem \ref{Bap and BT}.

We claim that this satisfies the given condition. Note first that, by Lemma~\ref{L1}(c), the set $\XX \cap \Bap_{\cA}(G)$ is $G$-invariant. It follows that we have
\[
f_\chi(\tau_t \omega)=0=\chi(t) f_\chi(\omega)
\]
for all $\omega \notin \Bap_\cA(G) \cap \XX$. For $\omega \in \Bap_\cA(G) \cap \XX$, it follows immediately from
the definition of the Fourier--Bohr coefficients and translation boundedness that
\begin{align*}
f_\chi(\tau_t\omega)
   &= \lim_{n\to\infty} \frac{1}{|A_n|} \overline{\int_{A_n} \overline{\chi(s)}\, \dd (\tau_t \mu)(s)} = \lim_{n\to\infty} \frac{1}{|A_n|}  \overline{\int_{A_n} (\tau_{-t} \overline{\chi})(s)\, \dd \mu(s) } \\
   &= \lim_{n\to\infty} \frac{1}{|A_n|}  \overline{ \int_{A_n}  \overline{\chi(s+t)}\, \dd \mu(s)} =\chi(t) \cdot f_\chi(\omega) \,.
\end{align*}

\smallskip

Next, we show that $f_\chi \in \mathcal{L}^1(\XX,m)$.
Choose $\varphi$ with $\widehat{\varphi}(\chi)=1$. As in the proof of Theorem~\ref{them cont eigenfunctions}, define $\AAA^\chi_n: C(\XX(\mu)) \to C(\XX(\mu))$ via
\[
\AAA^\chi_n(f) (\omega):=\frac{1}{|A_n|} \int_{A_n} \overline{\chi(s)}\, f( \tau_s \omega)\, \dd s  \,.
\]
Then, $\AAA_n^\chi(f_\varphi) \in C( \XX)$ for all $n$. Moreover, by definition, $\|\AAA_n^\chi ( \varphi) \|_\infty  \leq \| f_\varphi \|_\infty$.
For all $\omega \in \XX \cap \Bap_{\cA}(G)$, the Fourier--Bohr coefficients of $\omega$ exist by Corollary~\ref{bap2 char}. Hence,
\[
f_\chi(\omega) =\overline{a_\chi^\cA(\omega)}= \overline{a_\chi^\cA(\omega*\varphi)}= \lim_{n\to\infty}
\frac{1}{|A_n|} \overline{\int_{A_n} f_\varphi(\tau_t  \omega)\, \dd \omega} =
\lim_{n\to\infty} \overline{ \AAA_n^\chi(f_\varphi) (\omega)} \,.
\]
Since $\mm(\XX \cap \Bap_{\cA}(G))=1$, it follows that $\big(\overline{\AAA_n^\chi(f_\varphi)}\big)$ is a sequence of functions in $C(\XX) \subseteq \mathcal{L}^1(\XX,m)$, which is bounded by the constant function
$\| f_\varphi\|_\infty 1_{\XX} \in \mathcal{L}^1(\XX, \mm)$ and which converges almost everywhere to $f_\chi$.
The dominated convergence theorem implies that $f_\chi \in \mathcal{L}^1(\XX,\mm)$ as claimed.

Finally, as $\widehat{\gamma_\mu}(\{ \chi \}) \neq 0$ and as
$\gamma_\mu$ is almost surely the autocorrelation of $\omega \in \Bap_{\cA}(G)
\cap \XX$, we have
\[
0 \neq \widehat{\gamma_\mu}(\{ \chi \})= \left| a_\chi^\cA(\omega)\right|^2 =\left| f_\chi (\omega) \right|^2 \qquad \mbox{ for } \mm\mbox{-almost all } \omega \in \XX \,,
\]
by Theorem~\ref{thm:a-representation-besicov}. Therefore, $f_\chi$ is non-trivial, as well as the last claim.
\end{proof}

\begin{remark}
The direct implication $\Longrightarrow$ in Theorem~\ref{them: ergodic} can also be proven exactly as (iii) $\Longrightarrow$ (i) in Theorem~\ref{thm:wap}.   \exend
\end{remark}

As a last consequence, we list the following result, which shows that in the pure-point case, the notion of Wiener--Wintner point is equivalent to the points studied in \cite[Thm.~2]{KK}. A more general version of this can be found in
\cite{LS3}. As the proof is straightforward, we skip it and refer the reader to \cite{LS3} instead.

\begin{theorem}
Let $(\XX, \mm)$ be an ergodic TMDS with pure-point spectrum, let $\cA$ be a van Hove sequence along which the Birkhoff ergodic theorem holds, and let $\mu \in X$. Then, with respect to $\cA$ we have
\[
\XX_{\mathsf{ww}}=\XX_{\mathsf{g}} \cap \Bap_{\cA}(G) \,. \tag*{$\qed$}
\]
\end{theorem}

Let us cover another intriguing consequence of Theorem~\ref{bap-sg}.

\begin{theorem} Let $(\XX, G, \mathfrak{m})$ be a TMDS, and let $\cA$ be a van Hove sequence.
Assume that $\XX$ has pure-point spectrum and continuously representable eigenfunctions.
\begin{itemize}
  \item[(a)] One has $\XX_{\mathsf{ww}}=\XX_{\mathsf{g}}$.
  \item[(b)] We have $\XX_{\mathsf{g}} \subseteq \Bap_{\cA}(G)$.
  \item[(c)] The set $\XX_{\mathsf{g}}$ is non-empty if and only if $\mathfrak{m}$ is ergodic.
\end{itemize}
\end{theorem}
\begin{proof}
Since the eigenfunctions are continuous, the sets of generic and Wiener--Wintner points coincide. The claims follow.
\end{proof}

By combining all results in this section, we get:

\begin{coro}[Characterization of pure-point spectrum via almost periodicity]
\label{cor: ergodic}
Consider an ergodic TMDS $(\XX, G, \mm)$ with second countable $G$, and let $\cA$ be a van Hove sequence along which Birkhoff's ergodic theorem holds. Then, the following statements are equivalent:
\begin{itemize}
  \item [(i)] The system $(\XX, G, \mm)$ has a pure-point spectrum.
  \item [(ii)] One has $\mm(\XX \cap \Bap_{\cA}(G))=1$.
  \item [(iii)] One has $\mm(\XX \cap \Map_{\cA}(G))=1$.
\end{itemize}
\end{coro}
\begin{proof}
Combine Theorem~\ref{thm:DS-char-map} and Theorem~\ref{them: ergodic}.
\end{proof}

\begin{remark}\label{rem1} \phantom{X}
\begin{itemize}
\item[(a)] In Corollary~\ref{cor: ergodic}, we can have
\[
\XX \cap \Bap_{\cA}(G) \subsetneq  \XX \cap \Map_{\cA}(G) \,.
\]
Consider for example the hull $\XX:=\XX(\mu)$, where $\mu$ is the
$a$-defect of $\ZZ$ for some $a \in (0,1) \backslash \QQ$ from Proposition~\ref{single def prop}. Then, by Proposition~\ref{single def prop}, $(\XX, G)$ is uniquely ergodic, has a pure-point spectrum and $\XX \subseteq \Map_{\cA}(G)$ but $\mu \notin \Bap_{\cA}(G)$.
\item[(b)]
 Let $\XX$ be a uniquely ergodic TMDS with pure-point diffraction.
Then, all elements $\omega \in \XX$ are mean almost periodic, and
almost all elements $\omega \in \XX$ are Besicovitch almost
periodic. It is not necessarily true that all elements $\omega \in
\XX$ are Besicovitch almost periodic.
Again, the hull
$\XX:=\XX(\mu)$, where $\mu$ is the $a$-defect of $\ZZ$ for some $a
\in (0,1) \backslash \QQ$, provides such an example.    \exend
\end{itemize}
\end{remark}

Let us briefly discuss the difference between mean and Besicovitch almost periodicity in the context of ergodic dynamical systems (with pure-point spectrum). The following result is an immediate consequence of Theorem~\ref{thm:DS-char-map} and Theorem~\ref{bap-sg}.

\begin{coro}
Consider an ergodic TMDS $(\XX, G, \mm)$ with pure-point spectrum and second countable $G$, and let $\cA$ be a van Hove sequence along which Birkhoff's ergodic theorem holds. Let $\mu \in \XX_{\mathsf{g}}$. Then, the following holds:
\begin{itemize}
  \item[(a)] $\mu \in \Map_{\cA}(G)$.
  \item[(b)] $\mu \in \Bap_{\cA}(G)$ if and only if $\mu \in \XX_{\mathsf{ww}}$.\qed
\end{itemize}
\end{coro}

As a consequence, we get the following theorem.

\begin{theorem}
Let $G$ be a second countable group, and let $\sigma$ be a positive pure-point measure on $\widehat{G}$. Then, the following statements are equivalent.
\begin{itemize}
\item[(i)] There is an ergodic TMDS $(\XX, G, \mathfrak{m})$ with autocorrelation $\gamma$ such that $(\widehat{\gamma})_{\mathsf{pp}}=\sigma$.
\item[(ii)] There is an ergodic TMDS $(\XX, G, \mathfrak{m})$ with a pure-point spectrum, and diffraction $\sigma$.
\item[(iii)] There is a van Hove sequence $\cA$ and some $\mu \in \Bap_{\cA}(G) \cap \cM^\infty(G)$ such that $\sigma$ is the diffraction of $\mu$ with respect to $\cA$.
\end{itemize}
\end{theorem}
\begin{proof}
(i) $\Longrightarrow$ (ii) follows from \cite[Thm.~4.1]{JBA}.

\medskip

\noindent (ii) $\Longrightarrow$ (iii) Since $(\XX, G, \mm)$ has a pure-point spectrum, we have
\[
\mathfrak{m}(\XX \cap \Bap_{\cA}(G))=1
\]
by Corollary~\ref{cor: ergodic}. Now, let $\cA$ be a van Hove sequence along which Birkhoff's ergodic theorem holds. By \cite[Thm.~5(b)]{BL}, $\omega$ is almost surely the diffraction of $\nu \in \XX$. In particular, there exists some
$\mu \in \XX \cap \Bap_{\cA}(G)$ such that $\omega$ is the diffraction of $\mu$.

\medskip

\noindent (iii) $\Longrightarrow$ (i) This follows from Theorem~\ref{bap gives ergodic}.
\end{proof}

\section{Characterizing Weyl almost periodic measures  via TMDS}
We showed in the previous sections that having a pure-point spectrum for a TMDS
can be characterized via mean and Besicovitch almost periodicity.
Now, we show that, for a TMDS $(\XX, G)$, Weyl almost periodicity for
one/all elements is equivalent to having a pure-point dynamical spectrum and
continuous eigenfunctions and being uniquely ergodic.

\begin{theorem}\label{thm:wap}
Let $\mu \in \cM^\infty(G)$. Then, the following statements are equivalent:
\begin{itemize}
  \item [(i)] We have $\mu \in \Wap(G)$.
  \item [(ii)] We have $\XX(\mu) \subseteq \Wap(G)$.
  \item [(iii)] The dynamical system $\XX(\mu)$ is uniquely ergodic, has pure-point dynamical spectrum and continuous eigenfunctions.
\end{itemize}
In this case, for each $\chi$ with $a_{\chi}(\mu) \neq 0$,
the function
\begin{align*}
f_\chi &:
\XX(\mu)  \to \CC \\
f_{\chi}(\omega)\,&=\, \overline{a_{\chi}(\omega)}
\end{align*}
is a continuous eigenfunction for the system with eigenvalue $\chi$.
\end{theorem}
\begin{proof} (ii)$\implies$(i): This trivially holds.

\medskip

\noindent (iii)$\implies$(ii):
Denote the set of eigenvalues by $E$. Then, $E$ is a countable subgroup of $\widehat{G}$ by standard arguments. Choose a family $\{ f_\chi \}_{\chi \in E}$ of eigenfunctions which are continuous, normalized and such that $f_{\underline{1}}=1$.

Now, for each $\varphi \in \Cc(G)$ and each $\eps >0$, since $L^2(\XX, \mathfrak{m})$ has pure-point spectrum, exactly as in Theorem.~\ref{them: ergodic}, there exists some $F= \sum_{k=1}^N c_k f_{\chi_k}$ such that
\[
 \int_{\XX} \left| f_\varphi (\omega) - F(\omega) \right|^2\, \dd \mathfrak{m} (\omega) < \eps^2 \,.
\]
Next, fix some arbitrary $\nu \in \XX(\mu)$.
Since the eigenfunctions are continuous, so is $f_\varphi-F$. Therefore, by the unique ergodic theorem, we have
\begin{align*}
\int_{\XX} \left| f_\varphi (\omega) - F(\omega) \right|^2\, \dd \mathfrak{m} (\omega)
    &= \lim_{n\to\infty} \frac{1}{|A_n|} \int_{x+A_n} \left|(\nu*\varphi)(s)
      - F(\tau_s\nu) \right|^2 \, \dd s \\
    &= \lim_{n\to\infty} \frac{1}{|A_n|} \int_{x+A_n} \left|(\nu*\varphi)(s)
      - \sum_{k=1}^N c_k'  \chi_k(s) \right|^2\, \dd s
\end{align*}
uniformly in $x$, where
\[
c_k':=c_k f_{\chi_k}(\nu)\,.
\]
This shows that, for all $\varphi \in \Cc(G)$ and $\eps >0$, there exists a trigonometric polynomial $P=\sum_{k=1}^N c_k'  \chi_k$ such that
\[
\|\nu*\varphi - P\|_{\we,2} < \eps\,.
\]
This gives  $\nu*\varphi \in \wap^2(G)$ for all $\varphi \in \Cc(G)$.

\medskip

\noindent (i)$\implies$(iii): We split the proof into steps:

\smallskip

\noindent \underline{Step 1:} We prove unique
ergodicity.

Since $\mu \in \Wap(G)$, for all $\varphi \in \Cc(G)$, we have $\mu *\varphi \in \wap(G) \cap \Cu(G)$, and hence $\overline{\mu *\varphi} \in \wap(G)\cap \Cu(G)$.
It follows that
\[
A:=\{ \mu*\varphi, \overline{\mu*\varphi} : \varphi
\in \Cc(G) \} \subseteq \wap(G) \cap \Cu(G) \,.
\]
By Lemma~\ref{spi
lemma} (c), we get that $\{ \prod_{j=1}^n f_j : f_j \in A\}
\subseteq \wap(G)$. In particular, any product of elements in $A$ is
amenable. Unique ergodicity then follows from
Corollary~\ref{cor:char-ue}.

\medskip

\noindent \underline{Step 2:}  We prove that $(\XX(\mu), \mm)$ has pure-point
dynamical spectrum. Here, $\mathfrak{m}$ is the unique ergodic measure.

Let $\gamma$ be the
autocorrelation of the dynamical system. Then, by the unique
ergodicity, $\gamma$ is the autocorrelation of $\mu$ with respect to
some van Hove sequence $(A_n)$.

Since $\mu$ is Weyl almost
periodic, hence mean almost periodic, $\widehat{\gamma}$ is pure
point by Theorem~\ref{theorem-uniform-phase-problem}, and satisfies the (CPP)
\[
\widehat{\gamma}(\{ \chi \}) = |a_\chi(\mu)|^2 \,.
\]
\smallskip

\noindent \underline{Step 3:} We prove the continuity of the eigenfunctions.
Let $\chi \in \widehat{G}$ be any element such that $\widehat{\gamma}(\{ \chi \}) \neq 0$. Then, by the (CPP) we have $a_\chi(\mu) \neq 0$. By Lemma~\ref{lem:fb} the Fourier--Bohr coefficient $a_\chi(\mu)$ exists uniformly in translates.  Theorem~\ref{them cont eigenfunctions} then shows that the corresponding eigenfunction can be chosen continuous.
This shows that, for each $\chi$ with $\widehat{\gamma}(\{ \chi \}) \neq 0$, we can choose a continuous eigenfunction.
Since the pure-point dynamical spectrum is generated as a group by the Bragg spectrum and since the product of continuous eigenfunctions is a continuous eigenfunction, the claim follows.

\smallskip

\noindent The last claim follows from Theorem~\ref{them cont eigenfunctions}.
\end{proof}

As an immediate consequence, we get the following result.

\begin{coro}
Let $\mu$ be a Weyl almost periodic measure.
\begin{itemize}
  \item [(a)] For all $\omega \in \XX(\mu)$ and all $\chi \in \widehat{G}$, we have $\widehat{\gamma}(\{ \chi \})=\left| a_\chi(\omega) \right|^2$.
  \item [(b)] For each $\omega \in \XX(\mu)$, the dynamical spectrum of $\XX(\mu)$ is the group generated by $\{ \chi\in\widehat{G} \, :\, a_\chi( \omega) \neq 0 \}$.
  \item [(c)] For each $\chi \in \widehat{G}$ with $\widehat{\gamma}(\{ \chi \}) \neq 0$, the function $\omega \mapsto \overline{a_\chi(\omega)}$ is a continuous eigenfunction on $\XX(\mu)$.    \qed
 \end{itemize}
\end{coro}

\chapter{Some (counter)examples}\label{App:counter}
In this chapter, we consider some examples showing the strictness of
certain inclusions. We also show  that the space $\Map_{\cA}(G)$ does
not answer Lagarias' Question 6 \cite[Prob.~4.6]{LAG}. The
first example is relevant in various parts of the book.

\begin{definition}[$a$-defect]
Let $a \in (0,1)$. We define the \textit{$a$-defect of $\ZZ$} by
\[
\vL_a:= \{ -n : n \in \NN \} \cup \{ n+a : n \in \NN \}  \,.        \tag*{$\Diamond$}
\]
\end{definition}

We can prove that the $a$-defect of $\ZZ$ has the following
properties.

\begin{prop}\label{single def prop} \phantom{X}
\begin{itemize}
  \item[(a)] For each $a \in (0,1)$ and each van Hove sequence $\cA$, we have
  \[
  \delta_{\vL_a} \in \Map_\cA(G) \,.
  \]
  \item[(b)] For each $a \in (0,1)$ and $A_n=[-n,n]$, we have
  \[
  \delta_{\vL_a} \notin \Bap_\cA(G) \,.
  \]
   In particular, for all $1 \leq p < \infty$, we have
   \begin{align*}
      \delta_{\vL_a} &\notin \Bap^p_\cA(G)  \\
      \delta_{\vL_a} &\notin \Wap(G) \,.
   \end{align*}
  \item[(c)] For each $a \in (0,1), 1 \leq p < \infty$ and $A_n=[-n,n^2]$, we have
  \[
  \delta_{\vL_a} \in \Bap^p_\cA(G) \,.
  \]
  \item[(d)] For each $a \in (0,1)$ and each van Hove sequence $\cA$, the autocorrelation $\gamma$ of $\vL_a$ exists with respect to $\cA$, and \[
      \gamma=\delta_\ZZ\,.
      \]
  \item[(e)] For each $a \in (0,1)$ and each $b >0$, the Fourier--Bohr coefficients of $\vL_a$ exists with respect to $A_n=[-n, bn]$ and satisfy
\begin{displaymath}
a_\lambda(\vL_a)=
\begin{cases}
\frac{1+be^{2 \pi \im \lambda a} }{b+1} ,  & \mbox{ if } \lambda \in \ZZ \,, \\
0, & \mbox{ if } \lambda \notin \ZZ \,.
\end{cases}
\end{displaymath}
  \item[(f)] Let $a \in (0,1) \backslash \QQ$. For all $ \lambda \in \ZZ \backslash \{ 0 \}$, the Fourier--Bohr coefficients of $\delta_{\vL_a}$ don't exist with respect to the van Hove sequence \[A_n=[-n, (2+(-1)^n)n]\,.\]
  \item[(g)] Let $a \in (0,1) \backslash \QQ$, and let $A_n=[-n,n]$. For all $\lambda \in \ZZ \backslash \{ 0 \}$, the Fourier--Bohr coefficients of $\delta_{\vL_a}$ exist with respect to $\cA$ and
  \[
\widehat{\gamma}(\{ \lambda \}) \neq \left| a_\lambda ^\cA(
\delta_{\vL_a} ) \right|^2 \,.
  \]
  \item[(h)] The dynamical system $\XX( \vL_a)$ is uniquely ergodic.
  \item[(i)] Let $a \in (0,1) \backslash \QQ$. The dynamical spectrum of $(\XX(\vL_a), G)$ is $\ZZ$, while the topological dynamical spectrum is trivial.
\end{itemize}
\end{prop}
\begin{proof}
(a) If $n \in \ZZ$, then $\tau_n \delta_{\vL_a}-\delta_{\vL_a}$
is a measure with compact support. It follows immediately that,
for all van Hove sequences
$\cA$, we have
\begin{displaymath}
\|\tau_n\delta_{\vL_a}-\delta_{\vL_a}\|_{\we,\cA}=0 \,.
\end{displaymath}
The claim follows.

\medskip

\noindent (b) Fix $0<b < \min \{ a, 1-a \}$. Choose a function $\varphi \in \Cc(\RR)$ satisfying
\[
\varphi \geq 1_{ [-\frac{b}{8}, \frac{b}{8}]}
\]
and
\[
\supp(\varphi) \subseteq (-\frac{b}{4}, \frac{b}{4}) \,.
\]
Let $f$ be a Bohr almost periodic function.
Since $\delta_{\ZZ}*\varphi-f$ and $\delta_{a+\ZZ}*\varphi -f$ are Bohr almost periodic, the independence of the mean with respect to van Hove sequences gives
\[
\lim_{n\to\infty} \frac{1}{n} \int_{-n}^0 \big|(\delta_{\ZZ}*\varphi)(t)-f(t)\big|\, \dd t =M( | \delta_{\ZZ}*\varphi-f |)= \| \delta_{\ZZ}*\varphi-f \|_{\be,1,\cA}
\]
and
\[
\lim_{n\to\infty} \frac{1}{n} \int_{0}^n| (\delta_{a+\ZZ}*\varphi)(t)-f(t)|\, \dd t =M( | \delta_{a+\ZZ}*\varphi-f |)= \| \delta_{a+\ZZ}*\varphi-f \|_{\be,1,\cA}  \,.
\]
Therefore, we obtain
\begin{align*}
{}
    &\lim_{n\to\infty} \frac{1}{2n}\left( \int_{0}^n |(\delta_{a+\ZZ}
     *\varphi)(t)-f(t)|\, \dd t+ \int_{-n}^0 |(\delta_{\ZZ}
     *\varphi)(t)-f(t) |\, \dd t \right) \\
    &\phantom{=======}=\frac{1}{2} \left(\| \delta_{\ZZ}*\varphi-f \|_{\be,1,\cA} +
      \|\delta_{a+\ZZ}*\varphi-f \|_{\be,1,\cA}  \right) \\
    &\phantom{=======}\geq\frac{1}{2}  \| \delta_{\ZZ}*\varphi-
      \delta_{a+\ZZ}*\varphi \|_{\be,1,\cA} \,.
\end{align*}
Now, the choice of $\varphi$ implies that
\[
(\delta_{\vL_a}*\varphi)(x)
=\begin{cases}
(\delta_{a+\ZZ}*\varphi)(x), &  \mbox{ for all } x >1\,,  \\
(\delta_{\ZZ}*\varphi)(x), &  \mbox{ for all } x <1 \,.
\end{cases}
\]
This yields
\begin{align*}
{}
    &\| \delta_{\vL_a}*\varphi-f \|_{\be,1,\cA}  \\
    &\phantom{XXX}= \lim_{n\to\infty} \frac{1}{2n}\bigg( \int_{0}^n|(\delta_{a+\ZZ} *\varphi)(t)-f(t) |\, \dd t
    + \int_{-n}^0| (\delta_{\ZZ}*\varphi)(t) -f(t)|\, \dd t \bigg) \\
    &\phantom{XXX}\geq \frac{1}{2} \| \delta_{\ZZ}*\varphi-
      \delta_{a+\ZZ}*\varphi \|_{\be,1,\cA}  \,.
\end{align*}
Finally, the choice of the support of $\varphi$ implies that, for
each $x \in \RR$, at most one of $(\delta_{a+\ZZ}*\varphi)(x)$ and
$(\delta_{\ZZ}*\varphi)(x)$ can be non-zero. Therefore,
\[
 | \delta_{\ZZ}*\varphi- \delta_{a+\ZZ}*\varphi |=  | \delta_{\ZZ}*\varphi|+ |\delta_{a+\ZZ}*\varphi |= \delta_{\ZZ}*\varphi+ \delta_{a+\ZZ}*\varphi  \,,
\]
and a short computation gives
\[
\| \delta_{\ZZ}*\varphi-
      \delta_{a+\ZZ}*\varphi \|_{\be,1,\cA} =2\int_{\RR}
\varphi(t)\, \dd t \geq \frac{b}{2} \,.
\]
Thus, we get
\[
\| \delta_{\vL_a}*\varphi-f \|_{\be,1,\cA}\geq \frac{b}{4}  \,.
\]
This shows that, for all $f \in \sap(\RR)$, and, in particular, for all
trigonometric polynomials, one has
\[
\| \delta_{\vL_a}*\varphi-f \|_{\be,1,\cA} \geq \frac{b}{4}
\,.
\]
Therefore, $\delta_{\vL_a}*\varphi \notin \bap^1(\RR)$ and $\delta_{\vL_a} \notin \Bap_{\cA}(\RR)$ hold.
Since $ \Wap^p(\RR) \subseteq \Bap_{\cA}^p(\RR) \subset\Bap_{\cA}(\RR) $, the last claim
follows.

\medskip

\noindent (c) Let $\varphi \in \Cc(\RR)$, and let $A$ be such that $\supp(\varphi)
\subseteq [-A,A]$. Then, for all $x >A$, we have
$\delta_{\vL_a}*\varphi=\delta_{a+\ZZ}*\varphi$. Therefore,
\begin{align*}
{}
  &  \|\delta_{\vL_a}*\varphi -\delta_{a+\ZZ}*\varphi \|^p_{\be,p,\cA}  \\
  &\phantom{XXXXXX}=\limsup_{n\to\infty} \frac{1}{n^2+n} \int_{-n}^{n^2}
      |(\delta_{\vL_a}*\varphi)(t) -(\delta_{a+\ZZ}*\varphi)(t)|^p\,
      \dd t \\
    &\phantom{XXXXXX}\leq \limsup_{n\to\infty} \frac{1}{n^2+n} \int_{-n}^{A}
      |(\delta_{\vL_a}*\varphi)(t) -(\delta_{a+\ZZ}*\varphi)(t)|^p
        \, \dd t  \\
    &\phantom{XXXXXX}\phantom{===}+\limsup_{n\to\infty} \frac{1}{n^2+n} \int_{A}^{n^2}
      |(\delta_{\vL_a}*\varphi)(t) -(\delta_{a+\ZZ}*\varphi)(t)|^p
      \, \dd t \\
    &\phantom{XXXXXX}=  \limsup_{n\to\infty} \frac{1}{n^2+n} \int_{-n}^{A}
       |(\delta_{\vL_a}*\varphi)(t) -(\delta_{a+\ZZ}*\varphi)(t)|^p
       \, \dd t \\
    &\phantom{XXXXXX}\leq  \limsup_{n\to\infty} \frac{A+n}{n^2+n}\|(\delta_{\vL_a}
       -\delta_{a+\ZZ})*\varphi\|_\infty^p =0  \,.
\end{align*}
Since $\delta_{a+\ZZ}*\varphi \in \sap(\RR)$, the claim follows.

\medskip

\noindent (d) First, we show that, for each van Hove sequence, we have
\[
\lim_{n\to\infty} \frac{1}{|A_n|} \card ( \vL_a  \cap A_n)=1 \,.
\]
To do this, fix a function $\varphi \in \Cc(\RR)$ with
$\supp(\varphi) \in [0,2]$ and $(\delta_\ZZ *\varphi)(x)=1$ for all $x
\in \RR$.
Then, a trivial computation shows that
\[
(\delta_{\vL_a}*\varphi)(x)=1 \qquad \text{ for all } x \notin [-2,3] \,.
\]
Therefore,
\[1-
(\delta_{\vL_a}*\varphi)(x) \in \Cc(\RR) \,.
\]
It follows immediately that
\[
\lim_{n\to\infty} \frac{1}{|A_n|} \int_{A_n} (\delta_{\vL_a}*\varphi)(x)\, \dd
x =1 \,.
\]
Now, by a standard Fubini and van Hove computation, we get
\[
\lim_{n\to\infty} \frac{1}{|A_n|} \left( \int_{A_n}
(\delta_{\vL_a}*\varphi)(x)\, \dd x - \int_{\RR} \varphi(t)\, \dd t\,
\card ( \vL_a  \cap A_n) \right)=0 \,.
\]
As
\[
1= M(1)= M( \delta_\ZZ *\varphi) = \dens(\ZZ) \int_{\RR} \varphi(t)\,
\dd t
\]
the claim follows.

\smallskip

Next, since $\vL_a$ is a Meyer set, a standard argument, compare
\cite{bm}, shows that its autocorrelation exists with respect to
$\cA$ if and only if the limit
\[
\eta(z)= \lim_{n\to\infty} \frac{1}{|A_n|} \card ( \vL_a \cap (z+
\vL_a) \cap A_n)
\]
exists for all $z \in \ZZ$. In this case, we have
\[
\gamma=
\sum_{z \in \RR} \eta(z) \delta_z \,,
\]
compare \cite{bm}.
Now, it is easy to see that, for all $z \in \ZZ$, the sets $\vL_a$
and $z+\vL_a$ agree outside the compact set $[-|z|, |z|+1]$.
Therefore, we have
\[
\eta(z)=\lim_{n\to\infty} \frac{1}{|A_n|} \card ( \vL_a \cap (z+
\vL_a) \cap A_n)=\lim_{n\to\infty} \frac{1}{|A_n|} \card ( \vL_a  \cap
A_n)=1
\]
for all  $z \in \ZZ$.
Also, for all $z \notin \ZZ$,  it is easy to see that $\vL_a
\cap (z+ \vL_a)$ is a finite set. Hence,
\[
\eta(z)= \lim_{n\to\infty} \frac{1}{|A_n|} \card ( \vL_a \cap (z+
\vL_a) \cap A_n)  =0
\]
holds for all $z \notin \ZZ$. The claim now follows.

\medskip

\noindent (e) We compute
\begin{align*}
{}
    &\frac{1}{bn+n} \int_{-n}^{bn} e^{-2 \pi \im \lambda t}\, \dd \delta_{\vL_a}(t)  \\
    &\phantom{XXXXX}= \frac{1}{b+1} \frac{1}{n} \int_{-n}^{0} e^{-2 \pi \im \lambda t}\,
       \dd \delta_{\vL_a}(t)+ \frac{b}{b+1} \frac{1}{bn}
       \int_{0}^{bn} e^{-2 \pi \im \lambda t}\, \dd \delta_{\vL_a}(t)   \\
    &\phantom{XXXXX}= \frac{1}{b+1}\frac{1}{n} \int_{-n}^{0} e^{-2 \pi \im \lambda t} \,
       \dd \delta_{\ZZ}(t) + \frac{b}{b+1} \frac{1}{bn} \int_{0}^{bn}
        e^{-2 \pi \im \lambda t}\, \dd \delta_{a+\delta_{\ZZ}}(t) \,.
\end{align*}
Now, since $\delta_{\ZZ}$ and $\delta_{a+\ZZ} $ are
periodic measures, their Fourier--Bohr coefficients exists with
respect to any van Hove sequence, and they are independent of the
choice of the van Hove sequence, see for example \cite{LS2}. Therefore, we have
\begin{align*}
\lim_{n\to\infty} \frac{1}{bn+n} \int_{-n}^{bn} e^{-2 \pi \im \lambda t}\,
  \dd \delta_{\vL_a}(t)
    &= \frac{1}{b+1} a_{\lambda}(\delta_{\ZZ})+ \frac{b}{b+1}
      a_{\lambda}(\delta_{a+\ZZ})  \\
    &= \left( \frac{1}{b+1} +e^{2 \pi \im \lambda a}  \frac{b}{b+1} \right)
      a_{\lambda}(\delta_{\ZZ}) \,.
\end{align*}
Since $\reallywidehat{\delta_\ZZ}=\delta_\ZZ$, the claim follows.

\medskip

\noindent (f) By (d), we have
\[
\lim_{n\to\infty}  \frac{1}{|A_{2n}|} \int_{A_{2n}} e^{-2 \pi \im \lambda t}
 \, \dd \delta_{\vL_a}(t)  =  \frac{1+3e^{2 \pi \im \lambda a} }{4}
\]
and
\[
\lim_{n\to\infty}  \frac{1}{|A_{2n+1}|} \int_{A_{2n+1}} e^{-2 \pi \im \lambda t}\, \dd \delta_{\vL_a}(t)  =  \frac{1+e^{2 \pi \im \lambda a} }{2} \,.
\]
Now,
\[
\frac{1+3e^{2 \pi \im \lambda a} }{4}=\frac{1+e^{2 \pi \im \lambda
a} }{2} \iff e^{2 \pi \im \lambda a}=1 \iff
\lambda a \in \ZZ\,.
\]
Since $a \notin \QQ$,
we have $\frac{1+3e^{2 \pi \im \lambda a} }{4} \neq \frac{1+e^{2 \pi
i \lambda a} }{2}$ for all $\lambda \in \ZZ$. Hence, one has
\[
\lim_{n\to\infty}  \frac{1}{|A_{2n}|} \int_{A_{2n}} e^{-2 \pi \im \lambda t}\, \dd
\delta_{\vL_a}(t) \neq \lim_{n\to\infty}  \frac{1}{|A_{2n+1}|}
\int_{A_{2n+1}} e^{-2 \pi \im \lambda t}\, \dd \delta_{\vL_a}(t) \,,
\]
showing that $\big(\frac{1}{|A_{n}|} \int_{A_{n}} e^{-2 \pi \im \lambda t}\,
\dd \delta_{\vL_a}(t)\big)$ is not convergent.

\medskip

\noindent (g) By (d), we have
\[
\lim_{n\to\infty}  \frac{1}{|A_{n}|} \int_{A_{n}} e^{-2 \pi \im \lambda t}\, \dd
\delta_{\vL_a}(t)  =  \frac{1+e^{2 \pi \im \lambda a} }{2}
\]
for all $\lambda \in \ZZ$. Hence, the Fourier--Bohr coefficients exist. Furthermore, we
have
\[
|\frac{1+e^{2 \pi \im \lambda a} }{2}|^2=1
\]
if and only if
\[
e^{2 \pi \im \lambda a}=1 \,
\] which again, by the irrationality of $a$,
implies that
\[
\widehat{\gamma}(\{ \lambda \}) \neq \left| a_\lambda ^\cA(
\delta_{\vL_a} ) \right|^2 \qquad \mbox{ for all } \lambda \in \ZZ
\backslash \{ 0 \} \,.
\]

\smallskip

\noindent (h) Note that
\[
\XX(\mu)= \{ \tau_t  \mu : t \in \RR \} \sqcup \{
\delta_{t+\ZZ} : t \in \RR/\ZZ  \}\,.
\]
Set
\[\TT:=\{  \delta_{t+\ZZ} : t \in \RR/\ZZ  \} \,.
\]
This is a compact
Abelian group, and the action of $\RR$ is simply addition modulo 1:
$\tau_s \delta_{t+\ZZ}=\delta_{t+s+\ZZ}$. Also, set
\[
\Omega:=\{ \tau_t  \mu : t \in \RR \}\,.
\]
We show that any ergodic measure is equal to the probability Haar
measure on $\TT$. This proves unique ergodicity.
Let $\mathfrak{m}$ be a $\RR$-invariant ergodic measure on $\XX(\mu)$.

Define $\Phi : \Cc(\RR) \to C(\XX(\mu))$ via
\[
\Phi(f)(\omega)=
\begin{cases}
f(t) & \mbox{ if } \omega = \tau_t \mu \in \Omega\\
0 &\mbox{ otherwise } \,
\end{cases}.
\]
 It is immediate that $\Phi(f)$ is
indeed continuous. Define,
\[
\eta(f)=\mathfrak{m}(\Phi(f))\qquad \mbox{ for all } f \in \Cc(\RR) \,.
\]
It is easy to see that $\eta$ is linear, and we have
\[
\left| \eta(f) \right| = \left| \mathfrak{m}(\Phi(f)) \right| \leq \|
\Phi(f) \|_\infty = \|f \|_\infty \,,
\]
for all $f \in\Cc(\RR)$.

Therefore, by Riesz' representation theorem, $\eta$ is a finite
measure on $\RR$. Also, it is easy to see from the definition that,
for all $t \in \RR$ and all $f \in \Cc(\RR)$, we have
\[
\Phi(\tau_t
f)=\tau_t \Phi(f)\,.
\]
Therefore, since $\mathfrak{m}$ is $\RR$-invariant, so
is $\eta$.
This implies that $\eta$ is a multiple of the Lebesgue measure on $\RR$ and hence, since it is finite,
\[
\eta=0 \,.
\]

Next, for each $n\in\NN$, choose an $f_n \in \Cc(\RR)$ with
$1_{[-n,n]} \leq f_n \leq 1_{[-n-1,n+1]}$, and let $\psi_n:= \Phi
(f_n)$. Then, $(\psi_n)$ is an increasing sequence of functions in
$C(\XX(\mu))$ which converges pointwise to the characteristic
function $1_\Omega$ of $\Omega$.
Let $g \in C(\XX)$, and define $h_n:\RR \to \CC$ via $h_n(t)=
\psi_n(\tau_t  \mu) g(\tau_t  \mu)$. Then, $h_n \in \Cc(\RR)$ and
$\Phi(h_n)= \psi_n g$.
Now, the monotone convergence theorem implies
\begin{align*}
\int_{\XX(\mu)} g(\omega)\, \dd \mathfrak{m}(\omega)
    &= \int_{\TT}g(\omega)\, \dd \mathfrak{m}(\omega)+\int_{\Omega}g(\omega)\, \dd \mathfrak{m}(\omega)
       \\
    &=\int_{\TT}g(\omega)\, \dd \mathfrak{m}(\omega)+\lim_{n\to\infty} \int_{\Omega}\,
      \psi_n g(\omega)\, \dd \mathfrak{m}(\omega)\\
    &=\int_{\TT}g(\omega)\, \dd \mathfrak{m}(\omega)+\lim_{n\to\infty} \int_{\Omega}
      \Phi(h_n)(\omega)\, \dd \mathfrak{m}(\omega)\\
    &=\int_{\TT}g(\omega)\, \dd \mathfrak{m}(\omega)+\lim_{n\to\infty}  \mathfrak{m}(\varphi(h_n))\\
    &=\int_{\TT}g(\omega)\, \dd \mathfrak{m}(\omega)+\lim_{n\to\infty} \eta(h_n)
       =\int_{\TT}g(\omega)\, \dd \mathfrak{m}(\omega) \,.
\end{align*}
Consequently, $\mathfrak{m}$ is supported on $\TT$. Hence, it is an
$\RR$-invariant probability measure on $\TT$. It follows that $\mathfrak{m}$ is the
probability Haar measure on $\TT$.

\medskip

\noindent (i) The first part follows from (c).
Now, we show that if $f_\lambda \not\equiv 0$ is a continuous
eigenfunction, then $\lambda =0$. We know that the measurable
spectrum is $\ZZ$, so $\lambda \in \ZZ$.
Let $f_\lambda$ be a continuous eigenfunction, and let $c=
f_\lambda(\delta_{\vL_a})$.
Now, for $n \in \NN$, we have
\[
\lim_{n \to \infty} \tau_n \delta_{\vL_a} = \delta_{a+\ZZ}
\]
in the local topology. Hence, since $f_\lambda$ is continuous,
\[
e^{2 \pi \im \lambda a} f_\lambda(\delta_{\ZZ})= f_\lambda (
\delta_{a+\ZZ} ) = \lim_{n\to\infty}  f_\lambda( \tau_n \delta_{\vL_a}) =
\lim_{n\to\infty} e^{2 \pi \im \lambda n} f_\lambda( \delta_{\vL_a})
= c \,.
\]
In the same way, for $n \in \NN$, we have
\[
\lim_{n \to \infty} \tau_{-n} \delta_{\vL_a} = \delta_{\ZZ} \,.
\]
Hence, since $f_\lambda$ is continuous, one has
\[
f_\lambda(\delta_{\ZZ})= \lim_{n\to\infty}  f_\lambda( \tau_{-n}
\delta_{\vL_a}) = c \,.
\]
This yields
\[
c= e^{2 \pi \im \lambda a} f_\lambda(\delta_{\ZZ})= e^{2 \pi \im \lambda
a} c \,.
\]
Hence, $c=0$ or $e^{2 \pi \im \lambda a}=1$.
In the first case, we get $f_\lambda \equiv 0$, which is not
possible, while the second case gives $\lambda=0$.
\end{proof}

\begin{remark} \phantom{X}
\begin{itemize}
  \item[(a)] In Proposition~\ref{single def prop}, (d) also follows from (h).
  \item[(b)] In Proposition~\ref{single def prop}, (h) can alternatively be proved by Corollary \ref{cor:char-ue}.  \exend
\end{itemize}
\end{remark}

Next, we discuss an example of a mean almost periodic measure with
respect to some van Hove sequence $\cA$ such that the
autocorrelation does not exists with respect to $\cA$. Since the
computations are straightforward and similar to the ones done for
the proof of Proposition~\ref{single def prop}, we skip them.

\begin{example}
Let $\vL:= \{ n , -2n : n \in \NN \}$. Then, $\vL$ is mean
almost periodic with respect to $A_n=[-n, (2+(-1)^n)n]$, but its
autocorrelation does not exist with respect to this van Hove
sequence. This shows that mean almost periodicity does not suffice for a convincing diffraction theory (and does not provide an answer to Question 6 of Lagarias article \cite{LAG}).   \exend
\end{example}

We complete the chapter by providing a simple example of a measure which is Besicovitch but not Weyl almost periodic. Other examples of such measures are provided by weak model sets of maximal density.

\begin{example}\label{remark Bap but not Wap}
Let
\[
\mu=\sum_{n=1}^\infty \sum_{k=1}^n \delta_{2^n+k} \,.
\]
Then, $\mu$
is Besicovitch almost periodic with respect to $A_n=[-n,n]$, but not
Weyl almost periodic.

Indeed, let $\varphi \in \Cc(\RR)$ be
arbitrary. We show that
\[
\|\mu*\varphi\|_{\be,1,\cA}=0 \,,
\] which yields
Besicovitch almost periodicity.

Let $A$ be such that $\supp(\varphi)
\subseteq [-A,A]$. Since
\[
\mu( [0,
2^m])=1+2+\ldots+m-1=\frac{m(m-1)}{2}\,,
\] we have
\begin{align*}
 \frac{1}{2n} \int_{[-n,n]} \left| \mu*\varphi (t) \right| \dd t & =\frac{1}{2n} \int_{\RR} 1_{[-n,n]}(t) \left| \int_{\RR} \varphi(t-s) \dd \mu(s) \right| \dd t  \\
&\leq \frac{1}{2n} \int_{\RR} 1_{[-n,n]}(t)  \int_{\RR}\left|\varphi(t-s) \right| \dd \mu(s) \dd t  \\
&= \frac{1}{2n} \int_{\RR}   \int_{\RR}1_{[-n,n]}(t) \left|\varphi(t-s) \right|  \dd t \dd \mu(s) \,.
\end{align*}

Now, a simple computation shows that for all $t,s$ we have
\begin{displaymath}
 1_{[-n,n]}(t) \left|\varphi(t-s) \right| \leq 1_{[-n-A,n+A]}(s) \left|\varphi(t-s) \right| \,.
\end{displaymath}
Therefore,
\begin{align*}
 \frac{1}{2n} \int_{[-n,n]} \left| (\mu*\varphi) (t) \right| \dd t &\leq \frac{1}{2n} \int_{\RR}   \int_{\RR}1_{[-n-A,n+A]}(s) \left|\varphi(t-s) \right|  \dd t\, \dd \mu(s) \\
 &=\frac{1}{2n} \int_{\RR}   \int_{\RR}1_{[-n-A,n+A]}(s) \left|\varphi(z) \right|  \dd z\, \dd \mu(s) \\
 &=\| \varphi \|_1 \frac{1}{2n}  \mu([-n-A,n+A])\leq \| \varphi \|_1 \frac{1}{2n}\mu([0, 2^{m+1})\\
 &= \| \varphi \|_1 \frac{\frac{m(m+1)}{2}}{2n} \leq  \| \varphi \|_1 \frac{(\log_2(n+A))^2}{2n} \,.
\end{align*}

This shows that
\[
\frac{1}{2n} \int_{[-n,n]} \left| (\mu*\varphi) (t) \right|\, \dd t
\leq  \| \varphi \|_1 \frac{(\log_2(n+A))^2}{2n} \,.
\]
From here, Besicovitch almost periodicity follows. One can show via a
similar approximation that $\mu$ is not Weyl almost periodic.

This can also be seen via Theorem~\ref{thm:wap}: $\XX(\mu)$ contains both the zero measure $0$, and $\delta_{\ZZ}$. This implies that $\XX(\mu)$ contains the hulls of these two measures, each one being a distinct compact group. Therefore, $\mu$ is
not uniquely ergodic, and hence cannot be Weyl almost periodic.   \exend
\end{example}

\chapter*{Acknowledgments}
The authors would like to thank Michael Baake for some discussions
and comments. NS would like to thank Christoph Richard for many inspiring discussions on almost periodicity. 

The work was supported by NSERC with grants
2020-00038 and 2024-04853(NS),  by DFG with grant 415818660(TS) and the authors are
grateful for the support. 
\appendix

\appendix

\chapter{Cut-and-project schemes}\label{appendix:cp}
In this Appendix, we give a brief review of cut-and-project schemes
 (CPS).
 Cut-and-project schemes have been introduced by Meyer \cite{Meyer} in order to produce models, which Meyer called model sets,  that almost solve an infinite system of linear equations modulo $\ZZ$. They have been rediscovered by de Bruijn \cite{dBr} as a way to re-derive the vertices of the Penrose tiling \cite{Pen}.
 Kramer and Neri \cite{KN} expanded on the de Bruijn ideas to construct a 3-D tiling with icosahedral symmetry. This work came just before the discovery of quasicrystals by Shechtman \cite{She}. Since then model sets have been prime  candidates in the modeling of  quasicrystals. The underlying approaches were using various restrictions and were sometimes  inconsistent with each other. A consistent and unified was then developed and presented by Lagarias  \cite{LAG1} and Moody \cite{MOO,Moody2}.  For a review on the early history of the cur-and-project formalism we recommend \cite{Kra} and \cite{GR}.

 Regular model sets (see below for definition), are Delone sets with pure point diffraction. This is a central result in the theory of Aperiodic Order, and was first proved by Hof \cite{Hof1,Hof2} for model sets in fully Euclidean cut-and-project schemes with polygonal window. The result was proven for general model sets by Schlottmann \cite{Martin2} via dynamical systems. Baake and Moody \cite{bm} provided an alternate proof via the (strong) almost periodicity of the autocorrelation measure, and Richard and Strungaru \cite{CRS} showed that the pure-point nature of regular model sets follows from the Poisson Summation Formula. Meyer \cite{Mey2} showed that regular model sets are g-a-p, and hence Weyl almost periodic measures, and their pure point nature is then an immediate consequence of Theorem~\ref{theorem-uniform-phase-problem}. Recently, some of those results have been extended to weak model sets with topological windows \cite{BHS,KR}, and to weak model sets with pre-compact Borel windows \cite{KRS,S}, and these results can be proved via Besicovitch almost periodicity (compare Section~\ref{weak ms}).

 For a detailed review of cut and project schemes and their properties, we refer to  \cite{MOO,CR,LR,TAO,CRS,NS11,LRS}, just to name a few.

A triple $\cp$ is called a \textit{cut and project scheme} (CPS)\index{cut~and~project~scheme} if $G$
and $H$ are LCA groups and $\cL$ is a \textit{ lattice\/}\index{lattice} in $G\times H$, i.e., a cocompact discrete subgroup, such that
\begin{itemize}
  \item{}the restriction of canonical projection $\pi^G : G\times H \longrightarrow G$ to $\cL$ is one to one;
  \item{}$\pi^H(\cL)$ is dense in $H$.
\end{itemize}
Let $L:=\pi^G (\cL)$. We can then define the \textit{star mapping} $(\cdot)^\star\!: L \longrightarrow H$ as
follows:
For each $x\in L$, the injectivity of the first projection implies that there exists a unique $y\in H$ such that $(x,y)\in\cL$. We denote this $y$ by $x^\star$.
Under the mapping
\[
L \ni x \to x^\star \in H \,,
\]
we have
\begin{displaymath}
\cL= \{ (x, x^\star) : x \in L\}
\end{displaymath}
and
%
%
%
%
\begin{center}
\begin{tikzcd}
G & \arrow[swap]{l}{\pi^G} G\times H \arrow{r}{\pi^H}& H \\
L \arrow[hookrightarrow]{u}  \arrow[swap, bend right=10]{rr}{\star}
& \arrow[hookrightarrow]{u} \arrow[swap, bend right=5]{l}{1-1} \cL
\arrow[bend left=5]{r}{\pi^H} & \arrow[hookrightarrow,
swap]{u}{\text{dense}}L^\star
\end{tikzcd}
\end{center}
Given a cut and project scheme, we can associate to any $W \subseteq H$,
called the \textit{window}, the set
\begin{displaymath}
\oplam(W):= \{x \in L : x^{\star} \in W\}\,.
\end{displaymath}
If $W$ is relatively compact, then $\oplam(W)$ is called a \textit{weak
model set}\index{model~set!weak~model~set}. If, additionally,
$W^{\circ}\ne \varnothing$, the set $\oplam(W)$ is called a \textit{model
set}\index{model~set!model~set}. Any weak model set is uniformly discrete, and any model set is a Delone set.
If, in addition, the model set $\oplam(W)$ satisfies $|\partial W|=0$, it is called a \textit{regular model set}\index{model~set!regular~model~set}.

Given a CPS $\cp$, for each function $h :H \to \CC$, we can define a formal sum via
\begin{displaymath}
\omega_h:= \sum_{x \in L} h(x^\star) \delta_x \,.
\end{displaymath}
If $h$ is compactly supported and bounded, then $\omega_h$ is a measure. The same holds under various decaying conditions of $h$ \cite{CR,LR,CRS,NS11}.
When $h=1_W$ is the characteristic function of a window $W$, we have
\[
\omega_h =\delta_{\smoplam(W)} \,.
\]
We can define a new CPS $(\widehat{G}, \widehat{H}, \cL^0)$, called the \textit{dual lattice}\index{lattice!dual~lattice} where $\cL^0$ is the annihilator, or the dual lattice, of $\cL$ in $\widehat{G} \times \widehat{H} \simeq \widehat{G \times H}$, that is
\[
\cL^0:=\{( \chi,\psi) \in \widehat{G} \times\widehat{H} :  \chi(x) \psi(y)=1 \ \text{for all } (x,y) \in \cL \} \,.
\]
For details that this is a CPS see \cite{MOO,Moody}.

\smallskip

We now list some of the essential properties of such combs \cite{CR,LR,CRS,TAO,TAO2,NS11}.

\begin{theorem}\cite{CRS}\label{thm-dens}
Let $\cp$ be a CPS and $h \in \Cc(H)$. Then,
\begin{itemize}
\item[(a)] $\omega_h \in \SAP(G)$ and
\[
M(\omega_h)=\dens(\cL)\, \int_H h(t)\, \dd t \,.
\]
\item[(b)] $\omega_h$ is Fourier transformable if and only if $\widehat{h} \in \mathcal{L}^1(\widehat{H})$. Moreover, in this case  $\omega_{\widecheck{h}}$ is a measure in the dual CPS and
  \begin{displaymath}
\widehat{\omega_{h}}=\dens(\cL) \,\omega_{\widecheck{h}} \,.        \tag*{$\qed$}
  \end{displaymath}
\end{itemize}
\end{theorem}

Next, we introduce the concept of \textit{weak model sets of maximal density}. First, let us recall the following result.

\begin{prop}\cite{HR,NS11}\label{prop-dens-bounds}
Let $\cp$ be a CPS, and let $W \subseteq H$ be a pre-compact set. Then, for each van Hove sequence $\cA$, we have
\[
    \dens (\cL)\, |W^{\circ}| \, \leq \,
 \liminf_{m\to\infty} \frac{\delta_{\smoplam(W)}(A_m)}{|A_m|}  \, \leq \,
   \limsup_{m\to\infty} \frac{\delta_{\smoplam(W)}(A_m)}{|A_m|} \, \leq \,
       \dens (\cL)\, |\overline{W}| \,.
\]
\end{prop}

We can now introduce the following definition.

\begin{definition}\label{wms def} \cite{BHS,KR} Given a CPS $\cp$, a van Hove sequence $\cA$ and a compact set $W \subseteq H$, we say that the weak model set $\oplam(W)$ has \textit{maximal density with respect to $\cA$}\index{model~set!maximal~density~weak~model~set} if
\[
  \lim_{m\to\infty} \frac{\delta_{\smoplam(W)}(A_m)}{|A_m|} \, = \, \dens (\cL)\, |W| \,.   \tag*{$\Diamond$}
\]
\end{definition}

\chapter{Semi-measures and their Fourier transform}\label{appendix:semi}
In this section, we collect the basic results we need about
semi-measures; see Definition~\ref{semi-measure}.
Let us stat with the following consequence of the definition.

\begin{lemma} \label{lem semi measure K_2}
Let $\vartheta$ be a Fourier transformable semi-measure.
\begin{itemize}
\item [(a)] For all $\psi \in K_2(G)$, we have $\widecheck{\psi} \in \mathcal{L}^1(|\widehat{\vartheta}|)$ and
\begin{displaymath}
\vartheta(\psi)=\widehat{\vartheta}(\widecheck{\psi}) \,.
\end{displaymath}
\item [(b)] For all $\psi \in K_2(G)$, we have
\begin{displaymath}
(\vartheta*\psi)(t) = \int_{\widehat{G}} \chi(t)\, \widehat{\psi}(\chi)\, \dd \widehat{\vartheta}(\chi)= \reallywidecheck{ \widehat{\psi} \widehat{\vartheta}}(t) \,.
\end{displaymath}
\end{itemize}
\end{lemma}
\begin{proof}
(a) By the polarisation identity \cite[p. 244]{MoSt}, the claim is true for $\psi= \varphi*\phi$ with $\varphi, \phi \in \Cc(G)$. (a) follows now by linearity.

\medskip

\noindent(b) Since $K_2(G)$ is closed under reflection and translation, we get
  \begin{displaymath}
(\vartheta*\psi)(t) = \vartheta(\tau_t  \varphi^\dagger)=
\widehat{\vartheta}( \reallywidecheck{\tau_t  \varphi^\dagger})=
\reallywidecheck{ \widehat{\psi} \widehat{\vartheta}}(t) \,,
  \end{displaymath}
for all $t\in G$.
\end{proof}

Next, let us recall the following definition \cite{NS18}.

\begin{definition}
A measure $\mu$ on $G$ is called \textit{weakly admissible}\index{weakly~admissible~measure}, if, for all
$\varphi \in K_2(\widehat{G})$, we have $\widehat{\varphi} \in
\mathcal{L}^1(|\mu|)$.       \exend
\end{definition}

We start with the following result, which emulates the standard proof that positive definite measures are Fourier transformable \cite[Thm.~4.5]{BF}, \cite[Thm.~4.11.5]{MoSt}.

\begin{lemma}\label{spectral lemma}
Let $\{ \sigma_\varphi \}_{\varphi \in \Cc(G)}$ be a family of
finite measures on $\widehat{G}$ which satisfy the compatibility
condition
\begin{equation}\label{comp}
\left| \widehat{\varphi} \right|^2 \sigma_\psi =  \big| \widehat{\psi} \big|^2 \sigma_\varphi  \qquad \mbox{ for all }\varphi, \psi \in \Cc(G) \,.
\end{equation}
Then, there exists a weakly admissible measure $\sigma$ on $\widehat{G}$ such that, for all $\varphi \in \Cc(G)$, we have
\begin{displaymath}
\sigma_\varphi = \left| \widehat{\varphi} \right|^2 \sigma \,.
\end{displaymath}
\end{lemma}
\begin{proof}
We follow closely the proof of \cite[Thm.~4.5]{BF}.
For each $f \in \Cc(\widehat{G})$, pick some $\varphi \in \Cc(G)$ such that $\widehat{\varphi}$ is not vanishing on $\supp(f)$. Such a function always exists by \cite[Prop.~2.4]{BF}, \cite[Cor.~4.9.12]{MoSt}. Define
\[
\frac{f}{\big| \widehat{\varphi} \big|^2}(\chi)=
\left\{
\begin{array}{cc}
\frac{f(\chi)}{\big| \widehat{\varphi}(\chi) \big|^2}(\chi) & \mbox{ if } f(\chi) \neq 0 \,,\\
 0 & \mbox{ otherwise}\,.
\end{array}
\right.
\]
It is immediate that $\frac{f}{\big| \widehat{\varphi} \big|^2} \in\Cc(\widehat{G})$.
Additionally, define
\begin{displaymath}
\sigma(f):=\sigma_{\varphi} \left( \frac{f}{\big| \widehat{\varphi} \big|^2} \right) \,.
\end{displaymath}
The compatibly condition Eq.~\eqref{comp} ensures that our definition doesn't depend on the choice of $\varphi$. It is easy to see that $\sigma : \Cc(G) \to \CC$ is linear.

We show next that $\sigma$ is continuous with respect to the inductive topology. To do this, fix some compact set $K$. Fix some $\varphi \in \Cc(G)$ such that $\widehat{\varphi} \geq 1_K$. Such a function exists again by  \cite[Prop.~2.4]{BF}, \cite[Cor.~4.9.12]{MoSt}.
Then, for all $f \in \Cc(\widehat{G})$ with $\supp(f) \subset K$, we have $\left| \frac{f}{\left| \widehat{\varphi} \right|^2} \right| \leq \| f \|_\infty 1_K$ and hence
\begin{displaymath}
\left| \sigma(f) \right|\leq \left| \sigma_{\varphi} \right|(K)  \cdot \| f\ \|_\infty \,.
\end{displaymath}
Since $\sigma_\varphi$ is a (finite) measure, the claim follows.

Next, we show that $\sigma_\varphi = \left| \widehat{\varphi} \right|^2 \sigma$ for all $\varphi \in \Cc(G)$.
Let $\varphi \in \Cc(G)$ be arbitrary. Pick some $f \in \Cc(\widehat{G})$, and choose some $\psi \in \Cc(G)$, such that $\widehat{\psi}$ is not vanishing on $\supp(f)$. Then,
\begin{align*}
(\left| \widehat{\varphi} \right|^2\sigma)(f)
    &=\sigma (\left| \widehat{\varphi} \right|^2f )
    =\sigma_{\psi} \left( \frac{\left| \widehat{\varphi} \right|^2f}{\big|
      \widehat{\psi} \big|^2} \right)
    = \bigl(\left| \widehat{\varphi} \right|^2 \sigma_{\psi}\bigr)
      \left( \frac{f}{\big| \widehat{\psi} \big|^2} \right)\\
    &= \bigl( \big| \widehat{\psi} \big|^2 \sigma_\varphi\bigr)
      \left( \frac{f}{\big| \widehat{\psi} \big|^2} \right)
    =\sigma_\varphi(f) \,.
\end{align*}
This shows that
\begin{displaymath}
\sigma_\varphi = \left| \widehat{\varphi} \right|^2 \sigma \,.
\end{displaymath}

Finally, since $\sigma_\varphi$ is finite, so is $\left| \widehat{\varphi} \right|^2 \sigma$, which gives the weak admissibility of $\sigma$.
\end{proof}

We can now prove the following simple result.

\begin{prop}\label{FT of semi-measures}
Let $\mu$ be a measure on $\widehat{G}$. Then, there exists a
semi-measure $\vartheta$ on $G$ such that $\widehat{\vartheta}=\mu$
if and only if $\mu$ is weakly admissible.
\end{prop}
\begin{proof}
$\Longrightarrow$: This follows from the definition of the Fourier transformability.

\smallskip

\noindent $\Longleftarrow$: Since $\mu$ is weakly admissible, we have $|\widecheck{\psi}| \in \mathcal{L}^1(|\mu|)$ for all $\psi \in K_2(G)$.
Then, we can define a semi-measure $\vartheta$ via
\begin{displaymath}
\vartheta(\psi) := \mu( \widecheck{\psi}) \quad \mbox{ for all } \psi \in K_2(G) \,.
\end{displaymath}
\end{proof}

We can now give an example of a semi-measure which is not a measure.

\begin{example}
On $G=\RR$, the Lebesgue measure is weakly admissible, and hence, so is its restriction to $[0, \infty)$ \cite[Lem. 3.2(ii)]{NS18}.
By Proposition~\ref{FT of semi-measures},
\begin{equation}\label{eq varme}
\vartheta(f) := \int_0^\infty \widecheck{f}(s)\, \dd s \,, \qquad  \text{ for all } f \in K_2(G) \,,
\end{equation}
is well defined and a semi-measure on $\RR$.
However, $\vartheta$ is not a measure.
Indeed, assume by contradiction that it is. Then, by Eq.~\eqref{eq varme}, $\vartheta$ is Fourier transformable as a measure and its Fourier transform as a measure is $\nu:= \lambda|_{[0, \infty)]}$.  \cite[Thm.~11.1]{ARMA} then implies that $\nu \in \WAP(\RR)$ and hence $\nu$ has a mean which does not depend on the choice of the van Hove sequence.
This is a contradiction, as the mean of $\nu$ with respect to $A_n=[0,n]$ is $1$, while the mean of $\nu$ with respect to $[-n,0]$ is $0$.    \exend
\end{example}

We introduce the concept of positive definiteness for a semi-measure, similar to a measure.

\begin{definition}
A semi-measure $\vartheta$ is called \textit{positive definite}\index{positive~definite!semi-measure}, if for all
$\varphi \in \Cc(G)$, we have $\vartheta(\varphi*\tilde{\varphi}) \geq 0$.   \exend
\end{definition}

\begin{remark}\label{rem:pd implies FT}
Similarly to  \cite[Thm.~4.5]{BF}, \cite[Thm.~4.11.5]{MoSt} one can prove that a semi-measure $\vartheta$ is Fourier transformable with positive Fourier transform if and only if $\vartheta$ is positive definite, satisfies $(\vartheta*\varphi)*\psi=(\vartheta*\psi)*\varphi$ for all $\varphi,\psi \in K_2(G)$, and, for all $\varphi \in K_2(G)$, the function $\vartheta*\varphi$ is continuous at $t=0$.         \exend
\end{remark}

\smallskip

We complete this Appendix by discussing when a semi-measure is a measure.

\begin{lemma}\label{posi semi is measure} Let $\vartheta$ be a semi-measure. Then,
$\vartheta$ is a measure if and only if, for all $K \subset G$, there exists a constant $C_K>0$ such that, for all $\psi \in K_2(G)$ with $\supp(\psi) \subseteq K$, we have
\begin{displaymath}
|\vartheta(\psi)| \leq C_K\,  \| \psi \|_\infty \,.
\end{displaymath}
\end{lemma}
\begin{proof}
$\Longrightarrow$: This follows from the definition of measures.

\medskip

\noindent $\Longleftarrow$: Fix some $K \subseteq G$. The set
\[
C(G:K):= \{ \varphi  \in \Cc(G): \supp(\varphi) \subset K \}
\]
is a Banach space with respect to $\| \cdot \|_\infty$.
Now, the given relation says that $\vartheta$ is bounded on the dense subspace $C(G:K)\cap K_2(G)$ and hence, has a unique extension to a continuous mapping $\mu_{K} : C(G:K) \to \CC$.
Next, if $K', K''$ are arbitrary compacts with non-empty intersection $K=K' \cap K''$, it is easy to see that $\mu_{K'}|_{ C(G:K)}=\mu_{K''}|_{ C(G:K)}$. Therefore, we can define $\mu : \Cc(G) \to \CC$ via
\begin{displaymath}
\mu(\varphi)= \mu_{K}(\varphi) \,,
\end{displaymath}
where $K$ is any compact set containing the $\supp(\varphi)$. It is easy to see that $\mu$ is a measure.
\end{proof}

\chapter{Averaging along arbitrary van Hove sequences}

Here, we cover an important lemma which allows us relate the means of a bounded function with respect to different van Hove sequences. We denote the open ball around $z\in \CC$ with radius $r>0$ by $U_r (z)$.

\begin{prop}\label{prop:mother-of-uniform-van-Hove-results}
Let $h: G\longrightarrow \CC$ be a bounded measurable function.
Let $A$ be an open, relatively compact subset of $G$, and assume that there exist $r>0$ and $z\in \CC$ with
\[
\frac{1}{|A|} \int_{A+s} h(t)\, \dd t \in U_r (z)
\]
for all  $s\in G$. Then, for any van Hove sequence $\cB$ and any $R>r$, there exists a natural number $N$ with
\[
\frac{1}{|B_n|} \int_{B_n +v} h(t)\, \dd t \in U_R (z)
\]
for all $v\in G$ and $n\geq N$.
\end{prop}
\begin{proof}
A short computation shows
\[
\left| \int_{B_n +v} h(t + u) \dd t - \int_{B_n + v} h(t)\, \dd t \right|  \leq \|h\|_\infty\, |\partial^{A\cup (-A)} B_n|
\]
for all $u\in A$  and $v\in G$.
In particular, we have
\[
\left|\frac{1}{|A|}\int_A \left(\int_{B_n + v} h(t+u)\, \dd t\right) \dd u - \int_{B_n + v} h(t)\, \dd t\right|\leq \|h\|_\infty\, |\partial^{A\cup (-A)} B_n|
\]
for all $v\in G$. Now, from the assumption, we find
\[
\frac{1}{|A|}\int_A \left(\frac{1}{|B_n|}\int_{B_n + v} h(t+u)\, \dd t\right) \dd u = \frac{1}{|B_n|}\int_{B_n+v} \left(\frac{1}{|A|}\int_A h(t+u)\, \dd u \right) \dd t \in U_r (z)
\]
for all $v\in G$ and $n\in\NN$.  Taking these statements together, we
infer that
\[
\frac{1}{|B_n|} \int_{B_n + v} h(t)\, \dd t \in U_{r + \delta(n)} (z)
\]
with
\[
\delta(n) = \frac{1}{|B_n|}\, \|h\|_\infty\, |\partial^{A\cup (-A)} B_n|
\]
for all $v\in G$ and $n\in \NN$. This easily gives the statement.
\end{proof}

\backmatter

\bibliographystyle{amsalpha}



\printnomenclature
\printindex

\end{document}